\documentclass[leqno,english]{smfbook}
\usepackage[T1]{fontenc}
\usepackage{amsmath}
\usepackage{amsfonts}
\usepackage{amssymb}
\usepackage{smfthm}
\usepackage[dvips]{pstricks}
\usepackage{mathrsfs}
\usepackage{graphicx}
\usepackage{calc}
\usepackage{float}

\author{Titus Lupu}
\address{Laboratoire de Probabilités, Statistique et Modélisation, Sorbonne Université, 75005 Paris (France)}
\email{titus.lupu@upmc.fr}
\title{Poisson ensembles of loops of one-dimensional diffusions}
\alttitle{Ensembles poissoniens de boucles des diffusions unidimensionnelles}

\makeatletter
\newlength{\temp@wc@width}
\newlength{\temp@wc@height}
\newcommand{\widecheck}[1]{%
  \setlength{\temp@wc@width}{\widthof{$#1$}}%
  \setlength{\temp@wc@height}{\heightof{$#1$}}%
  #1\hspace{-\temp@wc@width}%
  \raisebox{\temp@wc@height+1pt}[\heightof{$\widehat{#1}$}]%
     {\rotatebox[origin=c]{180}{\vbox to 0pt{\hbox{$\widehat{\hphantom{#1}}$}}}}%
}
\makeatother

\numberwithin{equation}{section}
\numberwithin{figure}{chapter}
\NumberTheoremsIn{chapter}

\begin{document}

\newenvironment{property}{\begin{enonce}{Property}}{\end{enonce}}
\newenvironment{proc}{\begin{enonce}{Procedure}}{\end{enonce}}
\newenvironment{vervaat}{\begin{enonce*}{Theorem(Vervaat)}}{\end{enonce*}}
\newenvironment{dynkin}{\begin{enonce*}{Theorem(Dynkin's Isomorphism)}}{\end{enonce*}}

\frontmatter

\begin{abstract}
There is a natural measure on loops (time-parametrized trajectories that in the end return to the origin), which one can associate to a wide class of Markov processes. The Poisson ensembles of Markov loops are Poisson point processes with intensity proportional to these measures. In wide generality, these Poisson ensembles of Markov loops are related, at intensity parameter $1/2$, to the Gaussian free field,
and at intensity parameter 1, to the loops done by a Markovian sample path. Here, we study the specific case when the Markov process is a one-dimensional diffusion. We give a detailed description of the corresponding measure on loops. 
Further, we describe the Poisson point processes of loops, their occupation fields, and explain how to sample these Poisson ensembles of loops out of diffusion sample path perturbed at their successive minima. Finally, we introduce a couple of interwoven determinantal point processes on the line, which is a dual through Wilson's algorithm of Poisson ensembles of loops, and study the properties of these determinantal point processes. 
\end{abstract}

\begin{altabstract}
Il y a une mesure naturelle sur les boucles (trajectoires paramétrées par le temps, qui à la fin retournent à leur origine) qu'on peut associer à une large classe de processus de Markov. Les ensembles poissoniens de boucles markoviennes sont des processus ponctuels de Poisson d'intensité proportionnelle à ces mesures. Dans une grande généralité, ces ensembles poissoniens de boucles markoviennes sont reliés, au paramètre d'intensité $1/2$, au champ libre gaussien, 
et au paramètre d'intensité 1, aux boucles crées par une trajectoire markovienne. Ici nous étudions le cas spécifique où le processus de Markov est une diffusion unidimensionnelle. Nous donnons une description détaillée de la mesure sur le boucles correspondante. 
Ensuite, nous décrivons les processus ponctuels de Poisson des boucles, leurs champs d'occupation et expliquons comment séquencer ces ensembles poissoniens de boucles à partir de trajectoires de diffusions perturbées à leur minima successifs. Enfin, nous introduisons un couple de processus ponctuels déterminantaux sur la droite, entrelacés, qui est un dual, à travers l'algorithme de Wilson, de l'ensemble poissonien de boucles, et étudions les propriétés de ces processus ponctuels déterminantaux.
\end{altabstract}



\altkeywords{Ensembles poissoniens de boucles markoviennes, soupes des boucles, processus de diffusion unidimensionnel, transformation de Vervaat, processus de branchement avec immigration à  état continu, champ libre gaussien, isomorphisme de Dynkin, algorithme de Wilson, arbre couvrant uniforme, processus ponctuels déterminantaux
}

\maketitle
\tableofcontents

\mainmatter

\chapter{Introduction}
\label{Ch1}

\section{Measure on loops: history}

There is a measure on loops one can naturally associate to a wide class of Markov processes. In general, it is expressed as
\begin{equation}
\label{EqPatternLoop}
\mu(d\gamma):=\int_{t>0}\int_{x}\mathbb{P}_{x,x}^{t}(d\gamma)p_{t}(x,x)m(dx) \dfrac{dt}{t},
\end{equation}
where $p_{t}(x,y)$ are the transition densities with respect to a $\sigma$-finite measure $m(dy)$, and 
$\mathbb{P}_{x,y}^{t}$ are the bridge probability measures. Here, a loop $\gamma$ is a path parametrized by continuous time, which at the end returns to the origin ($\gamma(0)=\gamma(T(\gamma))$). 

Such a measure first appeared, to my knowledge, in the work of Symanzik \cite{Symanzik65Scalar,Symanzik66Scalar,Symanzik1969QFT}. There, studying the $\phi^{4}$ model in dimension $d$, he expressed the moments, and some other functionals, of such fields in terms of multiple intergals over measures on massive Brownian loops and massive Brownian paths, corresponding to a $d$-dimensional Brownian motion killed at independent exponential time. He called this loop expansion, and saw it as Euclidean analogue of Feynman diagrams in Minkowski space QFT. Later, in the work of Brydges, Frölich and Spencer \cite{BFS82Loop}, an analogous measure appeared, but for discrete time random walk loops.

Then, an analogously defined measure in the setting of the two-dimensional Brownian motion appeared in the work of Lawler and Werner
\cite{LawlerWerner2004ConformalLoopSoup}. They also considered the Poisson point processes of intensity proportional to this measure on loops, which they called "loop-soups". The main motivation for studying this object was that the obtained loops, up to reroooting and time-reparametrization, were unvariant in law by conformal transformations. This properties were used by Sheffield and Werner in \cite{SheffieldWerner2012CLE} to construct
the Conformal Loop Ensemble (CLE) as outer boundaries of clusters in a Brownian loop-soup. 

Lawler and Trujillo-Ferreras initiated in \cite{LawlerFerreras2007RWLoopSoup} the study of discrete time random walk loop-soups. 
See also a recent survey by Lawler \cite{Lawler2018Topics}.
The measure they used was actually the same that appeared in Brydges, Frölich and Spencer \cite{BFS82Loop}.
Le Jan considered loops parametrized by continuous time, associated to Markov jump process on an electrical network, with intensity following rhe pattern \eqref{EqPatternLoop}, rather than the discrete time random walk loops \cite{LeJan2011Loops}. He also called the object Poisson ensemble of Markov loops rather than loop-soup. We will adopt Le Jan's terminology. If one takes the discrete skeletton of Le Jan's loops, one gets the random walk loops studied in \cite{LawlerFerreras2007RWLoopSoup}.

Taking loops parametrized by continuous time allowed Le Jan to consider the occupation field of Poisson ensembles of Markov loops, that is to say the total time spent by loops on each vertex. Le Jan identified two properties universally satisfied by these Poisson ensembles of Markov loops. The first one, is that at intensity parameter $1/2$, the occupation field has same law as half the square of a Gaussian free field. This is an extension of Dynkin's isomorphism \cite{Dynkin1984Isomorphism,Dynkin1984PolynomOccupField,Dynkin1984IsomorphismPresentation}, and is closely related to Symanzik's identities \cite{Symanzik1969QFT}. The second universal property is that, at intensity parameter 1, the loops erased during the loop-erasure algorithm applied to a Markovian sample path, are part of a Poisson ensembles of Markov loops. This property has been previously observed in the two-dimensional Brownian setting. See \cite{LawlerWerner2004ConformalLoopSoup}, Conjecture $1$, and \cite{LawlerSchrammWerner2003ConformalRestr}, Theorem $7.3$. See also \cite{DaisukeSapozh2017LERW3D} for recent developments in dimension three. On an electrical network, Wilson's algorithm 
\cite{Wilson1996UST} samples a uniform spanning tree by running successive loop erased random walks. Le Jan observed that the loops erased during this algorithm give a Poisson ensembles of Markov loops of intensity parameter 1.

Since then, the definition of the measure of loops was extended to a wide class of Markov processes, including some not having transition densities
\cite{LeJanMarcusRosen2012Loops,FitzsimmonsRosen2012LoopsIsomorphism,FzLJRosen2015Loops}. Actually, one rather considers measures on unrooted loops, seen as parametrized by a circle, where the cut between the start and the endpoint is not identified. What makes such measures on unrooted loops natural is actually the covariance by time change by an inverse of a Continuous Additive Functional (CAF). The time change by an inverse of a CAF transforms a Markov process into an other one, possibly on a smaller state space. The measure on unrooted loops of the time changed process is the pushforward by time-change of the measure on loops of the initial Markov process. This covariance is no longer true at the level of rooted loops.
In dimension two for Brownian loops, the covariance by time-change implies the invariance by conformal transformations of the range
\cite{LawlerWerner2004ConformalLoopSoup}.

\section{Measure on loops: present work and recent developments}

These notes are devoted to the study of loop measures, Poisson ensembles of loops and of an analogue of Wilson's algorithm in the setting of one-dimensional diffusions. This is a work done during my Ph.D. under the direction of Le Jan at Université Paris-Sud, Orsay
\cite{Lupu2015Thesis}. The aim of these notes is to give a complete and self-contained presentation of the topic in the continuous one-dimensional setting, displaying both results true more generally and those specific to our setting, and making the connection with the accumulated knowledge on the Brownian motion and diffusions, in particular their path decomposition. The study will also extend to loop measures associated to some infinitesimal generators containing a creation of mass term, under a negativity condition, as those enter in the expression of some exponential moments of occupation fields of Poisson ensembles of loops.

The aspects covered in these notes are the following: We will study the measures on loops, their covariance by change of scale and of speed of the underlying diffusions, covariance by adding a killing or creation of mass term, and invariance by conjugation of the generators of diffusions by positive continuous functions (generalization of h-transforms). Since the loops we consider are unrooted, a natural operation to do is to root the loops at their minima. In the Brownian case, this is related to Vervaat's bridge-to-excursion transformation \cite{Vervaat1979BridgeToExc}.
For more general diffusion, this leads to a disintegrated analogue of Vervaat's transformation, where the law of bridge conditioned on its minimum, and after exchange of pre- and postminimum part, is identified as absolutely continuous with respect the law of the excursion 
(Proposition \ref{Ch3Sec6: PropConditionnedVervaat}). We will further study the Poisson point processes of loops. It turns out that in the Brownian case, the Poisson ensembles of loops under the form of excursions (not bridges) already appeared in the litterature, in the Lévy-Hincin decompostion of square Bessel processes \cite{LeGallYor86CarreBessel}, and in the decompositions of a family of perturbed Brownian motions
\cite{LeGallYor86CarreBessel,Perman96}. We will investigate the occupation field Poisson ensembles of loops, which is a sum of local times, and identify it to (in general non-homogeneous) continuous state branching processes with immigration. We will revisit the relation to the Gaussian free field for the intensity parameter $1/2$. Further, we will explain how to sample Poisson ensembles of loops by slicing sample paths of diffusions perturbed at their successive minima. At intensity parameter 1, one uses for this unperturbed diffusions, which is the duality between loops and loop-erasure.
Then, we study the analogue of Wilson's algorithm applied to one-dimensional diffusions with a killing measure. It returns on one hand a Poisson ensemble of loops of intensity parameter 1, and on the other hand a couple of interwoven determinantal point processes on the line, which can be seen as an analogue of a random spanning tree. One point process gives the points connected to the root/cemetery, and the other point process corresponds to "deleted edges". The duality between loops and uniform spanning trees/determinantal point processes given by Wilson's algorithm deserves further investigation. Le Jan in \cite{LeJan2011Loops}, Section 8.4, suggests that it might be related to the supersymmetry, as determinatal point processes are associated to fermionic fields, and occupation fields of Poisson ensembles of loops to bosonic fields. Finally, we study some monotone coupling properties for our determinantal point processes as the killing measure increases.

One elementary, but very important observation done in this work is that in continuous one dimensional setting, the clusters of loops are exactly delimited by the zeroes of the occupation field (Proposition \ref{Ch4Sec2 PropClusters}). This observation led the author of these notes to study the loops associated to diffusions on metric graphs \cite{Lupu2014LoopsGFF}. A metric graph is obtained by replacing the edges in an electrical network by continuous line segments. A diffusion on it behaves like a one-dimensional diffusion inside the edges, and performs excursions in all directions once it reaches a vertex \cite{BaxterChacon1984DiffusionsNetworks}. Loops of metric graph diffusions have both a non-trivial geometry and continuous local times. Out of this one obtains that on metric graphs, the sign components of a Gaussian free field are exactly the clusters of a Poisson ensemble of loops of intensity parameter $1/2$. In this Poisson ensemble of metric graph loops, the restriction of loops to vertices gives a random walk loop-soup, and the loops that do not visit any vertex form Poisson ensembles of loops of one-dimensional diffusions inside each edge.
This isomorphism with the signed Gaussian free field led to a proof of convergence on two-dimensional lattices of clusters of rescaled random walk loop-soups to clusters of a Brownian loop-soup \cite{Lupu2015ConvCLE}. Further, in a collaboration with Aru and Sepúlveda
\cite{ALS2017FPS,ALS2}, we show that the continuum Gaussian free field in dimension two, which is a random generalized function, actually lives on clusters of a two-dimensional Brownian loop-soup. For this, we use the approximation by metric graphs of two-dimensional lattices.

The author thanks Professor Yves Le Jan (Université Paris-Sud, Orsay) for fruitful discussions and its helpful advice in relation with this work.
The author also thanks Professor Jim Pitman (University of California, Berkeley) for his comments on a previous version of these notes and his bibliographical suggestions.

\section{Results and layout}

The layout of these notes is the following: In Chapter \ref{Ch2} we recall some facts on one-dimensional diffusions and set the important notations. In Section \ref{Ch2Sec1} we recall the properties of the solutions to the second order ODE
\begin{displaymath}
\dfrac{d^{2}u}{dx^{2}} + u \nu = 0,
\end{displaymath}
where $\nu$ is a signed measure. In Section \ref{Ch2Sec2} we review the theory of one-dimensional diffusions, their generators, Green's functions, transition densities, excursions, bridge measures, etc. 
In Section \ref{Ch2Sec3} we further consider "generators" with creation of mass term and characterize a class of such operators which up to a conjugation are equivalent to generator of diffusions
(Proposition \ref{Ch2Sec3: PropPositivity}). 

The Chapter \ref{Ch3} is devoted to the properties of the measure on loops of one-dimensional diffusions. In Section \ref{Ch3Sec1} we introduce the functional spaces of continuous rooted and unrooted loops. In Section
\ref{Ch3Sec2} we define the measures $\mu^{x,y}$ on paths joining two points $x$ and $y$,  and study their invariance and covariance properties by different transformations such as time-reversal, restriction to a subinterval, increase of killing measure, change of scale, change of time and conjugation. In Section \ref{Ch3Sec3} we introduce the measure $\mu$
on rooted loops, and $\mu^{\ast}$ on unrooted loops, and again study the invariance and covariance properties. In particular, one distinctive property of $\mu^{\ast}$ is the covariance by time change
(Corollary \ref{Ch3Sec3: CorTimeChange}), which justifies the study of unrooted loops rather then rooted ones. In Section \ref{Ch3Sec4} we deal with the multiple local times of loops, and show that these appear as densities of concatenation of independent paths
$\mu^{x_{1},x_{2}}\lhd\dots\lhd\mu^{x_{n-1},x_{n}}\lhd\mu^{x_{n},x_{1}}$
relative to $\mu^{\ast}$ (Proposition \ref{Ch3Sec4: PropMultLocTimeDens}).
We also show that two diffusions have the same measure on loops if and only if they are conjugate 
(Proposition \ref{Ch3Sec4: PropConversehTransfInv}).
In Section \ref{Ch3Sec5}, we make a connection between the Brownian measure on loops and the Lévy-Itô measure on Brownian excursions using the Vervaat's bridge-to-excursion transformation. The idea is to root loops at their minima.
This in turn leads us in Section \ref{Ch3Sec6} 
to a conditioned version of Vervaat's transformation that holds for any one-dimensional diffusion process 
(Proposition \ref{Ch3Sec6: PropConditionnedVervaat}). It relates a bridge conditioned on its minimum and the excursion measure above this minimum.
In Section \ref{Ch3Sec7} we show that by restricting continuous loops to a discrete subsets we get the natural measure on random walk loops
(Proposition \ref{Ch3Sec7: PropRestrictionLoops}). In Section
\ref{Ch3Sec8} we consider loop measure in case of creation of mass terms
and see what properties generalize to this case.

In Chapter \ref{Ch4} we poissonize the measure on loops and study the occupation fields of Poisson ensembles of Markov loops. In Section \ref{Ch4Sec1} we recall the properties of the continuous state branching process with immigration.
In Section \ref{Ch4Sec2} we show that these processes, parametrized by the space variable, appear as occupation fields of Poisson ensembles (Poisson point processes) of loop of one-dimensional diffusions
(Proposition \ref{Ch4Sec2: PropCBPI}). We also observe that the clusters of loops are exactly the excursions of the occupation field above zero
(Proposition \ref{Ch4Sec2: PropOccupFieldZeroes}).
In Section \ref{Ch4Sec3} we consider the particular case of intensity parameter 1/2. We identify occupation fields as squares of Gaussian free fields (Property \ref{Ch4Sec3: PropertyIso}), give a signed version of this isomorphism (Proposition \ref{Prop PolIso}), and show how it is possible to derive particular versions of Dynkin's isomorphism using the isomorphism with loops and Palm's identity for Poisson point processes.
 
In Chapter \ref{Ch5} we root each loop in its minimum and obtain this way a collection of positive excursions. Then, we order these excursions in the decreasing sense of their minima and glue them together.
In Section \ref{Ch5Sec1} we recall some deterministic facts on paths obtained by this procedure.
In Section \ref{Ch5Sec2} we apply this procedure to random loops in our Poisson point processes. We obtain this way continuous paths which can be described as diffusions perturbed at their successive minima.
(Propositions \ref{Ch5Sec2: PropBMContour} and
\ref{Ch5Sec2: PropGeneralContour}). We study the particular case of intensity parameter 1 in Section \ref{Ch5Sec3}. For this particular value of intensity, the loops can be recovered from unperturbed diffusion sample
paths, by slicing them (Propositions \ref{Ch5Sec3: PropDirichletSlice} and
\ref{Ch5Sec3: PropTwoDecomp}).

In Chapter \ref{Ch6} we apply an extension of Wilson's loop-erasure algorithm (used to sample uniform spanning trees on electrical networks out of random walks) to transient one-dimensional diffusions, and obtain a couple of interwoven determinantal point processes on $\mathbb{R}$, 
$(\mathcal{Y}_{\infty},\mathcal{Z}_{\infty})$, dual to Poisson ensembles of loops of intensity parameter 1. In Section \ref{Ch6Sec1} we describe our construction, which involves an arbitrary choice of a countable everywhere dense family of points on the line, which are the starting points for diffusions. In Section \ref{Ch6Sec2} we relate the paths erased during the algorithm to the Poisson ensemble of loops of intensity parameter $1$ (Proposition \ref{Ch6Sec2: PropErasedLoops}).
In Section \ref{Ch6Sec3} we consider our algorithm in case of Brownian motion on $\mathbb{R}$ with a killing measure $\kappa$. We show that the law of $(\mathcal{Y}_{\infty},\mathcal{Z}_{\infty})$ does not depend on the arbitrary choice of the countable everywhere dense family of starting points (Proposition \ref{Ch6Sec3: PropInvPerm}). We identify 
$\mathcal{Y}_{\infty}$ and $\mathcal{Z}_{\infty}$ separately as determinantal point processes 
(Propositions \ref{Ch3Sec3: PropDeterminantalY} and
\ref{Sec3PropDiamondsDet}), and then identify the joint law of 
$(\mathcal{Y}_{\infty},\mathcal{Z}_{\infty})$
(Proposition \ref{Ch6Sec3: PropJointLaw}). We also give a criterion for
$\mathcal{Y}_{\infty}$ and $\mathcal{Z}_{\infty}$ to be finite or bounded on one side (Proposition \ref{Ch6Sec3: PropFiniteness}).
In Section \ref{Ch6Sec4} we point out how our identities transform if one takes a general transient diffusion instead of a Brownian motion with killing measure (Proposition \ref{Ch6Sec4: PropJointLawGeneral}).

In Chapter \ref{Ch7} we prove some monotone coupling properties for the determinantal point processes introduced in Chapter \ref{Ch6}. These monotone couplings do not follow from the construction by loop-erasure, but rather from the form of determinantal kernels.
First, in Section \ref{Ch7Sec1} we describe the conditional laws for
$(\mathcal{Y}_{\infty},\mathcal{Z}_{\infty})$, when conditioning on 
$\mathcal{Z}_{\infty}$ not charging an interval 
(Proposition \ref{Ch7Sec1: PropEdgeContr}) 
or $\mathcal{Y}_{\infty}$ not charging an interval 
(Corollary \ref{Ch7Sec1: CorAbsenceRoot}). In Section \ref{Ch7Sec2} we show that by increasing the killing measure in the diffusion, one can increase $\mathcal{Z}_{\infty}$ and make $\mathcal{Y}_{\infty}$ satisfy some constraints (Propositions \ref{Ch7Sec2: PropCouplingPoisson}
and Proposition \ref{Ch7Sec2: PropStrongCoup}). The monotone couplings of determinantal point processes we obtain are explicit.

\section{A list of commonly used notations}

\textbf{Chapter \ref{Ch2}:}

\begin{description}
\item[$I$] $\mathbb{R}$ or an open subinterval of $\mathbb{R}$.
\item[$\nu$] a signed $\sigma$-finite measure on $I$.
\item[$\kappa$] a positive Radon measure on $I$, considered as a killing measure.
\item[$\operatorname{Supp}$] the support of a measure.
\item[$m$] density of a speed measure on $I$.
\item[$w$] density of a scale measure on $I$.
\item[$L$] either the infinitesimal generator of a, possibly killed, diffusion on $I$, or a more general second order differential operator containing a creation of mass term.
\item[$u$] in general, a function on $I$ with $\frac{d^{2}u}{dx^{2}}$ a signed measure, often also assumed to be positive.
\item[$u_{\uparrow}$] a positive non-decreasing solution to $Lu=0$.
\item[$u_{\downarrow}$] a positive non-increasing solution to $Lu=0$.
\item[$W(u_{1},u_{2})$] Wronskian of $u_{1}$ and $u_{2}$.
\item[$X_{t}$] diffusion of generator $L$.
\item[$\zeta$] $+\infty$ or killing time for $X_{t}$.
\item[$\ell^{x}_{t}(X)$] local times of $X_{t}$.
\item[$G_{L}(x,y)$ or $G(x,y)$] Green's function of $L$.
\item[$p_{t}(x,y)$] transition densities.
\item[$\mathbb{P}^{t}_{x,y}$] bridge probability measure, from $x$ to $y$, of duration $t$.
\item[$\eta_{\rm exc}^{>x}$] measure on excursions above $x$.
\item[$\eta_{\rm exc}^{<x}$] measure on excursions below $x$.
\item[$\gamma$] a generic continuous path or loop.
\item[$\operatorname{Scale}_{A}$] change of scale operator (by $A$) acting on paths.
\item[$\operatorname{Scale}_{A}^{\rm gen}$] change of scale operator (by $A$) acting on generators.
\item[$\operatorname{Speed}_{V}$] change of speed operator (by $V$) acting on paths.
\item[$L_{\vert \widetilde{I}}$] operator $L$ restricted to functions supported on an open subinterval $\widetilde{I}$ of $I$.
\item[$\operatorname{Conj}(u,L)$] conjugation of the differential operator $L$ by the function $u$, $u^{-1}Lu$.
\item[$\mathfrak{D}^{+}$] operators that are not conjugates of generators of diffusions.
\item[$\mathfrak{D}^{0,-}$] operators that are conjugates of generators of diffusions.
\item[$\mathfrak{D}^{0}$] operators that are conjugates of generators of recurrent diffusions.
\item[$\mathfrak{D}^{-}$] operators that are conjugates of generators of transient diffusions.
\end{description}

\textbf{Chapter \ref{Ch3}:}

\begin{description}
\item[$T(\gamma)$] total life-time of a path or loop $\gamma$.
\item[$\operatorname{shift}_{v}$] a cyclic translation of parametrization of a loop.
\item[$\mathfrak{L}$] space of rooted loops.
\item[$\mathfrak{L}^{\ast}$] space of unrooted loops.
\item[$\pi$] projection form $\mathfrak{L}$ to $\mathfrak{L}^{\ast}$.
\item[$d_{\rm paths}$] distance on paths.
\item[$d_{\mathfrak{L}^{\ast}}$] distance on unrooted loops.
\item[$\mathcal{B}_{\mathfrak{L}}$] Borel $\sigma$-algebra on $\mathfrak{L}$.
\item[$\mathcal{B}_{\mathfrak{L}^{\ast}}$] Borel $\sigma$-algebra on $\mathfrak{L}^{\ast}$.
\item[$\tau^{y}_{l}$] stopping time at local time at $y$ equal to $l$.
\item[$T_{a}$] first passage time at level $a$.
\item[$^{\wedge}$] given $\mu_{1}$ a measure on paths, $\mu_{1}^{\wedge}$ is the image of $\mu_{1}$ by time reversal.
\item[$\lhd$] given $\mu_{1}$ and $\mu_{2}$ two measures on paths, $\mu_{1}\lhd\mu_{2}$ is the image measure of 
$\mu_{1}\otimes\mu_{2}$ by the operation of concatenation of paths.
\item[$\mu^{x,y}_{L}$ or $\mu^{x,y}$] a measure on finite life-time paths, starting in $x$ and ending in $y$, associated to the generator $L$.
\item[$\mu_{L}$ or $\mu$] measure on rooted loops associated to the generator $L$. 
\item[$\mu^{\ast}_{L}$ or $\mu^{\ast}$] measure on unrooted loops associated to the generator $L$.
\item[$\ell^{x_{1},x_{2},\dots,x_{n}}$] multiple local times.
\item[$\ell^{\ast x_{1},x_{2},\dots,x_{n}}$] circular multiple local times.
\item[$\mathcal{V}(\gamma)$] transformation that exchanges the pre- and the postminimum part of a bridge $\gamma$ (Vervaat's transformation).
\item[$\mu_{BM}$] measure on unrooted loops associated to the standard Brownian motion on $\mathbb{R}$.
\item[$\mu^{\ast}_{BM}$] measure on unrooted loops associated to the standard Brownian motion on $\mathbb{R}$.
\item[$\mathbb{P}_{BM,0,0}^{t}$] law of standard Brownian bridge from $0$ to $0$ in time $t$.
\item[$\eta_{t,BM}^{>0}$] probability measure on Brownian excursions above $0$ in time $t$.
\item[$B_{t}$] Brownian motion.
\item[$\rho_{t}$] Bessel 3 process.
\item[$\eta_{t}^{>a}$] probability measure on excursions above $a$ in time $t$ for general diffusions.
\item[$p_{t}^{(a \times)}(x,y)$] transition densities of a diffusion on $I\cap (a,+\infty)$, killed in $a$, that is of
generator $L_{\vert I\cap (a,+\infty)}$.
\item[$\mathbb{P}_{x,y}^{(a \times),t}$] bridge probability measures associated to $L_{\vert I\cap (a,+\infty)}$.
\item[$\rho_{t}^{+,a}$] diffusion on $I\cap (a,+\infty)$ obtained by conditioning the original diffusion not to hit $a$.
\item[$\mathbb{P}^{+,a}_{a}$] law of $\rho_{t}^{+,a}$ starting from $a$.
\end{description}


\textbf{Chapter \ref{Ch4}:}

\begin{description}
\item[$d\mathbb{B}_{x}$] spatial Gaussian white noise.
\item[$\alpha$] a positive constant, intensity parameter.
\item[$\mathcal{L}_{\alpha,L}$ or $\mathcal{L}_{\alpha}$] Poisson ensemble of loops of intensity
$\alpha\mu^{\ast}_{L}$.
\item[$\widehat{\mathcal{L}}_{\alpha,L}^{x}$ or $\widehat{\mathcal{L}}_{\alpha}^{x}$]
occupation field of $\mathcal{L}_{\alpha,L}$.
\item[$\#$] cardinal.
\item[$\mathfrak{G}_{s,\tilde{\nu}}$] operator
$f\mapsto\int_{I} G_{L+s\tilde{\nu}}(x,y)f(y)\tilde{\nu}(dy)$.
\item[$\vert\mathfrak{G}_{s,\tilde{\nu}}\vert$] operator
$f\mapsto\int_{I} G_{L+s\tilde{\nu}}(x,y)f(y)\vert\tilde{\nu}\vert(dy)$.
\item[$\phi_{x}$] Gaussian free field on $I$ of covariance function $G(x,y)$.
\item[$\mathfrak{M}(\mathcal{E})$] space of locally finite measure on an abstarct Polish space $\mathcal{E}$.
\item[$\Phi$] abstract Poisson random measure.
\item[$\mathfrak{P}_{n}$] set of partitions of $\lbrace 1,\dots,n\rbrace$.
\end{description}

\textbf{Chapter \ref{Ch5}:}

\begin{description}
\item[$\mathtt{e}_{q}$] a generic excursion above $0$.
\item[$\mathcal{Q}$]  countable everywhere dense subset of $(-\infty, x_{0})$, indexing the excursions.
\item[$\xi(t)$] function obtained by gluing together the excursions $(q+\mathtt{e}_{q})_{q\in\mathcal{Q}}$ ordered in decreasing sense of their minima.
\item[$\theta(t)$] $\inf_{[0,t]}\xi$.
\item[$\mathcal{L}_{\alpha, BM}$] Poisson ensemble of Brownian loops of intensity $\alpha\mu^{\ast}_{BM}$.
\item[$\xi_{\alpha, BM}^{(x_{0})}(t)$] path starting in $x_{0}$, obtained by gluing together some loops in $\mathcal{L}_{\alpha, BM}$ under the form of excursions above the minimum.
\item[$\theta_{\alpha, BM}^{(x_{0})}(t)$] $\inf_{[0,t]}\xi_{\alpha, BM}^{(x_{0})}$.
\item[$\Xi_{\alpha, BM}^{(x_{0})}(t)$] $\big(\xi_{\alpha, BM}^{(x_{0})}(t),
\theta_{\alpha, BM}^{(x_{0})}(t)\big)$.
\item[$T^{+}(\mathbb{R}^{2})$] $\lbrace(x,a)\in\mathbb{R}^{2}\vert x\geq a\rbrace$.
\item[$\operatorname{Diag}(\mathbb{R}^{2})$] $\lbrace(x,x)\vert x\in\mathbb{R}\rbrace$.
\item[$\mathcal{D}_{\alpha, BM}$] a functional space of functions on $T^{+}(\mathbb{R}^{2})$, which is a core for the generator of
$\Xi_{\alpha, BM}^{(x_{0})}(t)$.
\item[$\xi_{\alpha, L}^{(x_{0})}(t)$] path starting in $x_{0}$, obtained by gluing together some loops in $\mathcal{L}_{\alpha, L}$ under the form of excursions above the minimum.
\item[$\theta_{\alpha, L}^{(x_{0})}(t)$] $\inf_{[0,t]}\xi_{\alpha, L}^{(x_{0})}$.
\item[$\Xi_{\alpha, L}^{(x_{0})}(t)$] $\big(\xi_{\alpha, L}^{(x_{0})}(t),
\theta_{\alpha, L}^{(x_{0})}(t)\big)$.
\item[$T^{+}(I^{2})$] $\lbrace(x,a)\in I^{2}\vert x\geq a\rbrace$.
\item[$\widehat{T^{+}(I^{2})}$] closure of $T^{+}(I^{2})$ in $(\inf I, \sup I]^{2}$.
\item[$\operatorname{Diag}(I^{2})$] $\lbrace(x,x)\vert x\in I\rbrace$.
\item[$\widehat{\mathcal{D}}_{\alpha, \tilde{L}}$] a functional space of functions on $T^{+}(I^{2})$, which is a core for the generator of
$\Xi_{\alpha, L}^{(x_{0})}(t)$.
\item[$\mathscr{L}^{1}((X_{t})_{0\leq t<\zeta})$]  first way to slice a diffusion sample path $(X_{t})_{0\leq t<\zeta}$ into loops of
$\mathcal{L}_{1, L}$, corresponding to the loop-erasure procedure.
\item[$\mathscr{L}^{2}((X_{t})_{0\leq t<\zeta})$] second way to slice a diffusion sample path $(X_{t})_{0\leq t<\zeta}$ into loops of
$\mathcal{L}_{1, L}$, corresponding to the loop-erasure procedure applied to the time-reversed path.
\end{description}


\textbf{Chapter \ref{Ch6}:}

\begin{description}
\item[$\mathcal{T}$] a tree on a graph, often a spanning tree.
\item[$C(e)$] positive weight, conductance of an edge $e$.
\item[$(\mathcal{Y}_{n},\mathcal{J}_{n})$] result obtained after $n$ first steps of Wilson's algorithm applied to one-dimensional diffusions.
$\mathcal{Y}_{n}$ is a finite set of points, where the killing by $\kappa$ occurred. $\mathcal{J}_{n}$ is a finite set of disjoint line segments,
corresponding to the branches of the "tree" discovered so far.
\item[$(\mathcal{Y}_{\infty},\mathcal{Z}_{\infty})$] final result of Wilson's algorithm applied to one-dimensional diffusions.
$\mathcal{Y}_{\infty}$ and $\mathcal{Z}_{\infty}$ are two interwoven point processes, such that between any tow points in $\mathcal{Y}_{\infty}$,
there is a point of $\mathcal{Z}_{\infty}$, and vice-versa. Points in $\mathcal{Y}_{\infty}$ are those connected to the root/cemetery, where the killing by $\kappa$ occurred. Points in $\mathcal{Z}_{\infty}$ are analogues of missing edges in a spanning tree, separating different branches connected to the root/cemetery.
\item[$Q_{n}$] cardinal of $\mathcal{Y}_{n}$.
\item[$Y_{n,1},Y_{n,2},\dots$,$Y_{n,Q(n)}$]
points in $\mathcal{Y}_{n}$ ordered in the increasing sense.
\item[${[p^{-}_{n,1},p^{+}_{n,1}],[p^{-}_{n,2},p^{+}_{n,2}],\dots,[p^{-}_{n,Q_{n}},p^{+}_{n,Q_{n}}]}$]
intervals in $\mathcal{J}_{n}$ ordered in the increasing sense.
\item[$E^{-}_{n}$] subset of indices $q\in\lbrace 1,\dots,Q_{n}\rbrace$ for which $Y_{n,q}=p^{-}_{n,q}$.
\item[$E^{+}_{n}$] subset of indices $q\in\lbrace 1,\dots,Q_{n}\rbrace$ for which $Y_{n,q}=p^{+}_{n,q}$.
\item[$E^{-,+}_{n}$] subset of indices $q\in\lbrace 1,\dots,Q_{n}\rbrace$ for which $p^{-}_{n,q}<Y_{n,q}<p^{+}_{n,q}$.
\item[$\mathfrak{G}_{\kappa}$] operator $f\mapsto\int_{\mathbb{R}}G(x,y)f(y)\kappa(dy)$.
\item[$G^{(x\times)}(y,z)$] $G(y,z)-\dfrac{G(x,y)G(x,z)}{G(x,x)}$.
\item[$\mathcal{K}(y,z)$] determinantal kernel for $\mathcal{Z}_{\infty}$.
\item[$\mathcal{K}^{(x\triangleright)}(y,z)$] 
$\mathcal{K}(y,z)-\dfrac{\mathcal{K}(x,y)\mathcal{K}(x,z)}{\mathcal{K}(x,x)}$.
\item[$M_{n}(\mathcal{Y}_{\infty},\mathcal{Z}_{\infty})(dy_{0},dz_{1},dy_{1},\dots dz_{n},dy_{n})$]
infinitesimal probability for $y_{0},y_{1},\dots y_{n}$ being $n+1$ consecutive points in $\mathcal{Y}_{\infty}$, and $z_{1},\dots z_{n}$ being the $n$ points in $\mathcal{Z}_{\infty}$ separating them.
\item[$\mathscr{C}_{n}(a_{0},b_{0},\tilde{a}_{1},\tilde{b}_{1},a_{1},b_{1},\dots,\tilde{a}_{n},\tilde{b}_{n},a_{n},b_{n})$]
an event corresponding to $\mathcal{Y}_{\infty}$ having points in $[a_{i},b_{i}]$-s, $\mathcal{Z}_{\infty}$ having points in
$[\tilde{a}_{i},\tilde{b}_{i}]$-s, and some additional conditions.
\end{description}

\textbf{Chapter \ref{Ch7}:}

\begin{description}
\item[$\Upsilon$] a uniform spanning tree on a weighted graph (electrical network).
\item[$\mathcal{Y}_{\infty}^{(x_{0}\times)},\mathcal{Z}_{\infty}^{(x_{0}\times)},
\mathcal{Y}_{\infty}^{(\times x_{0})},\mathcal{Z}_{\infty}^{(\times x_{0})},
\mathcal{Y}_{\infty}^{(x_{0}\triangleright)},\mathcal{Z}_{\infty}^{(x_{0}\triangleright)},
\mathcal{Y}_{\infty}^{(\triangleleft x_{0})},\mathcal{Z}_{\infty}^{(\triangleleft x_{0})}$]
conditioned versions of $(\mathcal{Y}_{\infty},\mathcal{Z}_{\infty})$, above and below $x_{0}$.
\item[$G^{(x_{0}\times)},\mathcal{K}^{(x_{0}\times)},
G^{(\times x_{0})},\mathcal{K}^{(\times x_{0})},
G^{(x_{0}\triangleright)},\mathcal{K}^{(x_{0}\triangleright)},
G^{(\triangleleft x_{0})},\mathcal{K}^{(\triangleleft x_{0})}$]
associated determinantal kernels.
\end{description}

\chapter{Preliminaries on generators and semi-groups}
\label{Ch2}

\section{A second order ODE}
\label{Ch2Sec1}

In this chapter we will introduce the one-dimensional diffusions we will consider throughout this work (Section \ref{Ch2Sec2}). In Section
\ref{Ch2Sec3} we will extend the framework to "generators" containing a mass-creation term. In Section \ref{Ch2Sec1} we will prove or recall some facts on functions harmonic for these generators.

Let $I$ be an open interval of $\mathbb{R}$ and $\nu$ a signed measure on $I$. By signed measure we mean that the total variation $\vert\nu\vert$ is a positive Radon measure, but not necessarily finite, and $\nu(dx) = \epsilon(x)\vert\nu\vert(dx)$ where $\epsilon$ takes values in $\lbrace \pm 1\rbrace$. We look for the solutions of the linear second order differential equation on $I$:
\begin{equation}
\label{Ch2Sec1: Eq2ndOrderODE}
\dfrac{d^{2}u}{dx^{2}} + u \nu = 0.
\end{equation}
Given a solution $u$ of \eqref{Ch2Sec1: Eq2ndOrderODE}, we will write $\frac{du}{dx}(x^{+})$ and $\frac{du}{dx}(x^{-})$ for the right-hand side respectively left-hand side derivative of $u$ at $x$. The two are related by
\begin{displaymath}
\dfrac{du}{dx}(x^{+}) - \dfrac{du}{dx}(x^{-}) = - u(x)\nu(\lbrace x \rbrace).
\end{displaymath} 

Using a standard fixed point argument one can show that \eqref{Ch2Sec1: Eq2ndOrderODE} satisfies a Cauchy-Lipschitz principle: if $x_{0}\in I$ and $u_{0}, v_{0} \in\mathbb{R}$, there is a unique solution $u$ of \eqref{Ch2Sec1: Eq2ndOrderODE}, continuous on $I$, satisfying $u(x_{0})= u_{0}$ and $\frac{du}{dx}(x_{0}^{+})=v_{0}$. Let $x_{1}\in I\cap(x_{0},+\infty)$. A continuous function $u$ on $[x_{0},x_{1}]$ is solution of 
\eqref{Ch2Sec1: Eq2ndOrderODE} with previous initial conditions at $x_{0}$ if and only if it is a fixed point of the affine operator $\mathfrak{I}$ on $\mathcal{C}([x_{0},x_{1}])$ defined as
\begin{displaymath}
(\mathfrak{I}u)(x):= u_{0}+(x-x_{0})v_{0}-\int_{(x_{0},x]}(x-y)u(y)\nu(dy).
\end{displaymath}
The Lipschitz norm of $\mathfrak{I}^{n}$ is smaller or equal to $\frac{\vert\nu\vert([x_{0}, x_{1}])^{n}(x_{1}-x_{0})^{n}}{n!}$. So for $n$ large enough $\mathfrak{I}^{n}$ is contracting and thus $\mathfrak{I}$ has a unique fixed point in $\mathcal{C}([x_{0},x_{1}])$. 

Let $W(u_{1},u_{2})(x)$ be the Wronskian of two functions $u_{1}, u_{2}$:
\begin{displaymath}
{W(u_{1},u_{2})(x):=u_{1}(x)\dfrac{du_{2}}{dx}(x^{+})-u_{2}(x)\dfrac{du_{1}}{dx}(x^{+})}.
\end{displaymath}
If $u_{1}, u_{2}$ are both solutions of \eqref{Ch2Sec1: Eq2ndOrderODE}, $W(u_{1},u_{2})$ is constant on $I$. Using this fact we get a results which is similar to Sturm's separation theorem for the case of a measure $\nu$ with a continuous density with respect to the Lebesgue measure (see Theorem $7$, Section $2.6$ in \cite{BirkhoffRota1989ODE}):

\begin{property} 
\label{Ch2Sec1: PropertySturm}
Given $x_{0}<x_{1}$ two points in $I$:
\begin{itemize}
\item[(i)] Let $u_{1}$ be a solution of \eqref{Ch2Sec1: Eq2ndOrderODE} satisfying $u_{1}(x_{0})=0$, 
$\frac{du_{1}}{dx}(x_{0}^{+})>0$, and  $u_{2}$ a solution such that $u_{2}(x_{0})>0$. Assume that $u_{2}\geq 0$ on $[x_{0},x_{1}]$.  Then $u_{1}> 0$ on $(x_{0},x_{1}]$.
\item[(ii)] Let $u_{1}, u_{2}$ be two solutions such that $u_{1}(x_{0})=u_{2}(x_{0})>0$ and $\frac{du_{1}}{dx}(x_{0}^{+})>\frac{du_{2}}{dx}(x_{0}^{+})$. Assume that $u_{2}\geq 0$ on $[x_{0},x_{1}]$.  Then $u_{1}> u_{2}$ on $(x_{0},x_{1}]$.
\item[(iii)] If there is a solution $u$ to \eqref{Ch2Sec1: Eq2ndOrderODE} positive on $(x_{0}, x_{1})$ and zero at $x_{0}$ and $x_{1}$ then any other linearly independent solution of \eqref{Ch2Sec1: Eq2ndOrderODE} has exactly one zero in $(x_{0}, x_{1})$.
\end{itemize}
\end{property}

Next we prove a lemma that will be useful in Section \ref{Ch2Sec3}.

\begin{lemm}
\label{Ch2Sec1: LemNonDecrSol}
Let $\nu_{+}$ be the positive part of $\nu$. Let $x_{0}<x_{1}\in I$. Let $f$ be a continuous positive function on $[x_{0},x_{1}]$ such that ${\min_{[x_{0},x_{1}]}f>\nu_{+}([x_{0},x_{1}])^{2}}$. Then the equation 
\begin{equation}
\label{Ch2Sec1: Eq2ndOrderODE2}
\dfrac{d^{2}u}{dx^{2}} + u \nu - u f = 0
\end{equation}
has a positive solution that is non-decreasing on $[x_{0},x_{1}]$.
\end{lemm}

\begin{proof}
Set $a:=\min_{[x_{0},x_{1}]}f$. Let $u$ be the solution to \eqref{Ch2Sec1: Eq2ndOrderODE2} with the initial values $u(x_{0})=1$, $\frac{du}{dx}(x_{0}^{+})=\sqrt{a}$. We will show that $u$ is non-decreasing on 
$[x_{0},x_{1}]$. Assume that this is not the case. This means that $\frac{du}{dx}(x^{+})$ takes negative values somewhere in $[x_{0},x_{1}]$. Let 
\begin{displaymath}
x_{2}:=\inf\Big\lbrace x\in [x_{0},x_{1}]\Big\vert \dfrac{du}{dx}(x^{+})\leq 0\Big\rbrace.
\end{displaymath}
Since $\frac{du}{dx}(x^{+})$ is right-continuous, $\frac{du}{dx}(x_{2}^{+})\leq 0$.
Let $r(x):=\frac{1}{u(x)}\frac{du}{dx}(x^{+})$. $u$ is positive on $[x_{0}, x_{2}]$ hence $r$ is defined 
$[x_{0}, x_{2}]$. $r(x_{0})=\sqrt{a}$. $r$ is cadlag and satisfies the equation
\begin{displaymath}
dr = (f - r^{2}) dx - d\nu.
\end{displaymath}
Let $x_{3}:=\sup \lbrace x\in [x_{0},x_{2}]\vert r(x)\geq \sqrt{a} \rbrace$. We have
\begin{displaymath}
r(x_{2}) = r(x_{3}^{-}) + \int_{x_{3}}^{x_{2}}(f(x) - r^{2}(x)) dx - \nu([x_{3},x_{2}]).
\end{displaymath}
By construction $r(x_{3}^{-})\geq \sqrt{a}$. By definition, $f - r^{2}\geq 0$ on $(x_{3},x_{2}]$. Thus,
\begin{displaymath}
r(x_{2})\geq \sqrt{a} - \nu([x_{3},x_{2}])>0.
\end{displaymath} 
It follows that $r(x_{2})>0$, which is absurd.
\end{proof}

In the case $\nu = -2\kappa$ where $\kappa$ is a non-zero positive Radon measure, the equation \eqref{Ch2Sec1: Eq2ndOrderODE} becomes:
\begin{equation}
\label{Ch2Sec1: Eq2ndOrderODE3}
\dfrac{1}{2}\dfrac{d^{2}u}{dx^{2}} - u \kappa = 0.
\end{equation}
This is the equation of exit probabilities of a Brownian motion with killing measure $\kappa$. In this case, the two-dimensional linear space of solutions is spanned by two convex positive solutions $u_{\uparrow}$ and $u_{\downarrow}$, $u_{\uparrow}$ being non-decreasing and $u_{\downarrow}$ non-increasing. Given $x_{0}\in I$, we can construct $u_{\uparrow}$ as the limit when $x_{1}\to \inf I$ of the unique solution which equals $0$ in $x_{1}$ and $1$ in $x_{0}$. For $u_{\downarrow}$ we take the limit as $x_{1}\rightarrow \sup I$. $u_{\uparrow}$ and $u_{\downarrow}$ are defined up to a positive multiplicative constant. See \cite{Breiman1992Probability}, Section $16.11$, or \cite{RevuzYor1999BMGrundlehren}, Appendix $8$, for more details. Next we give equivalent conditions on the asymptotic behavior of $u_{\uparrow}$ and $u_{\downarrow}$ that will be used in Chapter \ref{Ch6}.

\begin{prop}
\label{Ch2Sec1: PropAsymptotics}
In case $[0,+\infty)\subseteq I$, the following four conditions are equivalent:
\begin{itemize}
\item[(i)] $\int_{(0,+\infty)}x\kappa(dx) <+\infty$,
\item[(ii)] $u_{\downarrow}(+\infty)>0$,
\item[(iii)] There is $C>0$ such that for all $x\geq 1$,
$u_{\uparrow}(x)\leq C x$,
\item[(iv)] $\int_{(0,+\infty)}u_{\uparrow}(x)u_{\downarrow}(x)\kappa(dx)<+\infty$.
\end{itemize}
\end{prop}

\begin{proof}
We will prove in order that (ii) implies (i), (iii) implies (i), (i) implies (ii), (i) implies (iii) and (iv) implies (ii). (iv) is obviously implied by the combination of (i), (ii) and (iii).

(ii) implies (i): For all $x\in[0,+\infty)$,
\begin{displaymath}
-\dfrac{du_{\downarrow}}{dx}(x^{+})=2\int_{(x,+\infty)}u_{\downarrow}(y)\kappa(dy)
\leq 2u_{\downarrow}(+\infty)\kappa((x,+\infty)).
\end{displaymath}
$-\frac{du_{\downarrow}}{dx}(x^{+})$ is integrable on $(0,+\infty)$. Since $u_{\downarrow}(+\infty)>0$, this implies that,
\begin{displaymath}
\int_{(0,+\infty)}\kappa((x,+\infty))dx < +\infty.
\end{displaymath}
But
\begin{displaymath}
\int_{(0,+\infty)}\kappa((x,+\infty))dx = \int_{(0,+\infty)}y\kappa(dy), 
\end{displaymath}
and hence (i).

(iii) implies (i): If (iii) holds then for all $x\in[0,+\infty)$, $\frac{du_{\uparrow}}{dx}(x^{+})\leq C$. 
But
\begin{displaymath}
\dfrac{du_{\uparrow}}{dx}(x^{+})=\dfrac{du_{\uparrow}}{dx}(0^{+})+ 
2\int_{(0,x]}u_{\uparrow}(y)\kappa(dy).
\end{displaymath}
This implies that
\begin{displaymath}
\int_{(0,+\infty)}u_{\uparrow}(y)\kappa(dy)<+\infty.
\end{displaymath}
Since $u_{\uparrow}$ is convex, 
$u_{\uparrow}(y)\geq u_{\uparrow}(0)+\frac{du_{\uparrow}}{dy}(0^{+})y$. So (i) is satisfied.

(i) implies (ii): For all $y\in[0,+\infty)$,
\begin{displaymath}
u_{\downarrow}(y)-u_{\downarrow}(+\infty)=
2\int_{y}^{+\infty}\int_{(z,+\infty)}u_{\downarrow}(x)\kappa(dx)dz
\leq 2u_{\downarrow}(y)\int_{(y,+\infty)}(x-y)\kappa(dx).
\end{displaymath}
Condition (i) implies that
\begin{displaymath}
\lim_{y\rightarrow +\infty}2\int_{(y,+\infty)}(x-y)\kappa(dx)=0.
\end{displaymath}
So for $y$ large enough, $u_{\downarrow}(y)-u_{\downarrow}(+\infty)<u_{\downarrow}(y)$. Necessarily, $u_{\downarrow}(+\infty)>0$.

(i) implies (iii): For all $y<x\in[0,+\infty)$,
\begin{equation}
\label{Ch2Sec1: EqIneq1}
\dfrac{du_{\uparrow}}{dx}(x^{+})=\dfrac{du_{\uparrow}}{dy}(y^{+})+
2u_{\uparrow}(y)\kappa((y,x])+2\int_{(y,x]}(u_{\uparrow}(z)-u_{\uparrow}(y))\kappa(dz).
\end{equation}
Let $y$ be large enough such that
\begin{displaymath}
2\int_{(y,+\infty)}(z-y)\kappa(dz)<1.
\end{displaymath}
Then there is $C>0$ large enough such that
\begin{equation}
\label{Ch2Sec1: EqIneq2}
C>\dfrac{du_{\uparrow}}{dy}(y^{+})+2u_{\uparrow}(y)\kappa((y,+\infty)) + 2C\int_{(y,+\infty)}(z-y)\kappa(dz).
\end{equation}
Assume that there is $x\in[0,+\infty)$ such that $\frac{du_{\uparrow}}{dx}(x^{+})\geq C$. Let
\begin{displaymath}
x_{0}:=\inf\Big\lbrace x\geq y\Big\vert \dfrac{du_{\uparrow}}{dx}(x^{+})\geq C\Big\rbrace.
\end{displaymath}
$x\mapsto \frac{du_{\uparrow}}{dx}(x^{+})$ is right-continuous. Thus 
$\frac{du_{\uparrow}}{dx}(x_{0}^{+})\geq C$. By definition, for all $z\in[y,x_{0}]$, 
$\frac{du_{\uparrow}}{dz}(z^{+})\leq C$, and hence $u_{\uparrow}(z)-u_{\uparrow}(y)\leq C(z-y)$.
But then \eqref{Ch2Sec1: EqIneq1} and \eqref{Ch2Sec1: EqIneq2} imply that 
$\frac{du_{\uparrow}}{dx}(x_{0}^{+})< C$ which is contradictory. It follows that $\frac{du_{\uparrow}}{dx}(x^{+})$ is bounded by $C$, which implies property (iii).

(iv) implies (ii): Applying integration by parts we get that for all $x>0$,
\begin{equation*}
\begin{split}
2\int_{(0,x]}u_{\uparrow}(y)u_{\downarrow}(y)\kappa(dy)=&
\int_{(0,x]}u_{\downarrow}(y)d\Big(\dfrac{du_{\uparrow}}{dy}\Big)(dy)\\=&
\dfrac{du_{\uparrow}}{dx}(x^{+})u_{\downarrow}(x)-\dfrac{du_{\uparrow}}{dx}(0^{+})u_{\downarrow}(0)-
\int_{0}^{x}\dfrac{du_{\downarrow}}{dy}(y^{+})\dfrac{du_{\uparrow}}{dy}(y^{+})dy.
\end{split}
\end{equation*}
$\frac{du_{\uparrow}}{dx}(x^{+})u_{\downarrow}(x)$ is positive. We get that
\begin{equation}
\label{Ch2Sec1: EqIneq3}
-\int_{0}^{+\infty}\dfrac{du_{\downarrow}}{dy}(y^{+})\dfrac{du_{\uparrow}}{dy}(y^{+})dy\leq
2\int_{(0,+\infty)}u_{\uparrow}(y)u_{\downarrow}(y)\kappa(dy)+
\dfrac{du_{\uparrow}}{dx}(0^{+})u_{\downarrow}(0)<+\infty.
\end{equation}
Next,
\begin{equation}
\label{Ch2Sec1: EqIneq4}
\begin{split}
\dfrac{du_{\uparrow}}{dx}(x^{+})(u_{\downarrow}(x)-u_{\downarrow}(+\infty))=&
-\dfrac{du_{\uparrow}}{dx}(x^{+})\int_{x}^{+\infty}\dfrac{du_{\downarrow}}{dy}(y^{+})dy\\
\leq&-\int_{x}^{+\infty}\dfrac{du_{\downarrow}}{dy}(y^{+})\dfrac{du_{\uparrow}}{dy}(y^{+})dy.
\end{split}
\end{equation}
Assume that $u_{\downarrow}(+\infty)=0$. Then \eqref{Ch2Sec1: EqIneq4} implies that
\begin{displaymath}
\lim_{x\rightarrow +\infty}\dfrac{du_{\uparrow}}{dx}(x^{+})u_{\downarrow}(x)=0
\end{displaymath}
and
\begin{equation}
\label{Ch2Sec1: EqLim}
\lim_{x\rightarrow +\infty}-\dfrac{du_{\downarrow}}{dx}(x^{+})u_{\uparrow}(x)=
W(u_{\downarrow},u_{\uparrow})-\lim_{x\rightarrow +\infty}\dfrac{du_{\uparrow}}{dx}(x^{+})u_{\downarrow}(x)
=W(u_{\downarrow},u_{\uparrow}).
\end{equation}
\eqref{Ch2Sec1: EqIneq3} together with \eqref{Ch2Sec1: EqLim} imply that
\begin{displaymath}
\int_{0}^{+\infty}\dfrac{1}{u_{\uparrow}(y)}\dfrac{du_{\uparrow}}{dy}(y^{+})dy <+\infty.
\end{displaymath}
But this is impossible because $\log(u_{\uparrow}(+\infty))=+\infty$. Thus $u_{\downarrow}(+\infty)>0$.
\end{proof}

Next we deal with the continuity of $u_{\uparrow}$ and $u_{\downarrow}$ with respect the measure $\kappa$. We will write $u_{\kappa,\uparrow}$ and $u_{\kappa,\downarrow}$ to denote the dependence on $\kappa$.

\begin{lemm}
\label{Ch2Sec1: LemContKappa}
Let $x_{0}\in I$. Let $(\kappa_{n})_{n\geq 0}$ be a sequence of non-zero positive Radon measures on $I$ converging vaguely (i.e. against functions with compact support) to $\kappa$. Then 
$\frac{u_{\kappa_{n},\uparrow}}{u_{\kappa_{n},\uparrow}(x_{0})}$ converges to 
$\frac{u_{\kappa,\uparrow}}{u_{\kappa,\uparrow}(x_{0})}$, 
$\frac{u_{\kappa_{n},\downarrow}}{u_{\kappa_{n},\downarrow}(x_{0})}$ converges to
$\frac{u_{\kappa,\downarrow}}{u_{\kappa,\downarrow}(x_{0})}$ and the convergences are uniform on compact subsets of $I$.
\end{lemm}

\begin{proof}
We will deal with the convergence of $\frac{u_{\kappa_{n},\downarrow}}{u_{\kappa_{n},\downarrow}(x_{0})}$, the other one being similar. To simplify notations we will chose the normalization 
$u_{\kappa,\downarrow}(x_{0})=u_{\kappa_{n},\downarrow}(x_{0})=1$. Without loss of generality we will also assume that $\kappa(\lbrace x_{0}\rbrace)=0$. The proof will be made of two parts. First we will show that if $u$ is the solution of \eqref{Ch2Sec1: Eq2ndOrderODE3} and $u_{n}$ solution of
\begin{equation}
\label{Ch2Sec1: Eq2ndOrderODEn}
\dfrac{1}{2}\dfrac{d^{2}u}{dx^{2}} - u \kappa_{n} = 0,
\end{equation}
and if $u_{n}(x_{0})=u(x_{0})=1$ and 
$\frac{du}{dx}(x_{0}^{+})=\lim_{n\rightarrow +\infty}\frac{du_{n}}{dx}(x_{0}^{+})$, 
then $u_{n}$ converges to $u$ uniformly on compact subsets of $I$. After that we will show that 
$\frac{du_{\kappa_{n},\downarrow}}{dx}(x_{0}^{+})$
converges to $\frac{du_{\kappa,\downarrow}}{dx}(x_{0}^{+})$.

Let $x_{1}\in I\cap(x_{0},+\infty)$. Let $(v_{n})_{n\geq 0}$ be a sequence in $\mathbb{R}$ converging to $v$. 
Let $\mathfrak{I}_{n}$ respectively $\mathfrak{I}$ be the following affine operators on $\mathcal{C}([x_{0},x_{1}])$:
\begin{displaymath}
(\mathfrak{I}_{n}f)(x):=1+(x-x_{0})v_{n}+2\int_{(x_{0},x]}(x-y)f(y)\kappa_{n}(dy).
\end{displaymath}
\begin{displaymath}
(\mathfrak{I}f)(x):=1+(x-x_{0})v+2\int_{(x_{0},x]}(x-y)f(y)\kappa(dy).
\end{displaymath}
Let $u_{n}$ respectively $u$ be the fixed points of $\mathfrak{I}_{n}$ respectively $\mathfrak{I}$. Let $\varepsilon\in(0,1)$. The Lipschitz norm of $\mathfrak{I}_{n}^{j}$ is bounded by
$\frac{2^{j}}{j!}\kappa_{n}([x_{0}, x_{1}])^{j}(x_{1}-x_{0})^{j}$. For $j\geq j_{\varepsilon}$, for all $n\in\mathbb{N}$, this norm is less then $\varepsilon$. Then
\begin{equation*}
\begin{split}
\max_{[x_{0},x_{1}]}\vert u_{n}-u\vert=&
\max_{[x_{0},x_{1}]}\vert \mathfrak{I}^{j_{\varepsilon}}_{n}u_{n}-\mathfrak{I}^{j_{\varepsilon}}u\vert
\leq\max_{[x_{0},x_{1}]}\vert \mathfrak{I}^{j_{\varepsilon}}_{n}u-\mathfrak{I}^{j_{\varepsilon}}u\vert
+\max_{[x_{0},x_{1}]}\vert \mathfrak{I}^{j_{\varepsilon}}_{n}u_{n}-\mathfrak{I}^{j_{\varepsilon}}_{n}u\vert\\
\leq&\max_{[x_{0},x_{1}]}\vert \mathfrak{I}^{j_{\varepsilon}}_{n}u-\mathfrak{I}^{j_{\varepsilon}}u\vert
+\varepsilon\max_{[x_{0},x_{1}]}\vert u_{n}-u\vert.
\end{split}
\end{equation*}
Hence,
\begin{equation}
\label{Ch2Sec1: EqIneq5}
\max_{[x_{0},x_{1}]}\vert u_{n}-u\vert\leq \dfrac{1}{1-\varepsilon}
\max_{[x_{0},x_{1}]}\vert \mathfrak{I}^{j_{\varepsilon}}_{n}u-\mathfrak{I}^{j_{\varepsilon}}u\vert.
\end{equation}
For $y<x\in I$ and $i\in\mathbb{N}^{\ast}$, let
\begin{displaymath}
f_{n,i}(y,x):=\int_{y<y_{1}<\dots<y_{i-1}<x}(x-y_{i-1})
\dots(y_{2}-y_{1})(y_{1}-y)\kappa_{n}(dy_{1})\dots \kappa_{n}(dy_{i-1}),
\end{displaymath}
\begin{displaymath}
f_{i}(y,x):=\int_{y<y_{1}<\dots<y_{i-1}<x}(x-y_{i-1})\dots
(y_{2}-y_{1})(y_{1}-y)\kappa(dy_{1})\dots \kappa(dy_{i-1}),
\end{displaymath}
and $f_{0,i}(y,x)=f_{0}(y,x)=x-y$. $f_{n,i}$ and $f_{i}$ are continuous functions. Moreover, the vague convergence of $\kappa_{n}$ to $\kappa$ ensures that if $(y_{n},x_{n})_{n\geq 0}$ is a sequence converging to $(y,x)$, then $f_{n,i}(y_{n},x_{n})$
converges to $f_{i}(y,x)$.
\begin{equation*}
\begin{split}
(\mathfrak{I}^{j_{\varepsilon}}_{n}u)(x)=&1+(x-x_{0})v_{n}+
\sum_{i=0}^{j_{\varepsilon}-2}\int_{x_{0}}^{x}(1+(y-x_{0})v_{n})f_{n,i}(y,x)\kappa_{n}(dy)\\
+&\int_{x_{0}}^{x}u(y)f_{n,j_{\varepsilon}-1}(y,x)\kappa_{n}(dy),\\
(\mathfrak{I}^{j_{\varepsilon}}u)(x)=&1+(x-x_{0})v+
\sum_{i=0}^{j_{\varepsilon}-2}\int_{x_{0}}^{x}(1+(y-x_{0})v)f_{i}(y,x)\kappa(dy)\\
+&\int_{x_{0}}^{x}u(y)f_{j_{\varepsilon}-1}(y,x)\kappa(dy).
\end{split}
\end{equation*}
For fixed $x$, the functions $y\mapsto 1_{x_{0}<y<x}f_{n,i}(y,x)$ and $y\mapsto 1_{x_{0}<y<x}f_{i}(y,x)$ have a compact support but are discontinuous in $x_{0}$. If $(z_{n})_{n\geq 0}$ is a sequence in $[x_{0},x_{1}]$ converging to $z$, then the convergence of $v_{n}$ to $v$, the weak convergence of $\kappa_{n}$ to $\kappa$ and the condition $\kappa(\lbrace x_{0}\rbrace)=0$ ensure that $(\mathfrak{I}^{j_{\varepsilon}}_{n}u)(z_{n})$ converges to $(\mathfrak{I}^{j_{\varepsilon}}u)(z)$. This implies the uniform convergence of 
$\mathfrak{I}^{j_{\varepsilon}}_{n}u$ to $\mathfrak{I}^{j_{\varepsilon}}u$ on $[x_{0},x_{1}]$. From \eqref{Ch2Sec1: EqIneq5} follows that $u_{n}$ converges uniformly to $u$ on $[x_{0},x_{1}]$. The situation is similar for $x_{1}<x_{0}$ and we get the uniform convergence on compact sets of $u_{n}$ to $u$.

Let
\begin{displaymath}
\underline{v}:=\liminf_{n\rightarrow +\infty}\dfrac{du_{\kappa_{n},\downarrow}}{dx}
(x_{0}^{+}),\qquad
\overline{v}:=\limsup_{n\rightarrow +\infty}\dfrac{du_{\kappa_{n},\downarrow}}{dx}
(x_{0}^{+}).
\end{displaymath}
Let $v<\frac{du_{k,\downarrow}}{dx}(x_{0}^{+})$. There is $x_{1}\in I\cap(x_{0},+\infty)$ such that the solution of \eqref{Ch2Sec1: Eq2ndOrderODE3} with initial conditions $u(x_{0})=1$, $\frac{du}{dx}(x^{+})=v$, is zero in $x_{1}$. Since $u_{\kappa_{n},\downarrow}$ converges to $u_{\kappa,\downarrow}$ uniformly on $[x_{0},x_{1}]$ and $u_{\kappa,\downarrow}$ is positive on $[x_{0},x_{1}]$, we get that for $n$ large enough, $u_{\kappa_{n},\downarrow}$ is positive on $[x_{0},x_{1}]$ and $\frac{du_{\kappa_{n},\downarrow}}{dx}(x_{0}^{+})>v$. Thus 
$\underline{v}\geq\frac{du_{k,\downarrow}}{dx}(x_{0}^{+})$. 

Conversely, let $v<\overline{v}$. Let $u_{n}$ be the solution of 
\eqref{Ch2Sec1: Eq2ndOrderODEn} with initial conditions $u_{n}(x_{0})=1$, 
$\frac{du_{n}}{dx}(x_{0}^{+})=v$. If 
$\frac{du_{\kappa_{n},\downarrow}}{dx}(x_{0}^{+})>v$, then for any 
$x\in I\cap[x_{0},+\infty)$,
\begin{displaymath}
\dfrac{du_{n}}{dx}(x^{+})\leq \dfrac{du_{\kappa_{n},\downarrow}}{dx}(x^{+})
-\Big(\dfrac{du_{\kappa_{n},\downarrow}}{dx}(x_{0}^{+})-v\Big)\leq
-\Big(\dfrac{du_{\kappa_{n},\downarrow}}{dx}(x_{0}^{+})-v\Big),
\end{displaymath}
\begin{displaymath}
u_{n}(x)\leq u_{\kappa_{n},\downarrow}-
\Big(\dfrac{du_{\kappa_{n},\downarrow}}{dx}(x_{0}^{+})-v\Big)(x-x_{0}).
\end{displaymath}
If $\sup I<+\infty$, then by convexity of $u_{\kappa_{n},\downarrow}$,
\begin{displaymath}
u_{n}(x)\leq \dfrac{\sup I -x}{\sup I -x_{0}} - 
\Big(\dfrac{du_{\kappa_{n},\downarrow}}{dx}(x_{0}^{+})-v\Big)(x-x_{0}),
\end{displaymath}
and $u_{n}(z_{n})\leq 0$, where
\begin{displaymath}
z_{n}:=\dfrac{\sup I + \Big(\dfrac{du_{\kappa_{n},\downarrow}}{dx}
(x_{0}^{+})-v\Big)x_{0}(\sup I - x_{0})}
{1+\Big(\dfrac{du_{\kappa_{n},\downarrow}}{dx}(x_{0}^{+})-v\Big)
(\sup I - x_{0})}.
\end{displaymath}
This is also true if $\sup I = +\infty$, and in this case 
$z_{n}=x_{0}+\Big(\frac{du_{\kappa_{n},\downarrow}}{dx}
(x_{0}^{+})-v\Big)^{-1}$.
Let $u$ be the solution of of \eqref{Ch2Sec1: Eq2ndOrderODE3} with initial conditions $u(x_{0})=1$, $\frac{du}{dx}(x^{+})=v$ and 
\begin{displaymath}
z_{\infty}:=\dfrac{\sup I + (\overline{v}-v)x_{0}(\sup I - x_{0})}
{1+(\overline{v}-v)(\sup I - x_{0})}.
\end{displaymath}
Considering a subsequence along which $\frac{du_{\kappa_{n},\downarrow}}{dx}(x_{0}^{+})$ converges to $\overline{v}$, we get by uniform convergence of $u_{n}$ tu $u$ on compact sets that $u(z_{\infty})\geq 0$. It follows that 
$\frac{du_{\kappa,\downarrow}}{dx}(x_{0}^{+})\geq v$. Hence 
$\frac{du_{\kappa,\downarrow}}{dx}(x_{0}^{+})\geq \overline{v}$.

Finally, $\underline{v}=\overline{v}=\frac{du_{\kappa,\downarrow}}{dx}(x_{0}^{+})$, and this implies the uniform convergence on compact subsets of $u_{\kappa_{n},\downarrow}$
to $u_{\kappa,\downarrow}$.
\end{proof}

\section{One-dimensional diffusions}
\label{Ch2Sec2}

In this subsection we will describe the kind of linear diffusion we are interested in, recall some facts and introduce notations that will be used subsequently. For a detailed presentation of one-dimensional diffusions see \cite{ItoMcKean1974Diffusions} and \cite{Breiman1992Probability}, Chapter $16$.

Let $I$ be an open interval of $\mathbb{R}$, $m$ and $w$ continuous positive functions on $I$. We consider a diffusion $(X_{t})_{0\leq t<\zeta^{(0)}}$ on $I$ with generator
\begin{displaymath}
L^{(0)}:=\dfrac{1}{m(x)}\dfrac{d}{dx}\left(\dfrac{1}{w(x)}\dfrac{d}{dx}\right),
\end{displaymath}
and killed as it hits the boundary of $I$. In case $I$ is unbounded, we also allow $X$ to blow up to infinity in finite time. $\zeta^{(0)}$ is the first time $X$ either hits the boundary or explodes. To avoid some technicalities we will assume that $\frac{dw}{dx}$ is locally bounded, although this condition is not essential. Given such a diffusion, the speed measure $m(x)dx$ and the scale measure $w(x)dx$ are defined up to a positive multiplicative constant, but the product $mw$ is uniquely defined. A primitive $S$ of $w$ is a natural scale function of $X$. Consider the random time change $d\tilde{t}=\frac{1}{m(X_{t})}dt$. Then, 
$(\frac{1}{2}S(X_{\tilde{t}}))_{0\leq\tilde{t}<\tilde{\zeta}^{(0)}}$ is a standard Brownian motion on $S(I)$ killed when it first hits the boundary of $S(I)$. For all $f,g$ smooth, compactly supported in $I$, 
\begin{displaymath}
\int_{I}(L^{(0)} f)(x)g(x)m(x)dx =  \int_{I}f(x)(L^{(0)} g)(x)m(x)dx.
\end{displaymath}
The diffusion $X$ has a family of local times $(\ell_{t}^{x}(X))_{x\in I, t\geq 0}$ with respect to the measure $m(x)dx$ such that $(x,t)\mapsto \ell_{t}^{x}(X)$ is continuous. We can further consider diffusions with killing measures. Let $\kappa$ be a non-negative Radon measure on $I$. We kill $X$ as soon as $\int_{I}\ell_{t}^{x}(X)m(x)\kappa(dx)$ hits an independent exponential time with parameter $1$. The corresponding generator is
\begin{equation}
\label{Ch2Sec2: EqGenerator}
L = \dfrac{1}{m(x)}\dfrac{d}{dx}\left(\dfrac{1}{w(x)}\dfrac{d}{dx}\right) - \kappa.
\end{equation}

Let $(X_{t})_{0\leq t<\zeta}$ be the diffusion of generator 
\eqref{Ch2Sec2: EqGenerator}, which is killed either by hitting $\partial I$, or by exploding, or by the killing measure $\kappa$. 
For $x\in I$, let $\eta_{\rm exc}^{>x}$ and $\eta_{\rm exc}^{<x}$ be the excursion measures of $X$ above and below the level $x$ up to the last time $X$ visits $x$. The behavior of $X$ from the first to the last time it visits $x$ is a Poisson point process with intensity $\eta_{\rm exc}^{>x}+\eta_{\rm exc}^{<x}$, parametrized by the local time at $x$ up to the value $\ell_{t}^{\zeta}(X)$. $\eta_{\rm exc}^{>x}$ and $\eta_{\rm exc}^{<x}$ are obtained from the Lévy-Itô measure on Brownian excursions through scale change, time change and multiplication by a density function accounting for the killing. See \cite{SalminenValloisYor2007Excur} for details on excursion measures in case of recurrent diffusions.

If $X$ is transient, the Green's function of $L$,
\begin{displaymath}
G(x,y):= \mathbb{E}_{x}[\ell_{t}^{\zeta}(X)],
\end{displaymath}
is finite, continuous and symmetric. For $x\leq y$ it can be written
\begin{displaymath}
G(x,y)=u_{\uparrow}(x)u_{\downarrow}(y)
\end{displaymath} 
where $u_{\uparrow}(x)$ and $u_{\downarrow}(y)$ are positive , respectively non-decreasing and non-increasing solutions to the equation $Lu=0$, which through a change of scale reduces to an equation of form \eqref{Ch2Sec1: Eq2ndOrderODE3}. If $S$ is bounded from below, $u_{\uparrow}(\inf I^{+})=0$. If $S$ is bounded from above, $u_{\downarrow}(\sup I^{-})=0$. $u_{\uparrow}(x)$ and $u_{\downarrow}(y)$ are each determined up to a multiplication by a positive constant, but when entering the expression of $G$, the two constants are related. For $x\leq y\in I$,
\begin{displaymath}
\dfrac{u_{\uparrow}(x)}{u_{\uparrow}(y)}=\mathbb{P}_{y}(X~\text{hits}~x~\text{before time}~\zeta),\qquad
\dfrac{u_{\downarrow}(y)}{u_{\downarrow}(x)}=\mathbb{P}_{x}(X~\text{hits}~y~\text{before time}~\zeta).
\end{displaymath}
See \cite{ItoMcKean1974Diffusions} or \cite{Breiman1992Probability}, Chapter $16$, for details. Let $W(u_{\downarrow},u_{\uparrow})$ be the Wronskian of $u_{\downarrow}$ and $u_{\uparrow}$:
\begin{displaymath}
W(u_{\downarrow},u_{\uparrow})(x):=u_{\downarrow}(x)\dfrac{du_{\uparrow}}{dx}(x^{+})-u_{\uparrow}(x)\dfrac{du_{\downarrow}}{dx}(x^{+}).
\end{displaymath}
This Wronskian is actually the density of the scale measure: $W(u_{\downarrow},u_{\uparrow})\equiv w$. We may write $G_{L}$ when there is an ambiguity on $L$.

If the killing measure $\kappa$ is non zero, then the probability that $X$, starting from $x$, gets killed by $\kappa$ before reaching a boundary of $I$ or exploding equals ${\int_{I}G(x,y)m(y)\kappa(dy)}$. Conditionally on this event, the distribution of $X_{\zeta^{-}}$ is
\begin{displaymath}
\dfrac{1_{z\in I}G(x,z)m(z)\kappa(dz)}{\int_{I}G(x,y)m(y)\kappa(dy)}.
\end{displaymath}
Indeed, let $f$ be a non-negative compactly supported measurable function on $I$ and
\begin{displaymath}
\tau_{l}:=\inf\Big\lbrace t\in [0,\zeta^{(0)}\Big\vert \int_{I}\ell^{y}_{t}(X)
m(y)\kappa(dy)>l\Big\rbrace.
\end{displaymath}
Then, by definition,
\begin{displaymath}
\mathbb{E}_{x}\big[f(X_{\zeta^{-}})\big]=\int_{0}^{+\infty}e^{-l}
\mathbb{E}_{x}\big[f(X_{\tau_{l}\wedge\zeta^{(0)}})\big]dl=
\int_{0}^{+\infty}dv e^{-v}
\mathbb{E}_{x}\Big[\int_{0}^{v}f(X_{\tau_{l}\wedge\zeta^{(0)}})dl\Big].
\end{displaymath}
But
\begin{displaymath}
\int_{0}^{v}f(X_{\tau_{l}\wedge\zeta^{(0)}})dl=\int_{I}
\ell^{y}_{\tau_{v}\wedge\zeta^{(0)}}(X)m(y)\kappa(dy)
\end{displaymath}
(see Corollary $2.13$, Chapter X in \cite{RevuzYor1999BMGrundlehren}). 
It follows that
\begin{displaymath}
\mathbb{E}\big[f(X_{\zeta^{-}})\big]=
\int_{I}f(y)\Big[\int_{0}^{+\infty}e^{-v}\ell^{y}_{\tau_{v}\wedge\zeta^{(0)}}(X)dv\Big]m(y)\kappa(dy)=\int_{I}f(y)G(x,y)m(y)\kappa(dy).
\end{displaymath}

The semi-group of $L$ has positive transition densities $p_{t}(x,y)$ with respect to the speed measure $m(y)dy$ and ${(t,x,y)\mapsto p_{t}(x,y)}$ is continuous on 
$(0,+\infty)\times I\times I$. McKean gives a proof of this in \cite{McKean1956EigenDiffExp} in case when the killing measure $\kappa$ has a continuous density with respect to the Lebesgue measure. If this is not the case, we can take $u$ a positive continuous solution to $Lu=0$ and consider the h-transform of $L$ by $u$: $u^{-1}Lu$. The latter is the generator of a diffusion without killing measure and by \cite{McKean1956EigenDiffExp} this diffusion has continuous transition densities $\tilde{p}_{t}(x,y)$ with respect to $m(y)dy$. Then $u(x)\tilde{p}_{t}(x,y)\frac{1}{u(y)}$ are the transition densities of the semi-group of $L$. Transition densities with respect to the speed measure are symmetric: $p_{t}(x,y)=p_{t}(y,x)$. For all $x,y\in I$ and $t\geq 0$ the following equality holds:
\begin{equation}
\label{Ch2Sec2: EqTempsLocSemiGroup}
\mathbb{E}_{x}\big[\ell^{y}_{t\wedge \zeta}(X)\big]=\int_{0}^{t}p_{s}(x,y)ds. 
\end{equation}

Next we deal with bridge probability measures.
\begin{prop}
\label{Ch2Sec2: PropContinuityBridge}
The bridge probability measures $\mathbb{P}_{x,y}^{t}(\cdot)$ (bridge of $X$ from $x$ to $y$ in time $t$ conditioned neither to die nor to explode in the interval) satisfy: for all $x\in I$ the map ${(x,y,t)\mapsto \mathbb{P}_{x,y}^{t}(\cdot)}$ is continuous for the weak topology on probability measures on continuous paths.
\end{prop}

\begin{proof}
Our proof mainly relies on absolute continuity arguments of \cite{FitzsimmonsPitmanYor1993Bridges} and \cite{ChaumontBravo2011ContMarkBridges}, and the time reversal argument of \cite{FitzsimmonsPitmanYor1993Bridges}. \cite{ChaumontBravo2011ContMarkBridges} gives a proof of weak continuity of bridges for conservative Feller cadlag processes on second countable locally compact spaces. But since the proof contains an error and we do not restrict to conservative diffusions, we give here accurate arguments for the weak continuity.

First, we can restrict to the case $\kappa=0$. Otherwise, consider $u$ a solution to $Lu=0$, positive on $I$. The generator of the h-transform of $L$ by $u$ is 
\begin{displaymath}
\dfrac{1}{u(x)^{2}m(x)}\dfrac{d}{dx}\left(\dfrac{u(x)^{2}}{w(x)}\dfrac{d}{dx}\right), 
\end{displaymath}
and does not contain any killing measure. The h-transform preserves the bridge measures and changes the density functions relatively to $m(y)dy$ to $\frac{1}{u(x)}p_{t}(x,y)u(y)$, and thus preserves their continuity.

Then we normalize the length of bridges: if $(X^{(x,y,t)}_{s})_{0\leq s\leq t}$ is a path under the law $\mathbb{P}_{x,y}^{t}(\cdot)$, let $\widetilde{\mathbb{P}}_{x,y}^{t}(\cdot)$ be the law of 
$(X^{(x,y,t)}_{rt})_{0\leq r\leq 1}$. It is sufficient to prove that 
${(x,y,t)\mapsto \widetilde{\mathbb{P}}_{x,y}^{t}(\cdot)}$ is continuous. For $v\in[0,1]$, let $\widetilde{\mathbb{P}}_{x,y}^{t,v}(\cdot)$ be the law of $(X^{(x,y,t)}_{rt})_{0\leq r\leq v}$. Let $\widetilde{\mathbb{P}}_{x}^{t,v}(\cdot)$ be the law of the Markovian path $(X_{rt})_{0\leq r\leq v}$ starting from $x$. For $v\in [0,1)$, we have the following absolute continuity relationship:
\begin{equation}
\label{Ch2Sec2: EqBridgeAbsCont}
d\widetilde{\mathbb{P}}_{x,y}^{t,v} = 1_{vt<\zeta}\dfrac{p_{(1-v)t}(X_{vt},y)}{p_{t}(x,y)}d\widetilde{\mathbb{P}}_{x}^{t,v}.
\end{equation}

Let $(J_{n})_{n\geq 0}$ be an increasing sequence of compact subintervals of $I$ such that $I=\bigcup_{n\geq 0} J_{n}$. Let $T_{n}$ be the first exit time from $J_{n}$. Let $f_{n}$ be continuous compactly supported function on $I$ such that $0\leq f_{n}\leq 1$ and 
${f_{n\vert J_{n}}\equiv 1}$. We can further assume that the sequence $(f_{n})_{n\geq 0}$ is non-decreasing. The map
\begin{displaymath}
(x,y,t)\mapsto f_{n}(\sup_{[0,vt]}X)f_{n}(\inf_{[0,vt]}X)d\widetilde{\mathbb{P}}_{x}^{t,v}
\end{displaymath}
is weakly continuous. Let $(x_{j},y_{j},t_{j})_{j\geq 0}$ be a sequence converging to $(x,y,t)$. Let $F$ be a continuous bounded functional on $\mathcal{C}([0,v])$. Then, applying \eqref{Ch2Sec2: EqBridgeAbsCont}, we get
\begin{equation}
\label{Ch2Sec2: EqBridgeCont1}
\widetilde{\mathbb{P}}^{t_{j}, v}_{x_{j},y_{j}}(f_{n}(\sup_{[0,v]}\gamma)f_{n}(\inf_{[0,v]}\gamma)F(\gamma))-
\widetilde{\mathbb{P}}^{t, v}_{x,y}(f_{n}(\sup_{[0,v]}\gamma)f_{n}(\inf_{[0,v]}\gamma)F(\gamma))=
\end{equation}
\begin{equation}
\label{Ch2Sec2: EqBridgeCont2}
\widetilde{\mathbb{P}}^{t_{j}, v}_{x_{j}}\left(\dfrac{p_{(1-v)t}(\gamma(v),y)}{p_{t}(x,y)}f_{n}(\sup_{[0,v]}\gamma)f_{n}(\inf_{[0,v]}\gamma)F(\gamma)\right) 
\end{equation}
\begin{equation}
\label{Ch2Sec2: EqBridgeCont3}
-\widetilde{\mathbb{P}}^{t, v}_{x}\left(\dfrac{p_{(1-v)t}(\gamma(v),y)}{p_{t}(x,y)}f_{n}(\sup_{[0,v]}\gamma)f_{n}(\inf_{[0,v]}\gamma)F(\gamma)\right) 
\end{equation}
\begin{equation}
\label{Ch2Sec2: EqBridgeCont4}
+\widetilde{\mathbb{P}}^{t_{j}, v}_{x_{j}}\left(\dfrac{p_{(1-v)t_{j}}(\gamma(v),y_{j})}{p_{t_{j}}(x_{j},y_{j})}f_{n}(\sup_{[0,v]}\gamma)f_{n}(\inf_{[0,v]}\gamma)F(\gamma)\right) 
\end{equation}
\begin{equation}
\label{Ch2Sec2: EqBridgeCont5}
-\widetilde{\mathbb{P}}^{t_{j}, v}_{x_{j}}\left(\dfrac{p_{(1-v)t}(\gamma(v),y)}{p_{t}(x,y)}f_{n}(\sup_{[0,v]}\gamma)f_{n}(\inf_{[0,v]}\gamma)F(\gamma)\right). 
\end{equation}
Since $\frac{p_{(1-v)t}(\cdot,y)}{p_{t}(x,y)}$ is continuous and bounded on $J_{n}$, \eqref{Ch2Sec2: EqBridgeCont2}$-$\eqref{Ch2Sec2: EqBridgeCont3} converges to $0$. Moreover, for $j$ large enough, 
$\frac{p_{(1-v)t_{j}}(\cdot,y_{j})}{p_{t_{j}}(x_{j},y_{j})}$ is uniformly close on $J_{n}$ to $\frac{p_{(1-v)t}(\cdot,y)}{p_{t}(x,y)}$. Thus, 
\eqref{Ch2Sec2: EqBridgeCont4}$-$\eqref{Ch2Sec2: EqBridgeCont5} converges to $0$, and finally \eqref{Ch2Sec2: EqBridgeCont1} converges to $0$. 
Let $n_{0}\in\mathbb{N}$ and 
$n\geq n_{0}$. Then,
\begin{multline*}
\widetilde{\mathbb{P}}^{t_{j},v}_{x_{j},y_{j}}(1-f_{n}(\sup_{[0,v]}\gamma)f_{n}(\inf_{[0,v]}\gamma))=1-\widetilde{\mathbb{P}}^{t_{j}, v}_{x_{j},y_{j}}(f_{n}(\sup_{[0,v]}\gamma)f_{n}(\inf_{[0,v]}\gamma))\\
\leq 1-\widetilde{\mathbb{P}}^{t_{j},v}_{x_{j},y_{j}}(f_{n_{0}}(\sup_{[0,v]}\gamma)f_{n_{0}}(\inf_{[0,v]}\gamma))\rightarrow 1-\widetilde{\mathbb{P}}^{t,v}_{x,y}(f_{n_{0}}(\sup_{[0,v]}\gamma)f_{n_{0}}(\inf_{[0,v]}\gamma)),
\end{multline*}
and consequently
\begin{displaymath}
\lim_{n\rightarrow +\infty} \limsup_{j\rightarrow +\infty } \widetilde{\mathbb{P}}^{t_{j},v}_{x_{j},y_{j}}(1-f_{n}(\sup_{[0,v]}\gamma)f_{n}(\inf_{[0,v]}\gamma)) = 0.
\end{displaymath}
It follows that
\begin{displaymath}
\lim_{j\rightarrow +\infty}\widetilde{\mathbb{P}}^{t_{j}, v}_{x_{j},y_{j}}(F(\gamma))=\widetilde{\mathbb{P}}^{t,v}_{x,y}(F(\gamma)).
\end{displaymath}
From this we get that the law of any finite-dimensional family of marginals of $\widetilde{\mathbb{P}}^{t}_{x,y}(\cdot)$ depends continuously on $(x,y,t)$. To conclude we need a tightness result for $(x,y,t)\mapsto\widetilde{\mathbb{P}}^{t}_{x,y}(\cdot)$. 
We have already tightness for $(x,y,t)\mapsto\widetilde{\mathbb{P}}^{t,v}_{x,y}(\cdot)$. The image of $\widetilde{\mathbb{P}}^{t}_{x,y}(\cdot)$ through time reversal is $\widetilde{\mathbb{P}}^{t}_{y,x}(\cdot)$. So we also have tightness on intervals 
$[1-v',1]$ where $0<v'<1$. But if $v+v'>1$, tightness on $[0,v]$ and on 
$[1-v',1]$ implies tightness on $[0,1]$. This concludes. The article \cite{ChaumontBravo2011ContMarkBridges} contains an error in the proof of the tightness of bridge measures in the neighborhood of the endpoint.
\end{proof}

\section{"Generators" with creation of mass}
\label{Ch2Sec3}

In this section we consider more general operators 
\begin{equation}
\label{Ch2Sec3: GenSignedMeasure}
L = \dfrac{1}{m(x)}\dfrac{d}{dx}\left(\dfrac{1}{w(x)}\dfrac{d}{dx}\right) + \nu
\end{equation}
with zero Dirichlet boundary conditions on $\partial I$, where $\nu$ is a signed measure on $I$ which is no longer assumed to be negative. We set
\begin{displaymath}
L^{(0)}:=L-\nu.
\end{displaymath}
In the sequel we may call $L$ "generator" 
even in case the semi-group $(e^{tL})_{t\geq 0}$ does not make sense. Our main goal in this subsection is to characterize through a positivity condition the subclass of operators of form \eqref{Ch2Sec3: GenSignedMeasure} that are equivalent up to a conjugation to the generator of a diffusion of form \eqref{Ch2Sec2: EqGenerator}.

We will consider several kinds of transformations on operators of the form \eqref{Ch2Sec3: GenSignedMeasure}. First, the conjugation. Let $u$ be a positive continuous function on $I$ such that $\frac{d^{2}u}{dx^{2}}$ is a signed measure. We call $\operatorname{Conj}(u,L)$ the operator
\begin{displaymath}
\operatorname{Conj}(u,L) = \dfrac{1}{u(x)^{2}m(x)}\dfrac{d}{dx}\left(\dfrac{u(x)^{2}}{w(x)}\dfrac{d}{dx}\right) + \nu + \dfrac{1}{u}L^{(0)}u.
\end{displaymath}
If $f$ is smooth function compactly supported in $I$ then
\begin{displaymath}
\operatorname{Conj}(u,L)f = u^{-1}L(uf).
\end{displaymath}
In case $L$ is the generator of a diffusion and $Lu=0$, $\operatorname{Conj}(u,L)$ is a Doob's h-transform
(see \cite{ChungWalsh05MP}, Chapter 11).

Second, the change of scale. If $A$ is a $\mathcal{C}^{1}$ function on $I$ such that $\frac{dA}{dx}>0$ and $\frac{d^{2}A}{dx^{2}}\in\mathbb{L}^{\infty}_{\rm loc}(I)$ and $(\gamma(t))_{0\leq t\leq T}$ a continuous path in $I$, then we will set 
$\operatorname{Scale}_{A}(\gamma)$ to be the continuous path $(A(\gamma(s)))_{0\leq t\leq T}$ in $A(I)$. Let $\operatorname{Scale}_{A}^{\rm gen}(L)$ be the operator on functions on $A(I)$ with zero Dirichlet boundary conditions induced by this change of scale:
\begin{displaymath}
\operatorname{Scale}_{A}^{\rm gen}(L) = \dfrac{1}{m\circ A^{-1}(a)}\dfrac{d}{da}\left(\dfrac{1}{w\circ A^{-1}(a)}\dfrac{d}{da}\right) + A_{\ast}\nu,
\end{displaymath}
where $A_{\ast}\nu$ is the push-forward of the measure $\nu$ by $A$.

Third, the change of time. If $V$ is positive continuous on $I$ then we can consider the change of time $ds = V(\gamma(t))dt$. 
Let $\operatorname{Speed}_{V}$ be the corresponding transformation on paths. The corresponding "generator" is $\frac{1}{V}L$.

Finally, the restriction. If $\widetilde{I}$ is an open subinterval of $I$, then set $L_{\vert \widetilde{I}}$ to be the operator $L$ acting on functions supported on $\widetilde{I}$ and with zero Dirichlet conditions on $\partial\widetilde{I}$.

For the analysis of $L$ we will use a bit of spectral theory. If $[x_{0}, x_{1}]$ is a compact interval of $\mathbb{R}$ and $\tilde{m}$, $\tilde{w}$ are positive continuous functions on $[x_{0}, x_{1}]$, then the operator $\frac{1}{\tilde{m}(x)}\frac{d}{dx}\left(\frac{1}{\tilde{w}(x)}\frac{d}{dx}\right)$ with zero Dirichlet boundary conditions has a discrete spectrum of negative eigenvalues. Let $-\tilde{\lambda}_{1}$ be the first eigenvalue. It is simple. According to Sturm-Liouville theory (see for instance \cite{Teschl2012ODE}, Section $5.5$) we have the following picture:

\begin{property}
\label{Ch2Sec3: PropertyFundEigenval}
Let $\lambda >0$ and $u$ a solution to 
\begin{displaymath}
\dfrac{1}{\tilde{m}}\dfrac{d}{dx}\left(\dfrac{1}{\tilde{w}}\dfrac{d}{dx}\right) + \lambda u =0
\end{displaymath}
with initial conditions $u(x_{0})=0, \dfrac{du}{dx}(x_{0})>0$.
\begin{itemize}
\item[(i)] If $u$ is positive on $(x_{0},x_{1})$ and $u(x_{1})=0$, then $\lambda=\tilde{\lambda}_{1}$ and $u$ is the fundamental eigenfunction.
\item[(ii)] If $u$ is positive on $(x_{0},x_{1}]$, then $\lambda < \tilde{\lambda}_{1}$.
\item[(iii)] If $u$ changes sign on $(x_{0},x_{1})$, then $\lambda > \tilde{\lambda}_{1}$.
\end{itemize}
\end{property}

Next we state and prove the main result of this section.

\begin{prop}
\label{Ch2Sec3: PropPositivity}
The following two conditions are equivalent:
\begin{itemize}
\item[(i)] There is a positive continuous function $u$ on $I$ satisfying $Lu=0$.
\item[(ii)] For any $f$ smooth compactly supported in $I$,
\begin{equation}
\label{Ch2Sec3: EqNegativity}
\int_{I}(L^{(0)}f)(x)f(x) m(x)dx + \int_{I}f(x)^{2}m(x)\nu(dx) \leq 0.
\end{equation}
\end{itemize}
\end{prop}

\begin{proof}
(i) implies (ii): First, observe that the equation $Lu=0$ reduces through a change of scale to an equation of the form 
\eqref{Ch2Sec1: Eq2ndOrderODE}. Let $u$ be given by condition (i). Let $\widetilde{L}:=\operatorname{Conj}(u,L)$. Since $Lu=0$, $\widetilde{L}$ is a generator of a diffusion without killing measure. Let $\tilde{m}(x):=u^{2}(x)m(x)$. Then for all $g$ smooth compactly supported in $I$, 
$\int_{I}(\widetilde{L}g)(x)g(x)\tilde{m}(x)dx\leq 0$. But
\begin{displaymath}
\int_{I}(\widetilde{L}g)(x)g(x)\tilde{m}(x)dx = 
\int_{I}(L^{(0)}(ug))(x)(ug)(x)m(x)dx + \int_{I}(ug)(x)^{2}m(x)\nu(dx).
\end{displaymath}
Thus \eqref{Ch2Sec3: EqNegativity} holds for all $f$ positive compactly supported in $I$ such that $u^{-1}f$ is smooth. By density arguments, this holds for general smooth $f$.

(ii) implies (i): First we will show that for every compact subinterval $J$ of $I$, there is a positive continuous function $u_{J}$ on $\mathring{J}$ satisfying $Lu_{J} =0$ on $\mathring{J}$. Let $J$ be such an interval. According to Lemma 
\ref{Ch2Sec1: LemNonDecrSol}, there is $\lambda>0$ and $u_{\lambda}$ positive continuous on $J$ satisfying ${Lu_{\lambda}-\lambda u_{\lambda} = 0}$ on $J$. Let ${L_{\lambda}:=\operatorname{Conj}(u_{\lambda},L_{\vert \mathring{J}})}$. Then,
\begin{displaymath}
L_{\lambda} = \dfrac{1}{u^{2}m}\dfrac{d}{dx}\left(\dfrac{u^{2}}{w}\dfrac{d}{dx}\right)+\lambda.
\end{displaymath}
Let $L^{(0)}_{\lambda}:=L_{\lambda}-\lambda$. $L^{(0)}$ is the generator of a diffusion on $\mathring{J}$. We can apply the standard spectral theorem to $L^{(0)}_{\lambda}$. Let $-\lambda_{1}$ be its fundamental eigenvalue.  $L^{(0)}_{\lambda} + \lambda = L_{\lambda}$ is a non-positive operator because it is a conjugate of $L_{\vert J}$ which satisfies condition (ii). This implies that $\lambda\leq\lambda_{1}$. Let $\tilde{u}$ be a solution of $L^{(0)}_{\lambda}\tilde{u} + \lambda\tilde{u}=0$ with initial conditions $\tilde{u}(\min J)=0$ and $\frac{d\tilde{u}}{dx}(\min J)>0$. Since $\lambda\leq\lambda_{1}$, according to Property \ref{Ch2Sec3: PropertyFundEigenval}, $\tilde{u}$ is positive on $\mathring{J}$. We set $u_{J}:= u_{\lambda}\tilde{u}$. Then $u_{J}$ is positive continuous on $\mathring{J}$ and satisfies $Lu_{J}=0$. This finishes the proof of the first step.

Now consider a fixed point $x_{0}$ in $I$ and $(J_{n})_{n\geq 0}$ an increasing sequence of compact subintervals of $I$ such that $x_{0}\in \mathring{J_{0}}$ and $\bigcup_{n\geq 0} J_{n} =I$. Let $u_{J_{n}}$ be a positive $L$-harmonic function on $\mathring{J_{n}}$. We may assume that ${u_{J_{n}}(x_{0})=1}$. The sequence 
$\Big(\frac{du_{J_{n}}}{dx}(x_{0}^{+})\Big)_{n\geq 0}$ is bounded from below. Otherwise, some of the $u_{J_{n}}$ would change sign on $I\cap(x_{0},+\infty)$. Similarly, since none of the $u_{J_{n}}$ changes sign on $I\cap(-\infty, x_{0})$, $\Big(\frac{du_{J_{n}}}{dx}(x_{0}^{+})\Big)_{n\geq 0}$ is bounded from above. Let $v$ be an accumulation value of the sequence 
$\Big(\frac{du_{J_{n}}}{dx}(x_{0}^{+})\Big)_{n\geq 0}$. Then the $L$-harmonic function satisfying the initial conditions $u(x_{0})=1$ and $\frac{du}{dx}(x_{0}^{+})=v$ is positive on $I$.
\end{proof}

We will divide the operators of the form \eqref{Ch2Sec3: EqNegativity} in two sets: $\mathfrak{D}^{0,-}$ for those that satisfies the constraints of the proposition \ref{Ch2Sec3: PropPositivity}, and $\mathfrak{D}^{+}$ for those that do not. $\mathfrak{D}^{0,-}$ is made exactly of operators that are equivalent up to a conjugation to the generator of a diffusion. We will subdivide the set $\mathfrak{D}^{0,-}$ in two:  $\mathfrak{D}^{-}$ for the operators that are a
conjugate of the generator of a transient diffusion, and $\mathfrak{D}^{0}$ for those that are a conjugate of the generator of a recurrent diffusion. These two subclasses are well defined since a transient diffusion can not be a conjugate of a recurrent one. Observe that each of $L\in\mathfrak{D}^{-}$, $\mathfrak{D}^{0}$ and $\mathfrak{D}^{+}$ is stable under conjugations, changes of scale and of speed. Operators in $\mathfrak{D}^{-}$ and $\mathfrak{D}^{0}$ do not need to be generators of transient or recurrent diffusions themselves. For instance consider on $\mathbb{R}$
\begin{displaymath}
L = \dfrac{1}{2}\dfrac{d^{2}}{dx^{2}} + a_{+}\delta_{1} - a_{-}\delta_{-1},
\end{displaymath}
where $a_{+}, a_{-} >0$. If $3a_{+}-a_{-}>0$ then $L\in\mathfrak{D}^{+}$, if $3a_{+}-a_{-}=0$ then $L\in\mathfrak{D}^{0}$, if $3a_{+}-a_{-}<0$ then $L\in\mathfrak{D}^{-}$. 

If $L\in \mathfrak{D}^{0,-}$, the semi-group $(e^{tL})_{t\geq 0}$ is well defined. Indeed, let $X$ be the diffusion on $I$ of generator $L^{(0)}$ and $\zeta$ the first time it hits the boundary of $I$ or blows up to infinity. Let $u$ be a positive $L$-harmonic function and $\widetilde{L}:=\operatorname{Conj}(u,L)$. $\widetilde{L}$ is the generator of a diffusion $\widetilde{X}$ on $I$ without killing measure. Let $\tilde{\zeta}$ be the first time $\widetilde{X}$ hits the boundary of $I$ or blows up to infinity. Using Girsanov's theorem, one can show that for any $F$ positive measurable functional on paths, $x\in I$ and $t>0$ the following equality holds:
\begin{multline*}
\mathbb{E}_{x}\left[1_{t<\zeta}\exp\left(\int_{I}\ell_{t}^{y}(X)m(y)\nu(dy)\right)F((X_{s})_{0\leq s\leq t})\right]=\\
\dfrac{1}{u(x)}\mathbb{E}_{x}\left[1_{t<\tilde{\zeta}}u(\widetilde{X}_{t})
F((\widetilde{X}_{s})_{0\leq s\leq t})\right].
\end{multline*}
In case $L\in\mathfrak{D}^{-}$, let $(G_{\widetilde{L}}(x,y))_{x,y\in I}$ be the Green's function of $\widetilde{L}$ relatively to the measure $u(x)^{2}m(x)dx$. Then $L$ has a Green's function $(G_{L}(x,y))_{x,y\in I}$ that equals
\begin{displaymath}
G_{L}(x,y)=\mathbb{E}_{x}\left[\int_{0}^{\zeta}\exp\left(\int_{I}\ell_{t}^{z}(X)m(z)\nu(dz)\right)d_{t}\ell^{y}_{t}(X)\right] = u(x)u(y) G_{\widetilde{L}}(x,y).
\end{displaymath}
For $L\in\mathfrak{D}^{-}$, the Green's functions $G_{L}$ satisfy the following resolvent identities:

\begin{lemm}
\label{Ch2Sec3: LemResolvent}
If $L\in\mathfrak{D}^{-}$ and $\tilde{\nu}$ is a signed measure with compact support on $I$ such that $L+\tilde{\nu}\in\mathfrak{D}^{-}$, then for all $x,y\in I$,
\begin{equation*}
\begin{split}
G_{L+\tilde{\nu}}(x,y)-G_{L}(x,y)=&\int_{I}G_{L+\tilde{\nu}}(x,z)G_{L}(z,y)m(z)\tilde{\nu}(dz)\\
=&\int_{I}G_{L}(x,z)G_{L+\tilde{\nu}}(z,y)m(z)\tilde{\nu}(dz).
\end{split}
\end{equation*}
\end{lemm}
\begin{proof}
We decompose $L$ as $L=L^{(0)}+\nu$, where $L^{(0)}$ does not contain measures and $\nu$ is a signed measure on $I$.
Let $(X_{t})_{0\leq t<\zeta}$ be the diffusion of generator $L^{(0)}$. Then
\begin{displaymath}
G_{L}(x,y)=\mathbb{E}_{x}\left[\int_{0}^{\xi}\exp\left(\int_{I}\ell_{t}^{a}(X)m(a)\nu(da)\right)
d_{t}\ell^{y}_{t}(X)\right],
\end{displaymath}
\begin{displaymath}
G_{L+\tilde{\nu}}(x,y)=\mathbb{E}_{x}\left[\int_{0}^{\xi}\exp\left(\int_{I}\ell_{t}^{a}(X)m(a)(\nu+\tilde{\nu})(da)\right)d_{t}\ell^{y}_{t}(X)\right],
\end{displaymath}
and
\begin{equation*}
\begin{split}
\exp&\left(\int_{I}\ell_{t}^{a}(X)m(a)(\nu+\tilde{\nu})(da)\right)-
\exp\left(\int_{I}\ell_{t}^{a}(X)m(a)\nu(da)\right)\\&=
\exp\left(\int_{I}\ell_{t}^{a}(X)m(a)\nu(da)\right)\times
\left(\exp \left(\int_{I}\ell_{t}^{a}(X)m(a)\tilde{\nu}(da)\right)-1\right)\\
&=\exp\left(\int_{I}\ell_{t}^{a}(X)m(a)\nu(da)\right)\int_{I}\int_{0}^{t}
\exp\left(\int_{I}\ell_{s}^{a}(X)m(a)\tilde{\nu}(da)\right)d_{s}\ell_{s}^{z}(X)m(z)\tilde{\nu}(dz).
\end{split}
\end{equation*}
Thus, $G_{L+\tilde{\nu}}(x,y)-G_{L}(x,y)$ equals
\begin{equation}
\label{Ch2Sec3: EqDiffGreen}
\mathbb{E}_{x}\left[\int_{I}\int_{0}^{\xi}\int_{0}^{t}
\exp\left(\int_{I}m(a)(\ell_{t}^{a}(X)\nu(da)+\ell_{s}^{a}(X)\tilde{\nu}(da))\right)
d_{s}\ell_{s}^{z}(X)d_{s}\ell_{t}^{y}(X)m(z)\tilde{\nu}(dz)\right].
\end{equation}
We would like to interchange $\mathbb{E}_{x}\left[\cdot\right]$ and $\int_{I}(\cdot)m(z)\tilde{\nu}(dz)$. Let
$z\in I$ and $(X^{(x)}_{t})_{0\leq t<\zeta_{x}}$, $(X^{(z)}_{t})_{0\leq t<\zeta_{z}}$ be two independent diffusions of generator $L^{(0)}$ starting in $x$ respectively $z$. Applying Markov property, we get
\begin{equation*}
\begin{split}
\mathbb{E}_{x}\Bigg[\int_{0}^{\xi}\int_{0}^{t}
\exp\bigg(\int_{I}m(a)&(\ell_{t}^{a}(X)\nu(da)+\ell_{s}^{a}(X)\tilde{\nu}(da))\bigg)
d_{s}\ell_{s}^{z}(X)d_{s}\ell_{t}^{y}(X)\Bigg]\\
=&\mathbb{E}\Bigg[\int_{0}^{\zeta_{x}}\int_{0}^{\zeta_{z}}
\exp\left(\int_{I}m(a)(\ell_{s}^{a}(X^{(x)})(\nu+\tilde{\nu})(da)\right)
\\&\times\exp\bigg(\int_{I}m(a)(\ell_{u}^{a}(X^{(z)})\nu(da)\bigg)
d_{u}\ell^{z}_{u}(X^{(z)})d_{s}\ell^{x}_{s}(X^{(x)})\Bigg]\\
=&G_{L+\tilde{\nu}}(x,z)G_{L}(z,y).
\end{split}
\end{equation*}
Since $\tilde{\nu}$ has compact support,
\begin{equation*}
\begin{split}
\mathbb{E}_{x}\Bigg[\int_{I}\int_{0}^{\xi}\int_{0}^{t}
\exp\bigg(\int_{I}m(a)&(\ell_{t}^{a}(X)\nu(da)+\ell_{s}^{a}(X)\tilde{\nu}(da))\bigg)
d_{s}\ell_{s}^{z}(X)d_{s}\ell_{t}^{y}(X)m(z)\vert\tilde{\nu}\vert(dz)\Bigg]\\=&
\int_{I}\mathbb{E}_{x}\Bigg[\int_{0}^{\xi}\int_{0}^{t}
\exp\left(\int_{I}m(a)(\ell_{t}^{a}(X)\nu(da)+\ell_{s}^{a}(X)\tilde{\nu}(da))\right)
\\&d_{s}\ell_{s}^{z}(X)d_{s}\ell_{t}^{y}(X)\Bigg]m(z)\vert\tilde{\nu}\vert(dz)
\\=&\int_{I}G_{L+\tilde{\nu}}(x,z)G_{L}(z,y)m(z)\vert\tilde{\nu}\vert(dz)<+\infty.
\end{split}
\end{equation*}
Thus, in \eqref{Ch2Sec3: EqDiffGreen} we can interchange $\mathbb{E}_{x}\left[\cdot\right]$ and $\int_{I}(\cdot)m(z)\tilde{\nu}(dz)$ and get
\begin{displaymath}
G_{L+\tilde{\nu}}(x,y)-G_{L}(x,y)=\int_{I}G_{L+\tilde{\nu}}(x,z)G_{L}(z,y)m(z)\tilde{\nu}(dz).
\end{displaymath}
Since $L$ and $L+\tilde{\nu}$ play symmetric roles, we also have
\begin{displaymath}
G_{L}(x,y)-G_{L+\tilde{\nu}}(x,y)=\int_{I}G_{L}(x,z)G_{L+\tilde{\nu}}(z,y)m(z)(-\tilde{\nu})(dz).
\qedhere
\end{displaymath}
\end{proof}

The discrete analogue of the sets $\mathfrak{D}^{-}$, $\mathfrak{D}^{0}$ and $\mathfrak{D}^{+}$ are symmetric matrices with non-negative off-diagonal coefficients inducing a connected transition graph, with the highest eigenvalue that is respectively negative, zero and positive. However, in continuous case, the sets $L\in\mathfrak{D}^{-}$, $\mathfrak{D}^{0}$ and $\mathfrak{D}^{+}$ can not be defined spectrally because for operators from $L\in\mathfrak{D}^{-}$ and $\mathfrak{D}^{+}$ the maximum of the spectrum can also equal zero. However, the next result shows that the sets $\mathfrak{D}^{-}$ and $\mathfrak{D}^{+}$ are stable under small perturbations of the measure $\nu$ and that $\mathfrak{D}^{0}$ is not.

\begin{prop}
\label{Ch2Sec3: PropStability}
\begin{itemize}
\item[(i)] If $L\in\mathfrak{D}^{0}$ and $\kappa$ is a non-zero positive Radon measure on $I$, then $L-\kappa\in\mathfrak{D}^{-}$ and $L+\kappa\in\mathfrak{D}^{+}$.
\item[(ii)] If $L\in\mathfrak{D}^{-}$ and $J$ is a compact subinterval of $I$, then there is $K>0$ such that for any positive measure $\kappa$ supported in $J$ satisfying $\kappa(J)< K$, we have $L+\kappa\in\mathfrak{D}^{-}$.
\item[(iii)] If $L\in\mathfrak{D}^{+}$, then there is $K>0$ such that for any positive finite measure $\kappa$ satisfying $\kappa(I)< K$, we have $L-\kappa\in\mathfrak{D}^{+}$.
\item[(iv)] If $L\in\mathfrak{D}^{+}$, there is a positive Radon measure $\kappa$ on $I$ such that $L-\kappa\in\mathfrak{D}^{0}$.
\item[(v)] Let $L\in\mathfrak{D}^{+}$ and $x_{0}<x_{1}\in I$. Then 
${L_{\vert (x_{0},x_{1})}\in\mathfrak{D}^{0}}$ if and only if there is an $L$-harmonic function $u$ positive on $(x_{0},x_{1})$ and zero in $x_{0}$ and $x_{1}$.
\end{itemize}
\end{prop}

\begin{proof}
(i): Consider $u$ positive continuous on $I$ such that $\operatorname{Conj}(u,L)$ is the generator of a recurrent diffusion. Since $\operatorname{Conj}(u,L-\kappa)=\operatorname{Conj}(u,L)-\kappa$, $\operatorname{Conj}(u,L-\kappa)$ is the generator of a diffusion killed at rate $\kappa$ and thus 
$L-\kappa\in\mathfrak{D}^{-}$. Similarly we can not have $L+\kappa\in\mathfrak{D}^{0,-}$ because this would mean 
$L=(L+\kappa)-\kappa\in\mathfrak{D}^{-}$.  

(ii): Without loss of generality we may assume that $L$ is the generator of a transient diffusion and that it is at natural scale, that is to say
$L=\frac{1}{m(x)}\frac{d^{2}}{dx^{2}}$. Since the diffusion is transient, $I\neq\mathbb{R}$. We may assume that $x_{0}:=\inf I > -\infty$. Write $J=[x_{1},x_{2}]$. Let $\kappa$ be a positive measure supported in $[x_{1},x_{2}]$. Let $u$ be the solution to $Lu + u\kappa=0$ with initial conditions $u(x_{0})=0, \frac{du}{dx}(x_{0}^{+})=1$. $u$ is affine on $[x_{0},x_{1}]$ and on $[x_{2},\sup I)$. On $[x_{1},x_{2}]$ $u$ is bounded from above by $x_{2}-x_{0}$. Thus, if
\begin{displaymath}
\kappa([x_{1},x_{2}])\leq \dfrac{\min_{[x_{1},x_{2}]} m}{(x_{2}-x_{0})},
\end{displaymath}
then $u$ is non-decreasing on $I$ and hence positive. This implies that $L+\kappa\in\mathfrak{D}^{0,-}$. By the point (i) of current proposition, if $\kappa([x_{1},x_{2}])<\frac{\min_{[x_{1},x_{2}]} m}{(x_{2}-x_{0})}$ then $L+\kappa\in\mathfrak{D}^{-}$.

(iii): By definition, there is $f$ smooth compactly supported in $I$ such that \eqref{Ch2Sec3: EqNegativity} does not hold for $f$. Let $U$ be the value of the left-hand side in \eqref{Ch2Sec3: EqNegativity}. $U>0$. If $\kappa$ is a positive finite measure on $I$ satisfying
\begin{displaymath}
\kappa(I)<\dfrac{U}{\Vert f \Vert_{\infty}^{2}\max_{\operatorname{Supp}(f)}m},
\end{displaymath}
then if we replace $\nu$ by $\nu-\kappa$ in \eqref{Ch2Sec3: EqNegativity}, keeping the same function $f$, we still get something positive. Thus, 
$L-\kappa\in\mathfrak{D}^{+}$.

(iv): Let $f$ be a smooth function compactly supported in $I$ such that \eqref{Ch2Sec3: EqNegativity} does not hold for $f$. Let $J$ be a compact subinterval of $I$ containing the support of $f$. The set
\begin{displaymath}
\lbrace s\in[0,1]\vert L-\nu_{+}+s\,1_{J}\nu_{+}\in\mathfrak{D}^{-}\rbrace
\end{displaymath}
is not empty because it contains $0$, and open by Proposition 
\ref{Ch2Sec3: PropStability} (ii). Let $s_{\rm max}$ by its supremum.
Then $s_{\rm max}<1$ and $L-\nu_{+}+s_{\rm max}1_{J}\nu_{+}\in\mathfrak{D}^{0}$. Then
\begin{displaymath}
\kappa:=1_{I\setminus J}\nu_{+}+(1-s_{\rm max})1_{J}\nu_{+}
\end{displaymath}
is appropriate.

(v): First, assume that there is such a function $u$. Then, by definition, 
$L_{\vert (x_{0},x_{1})}\in\mathfrak{D}^{0,-}$. $\operatorname{Conj}(u,L_{\vert (x_{0}, x_{1})})$ does not have any killing measure and the derivative of its natural scale function is $\frac{w}{u^{2}}$. It is not integrable in the neighborhood of $x_{0}$ or $x_{1}$. Thus, the corresponding diffusion never hits $x_{0}$ or $x_{1}$. This means that it is recurrent. Conversely, assume that $L_{\vert (x_{0},x_{2})}\in \mathfrak{D}^{0}$. Let $u$ be a solution to $Lu = 0$ satisfying $u(x_{0})=0$ and 
$\frac{du}{dx}(x_{0}^{+})>0$. If $u$ changed its sign on $(x_{0},x_{1})$, then according to the preceding  we would have 
$L_{\vert (x_{0},x_{1})}\in\mathfrak{D}^{+}$. If $u$ were positive on an interval larger that $(x_{0},x_{1})$, we would have 
$L_{\vert (x_{0},x_{1})}\in\mathfrak{D}^{-}$. The only possibility is that $u$ is positive on $(x_{0},x_{1})$ and zero in $x_{1}$.
\end{proof}

\chapter{Measure on loops and its basic properties}
\label{Ch3}

\section{Spaces of loops}
\label{Ch3Sec1}

In this chapter, in Section \ref{Ch3Sec3}, we will introduce the infinite measure $\mu^{\ast}$ on loops which is at the center of this work. Prior to this, in the section \ref{Ch3Sec2} we will introduce measures $\mu^{x,y}$ on finite life-time paths which will be instrumental for defining $\mu^{\ast}$. In Sections \ref{Ch3Sec4}, \ref{Ch3Sec5}, \ref{Ch3Sec7}, \ref{Ch3Sec8} will be explored different aspects of $\mu^{\ast}$. 
In Section \ref{Ch3Sec6} we will extend the Vervaat's Brownian bridge to Brownian excursion transformation to general diffusions. This generalisation can be easily interpreted in terms of measure $\mu^{\ast}$ and is related to the results of section \ref{Ch3Sec5}. In Section \ref{Ch3Sec1} we  will introduce the spaces of paths and loops on witch will be defined the measures we will consider throughout the paper. 

First we will consider continuous, time parametrized paths on $\mathbb{R}$, 
$(\gamma(t))_{0\leq t\leq T(\gamma)}$, with finite life-time $T(\gamma)\in (0,+\infty)$. Given two such paths $(\gamma(t))_{0\leq t\leq T(\gamma)}$ and $(\gamma'(t))_{0\leq t\leq T(\gamma')}$, a natural distance between them is 
\begin{displaymath}
d_{\rm paths}(\gamma,\gamma'):=\vert \log(T(\gamma))-\log(T(\gamma'))\vert + \max_{v\in [0,1]}\vert \gamma(vT(\gamma))-\gamma '(vT(\gamma')) \vert.
\end{displaymath}

A rooted loop in $\mathbb{R}$ will be a continuous finite life-time path $(\gamma(t))_{0\leq t\leq T(\gamma)}$ such that $\gamma(T(\gamma))=\gamma(0)$ and $\mathfrak{L}$ will stand for the space of such loops. $\mathfrak{L}$ endowed with the metric $d_{\rm paths}$ is a Polish space. In the sequel we will use the corresponding Borel $\sigma$-algebra, $\mathcal{B}_{\mathfrak{L}}$, for the definition of measures on  $\mathfrak{L}$. 
For $v\in[0,1]$, we define a parametrization shift transformation $\operatorname{shift}_{v}$ on $\mathfrak{L}$: 
${\operatorname{shift}_{v}(\gamma)=\tilde{\gamma}}$ where $T(\tilde{\gamma})=T(\gamma)$ and 
\begin{displaymath}
\tilde{\gamma}(t) = \left\lbrace 
\begin{array}{ll}
\gamma(vT(\gamma)+t) & \text{if}~t\leq (1-v)T(\gamma), \\ 
\gamma(t-(1-v)T(\gamma)) & \text{if}~t\geq (1-v)T(\gamma).
\end{array} 
\right. 
\end{displaymath}
We introduce an equivalence relation on $\mathfrak{L}$: $\gamma'\sim \gamma$ if $T(\gamma')=T(\gamma)$ and there is $v\in[0,1]$ such that $\gamma'=\operatorname{shift}_{v}(\gamma)$. We call the quotient space $\raisebox{0.6ex}{$\mathfrak{L}$}\diagup\raisebox{-0.6ex}{$\sim$}$ the space of unrooted loops, or just loops, and denote it $\mathfrak{L}^{\ast}$. Let $\pi$ be the projection 
$\pi: \mathfrak{L} \rightarrow \mathfrak{L}^{\ast}$. 
There is a natural metric $\delta_{\mathfrak{L}^{\ast}}$ on $\mathfrak{L}^{\ast}$:
\begin{displaymath}
d_{\mathfrak{L}^{\ast}}(\pi(\gamma),\pi(\gamma')):=\min_{v\in[0,1]}d_{\rm paths}(\operatorname{shift}_{v}(\gamma),\gamma').
\end{displaymath}
$(\mathfrak{L}^{\ast},d_{\mathfrak{L}^{\ast}})$ is a Polish space and $\pi$ is continuous. For defining measures on $\mathfrak{L}^{\ast}$ we will use its Borel $\sigma$-algebra $\mathcal{B}_{\mathfrak{L}^{\ast}}$. $\pi^{-1}(\mathcal{B}_{\mathfrak{L}^{\ast}})$, the inverse image of 
$\mathcal{B}_{\mathfrak{L}^{\ast}}$ by $\pi$, is a sub-algebra of $\mathcal{B}_{\mathfrak{L}}$.

In the sequel we will consider paths and loops that have a continuous family of local times $(\ell^{x}_{t}(\gamma))_{x\in\mathbb{R}, 0\leq t\leq T(\gamma)}$ relatively to a measure $m(x)dx$ such that for any positive measurable function $f$ on $\mathbb{R}$ and any $t\in [0,T(\gamma)]$,
\begin{displaymath}
\int_{0}^{t}f(\gamma(s))ds = \int_{I}\ell^{x}_{t}(\gamma)m(x)dx.
\end{displaymath}
We will simply write $\ell^{x}(\gamma)$ for $\ell^{x}_{T(\gamma)}(\gamma)$.

In the sequel we will also consider transformations on paths and loops and the images of different measures by these transformation. We will use everywhere the following notation: If $\mathcal{E}$ and $\mathcal{E}'$ are two measurable spaces, $F :\mathcal{E}\mapsto\mathcal{E}'$ a measurable map and $\eta$ a positive measure on $\mathcal{E}$, $F_{\ast}\eta$ will be the measure on $\mathcal{E}'$ obtained as the image of $\eta$ trough $F$.

\section{Measures $\mu^{x,y}$ on finite life-time paths}
\label{Ch3Sec2}

First we recall the framework that Le Jan used in \cite{LeJan2011Loops}: $\mathbb{G}=(V,E)$ is a finite connected undirected graph. $L_{\mathbb{G}}$ is the generator of a symmetric Markov jump process with killing on $\mathbb{G}$. $m_{\mathbb{G}}$ is the duality measure for $L_{\mathbb{G}}$. $(p^{\mathbb{G}}_{t}(x,y))_{x,y\in V,t\geq 0}$ is the family of transition densities of the jump process and $(\mathbb{P}^{\mathbb{G}, t}_{x,y}(\cdot))_{x,y\in V,t\geq 0}$ the family of bridge probability measures. The measure on rooted loops associated with $L_{\mathbb{G}}$ is
\begin{equation}
\label{Ch3Sec2: EqMesLoopsGraph}
\mu_{L_{\mathbb{G}}}(\cdot)=\int_{t>0}\sum_{x\in V}\mathbb{P}^{\mathbb{G}, t}_{x,x}(\cdot)p^{\mathbb{G}}_{t}(x,x)m_{\mathbb{G}}(x)\dfrac{dt}{t}.
\end{equation}
$\mu^{\ast}_{L_{\mathbb{G}}}$ is the image of $\mu_{L_{\mathbb{G}}}$ by the projection on unrooted loops. The definition of $\mu^{\ast}_{L_{\mathbb{G}}}$ is the exact formal analogue of the definition used in \cite{LawlerWerner2004ConformalLoopSoup} for two-dimensional Brownian loops and in \cite{Symanzik1969QFT} for massive Brownian loops in dimension four. In \cite{LeJan2011Loops} also appear variable life-time bridge measures $(\mu^{x,y}_{L_{\mathbb{G}}})_{x,y\in V}$ which are related to $\mu^{\ast}_{L_{\mathbb{G}}}$:
\begin{equation}
\label{Ch3Sec2: EqVarBridgeGraph}
\mu^{x,y}_{L_{\mathbb{G}}}(\cdot)=\int_{0}^{+\infty}\mathbb{P}^{\mathbb{G}, t}_{x,y}(\cdot)p^{\mathbb{G}}_{t}(x,y)dt.
\end{equation}
In this subsection we will define and give the important properties of the formal analogue of the measures $\mu^{x,y}_{L_{\mathbb{G}}}$ in case of one-dimensional diffusions. In the next section \ref{Ch3Sec2} we will do the same with the measure on loops $\mu^{\ast}_{L_{\mathbb{G}}}$.

$I$ is an open interval of $\mathbb{R}$. $(X_{t})_{0\leq t<\zeta}$ is a diffusion on $I$ with a generator $L$ of the form \eqref{Ch2Sec2: EqGenerator}. We use the notations of Section \ref{Ch2Sec1}. Let $x,y\in I$. Following the pattern of \eqref{Ch3Sec2: EqVarBridgeGraph}, we define:

\begin{defi}
\label{Ch3Sec2: DefBridge}
\begin{displaymath}
\mu^{x,y}_{L}(\cdot):=\int_{0}^{+\infty}\mathbb{P}_{x,y}^{t}(\cdot)p_{t}(x,y)dt.
\end{displaymath}
\end{defi}

We will write $\mu^{x,y}$ instead of $\mu^{x,y}_{L}$ whenever there is no ambiguity on $L$. The definition of $\mu^{x,y}$ depends on the choice of $m$, but $m(y)\mu^{x,y}$ does not. Measures $\mu^{x,y}$ appear in \cite{Dynkin1984Isomorphism} and enter the expression of Dynkin's isomorphism between the Gaussian free field and the local times of random paths. Pitman and Yor studied these measures in \cite{PitmanYor81BesselDiv,PitmanYor1996MaximumBridge} in the setting of one-dimensional diffusions without killing measure ($\kappa=0$). Next we give a handy representation of $\mu^{x,y}$ in the setting of one-dimensional diffusions. It was observed and proved by Pitman and Yor in case $\kappa=0$. See also \cite{SalYenYor2015IntRep}. We consider the general case.

\begin{prop}
\label{Ch3Sec2: PropLocTimeDesinteg}
Let $F$ be a non-negative measurable functional on the space of variable life-time paths starting from $x$. Then
\begin{equation}
\label{Ch3Sec2: EqBridgeLocalTime}
\mu^{x,y}(F(\gamma))=\mathbb{E}_{x}\left[\int_{0}^{\zeta} F((X_{s})_{0\leq s\leq t})d_{t}\ell^{y}_{t}(X)\right].
\end{equation}
Equivalently
\begin{displaymath}
\mu^{x,y}(F(\gamma))=\mathbb{E}_{x}\left[\int_{0}^{\ell^{y}_{\zeta}(X)} F((X_{s})_{0\leq s\leq \tau^{y}_{l}})dl\right], 
\end{displaymath}
where $\tau^{y}_{l}:=\inf\lbrace t\geq 0\vert \ell^{z}_{t}(X)> l\rbrace$.
\end{prop}

\begin{proof}
It is enough to prove this for $F$ non-negative continuous bounded functional witch takes value $0$ if either the life-time of the paths exceeds some value $t_{\rm max}<+\infty$ or of it is inferior to some value $t_{\rm min}$, or if the endpoint of the path lies out of a compact subinterval $[z_{1},z_{2}]$ of $I$. For $j\leq n\in\mathbb{N}$, set $t_{j,n}:=t_{\rm min}+\frac{j(t_{\rm max}-t_{\rm min})}{n}$ and $\Delta t_{n}:=\frac{t_{\rm max}-t_{\rm min}}{n}$. Almost surely, $\int_{0}^{\zeta} F((X_{s})_{0\leq s\leq t})d_{t}l^{y}_{t}$ is a limit as $n\rightarrow +\infty$ of
\begin{equation}
\label{Ch3Sec2: EqApproxBridge}
\sum_{j=0}^{n-1}F((X_{s})_{0\leq s\leq t_{j,n}})(\ell^{y}_{t_{j+1,n}\wedge\zeta}(X)-\ell^{y}_{t_{j,n}\wedge\zeta}(X)).
\end{equation}
Moreover, \eqref{Ch3Sec2: EqApproxBridge} is dominated by $\Vert F\Vert_{\infty}l^{y}_{t_{\rm max}\wedge\zeta}$. It follows that the expectations converge too. Using the Markov property and \eqref{Ch2Sec2: EqTempsLocSemiGroup}, we get that the expectation of \eqref{Ch3Sec2: EqApproxBridge} equals
\begin{equation}
\label{Ch3Sec2: EqExpectApproxBridge}
\sum_{j=0}^{n-1}\int_{z\in I}\int_{0}^{\Delta t_{n}}\mathbb{P}_{x,z}^{t_{j,n}}
\left(F((X_{s})_{0\leq s\leq t_{j,n}})\right) p_{t_{j,n}}(x,z)p_{r}(z,y)dr m(z)dz.
\end{equation}
Using the fact that $p_{r}(\cdot,\cdot)$ is symmetric, we can rewrite \eqref{Ch3Sec2: EqExpectApproxBridge} as
\begin{equation}
\label{Ch3Sec2: EqExpectApproxBridge2}
\int_{z_{1}}^{z_{2}}\Big(\sum_{j=0}^{n-1}\Delta t_{n}\mathbb{P}_{x,z}^{t_{j,n}}
\left(F((X_{s})_{0\leq s\leq t_{j,n}})\right)p_{t_{j,n}}(x,z)\Big)
\frac{1}{\Delta t_{n}}\int_{0}^{\Delta t_{n}}p_{r}(y,z)dr m(z)dz.
\end{equation}
As $n\rightarrow +\infty$, the measure $\frac{1}{\Delta t_{n}}\int_{0}^{\Delta t_{n}}p_{r}(y,z) dr m(z)dz$ 
converges weakly to $\delta_{y}$. Using the weak continuity of bridge probabilities 
(Proposition \ref{Ch2Sec2: PropContinuityBridge}), we get that \eqref{Ch3Sec2: EqExpectApproxBridge2} converges to
\begin{displaymath}
\int_{t_{\rm min}}^{t_{\rm max}}\mathbb{P}_{x,y}^{t}\left(F((X_{s})_{0\leq s\leq t})\right)p_{t}(x,y)dt.
\qedhere
\end{displaymath}
\end{proof}

Proposition \ref{Ch3Sec2: PropLocTimeDesinteg} also holds in case of a Markov jump processes on a graph, where the local time is replaced by the occupation time in a vertex dived by its weight. Proposition \ref{Ch3Sec2: PropLocTimeDesinteg} shows that we can consider $\mu^{x,y}$ as a measure on paths $(\gamma(t))_{0\leq t\leq T(\gamma)}$ endowed with continuous occupation densities $(\ell_{t}^{z}(\gamma))_{z\in I, 0\leq t\leq T(\gamma)}$. Next we state several properties  that follow almost immediately from Definition \ref{Ch3Sec2: DefBridge} and Proposition \ref{Ch3Sec2: PropLocTimeDesinteg}.
\begin{property}
\label{Ch3Sec2: PropertyBridges}
\begin{itemize}
\item[(i)] The total mass of the measure $\mu^{x,y}$ is finite if and only if $X$ is transient and then it equals $G(x,y)$.
If it is the case, $\frac{1}{G(x,x)}\mu^{x,x}$ is the law of $X$, starting from $X(0)=x$, up to the last time it visits $x$. $\frac{1}{G(x,y)}\mu^{x,y}$ is the law of $X$, starting from $X(0)=x$, conditioned to visit $y$ before $\zeta$, up to the last time it visits $y$.
\item[(ii)] The measure $\mu^{y,x}$ is image of the measure $\mu^{x,y}$ by time reversal.
\item[(iii)] If $\widetilde{I}$ is an open subinterval of $I$, then
\begin{displaymath}
\mu^{x,y}_{L_{\vert \widetilde{I}}}(d\gamma) = 1_{\gamma~contained~in~\tilde{I}}\mu^{x,y}_{L}(d\gamma).
\end{displaymath}
\item[(iv)] If $\tilde{\kappa}$ is a positive Radon measure on $I$, then
\begin{displaymath}
\mu^{x,y}_{L-\tilde{\kappa}}(d\gamma) = \exp\left(-\int_{I}\ell^{z}(\gamma)m(z)\tilde{\kappa}(dz)\right)
\mu^{x,y}_{L}(d\gamma).
\end{displaymath}
\item[(v)] If $A$ is a change of scale function, then
\begin{displaymath}
\mu^{A(x), A(y)}_{\operatorname{Scale}_{A}^{\rm gen} L} = \operatorname{Scale}_{A \ast}\mu^{x,y}_{L}.
\end{displaymath}
\item[(vi)] If $V$ is a positive continuous function on $I$, then for the time changed diffusion of generator $\frac{1}{V}L$,
\begin{displaymath}
\mu^{x,y}_{\frac{1}{V}L} = \operatorname{Speed}_{V \ast}\mu^{x,y}_{L}.
\end{displaymath}
\item[(vii)] If $u$ is a positive continuous function on $I$ such that $\frac{d^{2}u}{dx^{2}}$ is a signed measure and $Lu$ is a non-positive measure, then
\begin{displaymath}
\mu^{x,y}_{\operatorname{Conj}(u,L)}=\frac{1}{u(x)u(y)}\mu^{x,y}_{L}.
\end{displaymath}
\end{itemize}
\end{property}

Previous equalities depend on a particular choice of the speed measure for the modified generator. For (iv) we keep the measure $m(y)dy$. For (iii) we restrict $m(y)dy$ to $\widetilde{I}$. For (v) we choose 
$\left(\frac{dA}{dx}\circ A^{-1}\right)^{-1} m\circ A^{-1}da$. For (vi) we choose $\frac{1}{V(y)}m(y)dy$. For (vii) we choose $u(y)^{2}m(y)\,dy$. Property (ii) follows from that ${p_{t}(x,y)=p_{t}(y,x)}$ and $\mathbb{P}_{y,x}^{t}(\cdot)$ is the image of $\mathbb{P}_{x,y}^{t}(\cdot)$ by time reversal. Property (vi) is not immediate from definition $1$ because fixed times are transformed by time change in random times, but follows from proposition 
\ref{Ch3Sec2: PropLocTimeDesinteg}. Property (vii) follows from that a conjugate does not change bridge probability measures and changes the semi-group $(p_{t}(x,y)m(y)dy)_{t\geq 0, x\in I}$ to 
$(\frac{1}{u(x)}p_{t}(x,y)u(y)m(y)dy)_{t\geq 0, x\in I}$. Property (ii) was proved by Pitman and Yor in case $\kappa=0$. See \cite{PitmanYor1996MaximumBridge}. The case $\kappa\neq 0$ can be obtained through conjugation.

One can decompose the measures $\mu^{x,y}$ at the minimum of the path. This is done in 
\cite{PitmanYor1996MaximumBridge} for the case $\kappa=0$. The general case can be obtained
through conjugation. This is a generalization of Williams path decomposition \cite{Williams74Decomposition}.
See also \cite{Fitzsimmons2013ExcMin}.

\begin{property}
\label{Ch3Sec2: PropertyBridgesMin}
Let $X$ and $\widetilde{X}$ be two independent Markovian paths of generator $L$ starting from $X(0)=x$ and $\widetilde{X}(0)=y$. For $a\leq x\wedge y$, we introduce $T_{a}$ and $\widetilde{T}_{a}$ the first time $X$ respectively $\widetilde{X}$ hits $a$. Let $\mathbb{P}_{x}^{T_{a}}$ be the first passage bridge of $X$ from $x$ to $a$, conditioned by the event $T_{a}<\zeta$. Let $\widetilde{\mathbb{P}}_{y}^{\widetilde{T}_{a}}$ be the analogue for $\widetilde{X}$. Let $\widetilde{\mathbb{P}}_{y}^{\widetilde{T}_{a} \wedge}$ be the image of 
$\widetilde{\mathbb{P}}_{y}^{\widetilde{T}_{a}}$ through time reversal and 
$\mathbb{P}_{x}^{T_{a} }\lhd\widetilde{\mathbb{P}}_{y}^{\widetilde{T}_{a} \wedge}$ the image of $\mathbb{P}_{x}^{T_{a}}\otimes\widetilde{\mathbb{P}}_{y }^{\widetilde{T}_{a} \wedge}$ through concatenation at $a$ of two paths, one ending and the other starting in $a$. Then,
\begin{displaymath}
\mu^{x,y}(\cdot) = \int_{a\in I, a\leq x\wedge y}\mathbb{P}_{x}(T_{a}<\zeta)\mathbb{P}_{y}
(\widetilde{T}_{a}<\tilde{\zeta}) \left(\mathbb{P}_{x}^{T_{a}}\lhd\widetilde{\mathbb{P}}_{y }^
{\widetilde{T}_{a} \wedge}\right)(\cdot) w(a)da.
\end{displaymath}
\end{property}

Next property was given without proof by Dynkin in \cite{Dynkin1984Isomorphism}.

\begin{lemm}
\label{Ch3Sec2: LemPathUntilKilling}
Assume $\kappa\neq 0$. Let $\mathbb{P}_{x}(\cdot)$ be the law of $(X_{t})_{0\leq t<\zeta}$ where $X(0)=x$. Then
\begin{displaymath}
\int_{y\in I} \mu^{x,y}(\cdot)m(y)\kappa(dy) = 1_{X~\text{killed by}~\kappa}\mathbb{P}_{x}(\cdot).
\end{displaymath}
\end{lemm}
\begin{proof}
Let $0<t_{1}<t_{2}<\dots<t_{n}$ and let $A_{1},A_{2},\dots A_{n}, A_{n+1}$ be Borel subsets of $I$. The measure 
$\mu^{x,y}$ satisfies the following Markov property
\begin{equation*}
\begin{split}
&\mu^{x,y}(T(\gamma)>t_{n},\gamma(t_{1})\in A_{1},\dots\gamma(t_{n})\in A_{n},\gamma(T(\gamma))\in A_{n+1})=
\\&\int\limits_{A_{1}\times\dots\times A_{n}}p_{t_{1}}(x,x_{1})m(x_{1})\dots p_{t_{n}-t_{n-1}}(x_{n-1},x_{n})m(x_{n})
\mu^{x_{n},y}(T(\gamma)\in A_{n+1}) dx_{1}\dots dx_{n}
\\&=1_{y\in A_{n+1}}\int\limits_{A_{1}\times\dots\times A_{n}}p_{t_{1}}(x,x_{1})m(x_{1})\dots p_{t_{n}-t_{n-1}}(x_{n-1},x_{n})m(x_{n})G(x_{n},y) dx_{1}\dots dx_{n}.
\end{split}
\end{equation*}
Hence,
\begin{multline}
\label{Ch3Sec2: EqMarginalsMu}
\int_{y\in I} \mu^{x,y}(T(\gamma)>t_{n},\gamma(t_{1})\in A_{1},\dots\gamma(t_{n})\in A_{n},
\gamma(T(\gamma))\in A_{n+1})m(y)\kappa(dy)=
\\ \int\limits_{A_{1}\times\dots\times A_{n+1}}p_{t_{1}}(x,x_{1})m(x_{1})\dots p_{t_{n}-t_{n-1}}(x_{n-1},x_{n})m(x_{n})
G(x_{n},y)m(y) dx_{1}\dots dx_{n}\kappa(dy).
\end{multline}
From Markov property of $X$ follows
\begin{multline*}
\mathbb{P}_{x}(\zeta > t_{n}, X_{t_{1}}\in A_{1},\dots, X_{t_{n}}\in A_{n},X_{\zeta^{-}}\in A_{n+1})=
\\\int\limits_{A_{1}\times\dots\times A_{n}}p_{t_{1}}(x,x_{1})m(x_{1})\dots p_{t_{n}-t_{n-1}}(x_{n-1},x_{n})m(x_{n})
P_{x_{n}}(X_{\zeta^{-}}\in A_{n+1}) dx_{1}\dots dx_{n}.
\end{multline*}
Since the distribution of $X_{\zeta^{-}}$ on the event of $X$ killed by $\kappa$ is 
$1_{y\in I}G(X_{0},y)m(y)\kappa(dy)$, we get
\begin{multline}
\label{Ch3Sec2: EqMarginalsX}
\mathbb{P}_{x}(\zeta > t_{n}, X_{t_{1}}\in A_{1},\dots, X_{t_{n}}\in A_{n},X_{\zeta^{-}}\in A_{n+1})=
\\\int\limits_{A_{1}\times\dots\times A_{n+1}}p_{t_{1}}(x,x_{1})m(x_{1})\dots p_{t_{n}-t_{n-1}}(x_{n-1},x_{n})m(x_{n})
G(x_{n},y)m(y) dx_{1}\dots dx_{n}\kappa(dy).
\end{multline}
The equality between \eqref{Ch3Sec2: EqMarginalsMu} and \eqref{Ch3Sec2: EqMarginalsX} implies the lemma.
\end{proof}

Next we study the continuity of $(x,y)\mapsto \mu^{x,y}$.

\begin{lemm}
\label{Ch3Sec2: LemUniformConvLocTime}
Let $J$ be a compact subinterval of $I$. Then the family of local times of $X$ satisfies: for every $\varepsilon>0$,
\begin{displaymath}
\lim_{t\rightarrow 0^{+}}\sup_{x\in J}\mathbb{P}_{x}\left(\sup_{y\in I}\ell^{y}_{t\wedge\zeta}(X)>\varepsilon\right) =0.
\end{displaymath}
\end{lemm}

\begin{proof}
It is enough to prove it in case the killing measure $\kappa$ is zero because adding a killing measure only lowers $\ell^{y}_{t\wedge\zeta}(X)$. Without loss of generality we may also assume that the diffusion is on its natural scale, that is to say $w\equiv 2$. Then, $X$ is just a time changed Brownian motion on some open subinterval of $\mathbb{R}$. For a Brownian motion $(B_{t})_{t\geq 0}$ the statement is clear. In this case $\mathbb{P}_{x}\left(\sup_{y\in \mathbb{R}}\ell^{y}_{t\wedge\zeta}(B)>\varepsilon\right)$ does not depend on $x$ and for a given $x$,
\begin{displaymath}
\lim_{t\rightarrow 0^{+}}\mathbb{P}_{x}\left(\sup_{y\in \mathbb{R}}\ell^{y}_{t\wedge\zeta}(B)>\varepsilon\right) =0.
\end{displaymath}
Otherwise, let 
\begin{displaymath}
\mathcal{I}_{t}:=\int_{0}^{t}m(X_{s})ds.
\end{displaymath}
Then, given the time change that transforms $X$ into a Brownian motion $B$, we have
\begin{displaymath}
\ell^{y}_{t}(X)=\ell^{y}_{\mathcal{I}_{t}}(B).
\end{displaymath}
Let $J=[x_{0}, x_{1}]$. Let $x_{\rm min}\in I$, $x_{\rm min}<x_{0}$ and $x_{\rm max}\in I$, $x_{\rm max}>x_{1}$. Let $T_{x_{\rm min},x_{\rm max}}$ be the first time $X$ hits either $x_{\rm min}$ or $x_{\rm max}$. Let $s>0$, $\varepsilon>0$ and $x\in J$. If 
$t\leq \frac{s}{\max_{[x_{\rm min},x_{\rm max}]} m}$, then on the event $T_{x_{\rm min},x_{\rm max}}\geq t$, $\mathcal{I}_{t}$ is less or equal to $s$. So, for $t$ small enough,
\begin{displaymath}
\mathbb{P}_{x}\left(\sup_{y\in I}\ell^{y}_{t\wedge\zeta}(X)>\varepsilon\right)\leq 
\mathbb{P}_{x}\left(\sup_{y\in \mathbb{R}}\ell^{y}_{s}(B)>\varepsilon\right)+ \mathbb{P}_{x}\left(T_{x_{\rm min},x_{\rm max}}<t\right).
\end{displaymath}
But,
\begin{displaymath}
\mathbb{P}_{x}\left(T_{x_{\rm min},x_{\rm max}}<t\right) =
\mathbb{P}_{x_{0}}\left(T_{x_{\rm min},x_{\rm max}}<t\right) +
\mathbb{P}_{x_{1}}\left(T_{x_{\rm min},x_{\rm max}}<t\right),
\end{displaymath}
and
\begin{displaymath}
\lim_{t\rightarrow 0^{+}}\sup_{x\in J} \mathbb{P}_{x}\left(T_{x_{\rm min},x_{\rm max}}<t\right) = 0.
\end{displaymath}
Thus,
\begin{displaymath}
\limsup_{t\rightarrow 0^{+}}\sup_{x\in J}\mathbb{P}_{x}\left(\sup_{y\in I}\ell^{y}_{t\wedge\zeta}(X)>\varepsilon\right)\leq
\mathbb{P}_{x}\left(\sup_{y\in \mathbb{R}}\ell^{y}_{s}(B)>\varepsilon\right).
\end{displaymath}
Letting $s$ go to $0$ we get the statement of the lemma.
\end{proof}

\begin{prop}
\label{Ch3Sec2: PropContinuityBridge}
Let $t_{\rm max}>0$. Let $F$ be a bounded functional on finite life-time paths endowed with continuous local times that depends continuously on the path $(\gamma_{t})_{0\leq t\leq T(\gamma)}$ and on $(l^{x}_{T(\gamma)}(\gamma))_{x\in I}$, where we take the topology of uniform convergence for the occupation densities on $I$. On top of that we assume that $F$ is zero if $T(\gamma)>t_{\rm max}$. Then the function $(x,y)\mapsto \mu^{x,y}(F(\gamma))$ is continuous on $I\times I$.
\end{prop}

\begin{proof}
If we had assumed that $F$ does only depend on the path regardless to its occupation field then the continuity of $(x,y)\mapsto \mu^{x,y}(F(\gamma))$ would just be a consequence of the continuity of transition densities and of the weak continuity of bridge probability measures. For our proof we further assume that $L$ does not contain any killing measure. If this is not the case, then we can consider a continuous positive $L$-harmonic function $u$. Then $\operatorname{Conj}(u,L)$ does not contain any killing measure and up to a continuous factor $u(x)u(y)$ gives the same measure $\mu^{x,y}$ (Property 
\ref{Ch3Sec2: PropertyBridges} (vii)). We will mainly rely on the representation given by Proposition 
\ref{Ch3Sec2: PropLocTimeDesinteg}.

Let $x,y\in I$ and $(x_{j},y_{j})_{j\geq 0}$ a sequence in $I\times I$ converging to $(x,y)$. Without loss of generality, we assume that $(x_{j})_{j\geq 0}$ is increasing. We consider sample paths $(X_{t})_{0\leq t<\zeta}$ and $(X^{(j)}_{t})_{0\leq t<\zeta_{j}}$ of the diffusion of generator $L$ starting from $x$ and each of $x_{j}$, coupled on a same probability space in the following way: First we sample $X$ starting from $x$. Then we sample $X^{(0)}$ starting from $x_{0}$. It starts independently from $X$ until the first time $X^{(0)}_{t}=X_{t}$. After that time $X^{(0)}$ sticks to $X$. This two paths may never meet if one of them dies to early. If 
$X,X^{(0)},\dots,X^{(j)}$ are already sampled, we start $X^{(j+1)}$ from $x_{j+1}$ independently from the preceding sample paths until it meets one of them. After that time $X^{(j+1)}$ sticks to the path it has met. Let
\begin{displaymath}
T^{(j)}:=\inf\lbrace t\geq 0\vert X^{(j)}_{t}=X_{t}\rbrace.
\end{displaymath} 
If $X^{(j)}$ does not meet $X$, we set $T^{(j)}=+\infty$. By construction, $(T^{(j)})_{j\geq 0}$ is a non-increasing sequence. Here we use that there is no killing measure. $T^{(j)}$ is equal in law to the first time two independent sample paths of the diffusion, one starting from $x$ and the other from $x_{j}$, meet. Thus, the sequence $(T^{(j)})_{j\geq 0}$ converges to $0$ in probability. Since it is decreasing, it converges almost surely to $0$.

We use reduction to absurdity. The sequence $(\mu^{x_{j},y_{j}}(F(\gamma)))_{j\geq 0}$ is bounded because $F$ is bounded and zero on paths with life-time greater then $t_{\rm max}$. Assume that it does not converge to $\mu^{x,y}(F(\gamma))$. Then there is a subsequence that converges to a value other than $\mu^{x,y}(F(\gamma))$. We may as well assume that the whole sequence $(\mu^{x_{j},y_{j}}(F(\gamma)))_{j\geq 0}$ converges to a value $v\neq\mu^{x,y}(F(\gamma))$. According to Lemma \ref{Ch3Sec2: LemUniformConvLocTime}, the sequence $((\ell^{z}_{T^{(j)}}(X^{(j)}))_{z\in I})_{j\geq 0}$ of occupation density functions converges in probability to the null function. Thus, there is an extracted  subsequence $((\ell^{z}_{T^{(j_{n})}}(X^{(j_{n})}))_{z\in I})_{n\geq 0}$ that converges almost surely uniformly to the null function. We will show that $(\mu^{x_{j_{n}},y_{j_{n}}}(F(\gamma)))_{n\geq 0}$ converges to $\mu^{x,y}(F(\gamma))$ and obtain a contradiction.

For $z\in I$ and $l>0$ let
\begin{displaymath}
\tau^{z}_{l}:=\inf\lbrace t\geq 0\vert \ell^{z}_{t}(X)> l\rbrace
\end{displaymath}
and 
\begin{displaymath}
\tau^{z}_{j,l}:=\inf\lbrace t\geq 0\vert \ell^{z}_{t}(X^{(j)})> l\rbrace.
\end{displaymath}
Then, according to Proposition \ref{Ch3Sec2: PropLocTimeDesinteg},
\begin{displaymath}
\mu^{x,y}(F(\gamma))=\mathbb{E}\bigg[\int_{0}^{\ell^{y}_{t_{\rm max}\wedge\zeta}(X)}
F((X_{s})_{0\leq s\leq }\tau^{y}_{l})dl\bigg], 
\end{displaymath}
\begin{displaymath}
\mu^{x_{j},y_{j}}(F(\gamma))=\mathbb{E}\bigg[\int_{0}^{\ell^{y_{j}}_{t_{\rm max}\wedge\zeta_{j}}(X^{(j)})}F((X^{(j)}_{s})_{0\leq s\leq }\tau^{y_{j}}_{j,l})dl\bigg].
\end{displaymath}
For any $z\in I$, if $\tau^{z}_{j,l}\in [T^{(j)},\zeta_{j})$, then $\tau^{z}_{j,l} = \tau^{z}_{l'}$, where 
\begin{displaymath}
l' = l + \ell^{z}_{T^{(j)}}(X)-\ell^{z}_{T^{(j)}}(X^{(j)}). 
\end{displaymath}
Along the subset of indices $(j_{n})_{n\geq 0}$, $\tau^{y_{j_{n}}}_{j_{n},l}$ converges to $\tau^{y}_{l}$ for every $l\in(0, l^{y}_{\zeta}(X))$ except possibly the countable set of values of $l$ where $l\mapsto \tau^{y}_{j,l}$ jumps. For any $l$ such that $\tau^{y_{j_{n}}}_{j_{n},l}$ converges to $\tau^{y}_{l}$, the path $(X^{(j)}_{s})_{0\leq s\leq \tau^{y_{j_{n}}}_{j_{n},l}}$ converges to the path $(X_{s})_{0\leq s\leq }\tau^{y}_{l}$. Moreover, for such $l$, the occupation densities $(l^{z}_{\tau^{y_{j_{n}}}_{j_{n},l}}(X^{(j_{n})}))_{z\in I}$ converge uniformly to $(l^{z}_{\tau^{y}_{l}}(X))_{z\in I}$. Indeed,
\begin{displaymath}
\ell^{z}_{\tau^{y_{j_{n}}}_{j_{n},l}}(X^{(j_{n})}) = \ell^{z}_{\tau^{y_{j_{n}}}_{j_{n},l}}(X)-\ell^{z}_{T^{(j)}}(X) + \ell^{z}_{T^{(j)}}(X^{(j_{n})}).
\end{displaymath}
Thus, for all $l\in(0, \ell^{y}_{\zeta}(X))$, except possibly countably many,
\begin{displaymath}
\lim_{n\rightarrow +\infty}F((X^{(j_{n})}_{s})_{0\leq s\leq }\tau^{y_{j_{n}}}_{j_{n}, l}) = 
F((X_{s})_{0\leq s\leq }\tau^{y}_{l}).
\end{displaymath}
For $n$ large enough, $\zeta_{j}=\zeta$ and $\ell^{y_{j_{n}}}_{t_{\rm max}\wedge\zeta_{j_{n}}}(X^{(j_{n})})$ converges to $\ell^{y}_{t_{\rm max}\wedge\zeta}(X)$. It follows that the following almost sure convergence holds
\begin{equation}
\label{Ch3Sec2: EqConvF}
\lim_{n\rightarrow+\infty}\int_{0}^{\ell^{y_{j_{n}}}_{t_{\rm max}\wedge\zeta_{j_{n}}}(X^{(j_{n})})}
F((X^{(j_{n})}_{s})_{0\leq s\leq }\tau^{y_{j_{n}}}_{j_{n}, l})dl=
\int_{0}^{l^{y}_{t_{\rm max}\wedge\zeta}(X)}F((X_{s})_{0\leq s\leq }\tau^{y}_{l})dl.
\end{equation}
The left-hand side of \eqref{Ch3Sec2: EqConvF} is dominated by $\Vert F\Vert_{+\infty} \ell^{y_{j_{n}}}_{t_{\rm max}\wedge\zeta_{j_{n}}}(X^{(j_{n})})$. In order to conclude that the almost sure convergence 
\eqref{Ch3Sec2: EqConvF} is also an $\mathbb{L}^{1}$ convergence we need only to show that
\begin{equation}
\label{Ch3Sec2: EqConvDom}
\mathbb{E}\left[\vert \ell^{y_{j_{n}}}_{t_{\rm max}\wedge\zeta_{j_{n}}}(X^{(j_{n})}) -
 \ell^{y}_{t_{\rm max}\wedge\zeta}(X) \vert \right] = 0.
\end{equation}
We already know that $\ell^{y_{j_{n}}}_{t_{\rm max}\wedge\zeta_{j_{n}}}(X^{(j_{n})})$ converges almost surely to $\ell^{y}_{t_{\rm max}\wedge\zeta}(X)$. Moreover,
\begin{displaymath}
\mathbb{E}\left[\ell^{y_{j_{n}}}_{t_{\rm max}\wedge\zeta_{j_{n}}}(X^{(j_{n})})\right]=
\int_{0}^{t_{\rm max}}p_{t}(x_{j_{n}},y_{j_{n}}),
\end{displaymath} 
and
\begin{displaymath}
\mathbb{E}\left[\ell^{y}_{t_{\rm max}\wedge\zeta}(X)\right]=
\int_{0}^{t_{\rm max}}p_{t}(x,y).
\end{displaymath}
It follows that the expectations converge. By Scheffe's lemma, the $\mathbb{L}^{1}$ convergence 
\eqref{Ch3Sec2: EqConvDom} holds.

We have shown that there is always a subsequence $(\mu^{x_{j_{n}},y_{j_{n}}}(F(\gamma)))_{n\geq 0}$ that converges to $\mu^{x,y}(F(\gamma))$, which contradicts the convergence of $(\mu^{x_{j},y_{j}}(F(\gamma)))_{j\geq 0}$ to a different value.
\end{proof}

\section{The measure $\mu^{\ast}$ on unrooted loops}
\label{Ch3Sec3}

The measure $\mu^{x,x}$ can be seen as a measure on the space of rooted loops $\mathfrak{L}$. Next we define a natural measure $\mu^{\ast}_{L}$ on $\mathfrak{L}^{\ast}$ following the pattern \eqref{Ch3Sec2: EqMesLoopsGraph}.

\begin{defi}
\label{Ch3Sec3: DefMesLoops}
Let $\mu_{L}$ be the following measure on $\mathfrak{L}$:
\begin{displaymath}
\mu_{L}(d\gamma):=\int_{t>0}\int_{x\in I}\mathbb{P}_{x,x}^{t}(d\gamma)p_{t}(x,x)m(x) dx \dfrac{dt}{t} = \dfrac{1}{T(\gamma)}\int_{x\in I} \mu^{x,x}_{L}(d\gamma) m(x) dx.
\end{displaymath}
$\mu^{\ast}_{L}:=\pi_{\ast} \mu_{L}$ is a measure on $\mathfrak{L}^{\ast}$.
\end{defi}

We will drop the subscript $L$ whenever there is no ambiguity on $L$. Definition \ref{Ch3Sec3: DefMesLoops} does not depend on the choice of the speed measure $m(x)\,dx$. The measures $\mu$ and $\mu^{\ast}$ are $\sigma$-finite but not finite. They satisfy the following elementary properties:

\begin{property}
\label{Ch3Sec3: PropertyTrivialitiesLoops}
\begin{itemize}
\item[(i)] $\mu$ is invariant by time reversal.
\item[(ii)] If $\widetilde{I}$ is an open subinterval of $I$, then
\begin{displaymath}
\mu_{L_{\vert \widetilde{I}}}(d\gamma) = 1_{\gamma~contained~in~\tilde{I}}~\mu_{L}(d\gamma).
\end{displaymath}
\item[(iii)] If $\tilde{\kappa}$ is a positive Radon measure on $I$, then
\begin{displaymath}
\mu_{L-\tilde{\kappa}}(d\gamma) = \exp\left(-\int_{I}\ell^{z}(\gamma)m(z)\tilde{\kappa}(dz)\right) \mu_{L}(d\gamma).
\end{displaymath}
\item[(iv)] If $A$ is a change of scale function, then
\begin{displaymath}
\mu_{\operatorname{Scale}_{A}^{\rm gen} L} = \operatorname{Scale}_{A \ast}\mu_{L}.
\end{displaymath}
\item[(v)] If $u$ is a positive continuous function on $I$ such that $\frac{d^{2}u}{dx^{2}}$ is a signed measure and $Lu$ is a non-positive measure, then
\begin{displaymath}
\mu_{\operatorname{Conj}(u,L)}=\mu_{L}.
\end{displaymath}
\end{itemize}
Same properties hold for $\mu^{\ast}$.
\end{property}

The measures $\mu$ and $\mu^{\ast}$ contain some information on the diffusion $X$ but the invariance by conjugation (Property \ref{Ch3Sec3: PropertyTrivialitiesLoops} (v)) shows that they do not capture its asymptotic behavior. In Section \ref{Ch3Sec4} we will prove a converse to Property \ref{Ch3Sec3: PropertyTrivialitiesLoops} (v). In our setting, most important examples of conjugates are:
\begin{itemize}
\item The Bessel $3$ process on $(0,+\infty)$ is a conjugate of the Brownian motion on $(0,+\infty)$, killed when hitting $0$, through the function $x\mapsto x$.
\item The Brownian motion on $\mathbb{R}$ killed with uniform rate $\kappa dx$ (i.e. $\kappa$ constant) is a conjugate of the drifted Brownian motion on $\mathbb{R}$ with constant drift $\sqrt{2\kappa}$, through the function 
$x\mapsto e^{-\sqrt{2 \kappa}x}$.
\end{itemize}

In the sequel we will be interested mostly in $\mu^{\ast}$ and not $\mu$. As it will be clear from next propositions, the measure $\mu^{\ast}$ has some nice features that $\mu$ does not.

\begin{prop}
\label{Ch3Sec3: PropShiftInv}
Let $v\in[0,1]$. Then $\operatorname{shift}_{v \ast} \mu = \mu$. In particular,
\begin{equation}
\label{Ch3Sec3: EqCoupure}
\mu(\cdot) = \int_{v\in[0,1]} \operatorname{shift}_{v \ast} \mu(\cdot) dv. 
\end{equation}
\end{prop}

\begin{proof}
For a rooted loop $\gamma$ of life-time $T(\gamma)$, we will introduce $\gamma_{1}$ the path restricted to time interval $[0,vT(\gamma)]$ and $\gamma_{2}$ the path restricted to $[vT(\gamma), T(\gamma)]$. By bridge decomposition property, the measure $\mu(d\gamma_{1},d\gamma_{2})$ equals
\begin{displaymath}
\int_{t>0}\int_{I}\int_{I}\mathbb{P}_{x,y}^{vt}(d\gamma_{1})\mathbb{P}_{y,x}^{(1-v)t}(d\gamma_{2})p_{vt}(x,y) p_{(1-v)t}(y,x)m(y)\,dy\,m(x)\,dx\,\dfrac{dt}{t}.
\end{displaymath}
Since $\gamma_{1}$ and $\gamma_{2}$ play symmetric roles, changing the order of $\gamma_{1}$ and $\gamma_{2}$ does not change the measure $\mu$.
\end{proof}

Formula \eqref{Ch3Sec3: EqCoupure} shows that we can get back to the measure $\mu$ from the measure $\mu^{\ast}$ by cutting the circle parametrizing a loop in $\mathfrak{L}^{\ast}$ in a point chosen uniformly on this circle, in order to separate the start from the end. 

\begin{coro}
\label{Ch3Sec3: CorProjSubalg}
Let $F$ be a positive measurable functional on $\mathfrak{L}$. Then, the map 
$\gamma\mapsto\int_{0}^{1}F(\operatorname{shift}_{v}(\gamma))dv$ is 
$\pi^{-1}(\mathcal{B}_{\mathfrak{L}^{\ast}})$-measurable and 
\begin{displaymath}
\dfrac{d(F(\gamma)\mu)}{d\mu}_{\vert \pi^{-1}(\mathcal{B}_{\mathfrak{L}^{\ast}})} = \int_{0}^{1}F(\operatorname{shift}_{v}(\gamma))dv.
\end{displaymath}
\end{coro}

\begin{proof}
We need only to show that for every $F'$ measurable functional on $\mathfrak{L}^{\ast}$,
\begin{equation}
\label{Ch3Sec3: EqRootedLoopsConditional}
\int_{\mathfrak{L}}F(\gamma)F'(\pi(\gamma))\mu(d\gamma) = \int_{0}^{1}\int_{\mathfrak{L}}
F(\operatorname{shift}_{v}(\gamma))F'(\pi(\gamma))\mu(d\gamma)dv.
\end{equation}
From Proposition \ref{Ch3Sec3: PropShiftInv} follows that for every $v\in[0,1]$,
\begin{equation}
\label{Ch3Sec3: EqDualityShiftv}
\int_{\mathfrak{L}}F(\gamma)F'(\pi(\gamma))\mu(d\gamma) = 
\int_{\mathfrak{L}}F(\operatorname{shift}_{v}(\gamma))F'(\pi(\gamma))\mu(d\gamma).
\end{equation}
Integrating \eqref{Ch3Sec3: EqDualityShiftv} on $[0,1]$ leads to \eqref{Ch3Sec3: EqRootedLoopsConditional}.
\end{proof}

The next identity appears in \cite{LeJan2011Loops} in the setting of Markov jump processes on graphs. It can be generalized to a wider class of Markov processes admitting local times (see Lemma $2.2$ in \cite{FitzsimmonsRosen2012LoopsIsomorphism}). We will give a short proof that suits our framework.
\begin{coro}
\label{Ch3Sec3: CorLoopLocTimeBridge}
Let $x\in I$. Then,
\begin{equation}
\label{Ch3Sec3: EqLoopMeasureLocTime}
\ell^{x}(\gamma) \mu^{\ast}(d\gamma) = \pi_{\ast} \mu^{x,x}(d\gamma).
\end{equation}
For $l>0$, let $\mathbb{P}_{x}^{\tau^{x}_{l}}(\cdot)$ be the law of the sample paths of a diffusion $X$ of generator $L$, started from $x$, until the time $\tau^{x}_{l}$ when $\ell^{x}_{t}(X)$ hits $l$, conditioned by $\tau^{x}_{l}<\zeta$. Then,
\begin{equation}
\label{Ch3Sec3: EqDecompLoopLocTimex}
1_{\gamma~visits~x}\mu^{\ast}(d\gamma)=\int_{0}^{+\infty}\pi_{\ast}\mathbb{P}_{x}^{\tau^{x}_{l}}(d\gamma)e^{-\frac{l}{G(x,x)}} \dfrac{dl}{l}.
\end{equation}
By convention, we set $G(x,x)=+\infty$ if $X$ is recurrent.
\end{coro}

\begin{proof}
Let $\varepsilon>0$ such that $[x-\varepsilon, x+\varepsilon]\subseteq I$. Let $T_{[x-\varepsilon, x+\varepsilon]}(\gamma)$ be the time a loop $\gamma$ spends in $[x-\varepsilon, x+\varepsilon]$. From the identity 
\eqref{Ch3Sec3: EqCoupure} follows that
\begin{displaymath}
\dfrac{T_{[x-\varepsilon, x+\varepsilon]}(\gamma)}{T(\gamma)} \mu^{\ast}(d\gamma) = 
\dfrac{1}{T(\gamma)}\int_{x-\varepsilon}^{x+\varepsilon}\pi_{\ast} \mu^{z,z}(d\gamma) m(z) dz,
\end{displaymath}
and simplifying $T(\gamma)$,
\begin{displaymath}
T_{[x-\varepsilon, x+\varepsilon]}(\gamma)\mu^{\ast}(d\gamma) = \int_{x-\varepsilon}^{x+\varepsilon}\pi_{\ast} \mu^{z,z}(d\gamma) m(z) dz.
\end{displaymath}
Using local times we rewrite the previous expression as
\begin{equation}
\label{Ch3Sec3: EqTimeSpendIntervalLocTime}
\dfrac{\int_{x-\varepsilon}^{x+\varepsilon}\ell^{z}(\gamma)m(z) dz}{\int_{x-\varepsilon}^{x+\varepsilon}m(z) dz} \mu^{\ast}(d\gamma) = \dfrac{1}{\int_{x-\varepsilon}^{x+\varepsilon}m(z) dz} \int_{x-\varepsilon}^{x+\varepsilon}\pi_{\ast} \mu^{z,z}(d\gamma) m(z) dz.
\end{equation}
Let $\varepsilon_{0}>0$ such that $[x-\varepsilon_{0}, x+\varepsilon_{0}]\subseteq I$. Let $F$ be a continuous bounded functional on loops endowed with continuous local times such that $F$ is zero if the life-time of the loop exceeds $t_{\rm max}>0$ and if $\sup_{z\in[x-\varepsilon_{0}, x+\varepsilon_{0}]} l^{z}(\gamma)$ exceeds $l_{\rm max}$. According to Proposition \ref{Ch2Sec2: PropContinuityBridge}, the right-hand side of \eqref{Ch3Sec3: EqTimeSpendIntervalLocTime} applied to $F$ converges as $\varepsilon\rightarrow 0$ to $(\pi_{\ast} \mu^{x,x})(F(\gamma))$. By dominated convergence 
it follows that the left-hand side of \eqref{Ch3Sec3: EqTimeSpendIntervalLocTime} applied to $F$ converges as 
$\varepsilon\rightarrow 0$ to
\begin{displaymath}
\int_{\mathfrak{L}^{\ast}}\ell^{x}(\gamma)F(\gamma)\mu^{\ast}(d\gamma).
\end{displaymath}
Thus, we have the equality
\begin{equation}
\label{Ch3Sec3: EqLoopMeasureLocTimeInt}
\int_{\mathfrak{L}^{\ast}}\ell^{x}(\gamma)F(\gamma)\mu^{\ast}(d\gamma) = (\pi_{\ast} \mu^{x,x})(F(\gamma)).
\end{equation}
The set of test functionals $F$ that satisfy \eqref{Ch3Sec3: EqLoopMeasureLocTimeInt} is large enough to deduce the equality \eqref{Ch3Sec3: EqLoopMeasureLocTime} between measures.

From Proposition \ref{Ch3Sec2: PropLocTimeDesinteg} follows that
\begin{displaymath}
\mu^{x,x}(\cdot)=\int_{0}^{+\infty}\mathbb{P}_{x}^{\tau^{x}_{l}}(\cdot)e^{-\frac{l}{G(x,x)}} dl.
\end{displaymath}
Applying \eqref{Ch3Sec3: EqLoopMeasureLocTime} to the above disintegration, we get 
\eqref{Ch3Sec3: EqDecompLoopLocTimex}. 
\end{proof}

\begin{coro}
\label{Ch3Sec3: CorTimeChange}
Let $V$ be a positive continuous function on $I$. We consider a time change with speed $V$: $ds=V(x)dt$. Then
\begin{equation}
\label{Ch3Sec3: EqLoopTimeChange}
\mu^{\ast}_{\frac{1}{V}L} = \operatorname{Speed}_{V \ast} \mu^{\ast}_{L}.
\end{equation}
\end{coro}

\begin{proof}
By Definition \ref{Ch3Sec3: DefMesLoops} and Property \ref{Ch3Sec2: PropertyBridges} (vi),
\begin{displaymath}
\mu_{\frac{1}{V}L}(d\gamma) = \frac{1}{T(\gamma)}\int_{0}^{T(\gamma)}\frac{V(\gamma(0))}{V(\gamma(s))} ds~\operatorname{Speed}_{V \ast} (\mu_{L}(d\gamma)).
\end{displaymath}
Applying Corollary \ref{Ch3Sec3: CorProjSubalg}, we obtain
\begin{displaymath}
\dfrac{d \operatorname{Speed}_{V \ast} \mu_{L}}{d \mu_{\frac{1}{V}L}}_{\vert \pi^{-1}(\mathcal{B}_{\mathfrak{L}^{\ast}})} = \frac{\int_{0}^{1}V^{-1}(\gamma(v T(\gamma))) dv}{\frac{1}{T(\gamma)}\int_{0}^{T(\gamma)}V^{-1}(\gamma(s)) ds} = 1.
\end{displaymath}
This concludes.
\end{proof}

In dimension two, the time change covariance of the measure $\mu^{\ast}$ on loops plays a key role for the construction of the Conformal Loop Ensembles (CLE) using loop-soups as in \cite{SheffieldWerner2012CLE}: Let  $D$ be an open domain of the complex plane, $(B_{t})_{0\leq t <\zeta}$ the two-dimensional standard Brownian motion in $D$ killed when hitting $\partial D$ and $\mu^{\ast}$ the corresponding measure on loops. If $f:D\rightarrow D$ is a conformal map, then $(f(B_{t}))_{0\leq t <\zeta}$ is a time changed Brownian motion. If we consider $\mu^{\ast}$ not as a measure on loops parametrized by time but a measure on the geometrical drawings of loops, then $\mu^{\ast}$ is invariant by the transformation ${(\gamma(t))_{0\leq t\leq T(\gamma)}\mapsto (f(\gamma(t)))_{0\leq t\leq T(\gamma)}}$. This is proved in \cite{LawlerWerner2004ConformalLoopSoup}.

Given that $\mu^{\ast}$ is invariant through conjugation and covariant with the change of scale and change of time, if $X$ is a recurrent diffusion, then up to a change of scale and time, $\mu^{\ast}$ is the same as for the Brownian motion on $\mathbb{R}$, and if $X$ is a transient diffusion, even if the killing measure $\kappa$ is non-zero, then up to a change of scale and time, $\mu^{\ast}$ is the same as for the Brownian motion on a bounded interval, killed when it hits the boundary.

\section{Multiple local times}
\label{Ch3Sec4}

In this subsection we define the multiple local time functional on loops. Corollary \ref{Ch3Sec3: CorLoopLocTimeBridge} gives a link between the measure $\mu^{\ast}$ and the measures $(\mu^{x,x})_{x\in I}$. Using multiple local times we will get a further relation between $\mu^{\ast}$ and $(\mu^{x,y})_{x,y\in I}$. This will allow us to prove a converse to Pproperty \ref{Ch3Sec3: PropertyTrivialitiesLoops} (v): two diffusions that have the same measure on unrooted loops are conjugates.

\begin{defi}
\label{Ch3Sec4: DefMultipleLocalTime}
If $(\gamma(t))_{0\leq t \leq T(\gamma)}$ is a continuous path in $I$ having a family of local times 
$(\ell^{x}_{t}(\gamma))_{x\in I, 0\leq t \leq T(\gamma)}$ relatively to the measure $m(x)dx$, we introduce multiple local times $\ell^{x_{1}, x_{2},\dots,x_{n}}(\gamma)$ for $x_{1}, x_{2},\dots, x_{n}\in I$:
\begin{displaymath}
\ell^{x_{1}, x_{2},\dots,x_{n}}(\gamma):=\int_{0\leq t_{1}\leq t_{2}\leq \dots\leq t_{n}\leq T(\gamma)} d_{t_{1}}\ell^{x_{1}}_{t_{1}}(\gamma)d_{t_{2}}\ell^{x_{2}}_{t_{2}}(\gamma)\dots d_{t_{n}}\ell^{x_{n}}_{t_{n}}(\gamma).
\end{displaymath} 
If $\gamma\in\mathfrak{L}$ and has local times, we introduce circular local times for $\gamma$:
\begin{displaymath}
\ell^{\ast x_{1}, x_{2},\dots,x_{n}}(\gamma):=
\sum_{\substack{
c~\text{circular} \\ 
\text{permutation} \\ 
\text{of}~\lbrace 1, 2,\dots, n\rbrace
}}
\ell^{x_{c(1)}, x_{c(2)},\dots,x_{c(n)}}(\gamma).
\end{displaymath}
$\ell^{\ast x_{1}, x_{2},\dots,x_{n}}$ being invariant under the transformations $(\operatorname{shift}_{v})_{v\in[0,1]}$, we see it as a functional defined on $\mathfrak{L}^{\ast}$.
\end{defi}

Multiple local times of form $\ell^{x, x,\dots,x}(\gamma)$, called self intersection local times, were studied by Dynkin in \cite{Dynkin1984PolynomOccupField}. Circular local times were introduced by Le Jan in \cite{LeJan2011Loops}.

Let $n\in\mathbb{N}^{\ast}$ and $p\in\lbrace 1,\dots,n\rbrace$. Let $\operatorname{Shuffle}_{p,n}$ be the set of permutations $\sigma$ of $\lbrace 1,\dots,n\rbrace$ such that for all $i\leq j\in \lbrace 1,\dots,p\rbrace$, $\sigma(i)\leq\sigma(j)$, and for all $i\leq j\in \lbrace p+1,\dots,n\rbrace$, $\sigma(i)\leq\sigma(j)$. Permutations in $\operatorname{Shuffle}_{p,n}$ are obtained by shuffling two card decks $\lbrace 1,\dots,p\rbrace$ and $\lbrace p+1,\dots,n\rbrace$. Let $\operatorname{Shuffle}'_{p,n}$ be the permutations of $\lbrace 1,\dots,n\rbrace$ of the form $\sigma\circ c$ where $c$ is a circular permutation of 
$\lbrace p+1,\dots,n\rbrace$, and $\sigma\in \operatorname{Shuffle}_{p,n}$ satisfies $\sigma(1)=1$. One can check the following:

\begin{property}
\label{Ch3Sec4: PropertyShuffleAlg}
For all $x_{1},\dots,x_{p}, x_{p+1},\dots, x_{n}\in I$,
\begin{itemize}
\item[(i)]
\begin{displaymath}
\ell^{x_{1},\dots,x_{p}}(\gamma)\ell^{x_{p+1},\dots,x_{n}}(\gamma) = 
\sum_{\sigma\in \operatorname{Shuffle}_{p,n}}\ell^{x_{\sigma(1)},\dots, x_{\sigma(p)}, x_{\sigma(p+1)},\dots,x_{\sigma(n)}}(\gamma),
\end{displaymath}
\item[(ii)]
\begin{displaymath}
\ell^{\ast x_{1},\dots,x_{p}}(\gamma)\ell^{\ast x_{p+1},\dots,x_{n}}(\gamma) = 
\sum_{\sigma'\in \operatorname{Shuffle}'_{p,n}}\ell^{x_{\sigma'(1)},\dots, x_{\sigma'(p)}, x_{\sigma'(p+1)},\dots,x_{\sigma'(n)}}(\gamma).
\end{displaymath}
\end{itemize}
\end{property}

The equality \ref{Ch3Sec4: PropertyShuffleAlg} (ii) appears in \cite{LeJan2011Loops}. It is also shown in \cite{LeJan2011Loops} that for transient Markov jump processes,
\begin{equation}
\label{Ch3Sec4: EqMultLocTimeDensity}
\int \ell^{\ast x_{1}, x_{2},\dots,x_{n}}(\gamma) \mu(d\gamma) = G(x_{1},x_{2})\times \dots\times G(x_{n-1},x_{n})
\times G(x_{n},x_{1}).
\end{equation}

It turns out that we have more. We consider $L$ a generator of a diffusion on $I$ of form \eqref{Ch2Sec2: EqGenerator}. If $\gamma_{i}$ for $i\in\lbrace 1, 2,\dots,n-1\rbrace$ is a continuous path from $x_{i}$ to $x_{i+1}$, then we can concatenate $\gamma_{1}, \gamma_{2},\dots, \gamma_{n-1}$ to obtain a continuous path  $\gamma_{1}\lhd\gamma_{2}\lhd\dots\lhd\gamma_{n-1}$ from $x_{1}$ to $x_{n}$. Let 
$\mu^{x_{1},x_{2}}\lhd\dots\lhd\mu^{x_{n-1},x_{n}}$ be the image measure
$\mu^{x_{1},x_{2}}\otimes\dots\otimes\mu^{x_{n-1},x_{n}}$ by this concatenation procedure.

\begin{prop}
\label{Ch3Sec4: PropMultLocTimeDens}
The following absolute continuity relations hold:
\begin{itemize}
\item[(i)] $(\mu^{x_{1},x_{2}}\lhd\dots\lhd\mu^{x_{n-1},x_{n}})(d\gamma) = 
\ell^{x_{2},\dots,x_{n-1}}(\gamma) \mu^{x_{1},x_{n}}(d\gamma)$,
\item[(ii)] $\pi_{\ast}(\mu^{x_{1},x_{2}}\lhd\dots\lhd\mu^{x_{n-1},x_{n}}\lhd\mu^{x_{n},x_{1}})(d\gamma) = 
\ell^{\ast x_{1}, x_{2},\dots,x_{n}}(\gamma)\mu^{\ast}(d\gamma)$.
\end{itemize}
\end{prop}

\begin{proof}
(i): Let $((X^{(j)}_{t})_{0\leq t<\zeta_{j}})_{0\leq j\leq n-1}$ be $n-1$ independent diffusions of generator $L$, with $X^{(j)}_{0}=x_{j}$. For $l \geq 0$, let
\begin{displaymath}
\tau_{j,l}^{x_{j+1}}:=\inf\left\lbrace t_{j}\geq 0\vert \ell^{x_{j+1}}_{t_{j}}(X^{(j)})> l \right\rbrace. 
\end{displaymath}
According to Proposition \ref{Ch3Sec2: PropLocTimeDesinteg}, 
$(\mu^{x_{1},x_{2}}\lhd\dots\lhd\mu^{x_{n-1},x_{n}})(F(\gamma))$ equals
\begin{equation}
\label{Ch3Sec4: EqMultLocTimeStopTimes}
\mathbb{E}\bigg[\int\limits_{l_{j}<\ell^{x_{j+1}}_{\zeta_{j}}(X^{(j)}), 1\leq j\leq n-1} 
F\big((X^{(1)}_{t})_{0\leq t \leq \tau_{1,l_{1}}^{x_{2}}}\lhd\dots
\lhd(X^{(n-1)}_{t})_{0\leq t\leq \tau_{n-1,l_{n-1}}^{x_{n}}}\big)dl_{1}\dots dl_{n-1}\bigg].
\end{equation}
Let $(X_{t})_{0\leq t<\zeta}$ be an other diffusion of generator $L$. Let
\begin{displaymath}
\tau_{l_{1}}:=\inf\lbrace t\geq 0\vert l^{x_{2}}_{t}(X)> l_{1}\rbrace,
\end{displaymath}
and recursively defined,
\begin{displaymath}
\tau_{l_{1},\dots,l_{j-1}, l_{j}}:=\inf\lbrace t\geq \tau_{l_{1},\dots,l_{j-1}}\vert 
\ell^{x_{j+1}}_{t}(X)> l_{j}\rbrace.
\end{displaymath}
Then, by strong Markov property, \eqref{Ch3Sec4: EqMultLocTimeStopTimes} equals
\begin{displaymath}
\mathbb{E}\left[\int 1_{\tau_{l_{1},\dots,l_{n-1}}\leq \zeta} F\left((X_{t})_{0\leq t\leq \tau_{l_{1},\dots,l_{n-1}}}\right) dl_{1}\dots dl_{n-1}\right],
\end{displaymath}
which in turn equals
\begin{equation}
\label{Ch3Sec4: EqMultLocTimes}
\mathbb{E}\left[ \int 1_{\forall j, t_{j}<\zeta} F\left((X_{t})_{0\leq t\leq t_{n-1}}\right) d_{t_{1}}\ell^{x_{2}}_{t_{1}}(X)\dots d_{t_{n-1}}\ell^{x_{n}}_{t_{n-1}}(X)\right].
\end{equation}
By Proposition \ref{Ch3Sec2: PropLocTimeDesinteg}, \eqref{Ch3Sec4: EqMultLocTimes} equals 
$\int \ell^{x_{1},\dots,x_{n-1}}(\gamma)F(\gamma)\mu^{x_{1}, x_{n}}(d\gamma)$.

(ii): According to the identity (i) and Corollary \ref{Ch3Sec3: CorProjSubalg}, we have
\begin{displaymath}
\pi_{\ast}(\mu^{x_{1},x_{2}}\lhd \dots\lhd\mu^{x_{n-1},x_{n}}\lhd\mu^{x_{n},x_{1}})(d\gamma)=
\int_{0}^{1}\ell^{x_{2},\dots,x_{n}}(\operatorname{shift}_{v}(\gamma))dv\,\pi_{\ast} \mu^{x_{1},x_{1}}(d\gamma).
\end{displaymath}
According to Corollary \ref{Ch3Sec3: CorLoopLocTimeBridge}, 
\begin{displaymath}
\int_{0}^{1} \ell^{x_{2},\dots,x_{n}}(\operatorname{shift}_{v}(\gamma))dv \pi_{\ast} \mu^{x_{1},x_{1}}(d\gamma) 
= \ell^{x_{1}}(\gamma)\int_{0}^{1} \ell^{x_{2},\dots,x_{n}}(\operatorname{shift}_{v}(\gamma))dv \mu^{\ast}(d\gamma).
\end{displaymath}
But,
\begin{displaymath}
\ell^{x_{1}}(\gamma)\int_{0}^{1}\ell^{x_{2},\dots,x_{n}}(\operatorname{shift}_{v}(\gamma)) dv =
\ell^{\ast x_{1}, x_{2},\dots,x_{n}}(\gamma),
\end{displaymath}
which ends the proof.
\end{proof}

Note that Proposition \ref{Ch3Sec4: PropMultLocTimeDens} (ii) implies \eqref{Ch3Sec4: EqMultLocTimeDensity}.

\begin{prop}
\label{Ch3Sec4: PropConversehTransfInv}
If $L$ and $\widetilde{L}$ are two generators of diffusions on $I$ of form \eqref{Ch2Sec2: EqGenerator} such that 
$\mu^{\ast}_{L} = \mu^{\ast}_{\widetilde{L}}$, then there is a positive continuous function $u$ on $I$ such that $\frac{d^{2}u}{dx^{2}}$ is a signed measure, $Lu$ a non-positive measure and 
$\widetilde{L}=\operatorname{Conj}(u,L)$. If the diffusion of generator $L$ is recurrent, then $\widetilde{L}=L$.
\end{prop}

\begin{proof}
Let $m(x)dx$ be a speed measure for $L$ and $\tilde{m}(x)dx$ be a speed measure for $\widetilde{L}$. First let us assume that both $L$ and $\widetilde{L}$ are generators of transient diffusions. Applying the identity 
\eqref{Ch3Sec4: EqMultLocTimeDensity} to $\int_{\mathfrak{L}^{\ast}} \ell^{\ast x,y}(\gamma)\mu^{\ast}(d\gamma)$ we get that for all $x,y\in I$,
\begin{equation}
\label{Ch3Sec4: EqEqualityGreenFuncTwo}
G_{\widetilde{L}}(x,y)G_{\widetilde{L}}(y,x)\tilde{m}(x)\tilde{m}(y) = G_{L}(x,y)G_{L}(y,x)m(x)m(y),
\end{equation}
and for all $x, y, z\in I$,
\begin{equation}
\label{Ch3Sec4: EqEqualityGreenFuncThree}
G_{\widetilde{L}}(x,y)G_{\widetilde{L}}(y,z)G_{\widetilde{L}}(z,x) \tilde{m}(x)\tilde{m}(y)\tilde{m}(z)=
G_{L}(x,y)G_{L}(y,z)G_{L}(z,x)m(x)m(y)m(z).
\end{equation}
Fix $x_{0}\in I$. Let $u$ be
\begin{displaymath}
u(x):=\dfrac{G_{\widetilde{L}}(x_{0},x)\tilde{m}(x)}{G_{L}(x_{0},x) m(x)}.
\end{displaymath}
$u$ is positive and continuous. $\frac{1}{u(x)}G_{L}(x,y)u(y)m(y)$ equals:
\begin{multline}
\label{Ch3Sec4: EqConjGreenFunc}
\dfrac{G_{L}(x_{0},x)G_{L}(x,y)G_{L}(y,x_{0})m(x_{0})m(x)m(y)}
{G_{\widetilde{L}}(x_{0},x)G_{\widetilde{L}}(x,y)G_{\widetilde{L}}(y,x_{0})\tilde{m}(x_{0})\tilde{m}(x)\tilde{m}(y)}\\
\times\dfrac{G_{\widetilde{L}}(x_{0},y)G_{\widetilde{L}}(y,x_{0})\tilde{m}(x_{0})\tilde{m}(y)}
{G_{L}(x_{0},y)G_{L}(y,x_{0})m(x_{0})m(y)}\times G_{\widetilde{L}}(x,y)\tilde{m}(y).
\end{multline}
Applying \eqref{Ch3Sec4: EqEqualityGreenFuncTwo} and \eqref{Ch3Sec4: EqEqualityGreenFuncThree} to 
\eqref{Ch3Sec4: EqConjGreenFunc}, we get that
\begin{equation}
\label{Ch3Sec4: EqConjGreenFunc2}
\dfrac{1}{u(x)}G(x,y)h(y)m(y) = G_{\widetilde{L}}(x,y)\tilde{m}(y).
\end{equation}
Applying \eqref{Ch3Sec4: EqConjGreenFunc2} once to $(x,y)$ and once do $(x,x)$, we get that
\begin{equation}
\label{Ch3Sec4: EqExprh}
u(y) = u(x)\dfrac{ G_{\widetilde{L}}(x,y)}{G(x,y)}\dfrac{G(y,y)}{ G_{\widetilde{L}}(y,y)}.
\end{equation}
From \eqref{Ch3Sec4: EqExprh} we deduce that $\frac{d^{2}u}{dx^{2}}$ is a signed measure.
From \eqref{Ch3Sec4: EqConjGreenFunc2} we deduce that $\widetilde{L}=\operatorname{Conj}(u,L)$. 
$-Lu$ is the killing measure of $\widetilde{L}$ and is non-negative.

If we no longer assume that $L$ and $\widetilde{L}$ generate transient diffusions then consider $\lambda>0$. Then $\mu^{\ast}_{L-\lambda}=\mu^{\ast}_{\widetilde{L}-\lambda}$. According to the above, there is $u$ positive continuous function on $I$ such that $\frac{d^{2}u}{dx^{2}}$ is a signed measure and
\begin{displaymath}
\widetilde{L}-\lambda = \operatorname{Conj}(u,L-\lambda) = \operatorname{Conj}(u,L)-\lambda.
\end{displaymath}
Then $\widetilde{L}=\operatorname{Conj}(u,L)$ and necessarily $Lu$ is a non-positive measure. 

The class of recurrent diffusions is preserved by conjugation. If $L$ is the generator of a recurrent diffusion then so is $\widetilde{L}$, and thus, $u$ is bound to satisfy $Lu=0$. But since the diffusion of $L$ is recurrent, the only solutions to $Lu = 0$ are constant functions. Thus, $\widetilde{L}=L$.
\end{proof}

\section{Measure on loops rooted at the minimum}
\label{Ch3Sec5}

By conditioning the measure $\mu$ by the life-time of loops we get a sum of bridge measures. In this section we will 
disintegrate the measure $\mu^{\ast}$ as a measure on the minimal value of the loop and its behavior above this value.
By doing this way we will obtain a sum of excursion measures $\eta^{>x}_{\rm exc}$. In case of Brownian loops on $\mathbb{R}$ this disintegration will follow from the Vervaat's bridge to excursion transformation. The case of general diffusion will be obtained using covariance of the measure on loops by time and scale change, restriction to a subinterval, killing, as well as invariance by conjugation.

\begin{vervaat}
(\cite{Vervaat1979BridgeToExc},\cite{Biane1986Vervaat}) 
Let $(\gamma(s))_{0\leq s\leq t}$ be a random path following the Brownian bridge probability measure $\mathbb{P}^{t}_{BM,0,0}(\cdot)$. Let $s_{\rm min}:=\operatorname{argmin}\gamma$. Then the path
\begin{displaymath}
s\mapsto -\min \gamma + (\operatorname{shift}_{\frac{s_{\rm min}}{t}}\gamma)(s)
\end{displaymath}
has the law of a positive Brownian excursion of life-time $t$.
\end{vervaat}

In the sequel if $\eta$ is a measure on paths and $x\in\mathbb{R}$, we will write  $(x+\eta)$ for the image of $\eta$ by $\gamma\mapsto x+ \gamma$. $\eta_{BM}^{>0}$ will be the Lévy-Itô measure on positive Brownian excursions and 
$\eta_{t, BM}^{>0}$ the probability measure on positive Brownian excursions of duration $t$. Given a continuous loop $(\gamma_{t})_{0\leq t\leq T(\gamma)}$ and $t_{\rm min}$ the first time $\gamma$ hits $\min\gamma$, let $\mathcal{V}(\gamma)$ be the transformation $\operatorname{shift}_{\frac{t_{\rm min}}{T(\gamma)}}$. $\mathcal{V}$ is $\mathcal{B}_{\mathfrak{L}}$-measurable.

\begin{prop}
\label{Ch3Sec5: PropDesintegMinBM}
Let $\mu^{\ast}_{BM}$ be the measure on loops associated to the Brownian motion on $\mathbb{R}$. Then,
\begin{equation}
\label{Ch3Sec5: EqDisintegBrownianLoops}
\mu^{\ast}_{BM}(d\gamma) = 2\int_{a\in\mathbb{R}}\pi_{\ast}(a + \eta_{BM}^{>0})(d\gamma)\,da.
\end{equation}
The measure on $(\min\gamma, \max\gamma)$ induced by $\mu^{\ast}_{BM}$ is $1_{a<b}(b-a)^{-2} da db$. Let 
$a<b\in \mathbb{R}$ and $\rho, \tilde{\rho}$ two independent Bessel $3$ processes starting from $0$. Let $T_{b-a}$ and $\widetilde{T}_{b-a}$ be the first times $\rho$ respectively $\tilde{\rho}$ hits $b-a$. Let 
$(\beta_{t})_{0\leq t\leq T_{b-a} + \widetilde{T}_{b-a}}$ be the path
\begin{displaymath}
\beta_{t}:=
\left\lbrace 
\begin{array}{ll}
a+\rho_{t} & \text{if}~t\leq T_{b-a}, \\ 
a+\tilde{\rho}_{T_{b-a}+\widetilde{T}_{b-a} - t}& \text{if}~t\geq T_{b-a}.
\end{array} 
\right. 
\end{displaymath}
Then the law of $(\beta_{t})_{0\leq t\leq T_{b-a} + \widetilde{T}_{b-a}}$ is the probability measure obtained by conditioning the measure $\mu^{\ast}_{BM}$ on $(\min\gamma, \max\gamma)=(a,b)$.
\end{prop}

\begin{proof}
For the Brownian motion on $\mathbb{R}$, $\mu_{BM}$ writes
\begin{displaymath}
\mu_{BM}(\cdot) = \int_{x\in \mathbb{R}}\int_{t>0}(x+\mathbb{P}_{BM,0,0}^{t})(\cdot)\dfrac{dt}{\sqrt{2\pi t^{3}}}dx.
\end{displaymath}
Let $\upsilon_{t}(a)da$ be the law of the minimum of the bridge under $\mathbb{P}_{BM,0,0}^{t}$. Applying the Vervaat's transformation, we get that
\begin{displaymath}
\mathcal{V}_{\ast}\mu_{BM}(\cdot) = \int_{a\in\mathbb{R}}\int_{t>0}\left(\int_{x>a}\upsilon_{t}(x-a)dx\right) 
(a+\eta_{t, BM}^{>0})(\cdot)\dfrac{dt}{\sqrt{2\pi t^{3}}}da.
\end{displaymath}
Since $\int_{x>a}\upsilon_{t}(x-a)\,dx=1$, the right-hand side above equals
\begin{displaymath}
\int_{a\in\mathbb{R}}\int_{t>0}(a+\eta_{t, BM}^{>0})(\cdot)\dfrac{dt}{\sqrt{2\pi t^{3}}}da.
\end{displaymath}
But,
\begin{displaymath}
\int_{t>0}(a+\eta_{t, BM}^{>0})(\cdot)\dfrac{dt}{\sqrt{2\pi t^{3}}} = 2(a+\eta^{>0}_{BM})(\cdot).
\end{displaymath}
The equality \eqref{Ch3Sec5: EqDisintegBrownianLoops} follows. The rest of the proposition is a consequence of the William's representation of Brownian excursions (see \cite{RevuzYor1999BMGrundlehren}, Chapter XII).
\end{proof}

\begin{coro}
\label{Ch3Sec5: CorInvRefBM}
Let $I$ be an open interval of $\mathbb{R}$ and $\lambda\geq 0$. Let $\mu^{\ast}$ be the measure on loops in $I$ associated to the generator $\frac{1}{2}\frac{d^{2}}{dx^{2}}-\lambda$. Given a loop $(\gamma(t))_{0\leq t\leq T(\gamma)}$, let $R(\gamma)$ be the loop 
\begin{displaymath}
R(\gamma):=(\max \gamma +\min \gamma - \gamma(t))_{0\leq t\leq T(\gamma)},
\end{displaymath}
that is to say the image of $\gamma$ through reflection relatively to $\frac{\max \gamma +\min \gamma}{2}$. Then,
\begin{displaymath}
R_{\ast}\mu^{\ast} = \mu^{\ast}.
\end{displaymath}
\end{coro}

\begin{proof}
It is enough to prove this in case $\lambda = 0$ and $I = \mathbb{R}$. Otherwise, we multiply the measure $\mu^{\ast}_{BM}$ by a density function that is invariant by $R$. Then we use the description of the measure $\mu^{\ast}_{BM}$ conditional the value of $(\min \gamma, \max \gamma)$ and the fact that if $a>0$, $(\rho_{t})_{t\geq}$ is a Bessel $3$ process starting from $0$ and $T_{b}$ is the first time it hits $b$, then $(y-\rho_{T_{b}-t})_{0\leq t\leq T_{b}}$ has the same law as $(\rho_{t})_{0\leq t\leq T_{b}}$ (see \cite{RevuzYor1999BMGrundlehren}, Section VII.4).
\end{proof}

Now we consider that $L$ is a generator of a diffusion on $I$ of form \eqref{Ch2Sec2: EqGenerator}. Given a point $x_{0}\in I$, $u^{+, x_{0}}$ and $u^{-, x_{0}}$ will be the $L$-harmonic functions satisfying the initial conditions $u^{+, x_{0}}(x_{0})=u^{-, x_{0}}(x_{0})=0$, $\frac{du^{+, x_{0}}}{dx}(x_{0}^{+})=1$ and 
$\frac{du^{-, x_{0}}}{dx}(x_{0}^{-})=-1$. If $x\leq y\in I$ then
\begin{equation}
\label{Ch3Sec5: EqEqualRatioWronsk}
w(y)u^{-,y}(x) = w(x) u^{+,x}(y).
\end{equation}
Indeed, the Wronskian $W(u^{-,y},u^{+,x})$ takes in $x$ the value $u^{-,y}(x)$ and in $y$ the value $u^{+,x}(y)$, and the ratio $\frac{1}{w(z)}W(u^{-,y},u^{+,x})(z)$ is constant. If $\kappa = 0$, then the both sides of 
\eqref{Ch3Sec5: EqEqualRatioWronsk} equal $\int_{x}^{y}w(z)dz$. $u^{+,x_{0}}$ is positive on $I\cap(x_{0},+\infty)$ and $u^{-,x_{0}}$ is positive on $I\cap(-\infty, x_{0})$. Let $L^{+,x_{0}}$ be $\operatorname{Conj}(u^{+,x_{0}}, L)$ restricted to $I\cap(x_{0},+\infty)$ and $L^{-,x_{0}}$ be $\operatorname{Conj}(u^{-,x_{0}},L)$ restricted to $I\cap(-\infty, x_{0})$. $L^{+,x_{0}}$ and  $L^{-,x_{0}}$ are generators of transient diffusions without killing measures. If $L$ is the generator of the Brownian motion on $\mathbb{R}$, then $L^{+,0}$ is just the generator of a Bessel $3$ process. In general case, $x_{0}$ is an entrance boundary for $L^{+,x_{0}}$ and $L^{-,x_{0}}$, that is to say a diffusion started from $x\neq x_{0}$ will never reach the boundary at $x_{0}$, and we can also start this diffusions at the boundary point $x_{0}$, in which case it will be immediately repelled away from $x_{0}$. Let $x\in I$ and $(\rho^{+,x}_{t})_{0\leq t< \zeta^{+,x}}$ be a diffusion of generator $L^{+,x}$ starting from $x$. Let $y\in I$, $y>x$. Let $T_{y}^{+,x}$ be the first time $\rho^{+,x}$ hits $y$ and $\widehat{T}_{y}^{+,x}$ the last time it visits $y$. Then 
$(\rho^{+,x}_{\widehat{T}_{y}^{+,x}+t})_{0\leq t< \zeta^{+,x}-\widehat{T}_{y}^{+,x}}$ is a diffusion of generator $L^{+,y}$ starting from $y$. Let $(\rho^{-,y}_{t})_{0\leq t<\zeta^{-,y}}$ be a diffusion of generator $L^{-,y}$ starting from $y$ and $T^{-,y}_{x}$ the first time it hits $x$. Then $(\rho^{+,x}_{t})_{0\leq t\leq T_{y}^{+,x}}$ and $(\rho^{-,y}_{T_{x}^{-,y}-t})_{0\leq t\leq T_{x}^{-,y}}$ are equal in law: Indeed let $C$ be the constant
\begin{displaymath}
C=\dfrac{w(z)}{W(u^{-,y},u^{+,x})(z)}.
\end{displaymath}
The Green's operator of $\rho^{+,x}$ killed in $y$ is 
\begin{displaymath}
((-L^{+,x}_{\vert (x,y)})^{-1}f)(x')=C\int_{x}^{y}u^{+,x}(x'\wedge y')u^{-,y}(x'\vee y')\dfrac{u^{+,x}(y')}{u^{+,x}(x')}m(y')dy',
\end{displaymath}
and the Green's operator of $\rho^{-,y}$ killed in $x$ is 
\begin{displaymath}
((-L^{-,y}_{\vert (x,y)})^{-1}f)(x')=C\int_{x}^{y}u^{+,x}(x'\wedge y')u^{-,y}(x'\vee y')\dfrac{u^{-,y}(y')}{u^{-,y}(x')}m(y')dy'.
\end{displaymath}
The potential measure of $(\rho^{+,x}_{t})_{0\leq t\leq T_{y}^{+,x}}$ starting from $x$ is
\begin{displaymath}
U(x')dx' = Cu^{+,x}(x')u^{-,y}(x')m(x')dx',
\end{displaymath}
and for any $f, g$ bounded functions on $(x,y)$,
\begin{equation}
\label{Ch3Sec5: EqDualityTimeRev}
\int_{x}^{y}((-L^{+,x}_{\vert (x,y)})^{-1}f)(x')g(x')U(x')dx'=
\int_{x}^{y}f(x')((-L^{-,y}_{\vert (x,y)})^{-1}g)(x')U(x')dx'.
\end{equation}
The time reversal property for $(\rho^{+,x}_{t})_{0\leq t\leq T_{y}^{+,x}}$ follows from the duality relation \eqref{Ch3Sec5: EqDualityTimeRev}. See \cite{RevuzYor1999BMGrundlehren}, Section VII.4 for details on time reversal.

\begin{coro}
\label{Ch3Sec5: CorDesintegMinGen}
If $L$ is a generator of a diffusion on $I$ of form \eqref{Ch2Sec2: EqGenerator}, then
\begin{equation}
\label{Ch3Sec5: EqDisintegGeneralLoops}
\mu^{\ast}(\cdot) = \int_{a\in I}\pi_{\ast}\eta^{>a}(\cdot)w(a)da.
\end{equation}
The measure on $(\min\gamma, \max\gamma)$ induced by $\mu^{\ast}$ is $1_{a<b\in I}\frac{da db}{u^{+,a}(b)u^{-,b}(a)}$. Let $a<b\in I$. Let $(\rho^{+,a}_{t})_{0\leq t< \zeta^{+,a}}$ and $(\rho^{-,b}_{t})_{0\leq t< \zeta^{-,b}}$ be two independent diffusion, the first of generator $L^{+,a}$ starting from $a$ and the second of generator $L^{-,b}$ starting from $b$. Let $T_{b}^{+,a}$ be the first time $\rho^{+,a}$ hits $b$ and $T^{-,b}_{a}$ the first time $\rho^{-,b}$ hits $a$. Let $(\beta_{t})_{0\leq t\leq T_{b}^{+,a} + T^{-,b}_{a}}$ be the path
\begin{displaymath}
\beta_{t}:=
\left\lbrace 
\begin{array}{ll}
\rho^{+,a}_{t} & \text{if}~t\leq T_{b}^{+,a}, \\ 
\rho^{-,b}_{t-T_{b}^{+,a}}& \text{if}~t\geq T_{b}^{+,a}.
\end{array} 
\right. 
\end{displaymath}
Then the law of $(\beta_{t})_{0\leq t\leq T_{b}^{+,a} + T^{-,b}_{a}}$ is the probability measure obtained by conditioning the measure $\mu^{\ast}$ on $(\min\gamma, \max\gamma)=(a,b)$.
\end{coro}

\begin{proof}
Both sides of \eqref{Ch3Sec5: EqDisintegGeneralLoops} are covariant by scale and time change. Moreover both sides satisfy Pproperty \ref{Ch3Sec3: PropertyTrivialitiesLoops} (ii) for the restriction to a subinterval and Property 
\ref{Ch3Sec3: PropertyTrivialitiesLoops} (iii) when adding a killing measure. Thus, the general case 
\eqref{Ch3Sec5: EqDisintegGeneralLoops} follows from the Brownian case \eqref{Ch3Sec5: EqDisintegBrownianLoops} by this covariance properties.

If $L$ is a generator without killing measure ($\kappa=0$) then the description of the measure on 
$(\min\gamma, \max\gamma)$ and the probabilities obtained after conditioning by  the value of $(\min\gamma,\max\gamma)$ follow through a change of scale and time from the analogous description in Proposition \ref{Ch3Sec5: PropDesintegMinBM}. If $\kappa\neq 0$, then we can take $u$ a positive $L$-harmonic function and deduce the result for $L$ from the result for $\operatorname{Conj}(u,L)$ using the fact that $\mu^{\ast}_{L}=\mu^{\ast}_{\operatorname{Conj}(u,L)}$.
\end{proof}

The relation between the measure on loops and the excursions measures in dimension $1$ 
(identity \eqref{Ch3Sec5: EqDisintegGeneralLoops}) has an analogue in the setting of two-dimensional Brownian loops.
See Propositions $7$ and $8$ in \cite{LawlerWerner2004ConformalLoopSoup} by Lawler and Werner. It is possible to disintegrate the measure on loops in dimension 2 using the so called bubble measures. 

\section{A generalization of the Vervaat's transformation}
\label{Ch3Sec6}

In this section we will show a conditioned version of the Vervaat's transformation that holds for any one-dimensional diffusion of form \eqref{Ch2Sec2: EqGenerator} and not just for the Brownian motion.  $L$ will be a generator of a diffusion on $I$ of form \eqref{Ch2Sec2: EqGenerator}. From Corollary \ref{Ch3Sec3: CorLoopLocTimeBridge} and identity \eqref{Ch3Sec5: EqDisintegGeneralLoops} follows that for every $x\in I$,
\begin{equation}
\label{Ch3Sec6: EqVervaatIntegrated}
\int_{t>0}\mathcal{V}_{\ast}\mathbb{P}^{t}_{x,x}(d\gamma)p_{t}(x,x)dt = 
\int_{a\in I, a<x}\ell^{x}(\gamma)\eta^{>a}(d\gamma)w(a)da.
\end{equation}
Let $\mathbb{P}^{t}_{x,x}(d\gamma\vert \min \gamma = a)$ be the bridge probability measure conditioned on the value of the minimum being to equal $a$. Further we will show that there is a version that depends continuously on $(a,t)$. Let $\eta^{>a}_{t}$ be the probability measure obtained from $\eta^{>a}$ by conditioning the excursion on having a life-time $t$. The identity \eqref{Ch3Sec6: EqVervaatIntegrated} suggests the following:

\begin{prop}
\label{Ch3Sec6: PropConditionnedVervaat}
For every $a<x\in I$ and $t>0$
\begin{equation}
\label{Ch3Sec6: EqConditionedVervaat}
\mathcal{V}_{\ast}\mathbb{P}^{t}_{x,x}(d\gamma\vert \min \gamma = a)= 
\dfrac{\ell^{x}_{t}(\gamma)\eta^{>a}_{t}(d\gamma)}{\eta^{>a}_{t}(\ell^{x}_{t}(\gamma))}.
\end{equation}
The distribution of $\min \gamma$ under $\mathbb{P}_{x,x}^{t}$ equals
\begin{equation}
\label{Ch3Sec6: EqDistribMinBridge}
w(a)\eta^{>a}_{t}(\ell^{x}_{t}(\gamma))\dfrac{1}{p_{t}(x,x)}\dfrac{\eta^{>a}(T(\gamma)\in (t, t+dt))}{dt} da,
\end{equation}
where $\frac{\eta^{>a}(T(\gamma)\in (t, t+dt))}{dt}$ is the density of the measure on the life-time of the excursion induced by $\eta^{>a}$. Given an excursion $\gamma$ following the law 
$\frac{\ell^{x}_{t}(\gamma)\eta^{>a}_{t}(d\gamma)}{\eta^{>a}_{t}(\ell^{x}_{t}(\gamma))}$, the local time in $x$ is a measure on $\lbrace s\in[0,t]\vert \gamma(s)=x\rbrace$. The transformation $\mathcal{V}$ sends the starting point of the bridge to a point $s\in[0,t]$ distributed, conditional on the excursion $\gamma$, according to the measure $\frac{d_{s}\ell^{x}_{s}(\gamma)}{\ell^{x}_{t}(\gamma)}$.
\end{prop}

Identities \eqref{Ch3Sec6: EqConditionedVervaat} and \eqref{Ch3Sec6: EqDistribMinBridge} can be viewed as a conditioned analogue of the Vervaat's relation between the Brownian bridge and the Brownian excursion. The latter can be deduced from \eqref{Ch3Sec6: EqConditionedVervaat} and \eqref{Ch3Sec6: EqDistribMinBridge} using the translation invariance of the Brownian motion. From \eqref{Ch3Sec6: EqConditionedVervaat} we can only deduce that \eqref{Ch3Sec6: EqConditionedVervaat} and \eqref{Ch3Sec6: EqDistribMinBridge} hold for Lebesgue almost all $t$ and $a$. We need to show the weak continuity in $(a,t)$ of conditioned bridge probabilities and biased conditioned excursion probabilities to conclude. It is enough to prove Proposition \ref{Ch3Sec6: PropConditionnedVervaat} for $L$ not containing any killing measure and such that for all $a<x\in I$, a diffusion starting from $x$ reaches $a$ almost surely. Indeed, for a general generator, $\operatorname{Conj}(u_{\downarrow},L)$ does satisfy the above constraints and if Proposition \ref{Ch3Sec6: PropConditionnedVervaat} is true for $\operatorname{Conj}(u_{\downarrow},L)$, then it is also true for $L$. From now on we assume that $L$ satisfies the above constraints. Next, we give a more constructive description of the conditioned bridges and biased conditioned excursions. We start with bridges.

Property \ref{Ch3Sec2: PropertyBridgesMin} shows that the measure $\mathbb{P}_{x}^{T_{a}}\lhd\tilde{\mathbb{P}}_{x}^{\tilde{T}^{a} \wedge}$ conditioned on $T_{a}+\tilde{T}_{a}=t$ is a version of $\mathbb{P}_{x,x}^{t}(d\gamma\vert \min \gamma = a)$. Let $p_{t}^{(a\times)}(x,y)$ be the transition density on $I\cap(a,+\infty)$ relatively to $m(y)dy$ of the semi-group generated by $L_{\vert I\cap(a,+\infty)}$. Then $p_{t}^{(a\times)}(x,a^{+})=0$. According to \cite{McKean1956EigenDiffExp}, for all $t>0$, 
$y\mapsto p_{t}^{(a\times)}(x,y)$ is $\mathcal{C}^{1}$. Let $\partial_{2}p_{t}^{(a\times)}(x,y)$ be the derivative relatively to $y$. It has a positive limit $\partial_{2}p_{t}^{(a\times)}(x,a^{+})$ as $y\rightarrow a^{+}$. Extended in this way, the map $(t,x,y)\mapsto\partial_{2}p_{t}^{(a\times)}(x,y)$ is continuous on 
$(0+\infty)\times I\cap(a,+\infty)\times I\cap[a,+\infty)$. The distribution of $T_{a}$ under $\mathbb{P}_{x}$ is 
(see \cite{ItoMcKean1974Diffusions}, page 154)
\begin{displaymath}
\dfrac{1}{w(a)}\partial_{2}p_{t}^{(a\times)}(x,a^{+})dt.
\end{displaymath}
Let $\mathbb{P}_{x,y}^{(a\times),t}$ be the bridge probability measure of $L_{\vert I\cap(a,+\infty)}$. It has a weak limit $\mathbb{P}_{x,a^{+}}^{(a\times), t}$ as $y\rightarrow a^{+}$. Let $\mathcal{F}_{s}$ be the sigma-algebra generated by the restriction of a continuous path to the time interval $[0,s]$. Let $\mathbb{P}^{+,a}_{a}$ be the law of $\rho^{+,a}$ starting from $a$. For all $s\in (0,t)$, we have the following absolute continuity relations:
\begin{equation}
\label{Ch3Sec6: EqBridgePlusAbsCont}
\dfrac{d\mathbb{P}_{x,a^{+}}^{(a\times),t}}{d\mathbb{P}_{x}}_{\vert \mathcal{F}_{s}} = 
1_{s<T_{a}}\dfrac{\partial_{2}p_{t-s}^{(a\times)}(X_{s},a^{+})}{\partial_{2}p_{t}^{(a\times)}(x,a^{+})},
\end{equation}
and for the time reversed bridge,
\begin{equation}
\label{Ch3Sec6: EqBridgePlusRevAbsCont}
\dfrac{d\mathbb{P}_{x,a^{+}}^{(a\times),t \wedge}}{d\mathbb{P}^{+,a}_{a}}_{\vert \mathcal{F}_{s}} = 
\dfrac{p_{t-s}^{(a\times)}(\rho^{+,a}_{s},x)}{\partial_{2}p_{t}^{(a\times)}(x,a^{+})}.
\end{equation}
Using the absolute continuity relation \eqref{Ch3Sec6: EqBridgePlusAbsCont} and \eqref{Ch3Sec6: EqBridgePlusRevAbsCont}, one can prove in a similar way as in Proposition \ref{Ch2Sec2: PropContinuityBridge} that the map $(t,y)\mapsto\mathbb{P}_{x,a^{+}}^{(a\times),t}$ is continuous for the weak topology. The first passage bridge $\mathbb{P}_{x}^{T_{a}}$ disintegrates as follows
\begin{equation}
\label{Ch3Sec6: EqDisintegFPBridge}
\mathbb{P}_{x}^{T_{a}}(\cdot) = \dfrac{1}{w(a)}\int_{t>0}\mathbb{P}_{x,a^{+}}^{(a\times),t}(\cdot)
\partial_{2}p_{t}^{(a\times)}(x,a^{+})dt.
\end{equation}
From Property \ref{Ch3Sec2: PropertyBridgesMin} and \eqref{Ch3Sec6: EqDisintegFPBridge} we get that:

\begin{property}
\label{Ch3Sec6: PropertyDesintegBridgeMin}
The distribution of $\min\gamma$ under $P^{t}_{x,x}$ is
\begin{equation}
\label{Ch3Sec6: EqDistribMinBridge2}
\dfrac{da}{w(a) p_{t}(x,x)}\int_{0}^{t}\partial_{2}p_{s}^{(a\times)}(x,a^{+})\partial_{2}p_{t-s}^{(a)}(x,a^{+})ds.
\end{equation}
There is a version of $\mathbb{P}_{x,x}^{t}(d\gamma\vert \min \gamma = a)$ that disintegrates as
\begin{equation}
\label{Ch3Sec6: EqDesintegMinBridge}
\dfrac{\int_{0}^{t}\left(\mathbb{P}_{x,a^{+}}^{(a\times),s}\lhd \mathbb{P}_{x,a^{+}}^{(a\times),t-s \wedge}\right)(d\gamma)\partial_{2}p_{s}^{(a\times)}(x,a^{+})\partial_{2}p_{t-s}^{(a\times)}(x,a^{+})ds}{\int_{0}^{t}\partial_{2}p_{s}^{(a\times)}(x,a^{+})\partial_{2}p_{t-s}^{(a\times)}(x,a^{+})ds}.
\end{equation}
\end{property}
Next we show that the probability measure given by \eqref{Ch3Sec6: EqDesintegMinBridge} depends continuously on $(a,t)$.

\begin{lemm}
\label{Ch3Sec6: LemContinuityDensities}
The functions  $(x,a,t)\mapsto p_{t}^{(a\times)}(x,a^{+})$ and $(x,a,t)\mapsto\partial_{2}p_{t}^{(a\times)}(x,a^{+})$ are continuous on $\lbrace (x,a)\vert x>a\in I\rbrace\times(0,+\infty)$.
\end{lemm}

\begin{proof}
As in \cite{McKean1956EigenDiffExp}, we can use the eigendifferential expansion of $L$ to express 
$p_{t}^{(a\times)}(x,a^{+})$ and $\partial_{2}p_{t}^{(a\times)}(x,a^{+})$. Let $x_{0}$. For $\lambda\in\mathbb{R}$, consider $e_{1}(\cdot,\lambda)$ and $e_{2}(\cdot,\lambda)$ two solutions to $Lu+\lambda u =0$ with initial conditions
\begin{displaymath}
e_{1}(x_{0},\lambda)=1,\quad\dfrac{\partial e_{1}}{\partial x}(x_{0},\lambda)=0,\quad
e_{2}(x_{0},\lambda)=0\quad\dfrac{\partial e_{2}}{\partial x}(x_{0},\lambda)=1.
\end{displaymath}
Let $\mathfrak{e}(x,\lambda)$ be the $2$-vector whose entries are $e_{1}(x,\lambda)$ and $e_{2}(x,\lambda)$. According to Theorems $3.2$ and $4.3$ in \cite{McKean1956EigenDiffExp}, for all $a\in I$, there is a Radon measure $\mathfrak{f}^{(a)}$ on $(-\infty,0]$ with values in the space of $2\times 2$ symmetric positive semi-definite matrices such that for all $x\in I\cap(a,+\infty)$,
\begin{displaymath}
p_{t}^{(a\times)}(x,a^{+})=\int_{-\infty}^{0}e^{t\lambda}{}^\intercal \mathfrak{e}(x,\lambda)\mathfrak{f}^{(a\times)}(d\lambda)\mathfrak{e}(a,\lambda),
\end{displaymath}
\begin{displaymath}
\partial_{2}p_{t}^{(a\times)}(x,a^{+})=\int_{-\infty}^{0}e^{t\lambda}{}^\intercal 
\mathfrak{e}(x,\lambda)\mathfrak{f}^{(a\times)}(d\lambda)\dfrac{\partial\mathfrak{e}}{\partial x}(a,\lambda).
\end{displaymath}
Let $x>a\in I$. Consider two sequences $(x_{n})_{n\geq 0}$ and $(a_{n})_{n\geq 0}$ in $I\cap(-\infty,x)$, converging to $x$ respectively $a$, such that for all $n\geq 0$, $x_{n}>a_{n}$. Let $(b_{j})_{j\geq 0}$ be an increasing sequence in $I\cap(x,\sup I)$ converging to $\sup I$. Let $\mathfrak{f}_{n,j}$ be the $2\times 2$-matrix valued measure on $(-\infty,0]$ corresponding to the eigendifferential expansion of $L$ restricted to $(a_{n},b_{j})$. $\mathfrak{f}_{n,j}$ charges only a discrete set of atoms. As shown in the proof of Theorem $3.2$ in \cite{McKean1956EigenDiffExp}, the total mass of the measures $1\wedge\vert\lambda\vert^{-2}\Vert\mathfrak{f}_{n,j}\Vert(d\lambda)$, $1\wedge\vert\lambda\vert^{-2}\Vert\mathfrak{f}^{(a_{n}\times)}\Vert(d\lambda)$ and $1\wedge\vert\lambda\vert^{-2}\Vert\mathfrak{f}^{(a\times)}\Vert(d\lambda)$ is uniformly bounded. Moreover, for a fixed $n$, as $j\rightarrow +\infty$, $1\wedge\vert\lambda\vert^{-2}\mathfrak{f}_{n,j}(d\lambda)$ converges vaguely, that is against continuous functions vanishing at infinity, to the measure 
$1\wedge\vert\lambda\vert^{-2}\mathfrak{f}^{(a_{n}\times)}(d\lambda)$. Moreover, for any increasing integer-valued sequence $(j_{n})_{n\geq 0}$ converging to $+\infty$, $1\wedge\vert\lambda\vert^{-2}\mathfrak{f}_{n,j_{n}}(d\lambda)$ converges vaguely as $n\rightarrow +\infty$ to $1\wedge\vert\lambda\vert^{-2}\mathfrak{f}^{(a\times)}(d\lambda)$. Since the sequence $(j_{n})_{n\geq 0}$ is arbitrary, this implies that $1\wedge\vert\lambda\vert^{-2}\mathfrak{f}^{(a_{n}\times)}(d\lambda)$ converges vaguely as 
$n\rightarrow +\infty$ to $1\wedge\vert\lambda\vert^{-2}\mathfrak{f}^{(a\times)}(d\lambda)$.

There are constants $C, c'>0$ such that for all $\lambda\leq 0$ and $n\geq 0$,
\begin{equation}
\label{Ch3Sec6: EqEncadrEigenDiff}
\Vert\mathfrak{e}(x_{n},\lambda)\Vert\leq C e^{c'\sqrt{\vert\lambda\vert}},\quad
\Vert\mathfrak{e}(a_{n},\lambda)\Vert\leq C e^{c'\sqrt{\vert\lambda\vert}},\quad
\Vert\dfrac{\partial\mathfrak{e}}{\partial x}(a_{n},\lambda)\Vert\leq C e^{c'\sqrt{\vert\lambda\vert}}.
\end{equation}
Let $t>0$ and $(t_{n})_{n\geq 0}$ a sequence of times converging to $t$. From \eqref{Ch3Sec6: EqEncadrEigenDiff} follows that
\begin{displaymath}
\lim_{\lambda \rightarrow -\infty}\sup_{n\geq 0}\vert\lambda\vert^{2}e^{t_{n}\lambda}\Vert
\mathfrak{e}(x_{n},\lambda)\Vert\times\Vert\mathfrak{e}(a_{n},\lambda)\Vert = 0.
\end{displaymath}
$\lambda\mapsto 1\vee\vert\lambda\vert^{2}e^{t_{n}\lambda}\left(\mathfrak{e}(x_{n},\lambda),\partial\mathfrak{e}
(a_{n},\lambda)\right)$ vanishes at infinity and converges uniformly on $(-\infty, 0]$ to 
$\lambda\mapsto 1\vee\vert\lambda\vert^{2}e^{t\lambda}\left(\mathfrak{e}(x,\lambda),\mathfrak{e}(a,\lambda)\right)$.
The vague convergence of measures implies that
\begin{displaymath}
\lim_{n\rightarrow +\infty}\int_{-\infty}^{0}e^{t_{n}\lambda}{}^{\intercal}
\mathfrak{e}(x_{n},\lambda)\mathfrak{f}^{(a_{n}\times)}(d\lambda)\mathfrak{e}(a_{n},\lambda) = \int_{-\infty}^{0}e^{t\lambda}{}^{\intercal}\mathfrak{e}(x,\lambda)\mathfrak{f}^{(a\times)}(d\lambda)\mathfrak{e}(a,\lambda).
\end{displaymath}
Similarly, $\partial_{2}p_{t_{n}}^{(a_{n}\times)}(x_{n},a_{n}^{+})$ converges to $\partial_{2}p_{t}^{(a\times)}(x,a^{+})$.
\end{proof}

\begin{lemm}
\label{Ch3Sec6: LemContinuityFromLastTime}
The map $a\mapsto\mathbb{P}^{+,a}_{a}$ is weakly continuous. 
\end{lemm}

\begin{proof}
Let $a_{0}\in I$. Consider the process $(\rho^{+,a_{0}}_{t})_{t\geq 0}$ following the law $\mathbb{P}^{+,a_{0}}_{a_{0}}$. For $a\in I\cap(a_{0},+\infty)$, let $\widehat{T}_{a}$ be the last time $\rho^{+,a_{0}}$ visits $a$. Then $(\rho^{+,a_{0}}_{\widehat{T}_{a}+t})_{t\geq 0}$ follows the law $\mathbb{P}^{+,a}_{a}$. The process-valued map 
$a\mapsto (\rho^{+,a_{0}}_{\widehat{T}_{a}+t})_{t\geq 0}$ is almost surely continuous on $I\cap(a_{0},+\infty)$ and thus, the laws depend weakly continuously on $a$.
\end{proof}

\begin{prop}
\label{Ch3Sec6: PropContinuityBridgeMin}
The version of $\mathbb{P}_{x,x}^{t}(d\gamma\vert \min \gamma = a)$ given by \eqref{Ch3Sec6: EqDesintegMinBridge} is weakly continuous in $(a,t)$.
\end{prop}

\begin{proof}
From the absolute continuity relations \eqref{Ch3Sec6: EqBridgePlusAbsCont} for the bridge $\mathbb{P}_{x,a^{+}}^{(a\times),t}$ and \eqref{Ch3Sec6: EqBridgePlusRevAbsCont} for its time reversal, together with the continuity of the densities which follows from Lemma \ref{Ch3Sec6: LemContinuityDensities}, and the weak continuity of $a\mapsto\mathbb{P}^{+,a}_{a}$, we can deduce in a very similar way as in Proposition \ref{Ch2Sec2: PropContinuityBridge} that the map $(a,t)\mapsto\mathbb{P}_{x,a^{+}}^{(a\times),t}$ is weakly continuous on $(0,+\infty)\times I\cap(-\infty,x)$, and hence $(a,s,t)\mapsto\mathbb{P}_{x,a^{+}}^{(a\times),s}\lhd \mathbb{P}_{x,a^{+}}^{(a\times),t-s \wedge}$ is weakly continuous. Finally, the densities that appear in expression \eqref{Ch3Sec6: EqDesintegMinBridge} are continuous with respect to $(a,s,t)$.
\end{proof}

Next we will give a decomposition of the measure $\eta^{>a}$ which is similar to Bismut's decomposition of Brownian excursions (see \cite{RevuzYor1999BMGrundlehren}, Section XII.4, Theorem 4.7). Biane used this Bismut's decomposition to give an alternative proof for the Brownian Vervaat's transformation (\cite{Biane1986Vervaat}). $\partial_{2}p_{t}^{(a\times)}(x,a^{+})$ is $\mathcal{C}^{1}$ relatively to $x$ and the derivative $\partial_{1,2}p_{t}^{(a\times)}(x,a^{+})$ has a positive limit $\partial_{1,2}p_{t}^{(a\times)}(a^{+},a^{+})$ as $y\rightarrow a^{+}$. Moreover $t\mapsto\partial_{1,2}p_{t}^{(a\times)}(a^{+},a^{+})$ is continuous. The measure on the life-time of the excursion induced by $\eta^{>a}$ is (see \cite{SalminenValloisYor2007Excur})
\begin{displaymath}
\dfrac{1}{w(a)^{2}}\partial_{1,2}p_{t}^{(a\times)}(a^{+},a^{+})dt.
\end{displaymath}
Let $s\in[0,t]$. The measure $\eta^{>a}_{t}(\cdot)$ disintegrates as (see \cite{SalminenValloisYor2007Excur})
\begin{equation}
\label{Ch3Sec6: EqDisintegExcursions}
\int_{x\in I, x>a}\left(\mathbb{P}_{x,a^{+}}^{(a\times),s \wedge}\lhd \mathbb{P}_{x,a^{+}}^{(a\times),t-s}\right)(\cdot)\dfrac{\partial_{2}p_{s}^{(a\times)}(x,a^{+})\partial_{2}p_{t-s}^{(a\times)}(x,a^{+})m(y)}{\partial_{1,2}p_{t}^{(a\times)}(a^{+},a^{+})}dy.
\end{equation}
For every $s_{1}<s_{2}\in[0,s]$, under the bridge measure $\mathbb{P}_{y,z}^{(a\times),s}$,
\begin{equation}
\label{Ch3Sec6: EqLocTimeBridgePlus1}
\mathbb{P}_{y,z}^{(a\times),t}(\ell^{x}_{s_{2}}(\gamma)-\ell^{x}_{s_{1}}(\gamma))= 
\int_{s_{1}}^{s_{2}}\dfrac{p_{r}^{(a\times)}(y,x)p_{s-r}^{(a\times)}(x,z)}{p_{s}^{(a\times)}(y,z)}dr,
\end{equation}
and under the bridge measure $\mathbb{P}_{y,a^{+}}^{(a\times),s}$,
\begin{equation}
\label{Ch3Sec6: EqLocTimeBridgePlus2}
\mathbb{P}_{y,a^{+}}^{(a\times),t}(\ell^{x}_{s_{2}}(\gamma)-\ell^{x}_{s_{1}}(\gamma))= \int_{s_{1}}^{s_{2}}\dfrac{p_{r}^{(a\times)}(y,x)\partial_{2}p_{s-r}^{(a\times)}(x,a^{+})}
{\partial_{2}p_{s}^{(a\times)}(y,a^{+})}dr.
\end{equation}
Combining \eqref{Ch3Sec6: EqDisintegExcursions} and \eqref{Ch3Sec6: EqLocTimeBridgePlus2} we get that for every $s_{1}<s_{2}\in[0,s]$,
\begin{equation}
\label{Ch3Sec6: EqLocTimeExc}
\eta_{t}^{>a}(\ell^{x}_{s_{2}}(\gamma)-\ell^{x}_{s_{1}}(\gamma))= \int_{s_{1}}^{s_{2}}\dfrac{\partial_{2}
p_{s}^{(a\times)}(x,a^{+})\partial_{2}p_{t-s}^{(a\times)}(x,a^{+})}{\partial_{1,2}p_{t}^{(a\times)}(a^{+},a^{+})}ds.
\end{equation}

\begin{prop}
\label{Ch3Sec6: PropExcLocTime}
Let $F_{1}$ and $F_{2}$ be two non-negative measurable functional on the paths with variable life-time. Then,
\begin{multline}
\label{Ch3Sec6: EqExcLocTimeDesinteg}
\eta^{>a}_{t}\left(\int_{0}^{t}F_{1}((\gamma(r))_{0\leq r\leq s})F_{2}((\gamma(s+r))_{0\leq r\leq t-s})
d_{s}\ell^{x}_{s}(\gamma)\right)=\\
\int_{0}^{t}\mathbb{P}_{x,a^{+}}^{(a\times),s\wedge}(F_{1})\mathbb{P}_{x,a^{+}}^{(a\times),t-s}(F_{2})\dfrac{\partial_{2}p_{s}^{(a\times)}(x,a^{+})\partial_{2}p_{t-s}^{(a\times)}(x,a^{+})}
{\partial_{1,2}p_{t}^{(a)}(a^{+},a^{+})}\,ds.
\end{multline}
In particular,
\begin{equation}
\label{Ch3Sec6: EqExcLocTimeDesinteg2}
\ell^{x}_{t}(\gamma)\eta^{>a}_{t}(d\gamma)=
\int_{0}^{t}\left(\mathbb{P}_{x,a^{+}}^{(a\times),s \wedge}\lhd\mathbb{P}_{x,a^{+}}^{(a\times),t-s}\right)(d\gamma) \dfrac{\partial_{2}p_{s}^{(a\times)}(x,a^{+})\partial_{2}p_{t-s}^{(a\times)}(x,a^{+})}
{\partial_{1,2}p_{t}^{(a\times)}(a^{+},a^{+})}ds.
\end{equation}
\end{prop}

\begin{proof}
It is enough to prove the result in case $F_{1}$ and $F_{2}$ are non-negative, continuous and bounded. On top of that we may assume that there are $s_{\rm min}<s_{\rm max}\in (0,t)$ such that $F_{1}$ respectively $F_{2}$ takes value $0$ if the life-time of a path is smaller than $s_{\rm min}$ respectively $t-s_{\rm max}$, and that there is $C\in I$, $C>a$, such that $F_{1}$ and $F_{2}$ take value $0$ if $\max \gamma >C$. For $j\leq n\in\mathbb{N}$, set 
$\Delta s_{n}:=\frac{1}{n}(s_{\rm max}-s_{\rm min})$ and $s_{j,n}:=s_{\rm min}+j\Delta s_{n}$. Then, almost surely,
\begin{multline}
\label{Ch3Sec6: EqExcLocTimeDesintegApprox}
\int_{0}^{t}F_{1}((\gamma(r))_{0\leq r\leq s})F_{2}((\gamma(s+r))_{0\leq r\leq t-s})d_{s}\ell^{x}_{s}(\gamma)=\\
\lim_{n\rightarrow +\infty}\sum_{j=0}^{n-1}\!F_{1}((\gamma(r))_{0\leq r\leq s_{j,n}})(\ell^{x}_{s_{j+1,n}}(\gamma)-\ell^{x}_{s_{j,n}}(\gamma))F_{2}((\gamma(s_{j+1,n}+r))_{0\leq r\leq t-s_{j+1,n}}).
\end{multline}
Moreover, the right-hand side of \eqref{Ch3Sec6: EqExcLocTimeDesintegApprox} is dominated by 
$\ell^{x}_{t}(\gamma)\Vert F_{1}\Vert_{\infty}\Vert F_{2}\Vert_{\infty}$. Thus, the $\eta_{t}^{>a}$-expectation converges too. Applying \eqref{Ch3Sec6: EqDisintegExcursions} and \eqref{Ch3Sec6: EqLocTimeBridgePlus1} we get
\begin{multline*}
\eta_{t}^{>a}\big(F_{1}((\gamma(r))_{0\leq r\leq s_{j,n}})(\ell^{x}_{s_{j+1,n}}(\gamma)-\ell^{x}_{s_{j,n}}(\gamma))F_{2}((\gamma(s_{j+1,n}+r))_{0\leq r\leq t-s_{j+1,n}})\big)=\\
\int_{0}^{\Delta s_{n}}\int_{(a,C)^{2}}
\mathbb{P}_{y,a^{+}}^{(a\times),s_{j,n} \wedge}(F_{1})\mathbb{P}_{z,a^{+}}^{(a\times),t-s_{j+1,n}}(F_{2})
q_{n}(r,y,z) m(y)dy m(z) dz dr,
\end{multline*}
where
\begin{displaymath}
q_{n}(r,y,z) = \dfrac{\partial_{2}p_{s_{j,n}}^{(a\times)}(y,a^{+})\partial_{2}p_{t-s_{j+1,n}}^{(a\times)}(z,a^{+})}{\partial_{1,2}p_{t}^{(a\times)}(a^{+},a^{+})}p_{r}^{(a\times)}(y,x)p_{\Delta s_{n}-r}^{(a\times)}(x,z).
\end{displaymath} 
The measure $1_{y,z>a\in I}\dfrac{1}{\Delta s_{n}}\int_{0}^{\Delta s_{n}}q_{n}(r,y,z)dr dy dz$ converges weakly as $n\rightarrow +\infty$ to $\delta_{(x,x)}$. The maps $(s,y)\mapsto\partial_{2}p_{s}^{(a\times)}(x,a^{+})$  and  $(s,y)\mapsto\mathbb{P}_{s}^{(a\times),y,a^{+}}(\cdot)$ are continuous. Moreover, 
$\partial_{2}p_{s_{j,n}}^{(a\times)}(y,a^{+})\partial_{2}p_{t-s_{j+1,n}}^{(a\times)}(z,a^{+})$ is uniformly bounded for $j\leq n\in \mathbb{N}$ and $y,z\in (a,C]$. All this ensures that the $\eta_{t}^{>a}$-expectation of the right-hand side of \eqref{Ch3Sec6: EqExcLocTimeDesintegApprox} converges as $n\rightarrow +\infty$ to the right-hand side of \eqref{Ch3Sec6: EqExcLocTimeDesinteg}.
\end{proof}

Now we need only to match the preceding descriptions to prove Proposition \ref{Ch3Sec6: PropConditionnedVervaat}. \eqref{Ch3Sec6: EqDesintegMinBridge} and \eqref{Ch3Sec6: EqExcLocTimeDesinteg2} imply 
\eqref{Ch3Sec6: EqConditionedVervaat}. \eqref{Ch3Sec6: EqDistribMinBridge2} and \eqref{Ch3Sec6: EqLocTimeExc} imply \eqref{Ch3Sec6: EqDistribMinBridge}. The fact that the point where the excursion is split is distributed according to $\frac{d_{s}\ell^{x}_{s}(\gamma)}{\ell^{x}_{t}(\gamma)}$ follows from \eqref{Ch3Sec6: EqExcLocTimeDesinteg}.

\section{Restricting loops to a discrete subset}
\label{Ch3Sec7}

Let $L$ be the generator of a diffusion on $I$ of form \eqref{Ch2Sec2: EqGenerator} and $(X_{t})_{0\leq t <\zeta}$ be the corresponding diffusion. Let $\mathbb{J}$ be a countable discrete subset of $I$. A Markov jump process to nearest neighbors on $\mathbb{J}$ is naturally embedded in the diffusion $X$. In this section we will show that, given any $x,y\in\mathbb{J}$, the image of the measure $\mu^{x,y}_{L}$ through the restriction application that sends a sample paths of the diffusion $(X_{t})_{0\leq t <\zeta}$ to a sample path of a Markov jump process on $\mathbb{J}$, is a measure on $\mathbb{J}$-valued paths that follows the pattern \eqref{Ch3Sec2: EqVarBridgeGraph}. From this we will deduce that the image of the measure $\mu^{\ast}_{L}$ through the restriction to $\mathbb{J}$ is a measure on $\mathbb{J}$-valued loops following the pattern \eqref{Ch3Sec2: EqMesLoopsGraph} and which was studied in \cite{LeJan2011Loops}. This property will be used in Section \ref{Ch4Sec2} to express the law of finite-dimensional marginals of the occupation field of a Poisson ensemble of intensity $\alpha\mu^{\ast}_{L}$.

For a continuous path $(\gamma(t))_{0\leq t\leq T(\gamma)}$ in $I$, endowed with continuous local times, let
\begin{displaymath}
\mathcal{I}^{\mathbb{J}}_{t}(\gamma):=\sum_{x\in \mathbb{J}} \ell^{x}_{t}(\gamma)m(x).
\end{displaymath}
For $s\geq 0$, we introduce the stopping time
\begin{displaymath}
\tau^{\mathbb{J}}_{s}(\gamma):=\inf\lbrace t\geq 0\vert \mathcal{I}^{\mathbb{J}}_{t}(\gamma) \geq s\rbrace.
\end{displaymath}
We write $\gamma^{\mathbb{J}}$ for the path 
$(\gamma(\tau^{\mathbb{J}}_{s}))_{0\leq s\leq\mathcal{I}^{\mathbb{J}}_{T(\gamma)}(\gamma)}$ on $\mathbb{J}$. Let $m_{\mathbb{J}}$ be the measure
\begin{displaymath}
m_{\mathbb{J}}:=\sum_{x\in\mathbb{J}}m(x)\delta_{x}.
\end{displaymath}
The occupation measure of $\gamma^{\mathbb{J}}$ is
\begin{displaymath}
\sum_{x\in\mathbb{J}}\ell^{x}(\gamma)m(x)\delta_{x},
\end{displaymath}
and $(\ell^{x}(\gamma))_{x\in \mathbb{J}}$ are also occupation densities of the restricted path $\gamma^{\mathbb{J}}$ with respect to $m_{\mathbb{J}}$.

The restricted diffusion $X^{\mathbb{J}}$ is a Markov jump process to nearest neighbors on $\mathbb{J}$, potentially with killing. If $x_{0}<x_{1}$ are two consecutive points in $\mathbb{J}$, the jump rate from $x_{0}$ to $x_{1}$ is 
$\frac{1}{m(x_{0})w(x_{0})}\frac{1}{u^{+,x_{0}}(x_{1})}$ and the jump rate from $x_{1}$ to $x_{0}$ is 
$\frac{1}{m(x_{1})w(x_{1})}\frac{1}{u^{-,x_{1}}(x_{0})}$. If $x_{0}<x_{1}<x_{2}$ are three consecutive points in $\mathbb{J}$, then the rate of killing while in $x_{1}$ is
\begin{displaymath}
\dfrac{1}{m(x_{1})w(x_{1})}\left( \dfrac{W(u^{-,x_{2}},u^{+,x_{0}})(x_{1})}{u^{-,x_{2}}(x_{1})u^{+,x_{0}}(x_{1})}
-\dfrac{1}{u^{-,x_{1}}(x_{0})}-\dfrac{1}{u^{+,x_{1}}(x_{2})}\right). 
\end{displaymath}
If $\mathbb{J}$ has a minimum $x_{0}$ and $x_{1}$ is the second lowest point in $\mathbb{J}$, then the killing rate while in $x_{0}$ is 
\begin{displaymath}
\dfrac{1}{m(x_{0})w(x_{0})}\left(\dfrac{W(u^{-,x_{1}},u_{\uparrow})(x_{0})}{u^{-,x_{1}}(x_{0})u_{\uparrow}(x_{0})}
-\dfrac{1}{u^{+,x_{0}}(x_{1})}\right). 
\end{displaymath}
An analogous expression holds for the killing rate while in a possible maximum of $\mathbb{J}$. $X^{\mathbb{J}}$ is transient if and only if $X$ is. Let $L_{\mathbb{J}}$ be the generator of $X^{\mathbb{J}}$. $L_{\mathbb{J}}$ is symmetric relatively to $m_{\mathbb{J}}$. Its Green's function relatively to $m_{\mathbb{J}}$ is $(G(x,y))_{x,y\in I}$, that is the restriction of the Green's function of $L$ to $\mathbb{J}\times \mathbb{J}$. $X^{\mathbb{J}}$ may not be conservative even if the diffusion $X$ is. In case if  $\mathbb{J}$ is not finite, $X^{\mathbb{J}}$ may blow up performing an infinite number of jumps in finite time. Measures $(\mu^{x,y}_{L})_{x,y\in I}$, $\mu_{L}$ and $\mu^{\ast}_{L}$ have discrete space analogues $(\mu^{x,y}_{L_{\mathbb{J}}})_{x,y\in \mathbb{J}}$, $\mu_{L_{\mathbb{J}}}$ and $\mu^{\ast}_{L_{\mathbb{J}}}$, as defined in \cite{LeJan2011Loops}, that follow the patterns \eqref{Ch3Sec2: EqVarBridgeGraph} and 
\eqref{Ch3Sec2: EqMesLoopsGraph}.

\begin{prop}
\label{Ch3Sec7: PropRestrictionLoops}
Let $x,y\in \mathbb{J}$. Then $\gamma\mapsto\gamma^{\mathbb{J}}$ transforms $\mu^{x,y}_{L}$ into $\mu^{x,y}_{L_{\mathbb{J}}}$ and $\mu^{\ast}_{L}$ into $\mu^{\ast}_{L_{\mathbb{J}}}$.
\end{prop}

\begin{proof}
The representation \eqref{Ch3Sec2: EqBridgeLocalTime} also holds for $\mu^{x,y}_{L_{\mathbb{J}}}$. For $l>0$, let 
\begin{displaymath}
\tau_{l}^{y}:=\inf\lbrace t\geq 0\vert \ell^{y}_{t}(X)> l\rbrace,
\end{displaymath}
and
\begin{displaymath}
\tau_{l}^{y,\mathbb{J}}:=\inf\lbrace s\geq 0\vert \ell^{y}_{s}(X^{\mathbb{J}})> l\rbrace.
\end{displaymath}
Then for any non-negative measurable functional $F$,
\begin{displaymath}
\mu^{x,y}_{L_{\mathbb{J}}}(F(\gamma))=\int_{0}^{+\infty}dl\mathbb{E}_{x}
\left[1_{\tau_{l}^{y,\mathbb{J}}<\mathcal{I}^{\mathbb{J}}_{\zeta}} 
F((X^{\mathbb{J}}_{s})_{0\leq s\leq \tau_{l}^{y,\mathbb{J}}})\right]. 
\end{displaymath}
But, $(X^{\mathbb{J}}_{s})_{0\leq s\leq \tau_{\lambda}^{y,\mathbb{J}}}$ is the image of 
$(X_{t})_{0\leq t\leq \tau_{\lambda}^{y}}$ by the map $\gamma\mapsto\gamma^{\mathbb{J}}$, and 
$\tau_{l}^{y,\mathbb{J}}<\mathcal{I}^{\mathbb{J}}_{\zeta}$ if and only if $\tau_{l}^{y}<\zeta$. Thus, $\mu^{x,y}_{L_{\mathbb{J}}}$ is the image of $\mu^{x,y}_{L}$ through the restriction of paths to $\mathbb{J}$. The second part of the proposition can be deduced from that for any $x\in \mathbb{J}$,
\begin{displaymath}
\ell^{x}(\gamma)\mu^{\ast}_{L}(d\gamma)=\pi_{\ast}\mu^{x,x}_{L}(d\gamma),
\end{displaymath}
and as noticed in \cite{LeJan2011Loops},
\begin{displaymath}
\ell^{x}(\gamma)\mu^{\ast}_{L_{\mathbb{J}}}(d\gamma^{\mathbb{J}})=\pi_{\ast}\mu^{x,x}_{L_{\mathbb{J}}}(d\gamma^{\mathbb{J}}).
\qedhere
\end{displaymath}
\end{proof}

Previous restriction property and the time-change covariance of $\mu^{\ast}$ (Corollary \ref{Ch3Sec3: CorTimeChange}) can be treated in a unified framework of the time change by the inverse of a continuous additive functional. This is done in \cite{FitzsimmonsRosen2012LoopsIsomorphism}, Section $7$.

\section{Measure on loops in case of creation of mass}
\label{Ch3Sec8}

We can further extend the definition of the measures $\mu^{x,y}$ on paths and $\mu$ and $\mu^{\ast}$ on loops to the case of $L$ being a "generator" on $I$ containing a creation of mass term, as in \eqref{Ch2Sec3: GenSignedMeasure}. Doing so will enable us to emphasize further the conjugation invariance of the measure on loops and will be useful in Section \ref{Ch4Sec2} to compute the exponential moments of the occupation field of Poisson ensembles of Markov loops. Let $\nu$ be signed measure on $I$. Let $L^{(0)}:=\frac{1}{m(x)}\frac{d}{dx}\left(\frac{1}{w(x)}\frac{d}{dx}\right)$ and $L:=L^{(0)}+\nu$.

\begin{defi}
\label{Ch3Sec8: DefLoopsSignedMes}
\begin{itemize}
\item The measure on paths is $$\mu_{L}^{x,y}(d\gamma):=\exp\left(\int_{I}l^{x}(\gamma)m(x)\nu(dx)\right)\mu_{L^{(0)}}^{x,y}(d\gamma),$$
\item the measure on rooted loops is $$\mu_{L}(d\gamma):=\exp\left(\int_{I}l^{x}(\gamma)m(x)\nu(dx)\right)\mu_{L^{(0)}}(d\gamma),$$
\item the measure on unrooted loops is $\mu^{\ast}_{L}:=\pi_{\ast}\mu_{L}$.
\end{itemize}
\end{defi}

Definition \ref{Ch3Sec8: DefLoopsSignedMes} is consistent with Properties \ref{Ch3Sec2: PropertyBridges} (iv) and 
\ref{Ch3Sec3: PropertyTrivialitiesLoops} (iii). If $\tilde{\nu}$ is any other signed measure on $I$, then
\begin{equation}
\label{Ch3Sec8: EqAddingSignedMeasure}
\mu_{L + \tilde{\nu}}^{x,y}(d\gamma):=\exp\left(\int_{I}\ell^{x}(\gamma)m(x)\tilde{\nu}(dx)\right)\mu_{L}^{x,y}(d\gamma).
\end{equation}
Same holds for $\mu$ and $\mu^{\ast}$. Under the extended definition, the measures $\mu^{x,y}$ still satisfy Properties \ref{Ch3Sec2: PropertyBridges} (ii), (iii), (v) and (vi). Proposition \ref{Ch3Sec2: PropContinuityBridge} remains true. $\mu$ still satisfies Properties \ref{Ch3Sec3: PropertyTrivialitiesLoops} (i), (ii) and (iv). Proposition 
\ref{Ch3Sec3: PropShiftInv} and Corollary \ref{Ch3Sec3: CorProjSubalg} still hold. The identities 
\eqref{Ch3Sec3: EqLoopMeasureLocTime} and \eqref{Ch3Sec3: EqLoopTimeChange} remain true for $\mu^{\ast}$.
Concerning the conjugation, we have:

\begin{prop}
\label{Ch3Sec8: ProphTransfInv}
Let $u$ be a continuous positive function on $I$ such that $\frac{d^{2}u}{dx^{2}}$ is a signed measure. $u(x)^{2}m(x)dx$ is a speed measure for $\operatorname{Conj}(u,L)$. Then for all $x,y\in I$, $\mu^{x,y}_{\operatorname{Conj}(u,L)} = \frac{1}{u(x)u(y)}\mu^{x,y}_{L}$, and 
$\mu_{\operatorname{Conj}(u,L)}=\mu_{L}$. Conversely, if $L$ and $\widetilde{L}$ are two "generators" with or without creation of mass such that $\mu_{L}^{\ast}  = \mu_{\widetilde{L}}^{\ast}$, then there is a positive continuous function $u$ on $I$ such that $\frac{d^{2}u}{dx^{2}}$ is a signed measure and $\widetilde{L}=\operatorname{Conj}(u,L)$.
\end{prop}

\begin{proof}
There is a positive Radon measure $\tilde{\kappa}$ on $I$ such that both $L-\tilde{\kappa}$ and $\operatorname{Conj}(h,L)-\tilde{\kappa}$ are generators of (killed) diffusions. But
\begin{displaymath}
\operatorname{Conj}(u,L)-\tilde{\kappa}=\operatorname{Conj}(u,L-\tilde{\kappa}).
\end{displaymath}
It follows that ${\mu^{x,y}_{\operatorname{Conj}(u,L)-\tilde{\kappa}} = \frac{1}{u(x)u(y)}\mu^{x,y}_{L-\tilde{\kappa}}}$, and $\mu_{\operatorname{Conj}(u,L)-\tilde{\kappa}}=\mu_{L-\tilde{\kappa}}$. Applying \eqref{Ch3Sec8: EqAddingSignedMeasure} 
we get the result.

If $\mu_{L}^{\ast}  = \mu_{\widetilde{L}}^{\ast}$, we can again consider $\tilde{\kappa}$ a positive Radon measure on $I$ such that both $L-\tilde{\kappa}$ and $\widetilde{L}-\tilde{\kappa}$ are generators of (killed) diffusions.
We have $\mu_{L-\tilde{\kappa}}^{\ast}  = \mu_{\widetilde{L}-\tilde{\kappa}}^{\ast}$. According to Proposition \ref{Ch3Sec4: PropConversehTransfInv}, there is a positive continuous function $u$ on $I$ such that $\frac{d^{2}u}{dx^{2}}$ is a signed measure and 
$\widetilde{L}-\tilde{k}=\operatorname{Conj}(u,L-\tilde{k})$. Then, $\widetilde{L}=\operatorname{Conj}(u,L)$.
\end{proof}

Similarly to the case of generators of diffusions (Section \ref{Ch3Sec5}), one can consider $L$-harmonic functions $u^{-,x}$ and $u^{+,x}$ in case of $L$ containing creation of mass. If $L\in\mathfrak{D}^{+}$, then $u^{-,x}$ respectively $u^{+,x}$ is not necessarily positive on $I\cap(-\infty,x)$ respectively $I\cap(x,+\infty)$. Let
\begin{displaymath}
M(x):=\sup\lbrace y\in I, y\geq x\vert\forall z\in (x,y),u^{+,x}(z)> 0\rbrace\in I\cup\lbrace\sup I\rbrace.
\end{displaymath}
If $L\in\mathfrak{D}^{0,-}$, then for all $x\in I$, $M(x)=\sup I$. Let $y\in I$, $y>x$. If $y<M(x)$, then 
$L_{\vert (x,y)}\in\mathfrak{D}^{-}$. If $y=M(x)$, then $L_{\vert (x,y)}\in\mathfrak{D}^{0}$. If $y>M(x)$, then 
$L_{\vert (x,y)}\in\mathfrak{D}^{+}$. The diffusion $\rho^{+,x}$ of generator 
$L^{+,x}=\operatorname{Conj}(u^{+,x},L^{+,x}_{\vert (x,M(x))})$ is defined on $(x,M(x))$. Similarly for $\rho^{-,y}$. Moreover, if $M(x)\in I$, then ${L^{+,x}_{\vert (x,M(x))}=L^{-,M(x)}_{\vert (x,M(x))}}$.

If $L\in \mathfrak{D}^{0,-}$, the description of the measure on $(\min\gamma,\max\gamma)$ induced by $\mu^{\ast}$ as well as of the probability measures obtained by conditioning $\mu^{\ast}$ by the value of $(\min\gamma,\max\gamma)$ is the same as given by Corollary \ref{Ch3Sec5: CorDesintegMinGen}, with the same formal expressions. Next we state what happens if $L\in\mathfrak{D}^{+}$:

\begin{prop}
\label{Ch3Sec8: PropMinMaxD+}
Let $L\in\mathfrak{D}^{+}$. The measure on $(\min \gamma, \max \gamma)$ induced by $\mu^{\ast}$ and restricted to the set $\lbrace a\in I, b\in (a,M(a))\rbrace$ is $1_{a\in I, b\in (a,M(a))}\frac{da db}{u^{+,a}(b)u^{-,b}(a)}$. If $a<b<M(a)$, then the probability measure obtained through conditioning by $(\min \gamma, \max \gamma)=(a,b)$ has the same description as in corollary \ref{Ch3Sec5: CorDesintegMinGen}. Outside the set $\lbrace a\in I, b\in (a,M(a))\rbrace$, the measure on $(\min \gamma, \max \gamma)$ is not locally finite. That is to say that, if $a<b\in I$ and $b\geq M(a)$, then for all $\varepsilon>0$.
\begin{equation}
\label{Ch3Sec8: EqInfiniteTailsLoops}
\mu^{\ast}(\lbrace\min \gamma \in (a,a+\varepsilon), \max \gamma \in (b-\varepsilon,b)\rbrace)=+\infty.
\end{equation}
\end{prop}

\begin{proof}
For the behavior on $\lbrace a\in I, b\in (a,M(a))\rbrace$: There is a countable collection $(I_{j})_{j\geq 0}$ of open subintervals of $I$ such that
\begin{displaymath}
\lbrace a\in I, b\in (a,M(a))\rbrace = \bigcup_{j\geq 0}\lbrace x<y\in I_{j}\rbrace.
\end{displaymath}
Since for all $j$, $L_{\vert I_{j}}\in\mathfrak{D}^{0,-}$, Corollary \ref{Ch3Sec5: CorDesintegMinGen} applies to 
$L_{\vert I_{j}}$. Combining the descriptions on different $\lbrace a<b\in I_{j}\rbrace$, we get the description on $\lbrace a\in I, b\in (a,M(a))\rbrace$.

For the behavior outside  $\lbrace a\in I, b\in (a,M(a))\rbrace$: Let $A<B\in\mathbb{R}$. Then, 
\begin{equation}
\label{Ch3Sec8: EqInfiniteTailsLoopsBM}
\mu^{\ast}_{BM}(\lbrace\min\gamma <A, \max\gamma >B\rbrace)=\int_{B}^{+\infty}\int_{-\infty}^{A}\dfrac{da db}{(b-a)^{2}} = +\infty.
\end{equation}
If $a<b\in I$ and $M(a)=b$, then $1_{a<\gamma <b}\mu^{\ast}$ is the image of $\mu^{\ast}_{BM}$ through a change of scale and time. In this case, \eqref{Ch3Sec8: EqInfiniteTailsLoops} follows from \eqref{Ch3Sec8: EqInfiniteTailsLoopsBM}. If $b>M(a)$, then $L_{\vert (a,b)}\in\mathfrak{D}^{+}$. According to Proposition \ref{Ch2Sec3: PropStability} (iv), there is a positive measure Radon measure $\kappa$ on $(a,b)$ such that $L_{\vert (a,b)}-\kappa\in\mathfrak{D}^{0}$. From what precedes, \eqref{Ch3Sec8: EqInfiniteTailsLoops} holds for $\mu^{\ast}_{L_{\vert (a,b)}-\kappa}$. Moreover, $\mu^{\ast}_{L_{\vert (a,b)}}\geq\mu^{\ast}_{L_{\vert (a,b)}-\kappa}$. So, \eqref{Ch3Sec8: EqInfiniteTailsLoops} holds for $\mu^{\ast}_{L_{\vert (a,b)}}$.
\end{proof}

\chapter{Occupation fields of the Poisson ensembles of Markov loops}
\label{Ch4}

\section{Inhomogeneous continuous state branching processes with immigration}
\label{Ch4Sec1}

In this chapter we will introduce the Poisson ensembles of Markov loops $\mathcal{L}_{\alpha}$ of intensity $\alpha\mu^{\ast}$ and consider their occupation fields, that is to say the sums over individual loops of the local times at some level.
We will identify these occupation fields as inhomogeneous continuous state branching processes with immigration. This will be done in Section \ref{Ch4Sec2}. In the Section \ref{Ch4Sec1} we will give the basic properties of such branching processes. In Section \ref{Ch4Sec3} we will deal with the particular case of the intensity being $\frac{1}{2}\mu^{\ast}$, in relation with Dynkin's isomorphism. 

Let $I$ be an open interval of $\mathbb{R}$. We will consider stochastic processes where $x\in I$ is the evolution variable. We do not call it time because in the sequel it will rather represent a space variable. Let $(\mathbb{B}_{x})_{x\in\mathbb{R}}$ be a standard Brownian motion. Consider the following SDEs:
\begin{equation}
\label{Ch4Sec1: EqBranchingMechanism}
d\widetilde{Z}_{x}=\sigma(x)\sqrt{\widetilde{Z}_{x}}d\mathbb{B}_{x} + b(x)\widetilde{Z}_{x}dx,
\end{equation}
\begin{equation}
\label{Ch4Sec1: EqBranchingImmigration}
dZ_{x} = \sigma(x)\sqrt{Z_{x}} d\mathbb{B}_{x} + b(x)Z_{x} dx + c(x) dx.
\end{equation}
$d\mathbb{B}_{x}$ is the spatial white noise. We use Itô's integral for it, as our space is ordered incresingly.

For our needs we will assume that $\sigma$ is positive and continuous on $I$, that $b$ and $c$ are only locally bounded and that $c$ is non negative. In this case existence and pathwise uniqueness holds for 
\eqref{Ch4Sec1: EqBranchingMechanism} and \eqref{Ch4Sec1: EqBranchingImmigration} (see \cite{RevuzYor1999BMGrundlehren}, Section IX.3), and $\widetilde{Z}$ and $Z$ take values in $\mathbb{R}_{+}$. $0$ is an absorbing state for $\widetilde{Z}$.

\eqref{Ch4Sec1: EqBranchingMechanism} satisfies the branching property: if $\widetilde{Z}^{(1)}$ and $\widetilde{Z}^{(2)}$ are two independent processes solutions in law to \eqref{Ch4Sec1: EqBranchingMechanism}, defined on 
$I\cap[x_{0},+\infty)$, then $\widetilde{Z}^{(1)} + \widetilde{Z}^{(2)}$ is a solution in law to 
\eqref{Ch4Sec1: EqBranchingMechanism}. If $\widetilde{Z}$ and $Z$ are two independent processes, $\widetilde{Z}$ solution in law to \eqref{Ch4Sec1: EqBranchingMechanism} and $Z$ solution in law to \eqref{Ch4Sec1: EqBranchingImmigration}, defined on $I\cap[x_{0},+\infty)$, then $Z+\widetilde{Z}$ is a solution in law to \eqref{Ch4Sec1: EqBranchingImmigration}. Solutions to \eqref{Ch4Sec1: EqBranchingImmigration} are (inhomogeneous) continuous state branching processes with immigration. The branching mechanism is given by \eqref{Ch4Sec1: EqBranchingMechanism} and the immigration measure is $c(x) dx$. The homogeneous case ($\sigma$, $b$ and $c$ constant) was extensively studied. 
See \cite{KawazuWatanabe1971HomCSBI,ShigaWatanabe73BesselCBI}.

The case of inhomogeneous branching without immigration reduces to the homogeneous case as follows: Let $x_{0}\in I$ and let
\begin{displaymath}
C(x):=\exp\left(-\int_{x_{0}}^{x} b(y)\,dy\right),\quad A(x):=\int_{x_{0}}^{x} \sigma(y)^{2}C(y)^{2}dy.
\end{displaymath}
If $(\widetilde{Z}_{x})_{x\in I}$ is a solution to \eqref{Ch4Sec1: EqBranchingMechanism}, then 
$(C(A^{-1}(a))\widetilde{Z}_{A^{-1}(a)})_{a\in A(I)}$ is a solution in law to
\begin{displaymath}
d\widetilde{\mathcal{Z}}_{a} = 2\sqrt{\widetilde{\mathcal{Z}}_{a}} d\mathbb{B}_{a}.
\end{displaymath}

Let $\widetilde{Z}$ be a solution to \eqref{Ch4Sec1: EqBranchingMechanism} defined on $I\cap[x_{0},+\infty)$, starting at $x_{0}$ with the initial condition $\widetilde{Z}_{x_{0}}=z_{0}\geq 0$. Then, for $\lambda\geq 0$ and $x\in I$, 
$x\geq x_{0}$:
\begin{displaymath}
\mathbb{E}_{\widetilde{Z}_{x_{0}}=z_{0}}\left[e^{-\lambda \widetilde{Z}_{x}}\right] = e^{-z_{0}\psi(x_{0},x,\lambda)}.
\end{displaymath}
$\psi(x_{0},x,\lambda)$ depends continuously on $(x_{0},x,\lambda)$. If $x=x_{0}$, then
\begin{equation}
\label{Ch4Sec1: EqCharBranchingInitialCond}
\psi(x_{0},x_{0},\lambda)=\lambda .
\end{equation}
If $x_{0}\leq x_{1}\leq x_{2}\in I$, then
\begin{displaymath}
\psi(x_{0},x_{2},\lambda) = \psi(x_{0},x_{1},\psi(x_{1},x_{2},\lambda)).
\end{displaymath}
$\psi$ satisfies the differential equation
\begin{equation}
\label{Ch4Sec1: EqCharEqBranching}
\dfrac{\partial \psi}{\partial x_{0}}(x_{0},x,\lambda) = 
\dfrac{\sigma(x_{0})^{2}}{2}\psi(x_{0},x,\lambda)^{2} - b(x_{0})\psi(x_{0},x,\lambda).
\end{equation}
If $b$ is not continuous, equation \eqref{Ch4Sec1: EqCharEqBranching} should be understand in the weak sense. If $b$ is continuous, then \eqref{Ch4Sec1: EqCharEqBranching} satisfies the Cauchy-Lipschitz conditions, and $\psi$ is uniquely determined by \eqref{Ch4Sec1: EqCharEqBranching} and the initial condition \eqref{Ch4Sec1: EqCharBranchingInitialCond}. This is also the case even if $b$ is not continuous. Indeed, by considering $C(x)\widetilde{Z}_{x}$ rather than $\widetilde{Z}_{x}$, that is to say considering $\frac{C(x)}{C(x_{0})}\psi(x_{0},x,\lambda)$ rather than 
$\psi(x_{0},x,\lambda)$, we get rid of $b$.

Inhomogeneous branching processes are related to the local times of general one-dimensional diffusions:

\begin{prop}
\label{Ch4Sec1: PropRayKnight}
Let $x_{0}\in I$ and let $(X_{t})_{0\leq t<\zeta}$ be a diffusion on $I$ of generator $L$ of form 
\eqref{Ch2Sec2: EqGenerator} starting from $x_{0}$. Let $z_{0}>0$ and
\begin{displaymath}
\tau^{x_{0}}_{z_{0}}:=\inf\lbrace t\geq 0\vert \ell^{x_{0}}_{t}(X)> z_{0}\rbrace.
\end{displaymath}
Then, conditional on $\tau^{x_{0}}_{z_{0}}<\zeta$, $(\ell^{x}_{\tau^{x_{0}}_{z_{0}}}(X))_{x\in I, x\geq x_{0}}$ is a solution in law to the SDE
\begin{equation}
\label{Ch4Sec1: EqFirstRayKnight}
d\widetilde{Z}_{x}=\sqrt{2w(x)}\sqrt{\widetilde{Z}_{x}} d\mathbb{B}_{x}+ 
2\dfrac{d\log u_{\downarrow}}{dx}(x)\widetilde{Z}_{x} dx.
\end{equation}
\end{prop}
\begin{proof}
If $X$ is the Brownian motion on $\mathbb{R}$, then $w\equiv 2$ and $u_{\downarrow}$ is constant. In this case the assertion is the second Ray-Knight theorem. See \cite{RevuzYor1999BMGrundlehren}, Section XI.2. The equation 
\eqref{Ch4Sec1: EqFirstRayKnight} is then the equation of a square of Bessel $0$ process. If $x_{\rm min}<x_{0}$ and $X$ is the Brownian motion on $(x_{\rm min},+\infty)$ killed in $x_{\rm min}$, then the law of 
$(\ell^{x}_{\tau^{x_{0}}_{z_{0}}}(X))_{x\in I, x\geq x_{0}}$ conditional on $\tau^{x_{0}}_{z_{0}}<\zeta$ does not depend on $x_{\rm min}$ and is the same as in case of the Brownian motion on $\mathbb{R}$. Equation 
\eqref{Ch4Sec1: EqFirstRayKnight} is still satisfied.

If $X$ is a diffusion on $I$ that satisfies that for all $x>a\in I$, starting from $x$, $X$ reaches almost surly $a$, which is equivalent to $u_{\downarrow}$ being constant, then through a change of scale and time $X$ is the Brownian motion on some $(x_{\rm min},+\infty)$, where $x_{\rm min}\in[-\infty,+\infty)$. Time change does not change the local times because we defined them relatively to the speed measure. Only the change of scale matters. If $S$ is a primitive of $w$, then, conditional on $\tau^{x_{0}}_{z_{0}}<\zeta$, $(\ell^{S^{-1}(2y)}_{\tau^{x_{0}}_{z_{0}}}(X))_{y\geq \frac{1}{2}S(x_{0})}$ is a square of Bessel $0$ process. The equation \eqref{Ch4Sec1: EqFirstRayKnight} follows from the equation of the square of Bessel $0$ process by deterministic change of variable $dy:=\frac{1}{2}w(x)dx$.

Now the general case. Let $(\widetilde{X}_{t})_{0\leq t<\tilde{\zeta}}$ be the diffusion of generator $\operatorname{Conj}(u_{\downarrow},L)$. $\frac{w(x)}{u_{\downarrow}(x)^{2}}dx$ is the natural scale measure of $\widetilde{X}$ and $u_{\downarrow}(x)^{2}m(x)dx$ is its speed measure. We assume that both $X$ and $\widetilde{X}$ start from $x_{0}$. The law of $\widetilde{X}$ up to the last time it visits $x_{0}$ is the same as for $X$. Let
\begin{displaymath}
\tilde{\tau}:=\inf\left\lbrace t\geq 0\vert \ell^{x_{0}}_{t}(\widetilde{X})> \frac{1}{u_{\downarrow}(x_{0})^{2}}z_{0}\right\rbrace. 
\end{displaymath}
Then, the law of $(\ell^{x}_{\tau^{x_{0}}_{z_{0}}}(X))_{x\in I, x\geq x_{0}}$ conditional on $\tau^{x_{0}}_{z_{0}}<\zeta$ is the same as the law of $(u_{\downarrow}(x)^{2}\ell^{x}_{\tilde{\tau}}(\widetilde{X}))_{x\in I, x\geq x_{0}}$ conditional on $\tilde{\tau}<\tilde{\zeta}$. The factor $u_{\downarrow}(x)^{2}$ comes from the fact that performing an h-transform we change the measure relatively to which the local times are defined. For any $a<x_{0}\in I$, $\widetilde{X}$ reaches $a$ a.s. Thus, $(\ell^{x}_{\tilde{\tau}}(\tilde{X}))_{x\in I, x\geq x_{0}}$ satisfies the SDE
\begin{displaymath}
d\widetilde{Z}_{x} = \dfrac{\sqrt{2w(x)}}{u_{\downarrow}(x)}\sqrt{\widetilde{Z}_{x}} d\mathbb{B}_{x} ,
\end{displaymath}
and $(u_{\downarrow}(x)^{2}\ell^{x}_{\tilde{\tau}^{x_{0}}_{z_{0}}}(\widetilde{X}))_{x\in I, x\geq x_{0}}$ satisfies \eqref{Ch4Sec1: EqFirstRayKnight}.
\end{proof}

If there is immigration: Let $Z$ be a solution to \eqref{Ch4Sec1: EqBranchingImmigration} defined on 
$I\cap[x_{0},+\infty)$, starting at $x_{0}$ with the initial condition $Z_{x_{0}}=z_{0}\geq 0$. Then, for $\lambda\geq 0$ and $x\in I$, $x\geq x_{0}$,
\begin{equation}
\label{Ch4Sec1: EqLTCSBI}
\mathbb{E}_{Z_{x_{0}}=z_{0}}\left[e^{-\lambda Z_{x}}\right] = \exp\left(-z_{0}\psi(x_{0},x,\lambda)- \int_{x_{0}}^{x}\psi(y,x,\lambda) c(y) dy \right).
\end{equation}

\section{Occupation field}
\label{Ch4Sec2}

Let $L$ be the generator of a diffusion on $I$ of form \eqref{Ch2Sec2: EqGenerator}. 

\begin{defi}
The Poisson point process of Markov loops associated to $L$, with intensity parameter $\alpha>0$, is the
Poisson ensemble of intensity $\alpha \mu^{\ast}_{L}$, denoted $\mathcal{L}_{\alpha, L}$.
$\mathcal{L}_{\alpha, L}$ is a random infinite countable collection of unrooted loops supported in $I$. It is also called "loop-soup".
\end{defi}
 
\begin{defi}
\label{Ch4Sec2: DefOccupationField}
The occupation field of $\mathcal{L}_{\alpha, L}$ is $(\widehat{\mathcal{L}}_{\alpha,L}^{x})_{x\in I}$, where
\begin{displaymath}
\widehat{\mathcal{L}}_{\alpha,L}^{x}:=\sum_{\gamma\in \mathcal{L}_{\alpha, L}} \ell^{x}(\gamma).
\end{displaymath}
\end{defi}
We will drop out the subscript $L$ whenever there is no ambiguity on $L$. In this section we will identify the law of $(\widehat{\mathcal{L}}_{\alpha}^{x})_{x\in I}$ as an inhomogeneous continuous state branching process with immigration.
If $\mathbb{J}$ is a discrete subset of $I$, then applying Proposition \ref{Ch3Sec7: PropRestrictionLoops}, we deduce that $(\widehat{\mathcal{L}}_{\alpha}^{x})_{x\in\mathbb{J}}$ is the occupation field of the Poisson ensemble of discrete loops of intensity $\alpha\mu^{\ast}_{L_{\mathbb{J}}}$, as defined in \cite{LeJan2011Loops}, Chapter $4$. This fact allows us to apply the results of \cite{LeJan2011Loops} in order to describe the finite-dimensional marginals of the occupation field. If the diffusion is recurrent, then for all $x\in I$, $\widehat{\mathcal{L}}_{\alpha}^{x}=+\infty$ a.s. If the diffusion is transient, then for all $x\in I$, $\widehat{\mathcal{L}}_{\alpha}^{x}<+\infty$ a.s. Next we state how does the occupation field behave if we apply various transformations on $L$.

\begin{property}
\label{Ch4Sec2: PropertyOccupFieldTransf}
Let $L$ be the generator of a transient diffusion. 
\begin{itemize}
\item[(i)] If $A$ is a change of scale function, then
\begin{displaymath}
\widehat{\mathcal{L}}_{\alpha, \operatorname{Scale}_{A}^{\rm gen}L}^{A(x)} = \widehat{\mathcal{L}}_{\alpha,L}^{x}.
\end{displaymath}
\item[(ii)] If $V$ is a positive continuous function on $I$, then
\begin{displaymath}
\widehat{\mathcal{L}}_{\alpha, \frac{1}{V}L}^{x} = \widehat{\mathcal{L}}_{\alpha,L}^{x}.
\end{displaymath}
\item[(iii)] If $u$ is a positive continuous function on $I$ such that $Lu$ is a non-positive measure, then
\begin{displaymath}
\widehat{\mathcal{L}}_{\alpha,\operatorname{Conj}(u,L)}^{x} = \dfrac{1}{u(x)^{2}}\widehat{\mathcal{L}}_{\alpha,L}^{x}.
\end{displaymath}
\end{itemize}
\end{property}

Previous equalities depend on a particular choice of the speed measure for the modification of $L$. For (i) we choose $\left(\frac{dA}{dx}\circ A^{-1}\right)^{-1} m\circ A^{-1}da$. For (ii) we choose $\frac{1}{V(x)}m(x)dx$. For (iii) we choose $u(x)^{2}m(x)dx$. The fact that 
$\widehat{\mathcal{L}}_{\alpha,\operatorname{Conj}(u,L)}^{x}\neq\widehat{\mathcal{L}}_{\alpha,L}^{x}$ despite $\mathcal{L}_{\alpha,\operatorname{Conj}(u,L)}= \mathcal{L}_{\alpha,L}$ comes from a change of speed measure.

Next we characterize the finite-dimensional marginals of the occupation field by stating the results that appear in \cite{LeJan2011Loops}, Chapter 4.

\begin{property}
\label{Ch4Sec2: PropertyMarginals}
The distribution of $\widehat{\mathcal{L}}_{\alpha}^{x}$ is
\begin{displaymath}
\frac{(G_{L}(x,x))^{\alpha}}{\Gamma(\alpha)}l^{\alpha - 1}\exp\left(-\frac{l}{G_{L}(x,x)}\right)1_{l>0}dl.
\end{displaymath}
Let $x_{1}, x_{2},\dots,x_{n}\in I$ and $\lambda_{1},\lambda_{2},\dots,\lambda_{n}\geq 0$. Then,
\begin{equation}
\label{Ch4Sec2: EqLaplaceTransformOccupField}
\mathbb{E}\left[\exp\left(-\sum_{i=1}^{n}\lambda_{i}\widehat{\mathcal{L}}_{\alpha}^{x_{i}} \right) \right] = 
\left( \dfrac{\det(G_{L-\sum_{i=1}^{n}\lambda_{i}\delta_{x_{i}}}(x_{i},x_{j}))_{1\leq i,j\leq n}}
{\det(G_{L}(x_{i},x_{j}))_{1\leq i,j\leq n}}\right)^{\alpha}.
\end{equation}
The moment $\mathbb{E}\left[\widehat{\mathcal{L}}_{\alpha}^{x_{1}}\widehat{\mathcal{L}}_{\alpha}^{x_{2}}\dots
\widehat{\mathcal{L}}_{\alpha}^{x_{n}} \right]$ is an $\alpha$-permanent:
\begin{displaymath}
\mathbb{E}\left[\widehat{\mathcal{L}}_{\alpha}^{x_{1}}\widehat{\mathcal{L}}_{\alpha}^{x_{2}}\dots
\widehat{\mathcal{L}}_{\alpha}^{x_{n}} \right]=
\sum_{\sigma\in\mathfrak{S}_{n}}\alpha^{\#~\text{cycles of}~\sigma}\prod_{i=1}^{n}G(x_{i},x_{\sigma(i)}).
\end{displaymath}
If $\mathbb{J}$ is a discrete subset of $I$, then $(\widehat{\mathcal{L}}_{\alpha}^{x})_{x\in\mathbb{J}}$, viewed as a stochastic process that evolves when  $x$ increases, is an inhomogeneous continuous state branching process with immigration defined on the discrete set $\mathbb{J}$. In particular, for any $x_{1}\leq x_{2}\leq\dots\leq x_{n}\in I$ and $p\in\lbrace 1,2,\dots, n\rbrace$, 
$\left( \widehat{\mathcal{L}}_{\alpha}^{x_{1}}, \widehat{\mathcal{L}}_{\alpha}^{x_{2}},\dots,\widehat{\mathcal{L}}_{\alpha}^{x_{p}}\right)$ and 
$\left( \widehat{\mathcal{L}}_{\alpha}^{x_{p}}, \widehat{\mathcal{L}}_{\alpha}^{x_{p+1}},\dots,\widehat{\mathcal{L}}_{\alpha}^{x_{n}}\right)$ are independent conditional on $\widehat{\mathcal{L}}_{\alpha}^{x_{p}}$.
\end{property}

Next we show that the processes $x\mapsto\widehat{\mathcal{L}}_{\alpha}^{x}$ parametrized by $x\in I$, where $x$ is assumed to increase, is an inhomogeneous branching process with immigration of form 
\eqref{Ch4Sec1: EqBranchingImmigration}. In particular, it has a continuous version and is inhomogeneous Markov.

\begin{prop}
\label{Ch4Sec2: PropCBPI}
$(\widehat{\mathcal{L}}_{\alpha}^{x})_{x\in I}$ has the same finite-dimensional marginals as a solution to the stochastic differential equation
\begin{equation}
\label{Ch4Sec2: EqOccupFieldCSBI}
dZ_{x}=\sqrt{2w(x)}\sqrt{Z_{x}}d\mathbb{B}_{x} + 2\dfrac{d\log u_{\downarrow}}{dx}(x)Z_{x}dx + \alpha w(x)dx.
\end{equation}
If $L$ is the generator of a Brownian motion on $(0,+\infty)$ killed when it hits $0$, then $(\widehat{\mathcal{L}}_{\alpha}^{x})_{x>0}$ has the same law as the square of a Bessel process of dimension $2\alpha$ starting from $0$ at $x=0$. If $L$ is the generator of a Brownian motion on $(0,x_{\rm max})$, killed when hitting the boundary, then $(\widehat{\mathcal{L}}_{\alpha}^{x})_{0<x<x_{\rm max}}$ has the same law as the square of a Bessel bridge of dimension $2\alpha$ from $0$ at $x=0$ to $0$ at $x=x_{\rm max}$.
\end{prop}

\begin{proof}
Let $x_{0}<x\in I$ and $\lambda_{0},\lambda\geq 0$. Applying the identity \eqref{Ch4Sec2: EqLaplaceTransformOccupField} to the case of two points, we get that
\begin{equation}
\label{Ch4Sec2: EqLaplTransfOccField2Points}
\mathbb{E}\left[\exp\left(-\lambda_{0}\widehat{\mathcal{L}}_{\alpha}^{x_{0}}
-\lambda\widehat{\mathcal{L}}_{\alpha}^{x}\right)\right] =
\left((1+\lambda_{0}G(x_{0},x_{0}))(1+\lambda G(x,x))-\lambda_{0}\lambda(G(x_{0},x))^{2}\right) ^{-\alpha}.
\end{equation}
Let
\begin{displaymath}
\Psi(x_{0},\lambda_{0}):=\mathbb{E}\left[e^{-\lambda_{0}\widehat{\mathcal{L}}_{\alpha}^{x_{0}}}\right]
=\left(\dfrac{G(x_{0},x_{0})}{G(x_{0},x_{0})+\lambda_{0}}\right)^{\alpha}.
\end{displaymath}
For $y\leq x$, let
\begin{displaymath}
\psi(y,x,\lambda):=\dfrac{G(x,y)G(y,x)\lambda}{G(y,y)(G(y,y)+\lambda\det_{y,x}G)},
\end{displaymath}
\begin{displaymath}
\varphi(y,x,\lambda):=-\log\left(\dfrac{G(y,y)}{G(y,y)+\lambda\det_{y,x}G}\right). 
\end{displaymath}
One can check that the right-hand side of \eqref{Ch4Sec2: EqLaplTransfOccField2Points} equals
\begin{displaymath}
\Psi(x_{0},\lambda_{0}+\psi(x_{0},x,\lambda))\exp(-\alpha \varphi(x_{0},x,\lambda)).
\end{displaymath}
In particular, for the conditional Laplace transform:
\begin{equation}
\label{Ch4Sec2: EqCondLTOccupField}
\mathbb{E}\left[\exp\left(-\lambda\widehat{\mathcal{L}}_{\alpha}^{x}\right)\vert \widehat{\mathcal{L}}_{\alpha}^{x_{0}}\right] = 
\exp\left(-\widehat{\mathcal{L}}_{\alpha}^{x_{0}}\psi(x_{0},x,\lambda)\right)\exp(-\alpha \varphi(x_{0},x,\lambda))~\text{a.s}.
\end{equation}
Moreover,
\begin{equation*}
\begin{split}
\dfrac{\partial\psi}{\partial y}(y,x,\lambda)=& W(u_{\downarrow},u_{\uparrow})(y)\psi(y,x,\lambda)^{2}
-\dfrac{2}{u_{\downarrow}(y)}\dfrac{du_{\downarrow}}{dy}(y)\psi(y,x,\lambda)\\
=& w(y)\psi(y,x,\lambda)^{2}-2\dfrac{d\log u_{\downarrow}}{dy}(y)\psi(y,x,\lambda),
\end{split}
\end{equation*}
and
\begin{displaymath}
\dfrac{\partial \varphi}{\partial y}(y,x,\lambda) = -W(u_{\downarrow},u_{\uparrow})(y)\psi(y,x,\lambda) 
= -w(y)\psi(y,x,\lambda),
\end{displaymath}
and we have the initial conditions $\psi(x,x,\lambda)=\lambda$ and $\varphi(x,x,\lambda)=0$.
Thus, \eqref{Ch4Sec2: EqCondLTOccupField} has the same form as \eqref{Ch4Sec1: EqLTCSBI} where $c(y)=\alpha w(y)$. Let $(Z_{y})_{y\in I, y\geq x_{0}}$ be a solution to \eqref{Ch4Sec2: EqOccupFieldCSBI} with the initial condition $Z_{x_{0}}$ being a gamma random variable of parameter $\alpha$ with mean $\alpha G(x_{0},x_{0})$. It follows from what precedes that $(\widehat{\mathcal{L}}_{\alpha}^{x_{0}},\widehat{\mathcal{L}}_{\alpha}^{x})$ has the same law as $(Z_{x_{0}},Z_{x})$. Using the conditional independence satisfied by the occupation field, we deduce that $(\widehat{\mathcal{L}}_{\alpha}^{y})_{y\in I, y\geq x_{0}}$ has the same finite-dimensional marginals as  
$(Z_{y})_{y\in I, y\geq x_{0}}$. Making $x_{0}$ converge to $\inf I$ along a countable subset, we get a consistent family of continuous stochastic processes, which induces a continuous stochastic process $(Z_{y})_{y\in I}$ defined on whole $I$. It satisfies \eqref{Ch4Sec2: EqOccupFieldCSBI} and has the same finite-dimensional marginals as $(\widehat{\mathcal{L}}_{\alpha}^{y})_{y\in I}$.

In case of a Brownian motion in $(0,+\infty)$ killed in $0$, the equation \eqref{Ch4Sec2: EqOccupFieldCSBI} becomes
\begin{displaymath}
dZ_{x} = 2\sqrt{Z_{x}} d\mathbb{B}_{x} + 2\alpha\,dx,
\end{displaymath}
which is the SDE satisfied by the square of a Bessel process of dimension $2\alpha$. Moreover $(\widehat{\mathcal{L}}_{\alpha}^{x})_{x>0}$ has the same one-dimensional marginals as the latter, more precisely $\widehat{\mathcal{L}}_{\alpha}^{x}$ is a gamma r.v. of parameter $\alpha$ with mean $2\alpha x$. This shows the equality in law.

In case of a Brownian motion in $(0,x_{\rm max})$ killed in $0$ and $x_{\rm max}$, the equation \eqref{Ch4Sec2: EqOccupFieldCSBI} becomes
\begin{displaymath}
dZ_{x} = 2\sqrt{Z_{x}} d\mathbb{B}_{x} + \dfrac{1}{x_{\rm max}-x}Z_{x} dx + 2\alpha dx,
\end{displaymath}
which is the SDE satisfied by the square of a Bessel bridge of dimension $2\alpha$ from $0$ at $x=0$ to $0$ at $x=x_{\rm max}$. Moreover, the latter process and $(\widehat{\mathcal{L}}_{\alpha}^{x})_{0<x<x_{\rm max}}$ have the same one-dimensional marginals, more precisely gamma r.v. of parameter $\alpha$ with mean $2\alpha(x_{\rm max}-x)\frac{x}{x_{\rm max}}$. Thus, the two have the same law.
\end{proof}

We showed that $(\widehat{\mathcal{L}}_{\alpha}^{x})_{x\in I}$ has the same finite-dimensional marginals as a continuous stochastic process. We will assume in the sequel and prove in Section \ref{Ch5Sec2} that one can couple the Poisson ensemble $\mathcal{L}_{\alpha}$ and a continuous version of its occupation field $(\widehat{\mathcal{L}}_{\alpha}^{x})_{x\in I}$ on the same probability space. This does not follow trivially from the fact that the process $(\widehat{\mathcal{L}}_{\alpha}^{x})_{x\in I}$ has a continuous version. Consider the following counterexample: Let $U$ be an uniform r.v. on $(0,1)$. Let $\mathscr{E}$ be a countable random set of Brownian excursions defined as follows: conditional on $U$, $\mathscr{E}$ is a Poisson ensemble with intensity $\eta_{BM}^{>U}+\eta_{BM}^{<U}$. Let $(\widehat{\mathscr{E}}_{x})_{x\in\mathbb{R}}$ be the occupation field of $\mathscr{E}$. Then, $\widehat{\mathscr{E}}$ is continuous on $(-\infty, U)$ and $(U,+\infty)$ but not at $U$. Indeed, $\widehat{\mathscr{E}}_{U}=0$ and
\begin{displaymath}
\lim_{x\rightarrow U^{-}}\widehat{\mathscr{E}}_{x}=\lim_{x\rightarrow U^{-}}\widehat{\mathscr{E}}_{x}=1.
\end{displaymath}
Let $(\widehat{\mathscr{E}}'_{x})_{x\in\mathbb{R}}$ be the field defined by: $\widehat{\mathscr{E}}'_{x}=\widehat{\mathscr{E}}_{x}$ if $x\neq U$ and $\widehat{\mathscr{E}}'_{U}=1$. $(\widehat{\mathscr{E}}'_{x})_{x\in\mathbb{R}}$ is continuous and for any fixed $x\in\mathbb{R}$ $\widehat{\mathscr{E}}'_{x}=\widehat{\mathscr{E}}_{x}$ a.s. Thus, $(\widehat{\mathscr{E}}'_{x})_{x\in\mathbb{R}}$ is a continuous version of the process $(\widehat{\mathscr{E}}_{x})_{x\in\mathbb{R}}$, but it can not be implemented as a sum of local time across the excursions in $\mathscr{E}$. As we will show in Section \ref{Ch5Sec2}, such a difficulty does not arise in case of $\mathcal{L}_{\alpha}$.

$(\widehat{\mathcal{L}}_{\alpha}^{x})_{x\in I}$ is an inhomogeneous continuous state branching with immigration. The branching mechanism is the same as for the local times of the diffusion $X$, given by \eqref{Ch4Sec1: PropRayKnight}. The immigration measure is $\alpha w(x)dx$. The interpretation is the following: given a loop in $\mathcal{L}_{\alpha}$, its family of local times performs a branching according to the mechanism \eqref{Ch4Sec1: PropRayKnight}, independently from the other loops. The immigration between $x$ and $x+\Delta x$ comes from the loops whose minima belong to 
$(x,x+\Delta x)$. For a better understanding of this, it is useful to keep in mind the representation of the unrooted loops as positive excursions above their minima, as in Corollary \ref{Ch3Sec5: CorDesintegMinGen}. In case of Brownian loops on $(0,+\infty)$ or $(0,1)$, one recovers in this way the Lévy-Hincin representation of squares of Bessel processes and Bessel bridges 
\cite{PitmanYor82DecompBessBridge,LeGallYor86CarreBessel,Perman96}.
We would also like to mention Pitman's work \cite{Pitman96LTCicle}, where he studied a cyclically stationary local time process on a circle, which he decomposed as a sum of individual local times over an infinite countable collection of Poisson distributed loops, called "pulses".

It is remarkable that, although the immigration measure is absolutely continuous with respect to Lebesgue measure, there is only a countable number of moments at which immigration occurs. These are the positions of the minima of loops in $\mathcal{L}_{\alpha}$. Moreover, the local time of each loop at its minimum is zero. For $x>a\in I$, let
\begin{displaymath}
\widehat{\mathcal{L}}_{\alpha}^{(a), x} : = 
\sum_{
\substack{
\gamma\in \mathcal{L}_{\alpha} \\ 
\min \gamma >a
} } 
\ell^{x}(\gamma).
\end{displaymath}
Let $a<b\in I$. For $j\leq n\in \mathbb{N}$, let $\Delta x_{n}:=\frac{1}{n}(b-a)$ and let $x_{j,n}:=a+j\Delta x_{n}$. Then, $\Big(\widehat{\mathcal{L}}_{\alpha}^{(x_{j-1}), x_{j}}\Big)_{1\leq j\leq n}$ is a sequence of independent gamma r.v. of parameter $\alpha$ and the mean of $\widehat{\mathcal{L}}_{\alpha}^{(x_{j-1}), x_{j}}$ is $\alpha\bigg(G(x_{j},x_{j})-\dfrac{G(x_{j-1},x_{j})G(x_{j},x_{j-1})}{G(x_{j-1},x_{j-1})}\bigg)$. For $n$ large,
\begin{displaymath}
G(x_{j},x_{j})-\dfrac{G(x_{j-1},x_{j})G(x_{j},x_{j-1})}{G(x_{j-1},x_{j-1})} = w(x_{j-1})\Delta x_{n} + o(\Delta x_{n}),
\end{displaymath}
and $o(\Delta x_{n})$ is uniform in $j$. Thus,
\begin{multline*}
\lim_{n \rightarrow +\infty}\mathbb{E}\Big[\sum_{j=1}^{n}\widehat{\mathcal{L}}_{\alpha}^{(x_{j-1}), x_{j}}\Big]= \\
\lim_{n \rightarrow +\infty}
\alpha\sum_{j=1}^{n}\left(G(x_{j},x_{j})-\dfrac{G(x_{j-1},x_{j})G(x_{j},x_{j-1})}{G(x_{j-1},x_{j-1})}\right) 
=\alpha\int_{a}^{b}w(x) dx,
\end{multline*}
and
\begin{multline*}
\lim_{n \rightarrow +\infty}\operatorname{Var}\Big(\sum_{j=1}^{n}\widehat{\mathcal{L}}_{\alpha}^{(x_{j-1}), x_{j}}\Big)
=\\\lim_{n \rightarrow +\infty}
\alpha\sum_{j=1}^{n}\!\left(G(x_{j},x_{j})-\dfrac{G(x_{j-1},x_{j})G(x_{j},x_{j-1})}{G(x_{j-1},x_{j-1})}\right)^{2}=0.
\end{multline*}
It follows that $\sum_{j=1}^{n}\widehat{\mathcal{L}}_{\alpha}^{(x_{j-1}), x_{j}}$ converges in probability to $\alpha\int_{a}^{b}w(x)\,dx$. This is consistent with our interpretation of immigration.

Next proposition deals with the zeroes of the occupation field.

\begin{prop}
\label{Ch4Sec2: PropOccupFieldZeroes}
Let $x_{0}\in I$. If $\int_{\inf I}^{x_{0}}w(x) dx<+\infty$, then
\begin{displaymath}
\lim_{x\rightarrow \inf I}\widehat{\mathcal{L}}_{\alpha}^{x}=0.
\end{displaymath}
Analogous result holds if $\int_{x_{0}}^{\sup I}w(x)\,dx<+\infty$.

If $\alpha\geq 1$, then the continuous process $(\widehat{\mathcal{L}}_{\alpha}^{x})_{x\in I}$ stays almost surely positive on $I$. If $\alpha<1$, then $(\widehat{\mathcal{L}}_{\alpha}^{x})_{x\in I}$ hits $0$ infinitely many times on $I$.
\end{prop}

\begin{proof}
If $\int_{\inf I}^{x_{0}}w(x) dx<+\infty$, then $L+\kappa$, where $\kappa$ is the killing measure of $L$, is also the generator of a transient diffusion. We can couple $(\widehat{\mathcal{L}}_{\alpha, L}^{x})_{x\in I}$ and $(\widehat{\mathcal{L}}_{\alpha, L+\kappa}^{x})_{x\in I}$ on the same probability space such that a.s. for all $x\in I$, $\widehat{\mathcal{L}}_{\alpha, L}^{x}\leq \widehat{\mathcal{L}}_{\alpha, L+\kappa}^{x}$. But, according to Property \ref{Ch4Sec2: PropertyOccupFieldTransf} (i), $(\widehat{\mathcal{L}}_{\alpha, L+\kappa}^{x})_{x\in I}$ is just a scale changed square of Bessel process starting from $0$ or square of a Bessel bridge from $0$ to $0$. Thus,
\begin{displaymath}
\lim_{x\rightarrow \inf I}\widehat{\mathcal{L}}_{\alpha, L}^{x}\leq 
\lim_{x\rightarrow \inf I}\widehat{\mathcal{L}}_{\alpha, L+\kappa}^{x}=0.
\end{displaymath}

Regarding the number of zeros of $(\widehat{\mathcal{L}}_{\alpha}^{x})_{x\in I}$ on $I$, Property 
\ref{Ch4Sec2: PropertyOccupFieldTransf} ensures that it remains unchanged if we apply scale, time changes and conjugations to $L$. Since any generator of a transient diffusion is equivalent through latter transformation to the generator of a Brownian motion on $(0,+\infty)$ killed in $0$, the result on the number of zeros of $(\widehat{\mathcal{L}}_{\alpha}^{x})_{x\in I}$ follows from standard properties of Bessel processes.
\end{proof}

Next we study the clusters of loops. We introduce an equivalence relation on  the loops of $\mathcal{L}_{\alpha}$: $\gamma$ is in the same class as $\tilde{\gamma}$ if there is a finite chain of loops $\gamma_{0}, \gamma_{1},\dots,\gamma_{n}$ in $\mathcal{L}_{\alpha}$ such that $\gamma_{0}=\gamma$, $\gamma_{n}=\tilde{\gamma}$, and for all 
$i\in\lbrace 0,1,\dots,n-1\rbrace$, $\gamma_{i}([0,T(\gamma_{i})])\cap\gamma_{i+1}([0,T(\gamma_{i+1})])\neq\emptyset$. A cluster is the union of all $\gamma([0,T(\gamma)])$ where the loops $\gamma$ belong to the same equivalence class. It is a subinterval of $I$. By definition, clusters corresponding to different equivalence classes are disjoint.

The study of clusters of a loop-soup was initiated in \cite{SheffieldWerner2012CLE} in the two-dimensional Brownian setting, where these clusters were used to construct the Conformal Loop Ensembles (CLE). In \cite{LeJanLemaire2012LoopClusters} were studied the clusters of loops of Markovian jump processes on a general graph. The particular case of the discrete circle was treated in
\cite{Chang2015Cricle}. The loop clusters on $\mathbb{Z}^{d}$, $d\geq 3$, and the question of their percolation, were studied in
\cite{ChangSapozhnikov2014PercLoops,Chang2017Supercritical}. In our setting of one dimensional diffusions the description of clusters of loops is simple: these are exactly the connected components of the positive set of the occupation field. This key observation led the author of these notes to study the Poisson ensembles of loops associated to diffusions on metric graphs \cite{Lupu2014LoopsGFF}, obtained by replacing discrete edges in a graph by continuous line segments. There, the loops combine both non-trivial geometry and continuous local times. There too, the clusters of loops are exactly the connected components of the positive set of the occupation field. Out of this, it was deduced that on the discrete half-plane
$\mathbb{Z}\times\mathbb{N}$, the critical intensity for the nearest neighbor random walk loops is
$\alpha=1/2$ \cite{Lupu2014LoopsHalfPlane}, the same as in the two-dimensional Brownian case studied in \cite{SheffieldWerner2012CLE},
and that one obtains CLE as scaling limits of discrete clusters of loops on $\mathbb{Z}\times\mathbb{N}$
\cite{Lupu2015ConvCLE}.

\begin{prop}
\label{Ch4Sec2 PropClusters}
Let $L$ be the generator of a transient diffusion on $I$. If $\alpha\geq 1$, the loops in $\mathcal{L}_{\alpha}$ form a single cluster: $I$. If $\alpha\in(0,1)$, there are infinitely many clusters. These are the maximal open intervals on which $(\widehat{\mathcal{L}}_{\alpha}^{x})_{x\in I}$ is positive. In case of the Brownian motion on $(0,+\infty)$ killed at $0$, the clusters correspond to the jumps of a stable subordinator with index $1-\alpha$. In case of a general diffusion, by performing a change of scale of derivative $\frac{1}{2}\frac{w}{u_{\downarrow}^{2}}$, we reduce the problem to the previous case. In case of the Brownian motion on $(0,+\infty)$ killed at $0$ and with uniform killing $\kappa$, the
clusters correspond to the jumps of a subordinator with Levy measure 
$1_{x>0}\frac{e^{2\sqrt{2\kappa}x}dx}{(e^{2\sqrt{2\kappa}x}-1)^{2-\alpha}}$.
\end{prop}

\begin{proof}
Assume that $\mathcal{L}_{\alpha}$ and a continuous version of $(\widehat{\mathcal{L}}_{\alpha}^{x})_{x\in I}$ are defined on the same probability space. Almost surely the following holds:
\begin{itemize}
\item Given $\gamma\neq\gamma'\in\mathcal{L}_{\alpha}$, $\min\gamma\neq\max\gamma'$ and $\max\gamma\neq\min\gamma'$.
\item For all $\gamma\in\mathcal{L}_{\alpha}$, $\ell^{\min\gamma}(\gamma)=\ell^{\max\gamma}(\gamma)=0$ and $\ell^{x}(\gamma)$ is positive for $x\in(\min\gamma, \max\gamma)$.
\end{itemize}
Whenever the above two conditions hold it follows deterministically that the clusters are the intervals on which $(\widehat{\mathcal{L}}_{\alpha}^{x})_{x\in I}$ stays positive. We deduce then the number of clusters from Proposition \ref{Ch4Sec2: PropOccupFieldZeroes}.

If $L$ is the generator of the Brownian motion on $(0,+\infty)$ killed at $0$, then $(\widehat{\mathcal{L}}_{\alpha}^{x})_{x\in I}$ is the square of a Bessel process of dimension $2\alpha$ and its excursions correspond to the jumps of a stable subordinator with index $1-\alpha$. 

In general, a generator $L$ has the same measure on loops as $\operatorname{Conj}(u_{\downarrow},L)$. A diffusion of generator 
$\operatorname{Conj}(u_{\downarrow},L)$ transforms through a change of time and a change of scale of density 
$\frac{1}{2}\frac{w}{u_{\downarrow}^{2}}$ into a Brownian motion on $(0,+\infty)$ killed at $0$. For the clusters, the change of time does not matter.

In case of a Brownian motion on $(0,+\infty)$ killed at $0$ and with uniform killing $\kappa$, we can take
$u_{\downarrow}(x)=e^{-\sqrt{2\kappa}x}$. The scale function is then
\begin{displaymath}
S(x)=\int_{0}^{x}\dfrac{dy}{u_{\downarrow}(y)^{2}}=\int_{0}^{x}e^{2\sqrt{2\kappa}y}dy=
\dfrac{1}{2\sqrt{2\kappa}}(e^{2\sqrt{2\kappa}x}-1).
\end{displaymath}
Let $(Y_{t})_{t\geq 0}$ be an $1-\alpha$ stable subordinator with Levy measure $1_{y>0}y^{-(2-\alpha)}dy$. The clusters 
of $\mathcal{L}_{\alpha,\frac{1}{2}\frac{d^{2}}{dx^{2}}-\kappa}$ correspond to the jumps of the process 
$(S^{-1}(Y_{t}))_{t\geq 0}$, which is not a subordinator. We will see that nevertheless the latter process has the same set of jumps as a subordinator with Levy measure $1_{x>0}\frac{e^{2\sqrt{2\kappa}x}dx}{(e^{2\sqrt{2\kappa}x}-1)^{2-\alpha}}$. Let $\varepsilon>0$ and $(Y_{\varepsilon,t})_{t\geq 0}$ be the process obtained from $(Y_{t})_{t\geq 0}$ by removing all the jumps of height less then $\varepsilon$. By construction 
$Y_{\varepsilon,t}\leq Y_{t}$. $(S^{-1}(Y_{\varepsilon,t}))_{t\geq 0}$ is a Markov process: given the position of 
$S^{-1}(Y_{\varepsilon,t})$ at time $t$, the process waits an exponential holding time with inverse of the mean
equal to
\begin{displaymath}
\int_{\varepsilon}^{+\infty}\dfrac{dy}{y^{2-\alpha}}=\dfrac{1}{(1-\alpha)\varepsilon^{1-\alpha}}.
\end{displaymath}
Once a jump occurs, the jump of $Y_{\varepsilon}$ is distributed according the probability
\begin{displaymath}
1_{y>\varepsilon}(1-\alpha)\varepsilon^{1-\alpha}\dfrac{dy}{y^{2-\alpha}}.
\end{displaymath}
The distribution of the corresponding jump of $S^{-1}(Y_{\varepsilon,t})$ is obtained by pushing forward the above probability by the map $y\mapsto S^{-1}(y+Y_{\varepsilon,t})-S^{-1}(Y_{\varepsilon,t})$, which gives
\begin{multline*}
1_{x>S^{-1}(\varepsilon +Y_{\varepsilon,t})-S^{-1}(Y_{\varepsilon,t})}
(1-\alpha)\varepsilon^{1-\alpha}\dfrac{(2\sqrt{2\kappa})^{2-\alpha}e^{2\sqrt{2\kappa}(x+S^{-1}(Y_{\varepsilon,t}))}dx}
{\big(e^{2\sqrt{2\kappa}(x+S^{-1}(Y_{\varepsilon,t}))}-e^{2\sqrt{2\kappa}S^{-1}(Y_{\varepsilon,t})}\big)^{2-\alpha}}
\\=1_{x>S^{-1}(\varepsilon +Y_{\varepsilon,t})-S^{-1}(Y_{\varepsilon,t})}
(1-\alpha)\varepsilon^{1-\alpha}(2\sqrt{2\kappa})^{2-\alpha}e^{-(1-\alpha)2\sqrt{2\kappa}S^{-1}(Y_{\varepsilon,t})}
\dfrac{e^{2\sqrt{2\kappa}x}dx}{(e^{2\sqrt{2\kappa}x}-1)^{2-\alpha}}
\\=1_{x>S^{-1}(\varepsilon +Y_{\varepsilon,t})-S^{-1}(Y_{\varepsilon,t})}
(1-\alpha)\varepsilon^{1-\alpha}\dfrac{(2\sqrt{2\kappa})^{2-\alpha}}
{(1+2\sqrt{2\kappa}Y_{\varepsilon,t})^{1-\alpha}}
\dfrac{e^{2\sqrt{2\kappa}x}dx}{(e^{2\sqrt{2\kappa}x}-1)^{2-\alpha}}.
\end{multline*}
Consider now the random time change
\begin{displaymath}
\tau_{\varepsilon}(v):=\inf\left\lbrace t\geq 0\Big\vert\int_{0}^{t}
\dfrac{(2\sqrt{2\kappa})^{2-\alpha}}{(1+2\sqrt{2\kappa}Y_{\varepsilon,s})^{1-\alpha}}ds\geq v\right\rbrace,
\end{displaymath}
and at the limit as $\varepsilon \rightarrow 0$,
\begin{displaymath}
\tau(v):=\inf\left\lbrace t\geq 0 \Big\vert\int_{0}^{t}
\dfrac{(2\sqrt{2\kappa})^{2-\alpha}}{(1+2\sqrt{2\kappa}Y_{\varepsilon,s})^{1-\alpha}}ds\geq v\right\rbrace.
\end{displaymath}
For the time-changed process  $(S^{-1}(Y_{\varepsilon,\tau_{\varepsilon}(v)}))_{v\geq 0}$, the rate of jumps of height
belonging to $[x,x+dx]$ is
\begin{displaymath}
\left\lbrace
\begin{array}{ll}
\dfrac{e^{2\sqrt{2\kappa}x}dx}{(e^{2\sqrt{2\kappa}x}-1)^{2-\alpha}} 
& \text{if}~x>S^{-1}(\varepsilon +Y_{\varepsilon,\tau_{\varepsilon}(v)})-S^{-1}(Y_{\varepsilon,\tau_{\varepsilon}(v)}),\\ 
0 & \text{otherwise}.
\end{array} 
\right.
\end{displaymath}
Thus, as $\varepsilon$ goes to $0$, on one hand the process $(S^{-1}(Y_{\varepsilon,\tau_{\varepsilon}(v)}))_{v\geq 0}$
converges in law to $(S^{-1}(Y_{\tau(v)}))_{v\geq 0}$, and on the other hand it converges in law to a subordinator with Levy measure $1_{x>0}\frac{e^{2\sqrt{2\kappa}x}dx}{(e^{2\sqrt{2\kappa}x}-1)^{2-\alpha}}$.
\end{proof}

The clusters coalesce when $\alpha$ increases and fragment when $\alpha$ decreases. Some information on the coalescence of clusters delimited by the zeroes of Bessel processes is given in \cite{BertoinPitman1999TwoCoalStableSubord}, Section $3$. This clusters can be obtained as a limit of clusters of discrete loops on discrete subsets. In case of a symmetric jump process to the nearest neighbors on $\varepsilon\mathbb{N}$, if $\alpha> 1$, there are finitely many clusters, and if $\alpha\in(0,1)$, there are infinitely many clusters and these clusters are given by the holding times of a renewal process, which suitable normalized converges in law as $\varepsilon\rightarrow 0^{+}$ to the inverse of a stable subordinator with index $1-\alpha$. See Remark $3.3$ in \cite{LeJanLemaire2012LoopClusters}.

We can consider the occupation field $(\widehat{\mathcal{L}}_{\alpha, L}^{x})_{x\in I}$ if $L$ is not the generator of a diffusion but contains creation of mass as in \eqref{Ch2Sec3: GenSignedMeasure}. In this setting, if $u$ is a positive continuous function on $I$ such that $\frac{d^{2}u}{dx^{2}}$ is a signed measure, then for all $x\in I$
\begin{displaymath}
\widehat{\mathcal{L}}_{\alpha,\operatorname{Conj}(u,L)}^{x} = \dfrac{1}{u(x)^{2}}\widehat{\mathcal{L}}_{\alpha,L}^{x}.
\end{displaymath}
It follows that if $L\in\mathfrak{D}^{-}$, then for all $x\in I$, $\widehat{\mathcal{L}}_{\alpha, L}^{x}<+\infty$ a.s. and if $L\in\mathfrak{D}^{0}$, then for all $x\in I$, $\widehat{\mathcal{L}}_{\alpha, L}^{x}=+\infty$ a.s. If $L\in\mathfrak{D}^{+}$, then according to Proposition \ref{Ch2Sec3: PropStability} (iv), there is a positive Radon measure $\tilde{\kappa}$ such that $L-\tilde{\kappa}\in\mathfrak{D}^{0}$. Then, for all $x\in I$, 
$\widehat{\mathcal{L}}_{\alpha, L}^{x}\geq\widehat{\mathcal{L}}_{\alpha, L-\tilde{\kappa}}^{x}=+\infty$. If $L\in\mathfrak{D}^{-}$, then Properties \ref{Ch4Sec2: PropertyOccupFieldTransf} (i) and (ii) still hold. The description given by Property \ref{Ch4Sec2: PropertyMarginals} of the finite-dimensional marginals of  $(\widehat{\mathcal{L}}_{\alpha}^{x})_{x\in I}$ is still true, although the case of creation of mass was not considered in \cite{LeJan2011Loops}. $(\widehat{\mathcal{L}}_{\alpha}^{x})_{x\in I}$ still satisfies the SDE 
\eqref{Ch4Sec2: EqOccupFieldCSBI}.

\begin{prop}
\label{Ch4Sec2: PropExpMoments}
Let $L\in\mathfrak{D}^{-}$ and $\tilde{\nu}$ a finite signed measure with compact support in $I$. Then there is equivalence between:
\begin{itemize}
\item[(i)] $\mathbb{E}\left[\exp\left(\int_{I}
\widehat{\mathcal{L}}_{\alpha, L}^{x}\tilde{\nu}(dx)\right)\right]<+\infty$,
\item[(ii)] $L+\tilde{\nu}\in\mathfrak{D}^{-}$ .
\end{itemize}
If $L+\tilde{\nu}\in\mathfrak{D}^{-}$, then for $s\in[0,1]$,
\begin{equation}
\label{Ch4Sec2: EqExpMom}
\mathbb{E}\left[\exp\left(\int_{I}\widehat{\mathcal{L}}_{\alpha, L}^{x}\tilde{\nu}(dx)\right)\right] = 
\exp\left(\alpha\int_{0}^{1}\int_{I}G_{L+s\tilde{\nu}}(x,x)\tilde{\nu}(dx)ds\right).
\end{equation}
\end{prop}

\begin{proof}
First, observe that $\int_{I}\widehat{\mathcal{L}}_{\alpha, L}^{x}\vert\tilde{\nu}\vert(dx)$ is almost surely finite because $\vert\tilde{\nu}\vert$ is finite and has compact support and $(\widehat{\mathcal{L}}_{\alpha, L}^{x})_{x\in I}$ is continuous. Also, observe that $\mathfrak{D}^{-}$ is convex. So, if $L+\tilde{\nu}\in\mathfrak{D}^{-}$, then for all $s\in [0,1]$, $ L+s\tilde{\nu}\in\mathfrak{D}^{-}$.

(i) implies (ii): Let $\mathbb{P}_{\mathcal{L}_{\alpha, L}}$ be the law of $\mathcal{L}_{\alpha, L}$ and  $\mathbb{P}_{\mathcal{L}_{\alpha, L+\tilde{\nu}}}$ be the law of $\mathcal{L}_{\alpha, L+\tilde{\nu}}$. There is an absolute continuity relation between the intensity measures:
\begin{displaymath}
\mu_{L+\tilde{\nu}}(d\gamma)=\exp\left(\int_{I}\ell^{x}(\gamma)\right)\mu_{L}(d\gamma). 
\end{displaymath}
In case (i) is true, $\mathbb{P}_{\mathcal{L}_{\alpha, L+\tilde{\nu}}}$ is absolutely continuous with respect to $\mathbb{P}_{\mathcal{L}_{\alpha, L}}$ and 
\begin{equation}
\label{Ch4Sec2: EqAbsContLoopSoup}
d\mathbb{P}_{\mathcal{L}_{\alpha, L+\tilde{\nu}}} = 
\dfrac{\exp\left(\int_{I}\widehat{\mathcal{L}}_{\alpha, L}^{x}\tilde{\nu}(dx)\right)}{\mathbb{E}\left[\exp\left(\int_{I}\widehat{\mathcal{L}}_{\alpha, L}^{x}\tilde{\nu}(dx)\right)\right]}
d\mathbb{P}_{\mathcal{L}_{\alpha, L}}.
\end{equation}
But this can not be if $L+\tilde{\nu}\not\in\mathfrak{D}^{-}$, because then, for any $x\in I$, $\widehat{\mathcal{L}}_{\alpha, L}^{x}<+\infty$, and $\widehat{\mathcal{L}}_{\alpha, L+\tilde{\nu}}^{x}=+\infty$. Thus, necessarily, $L+\tilde{\nu}\in\mathfrak{D}^{-}$.

(ii) implies (i): We first assume that $\tilde{\nu}$ is a positive measure and $L+\tilde{\nu}\in\mathfrak{D}^{-}$. Then, $\mathbb{P}_{\mathcal{L}_{\alpha, L}}$ is absolutely continuous with respect to 
$\mathbb{P}_{\mathcal{L}_{\alpha, L+\tilde{\nu}}}$ and
\begin{displaymath}
d\mathbb{P}_{\mathcal{L}_{\alpha, L}} = 
\dfrac{\exp\left(-\int_{I}\widehat{\mathcal{L}}_{\alpha, L+\tilde{\nu}}^{x}\tilde{\nu}(dx)\right)}{\mathbb{E}\left[\exp\left(-\int_{I}\widehat{\mathcal{L}}_{\alpha, L+\tilde{\nu}}^{x}\tilde{\nu}(dx)\right)\right]}
d\mathbb{P}_{\mathcal{L}_{\alpha, L+\tilde{\nu}}}.
\end{displaymath}
Inverting the above absolute continuity relation, we get that
\begin{displaymath}
\mathbb{E}\left[\exp\left(\int_{I}\widehat{\mathcal{L}}_{\alpha, L}^{x}\tilde{\nu}(dx)\right)\right] =
\mathbb{E}\left[\exp\left(-\int_{I}\widehat{\mathcal{L}}_{\alpha, L+\tilde{\nu}}^{x}\tilde{\nu}(dx)\right)\right]^{-1}<+\infty.
\end{displaymath}
If $\tilde{\nu}$ is not positive, let $\tilde{\nu}^{+}$ and $-\tilde{\nu}^{-}$ be its positive respectively negative part. Then,
\begin{equation*}
\begin{split}
\mathbb{E}\bigg[\exp\bigg(\int_{I}\widehat{\mathcal{L}}_{\alpha, L}^{x}\tilde{\nu}(dx)&\bigg)\bigg]\\
&=\mathbb{E}\left[\exp\left(\int_{I}\widehat{\mathcal{L}}_{\alpha, L-\tilde{\nu}^{-}}^{x}\,\tilde{\nu}^{+}(dx)\right)\right]
\mathbb{E}\left[\exp\left(-\int_{I}\widehat{\mathcal{L}}_{\alpha, L}^{x}\tilde{\nu}^{-}(dx)\right)\right]\\
&=\dfrac{\mathbb{E}\left[\exp\left(-\int_{I}\widehat{\mathcal{L}}_{\alpha, L}^{x}\tilde{\nu}^{-}(dx)\right)\right]}{\mathbb{E}\left[\exp\left(-\int_{I}\widehat{\mathcal{L}}_{\alpha, L+\tilde{\nu}}^{x}
\tilde{\nu}^{+}(dx)\right)\right]}<+\infty.
\end{split}
\end{equation*}

For the expression \eqref{Ch4Sec2: EqExpMom} of exponential moments: 
\begin{equation}
\label{Ch4Sec2: EqDiffLambdaExpMom}
\dfrac{d}{ds}\mathbb{E}\left[ \exp\left(s\int_{I}\widehat{\mathcal{L}}_{\alpha, L}^{x}\tilde{\nu}(dx)\right)\right] =
\mathbb{E}\left[ \int_{I}\widehat{\mathcal{L}}_{\alpha, L}^{x}\tilde{\nu}(dx)\exp\left(s\int_{I}\widehat{\mathcal{L}}_{\alpha, L}^{x}\tilde{\nu}(dx)\right)\right].
\end{equation}
From the absolute continuity relation \eqref{Ch4Sec2: EqAbsContLoopSoup} follows that the right-hand side of \eqref{Ch4Sec2: EqDiffLambdaExpMom} equals
\begin{displaymath}
\alpha\int_{I}G_{L+s\tilde{\nu}}(x,x)\tilde{\nu}(dx)\mathbb{E}
\left[\exp\left(s\int_{I}\widehat{\mathcal{L}}_{\alpha, L}^{x}\tilde{\nu}(dx)\right)\right].
\end{displaymath}
This implies \eqref{Ch4Sec2: EqExpMom}.
\end{proof}

As in discrete space case, the above exponential moments can be expressed using determinants. On the complex Hilbert space $\mathbb{L}^{2}(d\vert\tilde{\nu}\vert)$, define, for $s\in[0,1]$, the operators
\begin{displaymath}
(\mathfrak{G}_{s,\tilde{\nu}}f)(x):=\int_{I} G_{L+s\tilde{\nu}}(x,y)f(y)\tilde{\nu}(dy),
\end{displaymath}
\begin{displaymath}
(\vert\mathfrak{G}_{s,\tilde{\nu}}\vert f)(x):=\int_{I} G_{L+s\tilde{\nu}}(x,y)f(y)\vert\tilde{\nu}\vert(dy).
\end{displaymath}
The operator $\vert\mathfrak{G}_{s,\tilde{\nu}}\vert$ is self-adjoint, positive semi-definite with continuous kernel function, and according to \cite{Simon2005TraceIdeals}, Theorem $2.12$, it is trace class. Since trace class operators form a two-sided ideal in the algebra of bounded operators, $\mathfrak{G}_{s,\tilde{\nu}}$ is also trace class. Moreover,
\begin{equation}
\label{Ch4Sec2: EqTrace}
\operatorname{Tr}(\mathfrak{G}_{s,\tilde{\nu}}) = \int_{I} G_{L+s\tilde{\nu}}(x,x)\tilde{\nu}(dx).
\end{equation}
The determinant $\det(Id + \mathfrak{G}_{s,\tilde{\nu}})$ is well defined as a converging product of its eigenvalues 
(see \cite{Simon2005TraceIdeals}, Chapter $3$).

\begin{prop}
\label{Ch4Sec2: PropDeterminant}
\begin{displaymath}
\exp\left(\alpha\int_{0}^{1}\int_{I}G_{L+s\tilde{\nu}}(x,x)\tilde{\nu}(dx)ds\right)=
(\det(Id + \mathfrak{G}_{1,\tilde{\nu}}))^{\alpha}.
\end{displaymath}
\end{prop}

\begin{proof}
$\mathfrak{G}_{1,\tilde{\nu}}$ has only real eigenvalues. Indeed, let $\lambda$ be such an eigenvalue and $f$ a non zero eigenfunction for $\lambda$. The sign of $\tilde{\nu}$, $\operatorname{sign}(\tilde{\nu})$, is a $\lbrace -1, +1\rbrace$-valued function defined $d\vert\tilde{\nu}\vert$ almost everywhere.
\begin{equation}
\label{Ch4Sec2: EqRealEV}
\int_{I}(\operatorname{sign}(\tilde{\nu})\bar{f})(x)\vert\mathfrak{G}_{1,\tilde{\nu}}\vert(\operatorname{sign}(\tilde{\nu})f)(x)\vert\tilde{\nu}\vert(dx)=
\lambda\int_{I}\vert f\vert^{2}(x)\tilde{\nu}(dx).
\end{equation}
The left-hand side of \eqref{Ch4Sec2: EqRealEV} is non-negative. If the right-hand side of \eqref{Ch4Sec2: EqRealEV} is non-zero, then $\lambda$ is real. If it is zero, consider $f_{\varepsilon}:=f+\varepsilon \operatorname{sign}(\tilde{\nu})f$. Then,
\begin{displaymath}
\lambda=\lim_{\varepsilon\rightarrow 0^{+}}\dfrac{1}{2\varepsilon}
\left(\int_{I}(\operatorname{sign}(\tilde{\nu})\bar{f}_{\varepsilon})(x)\vert\mathfrak{G}_{1,\tilde{\nu}}
\vert(\operatorname{sign}(\tilde{\nu})f_{\varepsilon})(x)\vert\tilde{\nu}\vert(dx)\right) 
\left(\int_{I}\vert f\vert^{2}(x)\vert\tilde{\nu}\vert(dx)\right)^{-1},
\end{displaymath}
and thus $\lambda$ is real.

The operators $\mathfrak{G}_{s,\tilde{\nu}}$ are compact and the characteristic space corresponding to each of their non-zero eigenvalue is of finite dimension. Let $(\lambda_{i})_{i\geq 0}$ be the non-increasing sequence of positive eigenvalues of $\mathfrak{G}_{1,\tilde{\nu}}$. Each eigenvalue $\lambda_{i}$ appears as many times as the dimension of its characteristic space $\ker(\mathfrak{G}_{1,\tilde{\nu}}-\lambda_{i}Id)^{n}$ ($n$ large enough). Similarly, let $(-\tilde{\lambda}_{j})_{j\geq 0}$ be the non-decreasing sequence of the negative eigenvalues of 
$\mathfrak{G}_{1,\tilde{\nu}}$. Let $s\in[0,1]$. According to the resolvent identity (Lemma \ref{Ch2Sec3: LemResolvent}),
the operators $\mathfrak{G}_{1,\tilde{\nu}}$ and $\mathfrak{G}_{s,\tilde{\nu}}$ commute and satisfy the relation
\begin{equation}
\label{Ch4Sec2: EqResolventIdentity}
\mathfrak{G}_{1,\tilde{\nu}}\mathfrak{G}_{s,\tilde{\nu}}=\mathfrak{G}_{s,\tilde{\nu}}\mathfrak{G}_{1,\tilde{\nu}}=
\dfrac{1}{1-s}(\mathfrak{G}_{1,\tilde{\nu}}-\mathfrak{G}_{s,\tilde{\nu}}).
\end{equation}
Since $\mathfrak{G}_{1,\tilde{\nu}}$ and $\mathfrak{G}_{s,\tilde{\nu}}$ commute, these operators have common characteristic spaces. From \eqref{Ch4Sec2: EqResolventIdentity} follows that 
$\big(\frac{\lambda_{i}}{1+(1-s)\lambda_{i}}\big)_{i\geq 0}$ is a non-increasing sequence of positive eigenvalues of $\mathfrak{G}_{s,\tilde{\nu}}$. If $\frac{-1}{1-s}$ is not an eigenvalue of $\mathfrak{G}_{1,\tilde{\nu}}$, then $\big(\frac{-\tilde{\lambda}_{j}}{1-(1-s)\tilde{\lambda}_{j}}\big)_{j\geq 0}$ is also a sequence of eigenvalues of $\mathfrak{G}_{s,\tilde{\nu}}$. But the family of operators $(\mathfrak{G}_{s,\tilde{\nu}})_{s\in[0,1]}$ is bounded. Thus, none of $\frac{-\tilde{\lambda}_{j}}{1-(1-s)\tilde{\lambda}_{j}}$ can blow up when $s$ varies. So, it turns out that $\mathfrak{G}_{1,\tilde{\nu}}$ has no eigenvalues in $(-\infty,-1]$. From \eqref{Ch4Sec2: EqTrace} we get
\begin{displaymath}
\int_{I}G_{L+s\tilde{\nu}}(x,x)\tilde{\nu}(dx)=
\sum_{i\geq 0}\dfrac{\lambda_{i}}{1+(1-s)\lambda_{i}} -
\sum_{j\geq 0}\dfrac{\tilde{\lambda}_{j}}{1-(1-s)\tilde{\lambda}_{j}}.
\end{displaymath}
The above sum is absolutely convergent, uniformly for $s\in[0,1]$. Integrating over $[0,1]$ yields
\begin{displaymath}
\int_{0}^{1}\int_{I}G_{L+s\tilde{\nu}}(x,x)\tilde{\nu}(dx)ds=
\sum_{i\geq 0}\log(1+\lambda_{i})+\sum_{j\geq 0}\log(1-\tilde{\lambda}_{j}).
\end{displaymath}
This concludes the proof.
\end{proof}

\section{Isomorphism with the Gaussian free field}
\label{Ch4Sec3}

In \cite{LeJan2011Loops} Le Jan observed the equality in law between the occupation field of a Poisson ensemble of loops of a symmetric Markov jump process on an electircal network, at intensity parameter $\alpha=1/2$, and half the square of a discrete Gaussian free field (GFF). 
His identity is a generalization of Dynkin's isomorphism
\cite{Dynkin1984Isomorphism,Dynkin1984PolynomOccupField,Dynkin1984IsomorphismPresentation}, which relates the square of a GFF and the occupation times of Markovian trajectories. Dynkin's isomorphism itself can be seen as a reformulation of an identity by Symanzik \cite{Symanzik1969QFT},
who expressed the moments of a continuum GFF in $\mathbb{R}^{d}$ as a multiple intergal over a measure on Brownian loops, which is the higher dimensional analogue of our measure given by Definition \ref{Ch3Sec3: DefMesLoops}. Dynkin's isomorphism theorem has multiple variations, such as the generalized Ray-Knight theorems \cite{Eisenbaum94DynkinRK,EKMRS2000RK} (see also \cite{Ray1963Sojourn,Knight1963Sojourn} for the original Ray-Knigth theorems), Eisenbaum's isomorphism \cite{Eisenbaum1995Iso}, Sznitman's isomorphism for random interlacments \cite{Sznitman2012Isomorphism}, and the last in this line, Le Jan's isomorphism for Poisson ensembles of loops \cite{LeJan2011Loops}. Le Jan's isomorphism has the advantage to give the whole square of a GFF as an occupation field of Markovian trajectories, not just parts of it. See also \cite{MarcusRosen2006MarkovGaussianLocTime} and \cite{Sznitman2012LectureIso} for a review on this subject.

We are going to formulate Le Jan's isomorphism in our setting of one-dimensional diffusions.
Let $L$ be a generator of a transient diffusion on $I$ of form \eqref{Ch2Sec2: EqGenerator}. Let $(\phi_{x})_{x\in I}$ be a centered Gaussian process with covariance function
\begin{displaymath}
\mathbb{E}[\phi_{x}\phi_{y}]=G(x,y).
\end{displaymath}
$(\phi_{x})_{x\in I}$ is the Gaussian free field associated to $L$. Let $\widetilde{S}$ be a primitive of 
$\frac{w}{u_{\downarrow}^{2}}$. Then, $\widetilde{S}(\sup I)=+\infty$. Moreover, $\widetilde{S}(\inf I)>-\infty$ because $L$ is the generator of a transient diffusion. 
$\left(\frac{1}{u_{\downarrow}(\widetilde{S}^{-1}(a))}\phi_{\widetilde{S}^{-1}(a)}\right)_{a\in \tilde{S}(I)}$ is a standard Brownian motion starting from $0$ at $\widetilde{S}(\inf I)$. In particular $(\phi_{x})_{x\in I}$ is inhomogeneous Markov and has continuous sample paths.
As observed in \cite{LeJan2011Loops}, Chapter $5$, the following holds:

\begin{property}
\label{Ch4Sec3: PropertyIso}
When $\alpha=\frac{1}{2}$,
$(\widehat{\mathcal{L}}_{\frac{1}{2}}^{x})_{x\in I}$ has the same law as $(\frac{1}{2}\phi_{x}^{2})_{x\in I}$. In case of a Brownian motion on $(0,+\infty)$ killed in $0$,  $(\widehat{\mathcal{L}}_{\frac{1}{2}}^{x})_{x>0}$ is the square of a standard Brownian motion starting from $0$. In case of a Brownian motion on $(0, x_{\rm max})$, killed in $0$ and $x_{\rm max}$,  $(\widehat{\mathcal{L}}_{\frac{1}{2}}^{x})_{0<x<x_{\rm max}}$ is the square of a standard Brownian bridge on $[0,x_{\rm max}]$ from $0$ to $0$. In case of a Brownian motion on $\mathbb{R}$ with constant killing rate $\kappa$, $(\widehat{\mathcal{L}}_{\frac{1}{2}}^{x})_{x\in\mathbb{R}}$ is the square of a stationary Ornstein-Uhlenbeck process.
\end{property}

In our continuous one-dimensional setting one can "polarize" the above identity, in order to relate to $\mathcal{L}_{\frac{1}{2}}$ not only the square of the GFF, but also its sign.

\begin{prop}
\label{Prop PolIso}
Let $\alpha=\frac{1}{2}$. Let $\varsigma: I \rightarrow \lbrace -1,0,1\rbrace$ be a sign function, which is zero on the points not visited by loops in $\mathcal{L}_{\frac{1}{2}}$, and to visited points assigns a sign $-1$ or $1$, constant on each cluster of loops, 
uniform ($\mathbb{P}(\varsigma(x)=1)=\mathbb{P}(\varsigma(x)=-1)=\frac{1}{2}$), and conditional independent on each cluster. Then, the field
$(\varsigma(x)\sqrt{2\widehat{\mathcal{L}}_{\frac{1}{2}}^{x}})_{x\in I}$ is distributed like a Gaussian free field 
$(\phi_{x})_{x\in I}$.
\end{prop}

\begin{proof}
$(\sqrt{2\widehat{\mathcal{L}}_{\frac{1}{2}}^{x}})_{x\in I}$ gives the absolute value of the GFF
$(\vert\phi_{x}\vert)_{x\in I}$. 
According to Proposition \ref{Ch4Sec2 PropClusters}, the clusters of $\mathcal{L}_{\frac{1}{2}}$ correspond exactly to the connected components of
$\lbrace x\in I\vert \phi_{x}\neq 0\rbrace$. Moreover, conditional on $(\vert\phi_{x}\vert)_{x\in I}$,
$\operatorname{sign}(\phi)$ is distributed independently uniformly on each such connected component.
\end{proof}

Proposition \eqref{Prop PolIso} extends to the Poisson ensembles of loops associated to diffusions on metric graphs \cite{Lupu2014LoopsGFF}, as there too the occupation field is continuous, an in particular satisfies the intermediate value property. This relation to the GFF was used in 
\cite{Lupu2015ConvCLE} to show in dimension two the convergence of clusters of discrete loops to clusters of Brownian loops.
Further, the "polarized" version of Le Jan's isomorphism on metric graphs was applied to show that, in some sense, the continuum GFF in dimension two lives on clusters of a Poisson ensemble of Brownian loops, and that these clusters are its sign components \cite{ALS2017FPS,ALS2}. One procedes by approximating the GFF on a continuum two-dimensional domain by the metric graph GFF. See also \cite{QW2015}.

Next we recall the original Dynkin's isomorphism (see \cite{Dynkin1984Isomorphism,Dynkin1984PolynomOccupField,Dynkin1984IsomorphismPresentation}):

\begin{dynkin}
Let $x_{1}, x_{2},\dots, x_{2n}\in I$. Then for any non-negative measurable functional $F$ on continuous paths on $I$,
\begin{multline}
\label{Ch4Sec3: EqDynkinIso}
\mathbb{E}_{\phi}\left[\prod_{i=1}^{2n}\phi_{x_{i}}F((\frac{1}{2}\phi_{x}^{2})_{x\in I}) \right]=\\
\sum_{\text{pairings}}\int\mathbb{E}_{\phi}\Big[F((\frac{1}{2}\phi_{x}^{2}+\sum_{j=1}^{n}\ell^{x}(\gamma_{j}))_{x\in I}) \Big]\prod_{\text{pairs}}\mu^{y_{j},z_{j}}(d\gamma_{j}),
\end{multline}
where $\sum_{\text{pairings}}$ means that the $n$ pairs $\lbrace y_{j},z_{j}\rbrace$ are formed with all $2n$ points $x_{i}$ in all $\frac{(2n)!}{2^{n}n!}$ possible ways.
\end{dynkin}

We will show that in case $x_{i}=x_{i+n}$, for $i\in\lbrace 1,\dots, n\rbrace$ , i.e. $\prod_{i=1}^{2n}\phi_{x_{i}}$ being a product of squares $\prod_{i=1}^{n}\phi_{x_{i}}^{2}$, one can deduce the Dynkin's isomorphism from the relation between the square of the Gaussian free field and the occupation field.  In \cite{LeJanMarcusRosen2012Loops} and \cite{FitzsimmonsRosen2012LoopsIsomorphism} this is only done in case $n=1$ and $x_{1}=x_{2}$ using the Palm's identity for Poissonian ensembles and the analogue of the relation \eqref{Ch3Sec3: EqLoopMeasureLocTime}. To generalize for any $n$ we will use an extended version of Palm's identity and the absolute continuity relation given by Proposition 
\ref{Ch3Sec4: EqMultLocTimeDensity} (ii). 

\begin{lemm}
\label{Ch4Sec3: LemExtendPalm}
Let $\mathcal{E}$ be an abstract Polish space. Let $\mathfrak{M}(\mathcal{E})$ be the space of locally finite measures on $\mathcal{E}$ and let $\mathcal{M}\in\mathfrak{M}(\mathcal{E})$. Let $\Phi$ be a Poisson random measure of intensity $\mathcal{M}$. Let $H$ be a positive measurable function on $\mathfrak{M}(\mathcal{E})\times\mathcal{E}^{n}$. Let $\mathfrak{P}_{n}$ be the set of partitions of $\lbrace 1,\dots,n\rbrace$. If $\mathcal{P}\in \mathfrak{P}_{n}$ and $i\in\lbrace 1,\dots,n\rbrace$, then $\mathcal{P}(i)$ will be the equivalence class of $i$ under $\mathcal{P}$. The following identity holds:
\begin{multline}
\label{Ch4Sec3: EqExtendedPalm}
\mathbb{E}\Big[\int_{\mathcal{E}^{n}}H(\Phi, q_{1},\dots,q_{n})\prod_{i=1}^{n}\Phi(dq_{i})\Big]=\\
\sum_{\mathcal{P}\in\mathfrak{P}_{n}}\int_{\mathcal{E}^{\#\mathcal{P}}}\mathbb{E}
\Big[H(\Phi +\sum_{c\in\mathcal{P}}\delta_{q_{c}},  q_{\mathcal{P}(1)},\dots,q_{\mathcal{P}(n)})\Big]
\prod_{c\in\mathcal{P}}\mathcal{M}(dq_{c}).
\end{multline} 
\end{lemm}

\begin{proof}
We will make an induction over $n$. If $n=1$, \eqref{Ch4Sec3: EqExtendedPalm} is the Palm's identity for Poisson random measures. Assume that $n\geq 2$ and that \eqref{Ch4Sec3: EqExtendedPalm} holds for $n-1$. We set 
\begin{displaymath}
\widetilde{H}(\Phi, q_{1},\dots,q_{n-1}):=\int_{\mathcal{E}}H(\Phi,  q_{1},\dots,q_{n-1}, q_{n})\Phi(dq_{n}).
\end{displaymath}
Then,
\begin{multline}
\label{Ch4Sec3: EqPartitionsRec}
\mathbb{E}\Big[\int_{\mathcal{E}^{n}}H(\Phi, q_{1},\dots,q_{n-1}, q_{n})\prod_{i=1}^{n}\Phi(dq_{i})\Big]=
\mathbb{E}\Big[\int_{\mathcal{E}^{n-1}}\widetilde{H}(\Phi, q_{1},\dots,q_{n-1})\prod_{i=1}^{n-1}\Phi(dq_{i})\Big]\\=
\sum_{\mathcal{P}'\in\mathfrak{P}_{n-1}}\int_{\mathcal{E}^{\#\mathcal{P}'}}
\mathbb{E}\bigg[\int_{\mathcal{E}}H(\Phi +\sum_{c'\in\mathcal{P}'}\delta_{q_{c'}},q_{\mathcal{P}'(1)},\dots,q_{\mathcal{P}'(n-1)}, q_{n})\\\times(\Phi(dq_{n})+\sum_{c'\in\mathcal{P}'}\delta_{q_{c'}}(dq_{n}))\bigg]
\prod_{c'\in\mathcal{P}'}\mathcal{M}(dq_{c'}).
\end{multline}
Given a partition $\mathcal{P}'\in\mathfrak{P}_{n-1}$, one can extend it to a partition of $\lbrace1,\dots,n-1,n\rbrace$ either by deciding that $n$ is single in its equivalence class or by choosing an equivalence class  $c'\in \mathcal{P}'$ and adjoining $n$ to it. In the identity \eqref{Ch4Sec3: EqPartitionsRec}, the first case corresponds to the integration with respect to $\Phi(dq_{n})$, and according to Palm's identity
\begin{multline*}
\mathbb{E}\left[
\int_{\mathcal{E}}H(\Phi +\sum_{c'\in\mathcal{P}'}\delta_{q_{c'}},  q_{\mathcal{P}'(1)},\dots,q_{\mathcal{P}'(n-1)}, q_{n})\Phi(dq_{n})\right]=\\
\int_{\mathcal{E}}\mathbb{E}\left[H(\Phi +\sum_{c'\in\mathcal{P}'}\delta_{q_{c'}},  q_{\mathcal{P}'(1)},\dots,q_{\mathcal{P}'(n-1)}, q_{n})\right]\mathcal{M}(dq_{n}).
\end{multline*}
The second case corresponds to the integration with respect to $\delta_{q_{c'}}(dq_{n})$. Thus, the right-hand side of \eqref{Ch4Sec3: EqPartitionsRec} equals the right-hand side of \eqref{Ch4Sec3: EqExtendedPalm}.
\end{proof}

Next we show how derive a particular case of Dynkin's isomorphism using the above extended Palm's formula. Since  $(\widehat{\mathcal{L}}_{\frac{1}{2}}^{x})_{x\in I}$ and $(\frac{1}{2}\phi_{x}^{2})_{x\in I}$ are equal in law,
\begin{displaymath}
\mathbb{E}_{\phi}\left[\prod_{i=1}^{n}\phi_{x_{i}}^{2}F((\frac{1}{2}\phi_{x}^{2})_{x\in I}) \right]=
2^{n}\mathbb{E}_{\mathcal{L}_{\frac{1}{2}}}\left[\prod_{i=1}^{n}\widehat{\mathcal{L}}_{\frac{1}{2}}^{x_{i}}F((\widehat{\mathcal{L}}_{\frac{1}{2}}^{x})_{x\in I}) \right].
\end{displaymath}
Applying Lemma \ref{Ch4Sec3: LemExtendPalm}, we get that
\begin{multline*}
\mathbb{E}_{\mathcal{L}_{\frac{1}{2}}}\left[\prod_{i=1}^{n}\widehat{\mathcal{L}}_{\frac{1}{2}}^{x_{i}}F((\widehat{\mathcal{L}}_{\frac{1}{2}}^{x})_{x\in I}) \right]=\\
\sum_{\mathcal{P}\in\mathfrak{P}_{n}}\int\prod_{i=1}^{n}\ell^{x_{i}}(\gamma_{\mathcal{P}(i)})
\mathbb{E}\left[F((\widehat{\mathcal{L}}_{\frac{1}{2}}^{x}+\sum_{c\in\mathcal{P}}
\ell^{x}(\gamma_{c}))_{x\in I})\right]\prod_{c\in\mathcal{P}}\frac{1}{2}\mu^{\ast}(d\gamma_{c}).
\end{multline*}
Let $\mathfrak{S}_{n}(\mathcal{P})$ be all the permutations $\sigma$ of $\lbrace 1,\dots,n\rbrace$ such that the classes of the partition $\mathcal{P}$ are the supports of the disjoint cycles of $\sigma$. Given a class $c\in \mathcal{P}$, let $j_{c}$ be its smallest element. From Property \ref{Ch3Sec4: PropertyShuffleAlg} (ii) follows that 
\begin{displaymath}
\prod_{i=1}^{n}\ell^{x_{i}}(\gamma_{\mathcal{P}(i)})=\sum_{\sigma\in\mathfrak{S}_{n}(\mathcal{P})}\prod_{c\in\mathcal{P}}\ell^{\ast x_{j_{c}}, x_{\sigma(j_{c})},\dots,x_{\sigma^{\vert c\vert}(j_{c})}}(\gamma_{c}).
\end{displaymath} 
Proposition \ref{Ch3Sec4: PropMultLocTimeDens} (ii) states that
\begin{multline*}
\ell^{\ast x_{j_{c}}, x_{\sigma(j_{c})},\dots,x_{\sigma^{\vert c\vert}(j_{c})}}(\gamma_{c})\mu^{\ast}(d\gamma_{c})=\\
\pi_{\ast}(\mu^{j_{c},\sigma(j_{c})}(d\tilde{\gamma}_{j_{c}})\lhd\dots\lhd\mu^{\sigma^{\vert c\vert -1}(j_{c}),\sigma^{\vert c\vert}(j_{c})}(d\tilde{\gamma}_{\sigma^{\vert c\vert -1}(j_{c})})\lhd\mu^{\sigma^{\vert c\vert}(j_{c}),j_{c}}(d\tilde{\gamma}_{\sigma^{\vert c\vert}(j_{c})})),
\end{multline*}
and if the loop $\gamma_{c}$ is a concatenation of paths 
$\tilde{\gamma}_{j_{c}},\dots,\tilde{\gamma}_{\sigma^{\vert c\vert -1}(j_{c})}, \tilde{\gamma}_{\sigma^{\vert c\vert}(j_{c})}$, then
\begin{displaymath}
\ell^{x}(\gamma_{c}) = \ell^{x}(\tilde{\gamma}_{j_{c}}) +\dots + \ell^{x}(\tilde{\gamma}_{\sigma^{\vert c\vert -1}(j_{c})}) + \ell^{x}(\tilde{\gamma}_{\sigma^{\vert c\vert}(j_{c})}).
\end{displaymath}
It follows that
\begin{multline}
\label{Ch4Sec3: EqDynkinSquares}
2^{n}\mathbb{E}_{\mathcal{L}_{\frac{1}{2}}}\left[\prod_{i=1}^{n}
\widehat{\mathcal{L}}_{\frac{1}{2}}^{x_{i}}F((\widehat{\mathcal{L}}_{\frac{1}{2}}^{x})_{x\in I}) \right]=\\
\sum_{\sigma\in\mathfrak{S}_{n}}2^{n-\# ~\text{cycles of}~\sigma}
\int\mathbb{E}_{\mathcal{L}_{\frac{1}{2}}}\left[F((\widehat{\mathcal{L}}_{\frac{1}{2}}^{x}+\sum_{i=1}^{n}\ell^{x}(\tilde{\gamma}_{i}))_{x\in I}) \right]\prod_{i=1}^{n}\mu^{i,\sigma(i)}(d\tilde{\gamma}_{i}).
\end{multline}
But the right-hand side of \eqref{Ch4Sec3: EqDynkinSquares} is just the same as the right-hand side of 
\eqref{Ch4Sec3: EqDynkinIso} in the specific case when for all $i\in\lbrace 1,\dots, n\rbrace$, $x_{i+n}=x_{i}$. This finishes the derivation of the special case of Dynkin's isomorphism.

\chapter{Decomposing paths into Poisson ensembles of loops}
\label{Ch5}

\section{Gluing together excursions ordered by their minima}
\label{Ch5Sec1}

Let $L$ be the generator of a diffusion on $I$ of form \eqref{Ch2Sec2: EqGenerator}. A loop of $\mathcal{L}_{\alpha, L}$ rooted at its minimal point is a positive excursion. For a given $x_{0}\in I$, we will consider the loops $\gamma\in\mathcal{L}_{\alpha, L}$ such that $\min\gamma\in(\inf I,x_{0}]$. We will root these loops at their minima and then order the obtained excursions in the decreasing sense of their minima. Then we will glue all these excursions together and obtain a continuous paths 
$\xi_{\alpha, L}^{(x_{0})}$. The law of this path can be described as a one-dimensional projection of a two-dimensional Markov process. Moreover this path contains all the information on the ensemble of loops 
$\mathcal{L}_{\alpha, L}\cap\lbrace \gamma\in \mathfrak{L}^{\ast}\vert \min \gamma <x_{0}\rbrace$. So this is a way to sample the latter ensemble of loops. 

In case of Brownian loops, the paths $\xi_{\alpha, BM}^{(x_{0})}$ turn out to belong to a family of Brownian motions perturbed at their minima, studied for instance in \cite{LeGallYor86CarreBessel,CarmonaPetitYor94,Perman96}. The decomposition of such perturbed Brownian motions into Poisson point processes of excursions already appears in \cite{LeGallYor86CarreBessel,Perman96}.

We will also see that, for a general diffusion of generator $L$, the case $\alpha=1$ is particular. Indeed, $\xi_{1, L}^{(x_{0})}$ is the sample path of a one-dimensional diffusion. This is the analogue in dimension one of the link between $\mathcal{L}_{1}$ and the loop-erasure procedure already observed in \cite{LawlerWerner2004ConformalLoopSoup}, in the setting of two-dimensional Brownian loops, and in \cite{LeJan2011Loops}, Chapter $8$, in the setting of discrete loops on network, and will de described in detail in Section \ref{Ch5Sec3}. See also \cite{Fitzsimmons2013ExcMin}.

In Section \ref{Ch5Sec1} we will consider generalities about gluing together excursions ordered by their minima and probability laws will not be involved. In Section \ref{Ch5Sec2} we will deal with $\xi_{\alpha, L}^{(x_{0})}$ and identify its law. In Section \ref{Ch5Sec3} we will focus on the case $\alpha=1$ and describe other ways of slicing sample paths of diffusions into Poisson ensembles of loops.

Let $x_{0}\in\mathbb{R}$ and let $\mathcal{Q}$ be a countable everywhere dense subset of $(-\infty, x_{0})$. We consider a deterministic collection of excursions $(\mathtt{e}_{q})_{q\in\mathcal{Q}}$, where 
$(\mathtt{e}_{q}(t))_{0\leq t\leq T(\mathtt{e}_{q})}$ is a continuous excursion above $0$, $T(\mathtt{e}_{q})>0$, and
\begin{displaymath}
\mathtt{e}_{q}(0)=\mathtt{e}_{q}(T(\mathtt{e}_{q}))=0,
\end{displaymath}
\begin{displaymath}
\forall t\in (0,T(\mathtt{e}_{q})),~\mathtt{e}_{q}(t)>0.
\end{displaymath}
We also assume that for all $C>0$ and $a<x_{0}$, there are only finitely many $q\in \mathcal{Q}\cap(a,x_{0})$ such that $\max \mathtt{e}_{q}>C$, and that for all $a<x_{0}$,
\begin{equation}
\label{Ch5Sec1: EqCondFinT}
\sum_{q\in \mathcal{Q}\cap(a,x_{0})}T(\mathtt{e}_{q}) < +\infty.
\end{equation} 
Let $\mathbf{T}(y)$ be the function defined on $[0,+\infty)$ by
\begin{displaymath}
\mathbf{T}(y):=\sum_{q\in \mathcal{Q}\cap(x_{0}-y,x_{0})}T(\mathtt{e}_{q}).
\end{displaymath}
$\mathbf{T}$ is a non-decreasing function. Since $\mathcal{Q}$ is everywhere dense, $\mathbf{T}$ is increasing. $\mathbf{T}$ is right-continuous and jumps when $x_{0}-y\in\mathcal{Q}$. The height of the jump is then 
$T(\mathtt{e}_{-y})$. 

Let $T_{\rm max}:=\mathbf{T}(+\infty)\in (0,+\infty]$. For $t\in [0,T_{\rm max})$, we define
\begin{displaymath}
\theta(t):=x_{0}-\sup\lbrace y\in [0,+\infty)\vert \mathbf{T}(y)>t\rbrace.
\end{displaymath}
$\theta$ is a non-increasing function from $[0,T_{\rm max})$ to $(-\infty, x_{0}]$. Since $\mathbf{T}$ is increasing, $\theta$ is continuous. We define
\begin{displaymath}
b^{-}(t)=\inf\lbrace s\in [0,T_{\rm max})\vert \theta(s)=\theta(t)\rbrace,
\end{displaymath}
\begin{displaymath}
b^{+}(t)=\sup\lbrace s\in [0,T_{\rm max})\vert \theta(s)=\theta(t)\rbrace.
\end{displaymath}
$b^{-}(t)<b^{+}(t)$ if and only if $\theta(t)\in\mathcal{Q}$, and then, $b^{+}(t)-b^{-}(t)=T(\mathtt{e}_{\theta(t)})$. We introduce the set
\begin{displaymath}
\mathfrak{b}^{-}:=\lbrace t\in[0,T_{\rm max})\vert \theta(t)\in\mathcal{Q},\,b^{-}(t)=\theta(t)\rbrace.
\end{displaymath}
$\mathfrak{b}^{-}$ is in one to one correspondence with $\mathcal{Q}$ by $t\mapsto\theta(t)$.

Finally, we define on $[0,T_{\rm max})$ the function $\xi$:
\begin{displaymath}
\xi(t):=
\left\lbrace 
\begin{array}{ll}
\theta(t) & \text{if}~\theta(t)\not\in\mathcal{Q}, \\ 
\theta(t)+\mathtt{e}_{\theta(t)}(t-b^{-}(t)) & \text{if}~\theta(t)\in\mathcal{Q}.
\end{array} 
\right. 
\end{displaymath}
Intuitively, $\xi$ is the function obtained by gluing together the excursions $(q+\mathtt{e}_{q})_{q\in\mathcal{Q}}$ ordered in decreasing sense of their minima. See Figure \ref{FigBM} for an example of $\xi$ and $\theta$.

\begin{prop}
\label{Ch5Sec1: PropContinuity}
$\xi$ is continuous. For all $t\in[0,T_{\rm max})$,
\begin{equation}
\label{Ch5Sec1: EqThetaInfXi}
\theta(t)=\inf_{[0,t]}\xi.
\end{equation}
The set $\mathfrak{b}^{-}$ can be recovered from $\xi$ as follows:
\begin{equation}
\label{Ch5Sec1: EqLeftBound}
\mathfrak{b}^{-}=\lbrace t\in[0,T_{\rm max})\vert \xi(t)=\inf_{[0,t]}\xi,~\text{and}~
\exists\varepsilon>0, \forall s\in(0,\varepsilon),\xi(t+s)>\xi(t)\rbrace.
\end{equation}
If $t_{0}\in \mathfrak{b}^{-}$, then
\begin{equation}
\label{Ch5Sec1: EqRightBound}
b^{+}(t_{0})=\inf\lbrace t\in[t_{0},T_{\rm max}]\vert \xi(t)<\xi(t_{0})\rbrace.
\end{equation}
\end{prop}

\begin{proof}
Let $t\in[0,T_{\rm max})$. To prove the continuity of $\xi$ at $t$, we distinguish three case: the first case is when $\theta(t)\in\mathcal{Q}$ and $b^{-}(t)<t<b^{+}(t)$, the second case is when $\theta(t)\not\in\mathcal{Q}$, and the third case is when $\theta(t)\in\mathcal{Q}$ and either $b^{-}(t)=t$ or $b^{+}(t)=t$. 

In the first case, for all $s\in(b^{-}(t),b^{+}(t))$,
\begin{displaymath}
\xi(s)=\theta(t)+\mathtt{e}_{\theta(t)}(s-b^{-}(t)).
\end{displaymath}
$\mathtt{e}_{\theta(t)}$ being continuous, we get the continuity of $\xi$ at $t$.

In the second case, we consider a sequence $(t_{n})_{n\geq 0}$ in $[0,T_{\rm max})$ converging to $t$. Let $C>0$. There are only finitely many $q\in\mathcal{Q}$ such that there is $n\geq 0$ such that $\theta(t_{n})=q$ and $\max\mathtt{e}_{q}>C$. Moreover, for any $q\in\mathcal{Q}$, there are only finitely many $n\geq 0$ such that $\theta(t_{n})=q$. Thus, there are only finitely many $n\geq 0$ such that $\theta(t_{n})\in \mathcal{Q}$ and $\max\mathtt{e}_{\theta(t_{n})}>C$. So, for $n$ large enough,
\begin{equation}
\label{Ch5Sec1: EqEncadrXiTheta}
\theta(t_{n})\leq \xi(t_{n})\leq \theta(t_{n}) + C.
\end{equation}
But, $\xi(t)=\theta(t)$ and $\theta(t_{n})$  converges to $\theta(t)$. Since we may take $C$ arbitrarily small, \eqref{Ch5Sec1: EqEncadrXiTheta} implies that $\xi(t_{n})$ converges to  $\theta(t)$.

Regarding the third case, assume for instance that $\theta(t)\in\mathcal{Q}$ and $t=b^{-}(t)$. The right-continuity of $\xi$ at $t$ follows from the same argument as in the first case and left-continuity from the same argument as in the second case.

By definition, for all $t\in[0,T_{\rm max})$, $\theta(t)\leq\xi(t)$. $\theta$ being non-increasing, for all $t\in[0,T_{\rm max})$
\begin{displaymath}
\theta(t)\leq \inf_{[0,t]}\xi.
\end{displaymath}
For the converse inequality, we have
\begin{displaymath}
\theta(t)=\xi(b^{-}(t))\geq \inf_{[0,t]}\xi.
\end{displaymath}

Regarding \eqref{Ch5Sec1: EqLeftBound} and \eqref{Ch5Sec1: EqRightBound}, we have the following disjunction: If $\theta(t)\in\mathcal{Q}$ and $b^{-}(t)<t<b^{+}(t)$, then $\xi(t)>\theta(t)$. If $\theta(t)\in\mathcal{Q}$ and $t=b^{-}(t)$, then for all $s\in(0,b^{+}(t)-b^{-}(t))$, $\xi(t+s)>\xi(t)$. If either $\theta(t)\in\mathcal{Q}$ and $t=b^{+}(t)$ or $\theta(t)\not\in\mathcal{Q}$, then $\xi(t)=\theta(t)$, and there is a positive sequence $(s_{n})_{n\geq 0}$ decreasing to $0$ such that $\theta(t+s_{n})\not\in\mathcal{Q}$ and $\xi(t+s_{n})=\theta(t+s_{n})<\theta(t)$.
\end{proof}

\begin{figure}[H]

\centering{
\includegraphics[width=1\textwidth]{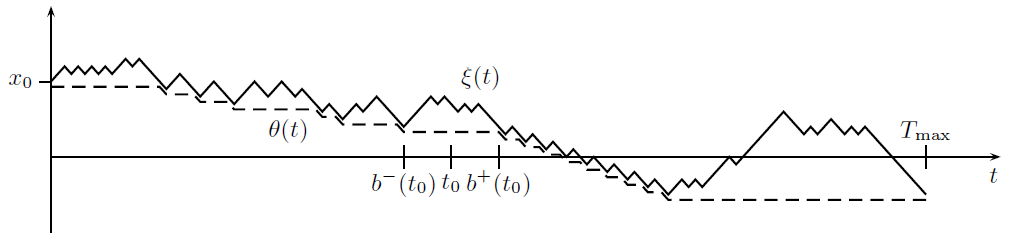}
}
\caption{Drawing of $\xi$ (full line) and $\theta$ (dashed line).}
\label{FigBM}
\end{figure}

Previous proposition shows that one can reconstruct $\mathcal{Q}$ and the family of excursions $(\mathtt{e}_{q})_{q\in\mathcal{Q}}$ only knowing $\xi$. 
\eqref{Ch5Sec1: EqThetaInfXi} shows how to recover $\theta$ from $\xi$. \eqref{Ch5Sec1: EqLeftBound} and \eqref{Ch5Sec1: EqRightBound} show how to recover the left and the right time boundaries of the excursions of $\xi$ above $\theta$. Also, observe that the set defined by the right-hand side of \eqref{Ch5Sec1: EqLeftBound} is countable whatever the continuous function $\xi$ is, even if it is not obtained by gluing together excursions.

\section{Loops represented as excursions and glued together}
\label{Ch5Sec2}

Let $\alpha >0$ and $\mathcal{L}_{\alpha, BM}$ the Poisson ensemble of loops of intensity $\alpha\mu^{\ast}_{BM}$, where $\mu^{\ast}_{BM}$ is the measure on loops associated to the Brownian motion on $\mathbb{R}$. Let $x_{0}\in\mathbb{R}$. We consider the random countable set $\mathcal{Q}$:
\begin{displaymath}
\mathcal{Q}:=\lbrace \min\gamma\vert\gamma\in
\mathcal{L}_{\alpha, BM}\rbrace\cap(-\infty,x_{0}).
\end{displaymath}
Almost surely, $\mathcal{Q}$ is everywhere dense in $(-\infty,x_{0})$, and for every $q\in\mathcal{Q}$, there is only one $\gamma\in\mathcal{L}_{\alpha, BM}$ such that $\min\gamma=q$. Almost surely, $\gamma\in\mathcal{L}_{\alpha, BM}$ reaches its minimum at one single moment. Given $q\in\mathcal{Q}$ and $\gamma\in\mathcal{L}_{\alpha, BM}$ such that $\min\gamma=q$, we consider $\mathtt{e}_{q}$ to be the excursion above $0$ equal to $\gamma - q$, where we root the unrooted loop $\gamma$ at $\operatorname{argmin} \gamma$. Then, the random set of excursions $(\mathtt{e}_{q})_{q\in\mathcal{Q}}$ almost surely satisfies the assumptions of Section \ref{Ch5Sec1}. In particular, the condition \eqref{Ch5Sec1: EqCondFinT} follows from the fact that, according to 
\eqref{Ch3Sec5: EqDisintegBrownianLoops},
\begin{displaymath}
\int_{\mathfrak{L}^{\ast}}1\wedge T(\gamma)1_{\min\gamma\in (a,x_{0})}\mu_{BM}^{\ast}(d\gamma)= (x_{0}-a)\int_{0}^{+\infty}\dfrac{t\wedge 1}{\sqrt{2\pi t^{3}}}dt <+\infty.
\end{displaymath}
Thus, we can consider the random continuous function 
$(\xi_{\alpha, BM}^{(x_{0})}(t))_{t\geq 0}$ constructed by gluing together the excursions $(q+\mathtt{e}_{q})_{q\in\mathcal{Q}}$ in the way described in section \ref{Ch5Sec1}. Let
\begin{displaymath}
\theta_{\alpha, BM}^{(x_{0})}(t)=\inf_{[0,t]}\xi_{\alpha, BM}^{(x_{0})},
\end{displaymath}
\begin{displaymath}
\Xi_{\alpha, BM}^{(x_{0})}(t):=\big(\xi_{\alpha, BM}^{(x_{0})}(t),
\theta_{\alpha, BM}^{(x_{0})}(t)\big).
\end{displaymath}
Next we will describe the law of the two-dimensional process 
$\big(\Xi_{\alpha, BM}^{(x_{0})}(t)\big)_{t\geq 0}$.

\begin{prop}
\label{Ch5Sec2: PropBMContour}
Let $(B_{t})_{t\geq 0}$ be a standard Brownian motion on $\mathbb{R}$ starting from $0$. $\big(\Xi_{\alpha, BM}^{(x_{0})}(t)\big)_{t\geq 0}$ 
has the same law as 
\begin{displaymath}
\left(x_{0}+\vert B_{t}\vert-\dfrac{1}{\alpha}\ell_{t}^{0}(B), 
x_{0}-\dfrac{1}{\alpha}\ell_{t}^{0}(B)\right)_{t\geq 0}. 
\end{displaymath}
In particular, for $\alpha=1$, $(\xi_{1, BM}^{(x_{0})}(t))_{t\leq 0}$ has the same law as a Brownian motion starting from $x_{0}$.
\end{prop}

\begin{proof}
For $a<x_{0}$, let $T_{a}$ be the first time $\theta_{\alpha, BM}^{(x_{0})}$ hits $a$. For $l>0$, let
\begin{displaymath}
\tilde{\tau}^{0}_{l}:=\inf\lbrace t>0\vert \ell_{t}^{0}(B)>l\rbrace.
\end{displaymath}
According to the disintegration \eqref{Ch3Sec5: EqDisintegBrownianLoops} of the measure $\mu^{\ast}_{BM}$ in Proposition \ref{Ch3Sec5: PropDesintegMinBM}, for all $a<x_{0}$, the family $(\mathtt{e}_{q})_{q\in \mathcal{Q}\cap(a,x_{0})}$ of excursions above $0$  is a Poisson point process of intensity $2\alpha\eta_{BM}^{>0}$. This implies the following equality in law:
\begin{displaymath}
\big(\xi_{\alpha, BM}^{(x_{0})}(t)-\theta_{\alpha, BM}^{(x_{0})}(t)\big)
_{0\leq t\leq T_{a}}  \stackrel{(\text{law})}{=} 
(\vert B_{t}\vert)_{0\leq t\leq \tilde{\tau}^{0}_{\alpha (x_{0}-a)}}.
\end{displaymath}
Since the above holds for all $a<x_{0}$, we have the following equality in law:
\begin{displaymath}
(\xi_{\alpha, BM}^{(x_{0})}(t)-\theta_{\alpha, BM}^{(x_{0})}(t),\alpha(x_{0}-\theta_{\alpha, BM}(t)))_{t\geq 0}\stackrel{(\text{law})}{=} 
(\vert B_{t}\vert, \ell_{t}^{0}(B))_{t\geq 0},
\end{displaymath}
which is exactly the equality in law we needed. Finally, for $\alpha=1$, $(x_{0}+\vert B_{t}\vert-\ell_{t}^{0}(B))_{t\geq 0}$ has the law of a Brownian motion starting from $x_{0}$. See \cite{RevuzYor1999BMGrundlehren}, Section VI.2.
\end{proof}

According to Proposition \ref{Ch5Sec2: PropBMContour}, a Brownian sample path can be decomposed into a Poisson process of positive excursion with decreasing minima. This decomposition is for instance described in \cite{LeGallYor86CarreBessel,Perman96} and \cite{KaratzasShreve2010BMStochCalc}, Section $6.2.D$. In case $\alpha=1$, Proposition \ref{Ch4Sec2: PropCBPI} states that the occupation field of a the Poisson ensemble of loops associated to the Brownian motion on $(0,+\infty)$ killed at $0$ is the square of a Bessel process of dimension $2$ starting from $0$ at $0$. This result can also be obtained using the fact that 
$(\xi_{1, BM}^{(x_{0})}(t))_{t\leq 0}$ is a Brownian sample path and applying the first Ray-Knight theorem which gives the law of the occupation field of a Brownian path stopped upon hitting $0$.

From Proposition \ref{Ch5Sec2: PropBMContour} follows in particular that 
$\big(\Xi_{\alpha, BM}^{(x_{0})}(t)\big)_{t\geq 0}$ is a sample path of a two-dimensional Feller process. Let
\begin{displaymath}
T^{+}(\mathbb{R}^{2}):=\lbrace(x,a)\in\mathbb{R}^{2}\vert x\geq a\rbrace,
\qquad
\operatorname{Diag}(\mathbb{R}^{2}):=\lbrace(x,x)\vert x\in\mathbb{R}\rbrace.
\end{displaymath}
For $(x_{0},a_{0})\in T^{+}(\mathbb{R}^{2})$, we define the process
\begin{multline}
\label{Ch5Sec2: EqDefXi}
\big(\Xi^{(x_{0},a_{0})}_{\alpha, BM}(t)\big)_{t\geq 0}=
\big(\xi^{(x_{0},a_{0})}_{\alpha, BM}(t),\theta^{(x_{0},a_{0})}_{\alpha, BM}(t)\big)_{t\geq 0}\\
:=\left(a_{0}+\vert x_{0}-a_{0}+B_{t}\vert-\dfrac{1}{\alpha}
\ell_{t}^{a_{0}-x_{0}}(B), 
a_{0}-\dfrac{1}{\alpha}\ell_{t}^{a_{0}-x_{0}}(B)\right)_{t\geq 0},
\end{multline}
where $(B_{t})_{t\geq 0}$ is a Brownian motion starting from $0$. 
$\Xi^{(x_{0},x_{0})}_{\alpha, BM}$ has the same law as $\Xi^{(x_{0})}_{\alpha, BM}$. The family of paths $\big(\Xi^{(x_{0},a_{0})}_{\alpha, BM}\big)_{x_{0}\geq a_{0}}$ are the sample paths of the same Feller semi-group on $T^{+}(\mathbb{R}^{2})$, starting from all possible positions. Next we describe this semi-group in terms of generator and domain. Let $f$ be a continuous function on $T^{+}(\mathbb{R}^{2})$, $\mathcal{C}^{2}$ on the interior of $T^{+}(\mathbb{R}^{2})$, such that all its second order derivatives extend continuously to $\operatorname{Diag}(\mathbb{R}^{2})$. This implies in particular that the first order derivatives also extend continuously to $\operatorname{Diag}(\mathbb{R}^{2})$. We write $\partial_{1}f$, $\partial_{2}f$ and $\partial_{1,1}f$ for the first order derivative relatively to the first variable, the second variable and the second order derivative relatively the first variable. Applying Itô-Tanaka's formula we get
\begin{multline*}
f\big(\Xi^{(x_{0},a_{0})}_{\alpha, BM}(t)\big)=f(x_{0},a_{0})+
\int_{0}^{t}\partial_{1}f\big(\Xi^{(x_{0},a_{0})}_{\alpha, BM}(s)\big)\operatorname{sign}(x_{0}-a_{0}+B_{s})d B_{s}+
\\\int_{0}^{t}\left(\left(1-\dfrac{1}{\alpha}\right)\partial_{1}-\dfrac{1}{\alpha}\partial_{2}\right)
f\big(\Xi^{(x_{0},a_{0})}_{\alpha, BM}(s)\big)d_{s}\ell_{s}^{a_{0}-x_{0}}(B) 
+\dfrac{1}{2}\int_{0}^{t}\partial_{1,1}f\big(\Xi^{(x_{0},a_{0})}_{\alpha, BM}(s)\big)ds.
\end{multline*}
Let $\mathcal{D}_{\alpha, BM}$ be the set of continuous functions $f$ on $T^{+}(\mathbb{R}^{2})$, $\mathcal{C}^{2}$ on the interior of $T^{+}(\mathbb{R}^{2})$, such that all the second order derivatives extend continuously to $\operatorname{Diag}(\mathbb{R}^{2})$ and that moreover satisfy the following constraints: $f$ and $\partial_{1,1}f$ are uniformly continuous and bounded (which also implies that $\partial_{1}f$ is bounded by the inequality $\Vert \partial_{1}f\Vert_{\infty}\leq 2\sqrt{\Vert f\Vert_{\infty} \Vert \partial_{1,1}f\Vert_{\infty}}$), and on 
$\operatorname{Diag}(\mathbb{R}^{2})$ the following equality holds:
\begin{displaymath}
\left(\left(1-\dfrac{1}{\alpha}\right)\partial_{1}-\dfrac{1}{\alpha}\partial_{2}\right)f(x,x) = 0.
\end{displaymath}
If $f\in\mathcal{D}_{\alpha, BM}$, then $\frac{1}{t}\big(\mathbb{E}\big[f(\Xi^{x_{0},a_{0}}_{\alpha, BM}(t))\big]-f(x_{0},a_{0})\big)$ converges as $t\rightarrow 0^{+}$, uniformly for $(x_{0},a_{0})\in T^{+}(\mathbb{R}^{2})$, to $\frac{1}{2}\partial_{1,1}f(x_{0},a_{0})$. Moreover, $\mathcal{D}_{\alpha, BM}$ is a core for  $\frac{1}{2}\partial_{1,1}$ in the space of continuous bounded function on $T^{+}(\mathbb{R}^{2})$.

Next we describe what we obtain if we glue together the loops, seen as excursion, ordered in the decreasing sense of their minima, where instead of $\mathcal{L}_{\alpha, BM}$ we use the Poisson ensemble of Markov loops associated to a general diffusion. Let $I$ be an open interval of $\mathbb{R}$ and $\widetilde{L}$ a generator on $I$ of form
\begin{displaymath}
\widetilde{L}=\dfrac{1}{\tilde{m}(x)}\dfrac{d}{dx}\left(\dfrac{1}{\tilde{w}(x)}\dfrac{d}{dx}\right) 
\end{displaymath}
with zero Dirichlet boundary conditions. Let $\widetilde{S}$ be a primitive of $\tilde{w}(x)$. We assume that $\widetilde{S}(\sup I)=+\infty$. Let
\begin{displaymath}
T^{+}(I^{2}):=\lbrace(x,a)\in I^{2}\vert x\geq a\rbrace,
\qquad
\operatorname{Diag}(I^{2}):=\lbrace(x,x)\vert x\in I\rbrace.
\end{displaymath}
Let $\widehat{T^{+}(I^{2})}$ be the closure of $T^{+}(I^{2})$ in $(\inf I, \sup I]^{2}$.

Given any $x'_{0}\geq a'_{0}>\frac{1}{2}\widetilde{S}(\inf I)$, let $\tilde{\zeta}_{\alpha}$ be the first time $\Xi^{(x'_{0},a'_{0})}_{\alpha, BM}$ hits $\frac{1}{2}\tilde{S}(\inf I)$. Let
\begin{displaymath}
\widetilde{I}_{t}:=\int_{0}^{t}\dfrac{1}{\tilde{m}}(\widetilde{S}^{-1}
\big(2\xi^{(x'_{0},a'_{0})}_{\alpha, BM}(s))\big)ds.
\end{displaymath}
Let $(\widetilde{I}^{-1}_{t})_{0\leq t< \widetilde{I}_{\tilde{\zeta}_{\alpha}}}$ be the inverse function of $(\widetilde{I}_{t})_{0\leq t<\tilde{\zeta}_{\alpha}}$. It is a family of stopping times for $\Xi^{(x'_{0},a'_{0})}_{\alpha, BM}$. For 
$x_{0}\geq a_{0}\in I$ and $t<\widetilde{I}_{\tilde{\zeta}_{\alpha}}$, let 
\begin{displaymath}
\Xi^{(x_{0},a_{0})}_{\alpha, \widetilde{L}}(t)=\big(\xi^{(x_{0},a_{0})}_{\alpha, \widetilde{L}}(t),\theta^{(x_{0},a_{0})}_{\alpha, \widetilde{L}}(t)\big)
:=\Xi^{(\widetilde{S}(2x_{0}),\widetilde{S}(2a_{0}))}_{\alpha, BM}(\widetilde{I}^{-1}_{t}).
\end{displaymath}
If $\alpha=1$, then $\xi^{(x_{0},a_{0})}_{\alpha, \tilde{L}}$ is just the sample paths starting $x_{0}$ of a diffusion of generator $\widetilde{L}$.
Let $\widehat{\mathcal{D}}_{\alpha, \tilde{L}}$ be the space of continuous functions $f$ on $T^{+}(I^{2})$ satisfying
\begin{itemize}
\item $f\circ\widetilde{S}^{-1}$ is $\mathcal{C}^{2}$ on the interior of $T^{+}(I^{2})$ and all the second order derivatives extend continuously to $\operatorname{Diag}(I^{2})$.
\item $f(x,a)$ and $\frac{1}{\tilde{m}(x)}\partial_{1}\left(\frac{1}{\tilde{w}(x)}\partial_{1}f(x,a)\right)$ are bounded on $T^{+}(I^{2})$ and extend continuously to $\widehat{T^{+}(I^{2})}$.
\item $f(x,a)$ and $\frac{1}{\tilde{m}(x)}\partial_{1}\left(\frac{1}{\tilde{w}(x)}\partial_{1}f(x,a)\right)$ converge to $0$ as $a$ converges to $\inf I$ uniformly in $x$.
\item On $\operatorname{Diag}(I^{2})$ the following equality holds:
\begin{equation}
\label{Ch5Sec2: EqBorderConstraint}
\left(\left(1-\dfrac{1}{\alpha}\right)\partial_{1}-\dfrac{1}{\alpha}\partial_{2}\right)f(x,x) = 0.
\end{equation}
\end{itemize}

\begin{lemm}
\label{Ch5Sec2: LemCore}
$\big(\Xi^{(x_{0},a_{0})}_{\alpha, \widetilde{L}}\big)_{x_{0}\geq a_{0}\in I}$ is a family of sample path starting from all possible positions of the same Markovian or sub-Markovian semi-group on $T^{+}(I^{2})$. The law of the path 
$\Xi^{(x_{0},a_{0})}_{\alpha, \tilde{L}}$ depends weakly continuously on the starting point $(x_{0},a_{0})$. The domain of the generator of this semi-group contains $\widehat{\mathcal{D}}_{\alpha, \tilde{L}}$, and on this space the generator equals
\begin{displaymath}
\dfrac{1}{\tilde{m}(x)}\partial_{1}\left(\dfrac{1}{\tilde{w}(x)}\partial_{1}\right).
\end{displaymath}
Moreover, there is only one Markovian or sub-Markovian semi-group with such generator on $\widehat{\mathcal{D}}_{\alpha, \widetilde{L}}$.
\end{lemm}

\begin{proof}
Since a change of scale does not alter the validity of the above statement, we can assume that $\tilde{w}\equiv 2$. Then, $\sup I = +\infty$. 
$\big(\Xi^{(x_{0},a_{0})}_{\alpha, \tilde{L}}(t)\big)_{0\leq t \leq \widetilde{I}_{\tilde{\zeta}_{\alpha}}}$ is then obtained from $\big(\Xi^{(x_{0},a_{0})}_{\alpha, BM}(t)\big)_{0\leq t<\tilde{\zeta}_{\alpha}}$ by a random time change. The Markov property and the continuous dependence on the starting point for $\Xi^{(x_{0},a_{0})}_{\alpha, \widetilde{L}}$ follows from analogous  properties for $\Xi^{(x_{0},a_{0})}_{\alpha, BM}$. If 
$f\in\widehat{\mathcal{D}}_{\alpha, \widetilde{L}}$, then
\begin{displaymath}
\left(f\big(\Xi^{(x_{0},a_{0})}_{\alpha, BM}(\widetilde{I}^{-1}_{t}\wedge\tilde{\zeta}_{\alpha})\big)-\dfrac{1}{2}\int_{0}^{\widetilde{I}^{-1}_{t}\wedge\tilde{\zeta}_{\alpha}}\partial_{1,1}
f\big(\Xi^{(x_{0},a_{0})}_{\alpha, BM}(s)\big)ds\right)_{t\geq 0}
\end{displaymath}
is a local martingale. We can rewrite it as 
\begin{displaymath}
\left(f\big(\Xi^{(x_{0},a_{0})}_{\alpha, \widetilde{L}}(t\wedge\widetilde{I}_{\tilde{\zeta}_{\alpha}})\big)-\int_{0}^{t}\dfrac{1}{2\tilde{m}\big(\xi^{(x_{0},a_{0})}_{\alpha, \widetilde{L}}(s)\big)}\partial_{1,1}f\big(\Xi^{(x_{0},a_{0})}_{\alpha, \widetilde{L}}(s)\big)1_{s<\widetilde{I}_{\tilde{\zeta}_{\alpha}}}ds\right)_{t\geq 0}.
\end{displaymath}
The above local martingale is bounded on all finite time intervals and thus is a true martingale. Since $\frac{1}{2\tilde{m}(x)}\partial_{1,1}f(x,a)$ converges to $0$ as $a$ converges to $\inf I$, uniformly in $x$, it follows that
\begin{displaymath}
f\big(\Xi^{(x_{0},a_{0})}_{\alpha,\widetilde{L}}(t\wedge\widetilde{I}_{\tilde{\zeta}_{\alpha}})\big) 
= 1_{t<\widetilde{I}_{\tilde{\zeta}_{\alpha}}}f\big(\Xi^{(x_{0},a_{0})}
_{\alpha,\widetilde{L}}(t)\big).
\end{displaymath}
Thus,
\begin{displaymath}
\lim_{t\rightarrow 0^{+}}\dfrac{1}{t}
\left(\mathbb{E}\left[1_{t<\widetilde{I}_{\tilde{\zeta}_{\alpha}}}
f\big(\Xi^{(x_{0},a_{0})}_{\alpha,\widetilde{L}}(t)\big)\right]- f(x_{0},a_{0})\right)=
\dfrac{1}{2\tilde{m}(x_{0})}\partial_{1,1}f(x_{0},a_{0}).
\end{displaymath}
Moreover, the above convergence is uniform in $(x_{0},a_{0})$ because 
$\frac{1}{2\tilde{m}(x)}\partial_{1,1}f(x,a)$ extends continuously to 
$\widehat{T^{+}(I^{2})}$.

To prove the uniqueness of the semi-group, we need to show that there is $\lambda>0$ such that 
\begin{displaymath}
\left(\dfrac{1}{2\tilde{m}(x)}\partial_{1,1}-\lambda\right)(\widehat{\mathcal{D}}_{\alpha, \tilde{L}})
\end{displaymath}
is sufficiently large, for instance that it contains all functions with compact support in $T^{+}(I^{2})$. Let $g$ be such a function and $\lambda>0$. Consider the equation
\begin{equation}
\label{Ch5Sec2: EqEqRes}
\dfrac{1}{2\tilde{m}(x)}\partial_{1,1}f(x,a)-\lambda f(x,a)=g(x,a).
\end{equation}
Let $\tilde{u}_{\lambda, \downarrow}$ be a positive decreasing solution to 
\begin{displaymath}
\dfrac{1}{2\tilde{m}(x)}\dfrac{d^{2}u}{dx^{2}}(x)-\lambda u(x) = 0.
\end{displaymath}
Let
\begin{displaymath}
f_{0}(x,a):=\tilde{u}_{\lambda, \downarrow}(x)\int_{x}^{+\infty}\int_{y}^{+\infty}2\tilde{m}(z)g(z,a)\tilde{u}_{\lambda, \downarrow}(z)\,dz\,\dfrac{dy}{\tilde{u}_{\lambda, \downarrow}(y)^{2}}.
\end{displaymath}
Then, $f_{0}$ is a solution to \eqref{Ch5Sec2: EqEqRes} and it is compactly supported in 
$T^{+}(I^{2})$. We look for the solutions to \eqref{Ch5Sec2: EqEqRes} of form
\begin{displaymath}
f(x,a) = f_{0}(x,a) + C(a)\tilde{u}_{\lambda, \downarrow}(x).
\end{displaymath}
$f$ satisfies the constraint \eqref{Ch5Sec2: EqBorderConstraint} if and only if $C$ satisfies
\begin{displaymath}
-\dfrac{1}{\alpha}\tilde{u}_{\lambda, \downarrow}(a)\dfrac{dC}{da}(a)+\left(1-\dfrac{1}{\alpha}\right)\dfrac{d\tilde{u}_{\lambda, \downarrow}}{dx}(a)C(a)+h(a)=0, 
\end{displaymath}
where
\begin{displaymath}
h(a)=\left(\left(1-\dfrac{1}{\alpha}\right)\partial_{1}-\dfrac{1}{\alpha}\partial_{2}\right)f_{0}(a,a).
\end{displaymath}
$h$ is compactly supported in $I$. We can set
\begin{displaymath}
C(a)=\tilde{u}_{\lambda, \downarrow}(a)^{\alpha - 1}\int_{\inf I}^{x}\dfrac{h(y)}{\tilde{u}_{\lambda, \downarrow}(y)^{\alpha}}dy.
\end{displaymath}
$C$ is zero in the neighborhood of $\inf I$. Moreover, $\tilde{u}_{\lambda, \downarrow}$ has a limit at $+\infty$. It follows that 
$f\in\widehat{\mathcal{D}}_{\alpha, \tilde{L}}$. 
\end{proof}

Let $L$ be the generator of a diffusion on $I$ of form \eqref{Ch2Sec2: EqGenerator}. Let $x_{0}\in I$. Consider the loops $\gamma$ in $\mathcal{L}_{\alpha, L}$ such that 
$\min \gamma<x_{0}$, rooted at $\operatorname{argmin} \gamma$, seen as excursions. Let 
$(\xi_{\alpha, L}^{(x_{0})}(t))_{0\leq t<\zeta_{\alpha}}$ be the path on $I$ obtained by gluing together this excursions ordered in the decreasing sense of their minima. Let
\begin{displaymath}
\theta_{\alpha, L}^{(x_{0})}(t):=\min_{[0,t]}\xi_{\alpha, L}^{(x_{0})},
\end{displaymath}
\begin{displaymath}
\Xi_{\alpha, L}^{(x_{0})}:=\big(\xi_{\alpha, L}^{(x_{0})},\theta_{\alpha, L}^{(x_{0})}\big).
\end{displaymath}

\begin{prop}
\label{Ch5Sec2: PropGeneralContour}
Let $\widetilde{L}:=\operatorname{Conj}(u_{\downarrow},L)$. Then, 
$\big(\Xi_{\alpha, L}^{(x_{0})}(t)\big)_{0\leq t<\zeta_{\alpha}}$ has the same law as $\big(\Xi^{(x_{0},x_{0})}_{\alpha, \widetilde{L}}(t)\big)
_{0\leq t<\tilde{\zeta}_{\alpha}}$. So it is a sample path of a two-dimensional Feller process. In particular, for $\alpha=1$, $\xi_{1, L}^{(x_{0})}$ is the sample path of a diffusion of generator $\widetilde{L}$. For all $\alpha>0$,
\begin{displaymath}
\liminf_{t\rightarrow \zeta_{\alpha}}\xi_{\alpha, L}^{(x_{0})}(t)= \inf I.
\end{displaymath}
If $L$ is the generator of a recurrent diffusion, then
\begin{displaymath}
\limsup_{t\rightarrow \zeta_{\alpha}}\xi_{\alpha, L}^{(x_{0})}(t)= \sup I. 
\end{displaymath}
Otherwise,
\begin{displaymath}
\limsup_{t\rightarrow \zeta_{\alpha}}\xi_{\alpha, L}^{(x_{0})}(t)= \inf I. 
\end{displaymath}
\end{prop}

\begin{proof}
First, notice that if $L$ is the generator of a recurrent diffusion then $\widetilde{L}=L$. Otherwise, a diffusion of generator $\widetilde{L}=L$ is, put informally, a diffusion of generator $L$ conditioned to converge to $\inf I$ (which may occur with zero probability). From h-transform invariance of the measure on loops follows that $\mathcal{L}_{\alpha, L}=\mathcal{L}_{\alpha, \widetilde{L}}$. From Property \ref{Ch3Sec3: PropertyTrivialitiesLoops} (iv) and Corollary 
\ref{Ch3Sec3: CorTimeChange} follows that $\Xi_{\alpha, L}^{(x_{0})}$ is obtained from $\Xi_{\alpha, BM}$ by scale and time change in the same way as $\Xi^{(x_{0},x_{0})}_{\alpha, \widetilde{L}}$, and thus, $\Xi_{\alpha, L}^{(x_{0})}$ and $\Xi^{(x_{0},x_{0})}_{\alpha, \widetilde{L}}$ have the same law. Regarding the limits of $\xi_{\alpha, L}^{(x_{0})}$ at $\zeta_{\alpha}$, we need just to observe that they hold if $L$ is the generator of the Brownian motion on an interval of form $(a,+\infty)$, $a\in[-\infty,+\infty)$, and by time and scale change they hold in general.
\end{proof}

As explained in Proposition \ref{Ch5Sec1: PropContinuity}, the knowledge of the path $\big(\xi_{\alpha, L}^{(x_{0})}(t)\big)_{0\leq t<\zeta_{\alpha}}$ alone is enough to reconstruct $\mathcal{L}_{\alpha, L}\cap\lbrace \gamma\in \mathfrak{L}^{\ast}\vert \min \gamma <x_{0}\rbrace$. From this we deduce the following:

\begin{coro}
\label{Ch5Sec2: CorLoopsField}
If $L$ is the generator of a transient diffusion, it is possible to construct on the same probability space $\mathcal{L}_{\alpha, L}$ and a continuous version of the occupation field $(\widehat{\mathcal{L}}_{\alpha, L}^{x})_{x\in I}$.
\end{coro}

\begin{proof}
By scale and time change covariance and invariance by conjugation of the Poisson ensembles of loops, it is enough to prove the statement in case of a Brownian motion on $(0,+\infty)$ killed at $0$. Let $(x_{n})_{n\geq 0}$ be an increasing sequence in $(0,+\infty)$ converging to $+\infty$. We consider a sequence of independent paths $\big(\xi_{\alpha, BM}^{(x_{n},x_{n})}\big)_{n\geq 0}$ defined by 
\eqref{Ch5Sec2: EqDefXi}. Let
\begin{displaymath}
T_{n,x_{n-1}}:=\inf\big\lbrace t\geq 0\vert \xi_{\alpha, BM}^{(x_{n},x_{n})}(t)=x_{n-1}\big\rbrace,
\end{displaymath}
where conventionally we set $x_{-1}:=0$. By decomposing on $[0,T_{n,x_{n-1}}]$ the restricted path $\big(\xi_{\alpha, BM}^{(x_{n},x_{n})}(t)\big)_{0\leq t<T_{n,x_{n-1}}}$, one can reconstruct a family of loops $\gamma$ such that $\min\gamma\in (x_{n-1},x_{n})$: there is a random countable set $\mathscr{B}_{n}$ of disjoint compact subintervals $[b^{-},b^{+}]$ of $[0,T_{n,x_{n-1}}]$ such that 
\begin{displaymath}
\big\lbrace\big(\xi_{\alpha, BM}^{(x_{n},x_{n})}(b^{-}+t)\big)_{0\leq t
\leq b^{+}-b^{-}}\vert[b^{-},b^{+}] \in \mathscr{B}_{n}\big\rbrace = 
\mathcal{L}_{\alpha, BM}\cap\lbrace \gamma\in\mathfrak{L}^{\ast}\vert\min\gamma \in(x_{n-1},x_{n})\rbrace,
\end{displaymath}
(see \eqref{Ch5Sec1: EqLeftBound}). The union of all previous families of loops for $n\geq 0$ is a Poisson ensemble of loops $\mathcal{L}_{\alpha, BM}\cap\lbrace \gamma\in \mathfrak{L}^{\ast}\vert \min \gamma >0\rbrace$. 

Each of $\xi_{\alpha, BM}^{(x_{n},x_{n})}$ is a semi-martingale and its quadratic variation is
\begin{displaymath}
\big\langle \xi_{\alpha, BM}^{(x_{n},x_{n})}, \xi_{\alpha, BM}^{(x_{n},x_{n})}\big\rangle_{t} = t.
\end{displaymath}
Moreover, for all $x\in\mathbb{R}$,
\begin{displaymath}
\int_{0}^{t}1_{\xi_{\alpha, BM}^{x_{n},x_{n}}=x} d\xi_{\alpha, BM}^{(x_{n},x_{n})}(s)=\left(1-\dfrac{1}{\alpha}\right)\int_{0}^{t}1_{\ell_{s}^{0}(B)=\alpha x} d_{s}\ell_{s}^{0}(B) = 0.
\end{displaymath}
From Theorems $1.1$ and $1.7$ in \cite{RevuzYor1999BMGrundlehren}, Section VI.1, follows that we can construct on the same probability space 
$\xi_{\alpha, BM}^{(x_{n},x_{n})}$ and a space-time continuous version of local times $\big(\ell^{x}_{t}\big(\xi_{\alpha, BM}^{(x_{n},x_{n})}\big)\big)_{x\in\mathbb{R}, t\geq 0}$ of $\xi_{\alpha, BM}^{(x_{n},x_{n})}$ relatively to the Lebesgue measure.  In particular ${x\mapsto\ell^{x}_{T_{n,x_{n-1}}}\big(\xi_{\alpha, BM}^{(x_{n},x_{n})}\big)}$ is continuous. If $[b^{-},b^{+}]\in\mathcal{J}_{n}$, then 
\begin{displaymath}
\big(\ell^{x}_{b^{+}}\big(\xi_{\alpha, BM}^{(x_{n},x_{n})})\big)-\ell^{x}_{b^{-}}\big(\xi_{\alpha, BM}^{(x_{n},x_{n})}\big)\big)_{x>0}
\end{displaymath}
is the occupation field of the loop corresponding to the time interval $[b^{-},b^{+}]$. We need to check that
\begin{equation}
\label{Ch5Sec2: EqCoincSumLocTimes}
\text{a.s.}, \forall x>0,\ell^{x}_{T_{n,x_{n-1}}}\big(\xi_{\alpha, BM}^{(x_{n},x_{n})}\big)=
\sum_{[b^{-},b^{+}]\in\mathscr{B}_{n}}\ell^{x}_{b^{+}}\big(
\xi_{\alpha, BM}^{(x_{n},x_{n})}\big)-\ell^{x}_{b^{-}}\big(\xi_{\alpha, BM}^{(x_{n},x_{n})}\big).
\end{equation}
For $x>0$, consider the random set of times
\begin{equation}
\label{Ch5Sec2: EqResidualTimes}
\big\lbrace t\in[0,T_{n,x_{n-1}}]\vert\xi_{\alpha, BM}^{(x_{n},x_{n})}(t)=x \big\rbrace\setminus\bigcup_{[b^{-},b^{+}]\in\mathscr{B}_{n}}[b^{-},b^{+}].
\end{equation}
If $x$ is a minimum of a loop embedded in $\big(\xi_{\alpha, BM}^{(x_{n},x_{n})}(t)\big)_{0\leq t<T_{n,x_{n-1}}}$ or if $x\not\in(x_{n-1},x_{n})$ then the set 
\eqref{Ch5Sec2: EqResidualTimes} is empty. Otherwise, it is reduced to one point: the first hitting time of the level $x$. Almost surely, for all $x>0$, the measure $d_{t}\ell^{x}_{t}\big(\xi_{\alpha, BM}^{(x_{n},x_{n}})\big)$ is supported in 
$\big\lbrace t\geq 0\vert \xi_{\alpha, BM}^{(x_{n},x_{n})}(t)=x\big\rbrace$ and has no atoms, and thus does not charge the set \eqref{Ch5Sec2: EqResidualTimes}. This implies \eqref{Ch5Sec2: EqCoincSumLocTimes}. Finally, we can conclude that $\big(\ell^{x}_{T_{n,x_{n-1}}}\big(\xi_{\alpha, BM}^{(x_{n},x_{n})}\big)\big)_{x>0}$ is the occupation field of $\mathcal{L}_{\alpha, BM}\cap\lbrace \gamma\in \mathfrak{L}^{\ast}\vert \min \gamma \in (x_{n-1},x_{n})\rbrace$.

The occupation field of $\mathcal{L}_{\alpha, BM}\cap\lbrace \gamma\in \mathfrak{L}^{\ast}\vert \min \gamma >0\rbrace$ is
\begin{displaymath}
\bigg(\sum_{n\geq 0} \ell^{x}_{T_{n,x_{n-1}}}\big(\xi_{\alpha, BM}^{(x_{n},x_{n})}\big)\bigg)_{x>0}.
\end{displaymath}
The above sum is locally finite and thus varies continuously with $x$.
\end{proof}

\section{The case $\alpha=1$}
\label{Ch5Sec3}

According to Proposition \ref{Ch5Sec2: PropGeneralContour}, in case $\alpha=1$, the Poisson ensemble of loops $\mathcal{L}_{1, L}$ can be recovered from sample paths of one-dimensional diffusions. A similar property was observed for loops of the two-dimensional Brownian Motion and of Markov jump processes on graphs. In \cite{LeJan2011Loops}, Chapter $8$, it is shown that by launching consecutively symmetric Markov jump processes from different vertices of a finite graph and applying Wilson's algorithm (\cite{Wilson1996UST}), one can simultaneously construct a uniform spanning tree of the graph with prescribed weights on the edges and an independent Poisson ensemble of Markov loops of parameter $\alpha=1$. If $\mathbb{D}$ is a simply-connected open domain of $\mathbb{C}$ other than $\mathbb{C}$, it was shown in \cite{Zhan2012LEBMSLE2} that one can couple a Brownian motion on $\mathbb{D}$, killed at hitting $\partial\mathbb{D}$, and a simple curve (SLE$_{2}$) with same extremal points such that the latter appears as the loop-erasure of the first. It is conjectured that given this loop-erased Brownian motion and an independent  Poisson ensemble of Brownian loops of parameter $1$, by attaching to the simple curve the loops that cross it one reconstructs a Brownian sample path. See \cite{LawlerWerner2004ConformalLoopSoup}, Conjecture $1$, and \cite{LawlerSchrammWerner2003ConformalRestr}, Theorem $7.3$.
More recently, a similar property was proved for loops of the three-dimensional Brownian motion \cite{DaisukeSapozh2017LERW3D}.

In case of one-dimensional diffusions one can partially recover $\mathcal{L}_{1, L}$ from Markovian sample paths otherwise than slicing $\xi_{1, L}^{(x_{0})}$ in excursions. The next result has an analogue for loops of Markov jump processes on graphs. See \cite{LeJan2011Loops}, Remark $21$.

\begin{prop}
\label{Ch5Sec3: PropDirichletSlice}
Assume that $L$ is the generator of a transient diffusion. Let $x\in I$. Let $(X_{t})_{0\leq t<\zeta}$ be the sample path of a diffusion of generator $L$ started from $x$. Let $\widehat{T}_{x}$ the last time $X$ visits $x$. For $l\geq 0$, let
\begin{displaymath}
\tau_{l}^{x}:=\lbrace t\geq 0\vert \ell^{x}_{t}(X)>l\rbrace.
\end{displaymath}
Let $(q_{j})_{j\in\mathbb{N}}$ be a Poisson-Dirichlet partition $PD(0,1)$ of $[0,1]$, independent from $X$, ordered in an arbitrary way. Let
\begin{displaymath}
l_{j}:=\ell^{x}_{\zeta}(X)\sum_{i=0}^{j} q_{i}.
\end{displaymath}
The family of bridges $((X_{t})_{\tau_{l_{j-1}}^{x}\leq t\leq \tau_{l_{j}}^{x}})_{j\geq 0}$ has, up to unrooting, the same law as the loops in
\begin{displaymath}
\mathcal{L}_{1, L}\cap\lbrace\gamma\in\mathfrak{L}^{\ast}\vert x\in \gamma([0,T(\gamma)])\rbrace.
\end{displaymath}
In particular, $(X_{t})_{0\leq t\leq \widehat{T}_{x}}$ can be obtained through sticking together all the loops in $\mathcal{L}_{\alpha, L}$ that visit $x$.
\end{prop}

\begin{proof}
According to Corollary \ref{Ch3Sec3: CorLoopLocTimeBridge}, 
$(\ell^{x}(\gamma))_{\gamma\in\mathcal{L}_{\alpha, L},\gamma~\text{visits}~x}$ is a Poisson ensemble of intensity 
$e^{-\frac{l}{G(x,x)}}\frac{dl}{l}$. Thus, $\widehat{\mathcal{L}}^{x}_{\alpha, L}$ is an exponential r.v. with mean $G(x,x)$ and has the same law as $\ell^{x}_{\zeta}(X)$. Moreover, the Poisson ensemble 
$(\ell^{x}(\gamma))_{\gamma\in\mathcal{L}_{\alpha, L},\gamma~\text{visits}~x}$ has, up to reordering, the same law as $(l_{j}-l_{j-1})_{j\geq 0}$. Almost surely, $l\mapsto\tau^{x}_{l}$ does not jump at any $l_{j}$. Conditional on $(l_{j})_{j\geq 0}$, $((X_{t})_{\tau_{l_{j-1}}^{x}\leq t\leq \tau_{l_{j}}^{x}})_{j\geq 0}$ is an independent family of bridges and $(X_{t})_{\tau_{l_{j-1}}^{x}\leq t\leq \tau_{l_{j}}^{x}}$ has the same law as $(X_{t})_{0\leq t\leq \tau_{l_{j}-l_{j-1}}^{x}}$. We conclude using identity \eqref{Ch3Sec3: EqDecompLoopLocTimex} and the theory of marked Poisson ensembles.
\end{proof}

Assume that $L$ is the generator of a transient diffusion. Let $x\in I$ and 
let $(X_{t})_{0\leq t<\zeta}$ be a sample path starting from $x$ of the diffusion corresponding to $L$. We will describe two different ways to slice $(X_{t})_{0\leq t<\zeta}$ so as to obtain the loops
\begin{displaymath}
\mathcal{L}_{1,L}\cap\lbrace\gamma\in\mathfrak{L}^{\ast}\vert
\gamma([0,T(\gamma)])\cap[X(0),X(\zeta^{-})](\text{or}~[X(\zeta^{-}),X(0)])\neq\emptyset\rbrace.
\end{displaymath}
The first method corresponds to the "loop-erasure procedure" applied to $(X_{t})_{0\leq t<\zeta}$ and the second to the "loop-erasure procedure" applied to the time-reversed path $(X_{\zeta-t})_{0<t\leq\zeta}$. Let $\widehat{T}_{x}$ be the last time $(X_{t})_{0\leq t<\zeta}$ visits $x$. Let $\widetilde{T}$ be the first time $X$ hits 
$X_{\zeta^{-}}$. If $X_{\zeta^{-}}\in\partial I$, then $\widetilde{T}=\zeta$. Let $(q_{j})_{j\in\mathbb{N}}$ be a Poisson-Dirichlet partition $PD(0,1)$ of $[0,1]$, independent from $X$. The first method of decomposition is the following:
\begin{itemize}
\item The path $(X_{t})_{0\leq t\leq\widehat{T}_{x}}$ is decomposed into bridges $((X_{t})_{\tau_{l_{j-1}}^{x}\leq t\leq \tau_{l_{j}}^{x}})_{j\geq 0}$ from $x$ to $x$ by applying the Poisson-Dirichlet partition $(q_{j})_{j\in\mathbb{N}}$ to 
$\ell^{x}_{\zeta}(X)$, as described in Proposition \ref{Ch5Sec3: PropDirichletSlice}.
\item Given the path $(X_{\widehat{T}_{x}+t})_{0\leq t<\zeta-\widehat{T}_{x}}$, if $X_{\zeta^{-}}<x$, we define
\begin{multline*}
\mathfrak{b}^{+}:=\Big\lbrace t \in [0,\zeta-\widehat{T}_{x})\vert
X_{\widehat{T}_{x}+t}=\sup_{s\in[t,\zeta-\widehat{T}_{x})}X_{\widehat{T}_{x}+s}\\\text{and}~\exists
\varepsilon\in(0,t)~\text{s.t.}~\forall s\in(t-\varepsilon,t), X_{\widehat{T}_{x}+s}<X_{\widehat{T}_{x}+t} \Big\rbrace.
\end{multline*}
$\mathfrak{b}^{+}$ is countable, and we define on $\mathfrak{b}^{+}$ the map $b^{-}$:
\begin{displaymath}
b^{-}(t):=\sup\big\lbrace s\in[0,t)\vert X_{\widehat{T}_{x}+s}=X_{\hat{T}_{x}+t}\big\rbrace.
\end{displaymath}
$((X_{\widehat{T}_{x}+b^{-}(t)+s})_{0\leq s \leq t-b^{-}(t)})_{t\in \mathfrak{b}^{+}}$ is the family of negative excursions of the path 
$(X_{\widehat{T}_{x}+t})_{0\leq t<\zeta-\widehat{T}_{x}}$ below
$(\sup_{[\widehat{T}_{x}+t,\zeta)}X)_{0\leq t<\zeta-\widehat{T}_{x}}$. If $X_{\zeta^{-}}>x$, then
\begin{multline*}
\mathfrak{b}^{+}:=\Big\lbrace t\in[0,\zeta-\widehat{T}_{x})\vert
X_{\widehat{T}_{x}+t}=\inf_{s\in[t,\zeta
-\widehat{T}_{x})}X_{\widehat{T}_{x}+s}\\\text{and}~\exists
\varepsilon\in(0,t)~\text{s.t.}~\forall s\in(t-\varepsilon,t), X_{\widehat{T}_{x}+s}>X_{\widehat{T}_{x}+t}\Big\rbrace.
\end{multline*}
We define on $\mathfrak{b}^{+}$ the map $b^{-}$:
\begin{displaymath}
b^{-}(t):=\sup\big\lbrace s\in[0,t)\vert X_{\widehat{T}_{x}+s}=X_{\widehat{T}_{x}+t}\big\rbrace.
\end{displaymath}
$((X_{\widehat{T}_{x}+b^{-}(t)+s})_{0\leq s \leq t-b^{-}(t)})_{t\in \mathfrak{b}^{+}}$ are the positive excursions of $(X_{\widehat{T}_{x}+t})_{0\leq t<\zeta-\widehat{T}_{x}}$ above $(\inf_{[\widehat{T}_{x}+t,\zeta)}X)_{0\leq t<\zeta-\widehat{T}_{x}}$.
\item We denote $\mathscr{L}^{1}((X_{t})_{0\leq t<\zeta})$ the set of loops 
\begin{displaymath}
\big\lbrace
(X_{\tau_{l_{j-1}}^{x}+s})_{0\leq s\leq \tau_{l_{j}}^{x}-\tau_{l_{j-1}}^{x}}\vert j\geq 0\big\rbrace\cup
\big\lbrace(X_{\widehat{T}_{x}+b^{-}(t)+s})_{0\leq s \leq t-b^{-}(t)}\vert t\in\mathfrak{b}^{+}\big\rbrace,
\end{displaymath}
where the loops are considered to be unrooted.
\end{itemize}
The second method of decomposition is the following:
\begin{itemize}
\item If $X_{\zeta^{-}}<x$, we define
\begin{displaymath}
\mathfrak{b}^{-}:=\Big\lbrace t\in[0,\widetilde{T})\vert X_{t}=\inf_{[0,t]}X~\text{and}~\exists\varepsilon>0~\text{s.t.}~\forall s\in (t,t+\varepsilon), X_{s}>X_{t}\Big\rbrace.
\end{displaymath}
On $\mathfrak{b}^{-}$ we define the map $b^{+}$:
\begin{displaymath}
b^{+}(t):=\inf\lbrace s\in (t,\widetilde{T})\vert X_{s}=X_{t}\rbrace.
\end{displaymath}
$((X_{t+s})_{0\leq s\leq b^{+}(t)-t})_{t\in \mathfrak{b}^{-}}$ are the positive excursions of the path $(X_{t})_{0\leq t<\widetilde{T}}$ above 
$(\inf_{[0,t]}X)_{0\leq t\leq \widetilde{T}}$. This is exactly the decomposition described in the previous Section \ref{Ch5Sec2}. If $X_{\zeta^{-}}>x$, then
\begin{displaymath}
\mathfrak{b}^{-}:=\Big\lbrace t\in[0,\widetilde{T})\vert X_{t}=\sup_{[0,t]}X~\text{and}~\exists\varepsilon>0~\text{s.t.}~\forall s\in (t,t+\varepsilon), X_{s}<X_{t}\Big\rbrace.
\end{displaymath}
The map $b^{+}$ defined on $\mathfrak{b}^{-}$ is
\begin{displaymath}
b^{+}(t):=\inf\lbrace s\in (t,\widetilde{T})\vert X_{s}=X_{t}\rbrace.
\end{displaymath}
$((X_{t+s})_{0\leq s\leq b^{+}(t)-t})_{t\in \mathfrak{b}^{-}}$ are the negative excursions of the path $(X_{t})_{0\leq t<\widetilde{T}}$ below 
$(\sup_{[0,t]}X)_{0\leq t\leq \widetilde{T}}$.
\item If $\widetilde{T}<\zeta$, we introduce:
\begin{displaymath}
\tilde{l}_{j}:=\ell_{\zeta}^{X_{\zeta^{-}}}(X)\sum_{i=0}^{j}q_{i},
\end{displaymath}
and
\begin{displaymath}
\tau_{\tilde{l}_{j}}:=\inf\lbrace t\in[\widetilde{T},\zeta)\vert\ell_{t}^{X_{\zeta^{-}}}(X)>\tilde{l}_{j}\rbrace.
\end{displaymath}
We decompose the path $(X_{t})_{\widetilde{T}\leq t<\zeta}$ into bridges
$((X_{t})_{\tau_{\tilde{l}_{j-1}}\leq t\leq \tau_{\tilde{l}_{j}}})_{j\geq 0}$
from $X_{\zeta^{-}}$ to $X_{\zeta^{-}}$.
\item We denote $\mathscr{L}^{2}((X_{t})_{0\leq t<\zeta})$ the set of loops 
\begin{displaymath}
\big\lbrace(X_{t+s})_{0\leq s \leq b^{+}(t)-t}\vert t\in\mathfrak{b}^{-}\big\rbrace\cup\big\lbrace
(X_{\tau_{\tilde{l}_{j-1}}+s})_{0\leq s\leq \tau_{\tilde{l}_{j}}-\tau_{\tilde{l}_{j-1}}}\vert j\geq 0\big\rbrace,
\end{displaymath}
where the loops are considered to be unrooted.
\end{itemize}
The loops in $\mathscr{L}^{1}((X_{t})_{0\leq t<\zeta})$ and $\mathscr{L}^{2}((X_{t})_{0\leq t<\zeta})$ are not the same but follow the same law.

\begin{prop}
\label{Ch5Sec3: PropTwoDecomp}
$\mathscr{L}^{1}((X_{t})_{0\leq t<\zeta})$ and $\mathscr{L}^{2}((X_{t})_{0\leq t<\zeta})$, considered as collections of unrooted loops, have the same law. Let $\mathcal{L}_{1,L}$ be a Poisson ensemble of loops independent from $X_{\zeta^{-}}$. Then $\mathscr{L}^{1}((X_{t})_{0\leq t<\zeta})$ and
$\mathscr{L}^{2}((X_{t})_{0\leq t<\zeta})$ have the same law as
\begin{equation}
\label{Ch5Sec3: EqLoopsTwoPoints}
\mathcal{L}_{1,L}\cap\lbrace\gamma\in\mathfrak{L}^{\ast}\vert
\gamma([0,T(\gamma)])\cap[X(0),X(\zeta^{-})]~(\text{or}~[X(\zeta^{-}),X(0)])
\neq\emptyset\rbrace.
\end{equation}
\end{prop}

\begin{proof}
First, we will prove that $\mathscr{L}^{2}((X_{t})_{0\leq t<\zeta})$ has the same law as \eqref{Ch5Sec3: EqLoopsTwoPoints}. If $\mathbb{P}(X_{\zeta^{-}}=\inf I)>0$, then conditional on $X_{\zeta^{-}}=\inf I$, $(X_{t})_{0\leq t<\zeta}$ has the law of a sample path corresponding to the generator $\operatorname{Conj}(u_{\downarrow},L)$. If 
$y\in I\cap(-\infty,x]$ and $y$ is in the support of $\kappa$ (the killing measure in $L$), then, conditional on $X_{\zeta^{-}}=y$, 
$(X_{t})_{0\leq t<\zeta}$ is distributed according the measure $\frac{1}{G(x,y)}\mu_{L}^{x,y}$ (Property 
\ref{Ch3Sec2: PropertyBridges} (i)). According to Lemma 
\ref{Ch3Sec2: LemPathUntilKilling}, $(X_{t})_{0\leq t\leq\widetilde{T}}$ and 
$(X_{\widetilde{T}+t})_{0\leq t\leq\zeta-\tilde{T}}$ are independent conditional on $X_{\zeta^{-}}=y$, $(X_{t})_{0\leq t\leq\widetilde{T}}$ having the law of a sample path corresponding to the generator $\operatorname{Conj}(u_{\downarrow},L)$, run until hitting $y$, and
$(X_{\widetilde{T}+t})_{0\leq t\leq\zeta-\tilde{T}}$ following the law $\frac{1}{G(y,y)}\mu_{L}^{y,y}$. From Proposition \ref{Ch5Sec2: PropGeneralContour} and 
\ref{Ch5Sec3: PropDirichletSlice} follows that $\mathscr{L}^{2}((X_{t})_{0\leq t<\zeta})$ and \eqref{Ch5Sec3: EqLoopsTwoPoints} have the same law on the event 
$X_{\zeta^{-}}\leq x$. Symmetrically, this also true on the event $X_{\zeta^{-}}\geq x$.

The decomposition $\mathscr{L}^{1}((X_{t})_{0\leq t<\zeta})$ is obtained by first applying the decomposition $\mathscr{L}^{2}$ to the time-reversed path 
$(X_{\zeta-t})_{0<t\leq\zeta}$ and then applying again the time-reversal to the obtained loops. The law of the loops in \eqref{Ch5Sec3: EqLoopsTwoPoints} is invariant by time-reversal. Let $y\in I$, $y$ in the support of $\kappa$. Conditional on $X_{\zeta^{-}}=y$, the law of $(X_{\zeta-t})_{0<t\leq\zeta}$ is 
$\frac{1}{G(x,y)}\mu^{y,x}$. So, applying the decomposition $\mathscr{L}^{2}$ to the path $(X_{\zeta-t})_{0<t\leq\zeta}$ conditioned on $X_{\zeta^{-}}=y$ gives
\begin{displaymath}
\mathcal{L}_{1,L}\cap\lbrace\gamma\in\mathfrak{L}^{\ast}\vert
\gamma([0,T(\gamma)])\cap[y,x]~(\text{or}~[x,y])\neq\emptyset\rbrace.
\end{displaymath}
If $\mathbb{P}(X_{\zeta^{-}}=\inf I)>0$, then, conditional on $X_{\zeta^{-}}=\inf I$, the path $(X_{t})_{0\leq t<\zeta}$ is a limit as $y\rightarrow \inf I$ of paths following the law $\frac{1}{G(x,y)}\mu^{x,y}$ (i.e. the latter are restrictions of the former). Thus, conditional on $X_{\zeta^{-}}=\inf I$, $\mathscr{L}^{1}((X_{t})_{0\leq t<\zeta})$ is an increasing limit as $y\rightarrow \inf I$ of 
\begin{displaymath}
\mathcal{L}_{1,L}\cap\lbrace\gamma\in\mathfrak{L}^{\ast}\vert
\gamma([0,T(\gamma)])\cap[y,x]\neq\emptyset\rbrace,
\end{displaymath}
which is 
\begin{displaymath}
\mathcal{L}_{1,L}\cap\lbrace\gamma\in\mathfrak{L}^{\ast}\vert
\gamma([0,T(\gamma)])\cap[\inf I,x]\neq\emptyset\rbrace.
\end{displaymath}
Similar is true conditional on $X_{\zeta^{-}}=\sup I$.
\end{proof}

\chapter{Wilson's algorithm in dimension one}
\label{Ch6}

\section{Description of the algorithm}
\label{Ch6Sec1}

Given a finite undirected connected graph $\mathbb{G}=(V,E)$ and $C$ a positive weight function on its edges, which we interpret as an electrical network with conductance $C$, a uniform spanning tree of the weighted graph $\mathbb{G}$ is a random spanning tree with the occurrence probability of a spanning tree $\mathcal{T}$ proportional to 
\begin{displaymath}
\prod_{e~\text{edge of}~\mathcal{T}}C(e).
\end{displaymath}
The edges belonging to the uniform spanning tree are a determinantal point process (transfer current theorem). In \cite{Wilson1996UST} Wilson showed how to sample a uniform spanning tree using successive random walks to nearest neighbors, with transition probabilities proportional to $C$, starting from different vertices, and erasing the loops created by these random walks. The edges left after loop-erasure form a 
uniform spanning tree. This is known as Wilson's algorithm. See \cite{BenjaminiLyonsPeresSchramm2001UnifSpanFor, Lawler2018Topics} for a review. In \cite{LeJan2011Loops}, chapter $8$, Le Jan shows that the loops erased during the execution of Wilson's algorithm are related to the Poisson ensemble of Markov loops of intensity parameter $\alpha=1$. 

In \cite{LeJan2011Loops}, Chapter $10$, Le Jan suggested that Wilson's algorithm can be adapted to the situation where the random walk on a graph is replaced by a transient diffusion on a subinterval $I$ of $\mathbb{R}$. In this section we will describe the algorithm in the latter setting. The algorithm returns on one hand a sequence of one-dimensional paths which can be decomposed into a Poisson ensemble of Markov loops of parameter $1$ (section \ref{Ch6Sec2}), and on the other hand a pair of interwoven determinantal point processes on $I$, which may be interpreted as some kind of 
uniform spanning tree. In Section \ref{Ch6Sec3} we will derive the law of this pair of determinantal point processes in the setting where the underlying is a Brownian motion on $\mathbb{R}$ with a killing measure. In Section \ref{Ch6Sec4} we will give without proof the law in general case as it follows directly from the Brownian case.

Let $I$ be a subinterval of $\mathbb{R}$ and $L$ a generator of a transient diffusion on $I$ of form \ref{Ch2Sec2: EqGenerator}. Let $\kappa$ be the killing measure in $L$, which may be zero. Let $(x_{n})_{n\geq 1}$ be a sequence of pairwise distinct points in $I$ which is dense in $I$. Let 
$\left(\big(X^{(x_{n})}_{t}\big)_{0\leq t<\zeta_{n}}\right)_{n\geq 1}$ be a sequence of independent sample paths of the diffusion of generator $L$, with starting points
$X^{(x_{n})}_{0}=x_{n}$. In the first step of  Wilson's algorithm we will recursively define sequences $(T_{n})_{n\geq 1}$, $(\mathcal{Y}_{n})_{n\geq 1}$ and $(\mathcal{J})_{n\geq 1}$ where $T_{n}$ is a killing time for $X^{(x_{n})}$, $\mathcal{Y}_{n}$ is a finite subset of $\operatorname{Supp}(\kappa)\cup\partial I$ and $\mathcal{J}_{n}$ is a finite set of disjoint compact subintervals of $\bar{I}$, some of which may be reduced to one point:
\begin{itemize}
\item $T_{1}:=\zeta_{1}$, 
$\mathcal{Y}_{1}:=\big\lbrace X^{(x_{1})}_{T_{1}^{-}}\big\rbrace$,
$\mathcal{J}_{1}:=\big\lbrace\big[x_{1},X^{(x_{1})}_{T_{1}^{-}}\big]\big\rbrace$
\big(or $\big\lbrace\big[B^{(x_{1})}_{T_{1}^{-}},x_{1}\big]\big\rbrace$\big).
\item Assume that $\mathcal{Y}_{n}$ and $\mathcal{J}_{n}$ are constructed. 
If $x_{n+1}\in\bigcup_{J\subseteq\mathcal{J}_{n}}J$, then we set $T_{n+1}:=0$, ${\mathcal{Y}_{n+1}:=\mathcal{Y}_{n}}$ and $\mathcal{J}_{n+1}:=\mathcal{J}_{n}$.
If $x_{n+1}\not\in\bigcup_{J\subseteq\mathcal{J}_{n}}J$, then we define
\begin{displaymath}
T_{n+1}:=\min\Big(\zeta_{n},\inf\Big\lbrace t\geq 0\vert X^{(x_{n+1})}_{t}\in\bigcup_{J\subseteq\mathcal{J}_{n}}J
\Big\rbrace\Big). 
\end{displaymath}
If $X^{(x_{n+1})}_{T_{n+1}^{-}}\in\bigcup_{J\subseteq\mathcal{J}_{n}}J$, then
there is a unique $J\in\mathcal{J}_{n}$ such that $X^{(x_{n+1})}_{T_{n+1}^{-}}\in J$. In this case, we set $\mathcal{Y}_{n+1}:=\mathcal{Y}_{n}$ and
\begin{multline*}
\mathcal{J}_{n+1}:=(\mathcal{J}_{n}\setminus\lbrace J\rbrace)\cup\big\lbrace J\cup\big[x_{n+1},X^{(x_{n+1})}_{T_{n+1}^{-}}\big]\big\rbrace
\\\Big(or~(\mathcal{J}_{n}\setminus\lbrace J\rbrace)\cup\big\lbrace J\cup\big[X^{(x_{n+1})}_{T_{n+1}^{-}},x_{n+1}\big]\big\rbrace\Big).
\end{multline*}
If $X^{(x_{n+1})}_{T_{n+1}^{-}}\not\in\bigcup_{J\subseteq\mathcal{J}_{n}}J$, then we set $\mathcal{Y}_{n+1}:=\mathcal{Y}_{n}\cup\big\lbrace X^{(x_{n+1})}_{T_{n+1}^{-}}\big\rbrace$ and
\begin{displaymath}
\mathcal{J}_{n+1}:=\mathcal{J}_{n}\cup\big\lbrace \big[x_{n+1},X^{(x_{n+1})}_{T_{n+1}^{-}}\big]\big\rbrace
~\Big(~\text{or}~\mathcal{J}_{n}\cup\big\lbrace \big[X^{(x_{n+1})}_{T_{n+1}^{-}},x_{n+1}\big]\big\rbrace\Big).
\end{displaymath}
\end{itemize}

It is immediate to check by induction the following facts: 
\begin{itemize}
\item $\mathcal{Y}_{n}\subseteq \operatorname{Supp}(\kappa)\cup \partial I$. More precisely,
$\mathcal{Y}_{n}\subseteq \operatorname{Supp}(\kappa)\cup \big\lbrace y\in\partial I\vert 
\mathbb{P}\big(X^{(x_{n})}_{\zeta_{n}^{-}}=y\big)>0\big\rbrace$.
\item The intervals in $\mathcal{J}_{n}$ are pairwise disjoint.
\item  $\# \mathcal{Y}_{n} = \# \mathcal{J}_{n}\leq n$.
\item For every $y\in\mathcal{Y}_{n}$, there is one single $J\in\mathcal{J}_{n}$ such that $y\in J$.
\item $\mathcal{Y}_{n}\subseteq \mathcal{Y}_{n+1}$.
\item If $n\leq n'$, then for every $J\in\mathcal{J}_{n}$ there is one single $J'\in \mathcal{J}_{n'}$ such that $J\subseteq J'$. We denote $\imath_{n,n'}$ the corresponding map from $\mathcal{J}_{n}$ to $\mathcal{J}_{n'}$. The map $\imath_{n,n'}$ is injective. Trivially, for $n\leq n'\leq n''$, $\imath_{n,n''}=\imath_{n',n''}\circ\imath_{n,n'}$
\item For any $J\in \mathcal{J}_{n}$, $\partial J\subseteq \mathcal{Y}_{n}\cup \lbrace x_{1},\dots,x_{n}\rbrace$.
\end{itemize}

In the second step of Wilson's algorithm we will take the limit of the sequence $((\mathcal{Y}_{n},\mathcal{J}_{n}))_{n\geq 1}$ and define 
$(\mathcal{Y}_{\infty},\mathcal{J}_{\infty})$ as follows:
\begin{displaymath}
\mathcal{Y}_{\infty}:=\bigcup_{n\geq 1}\mathcal{Y}_{n},
\qquad
\mathcal{J}_{\infty}:=\bigcup_{n\geq 1}\bigcup_{J\in \mathcal{J}_{n}}\bigg\lbrace \bigcup_{n'\geq n}\imath_{n,n'}(J)\bigg\rbrace.
\end{displaymath}
$\mathcal{Y}_{\infty}$ is a finite or countable subset of $\operatorname{Supp}(\kappa)\cup\partial I$. $\mathcal{J}_{\infty}$ is a finite of countable set of disjoint subintervals of $\bar{I}$, but these subintervals are not necessarily closed or bounded. For any 
$y\in\mathcal{Y}_{\infty}$, there is a single $J\in \mathcal{J}_{\infty}$ such that 
$y\in J$, and this induces a bijection between $\mathcal{Y}_{\infty}$ and $\mathcal{J}_{\infty}$. For any $J\in \mathcal{J}_{n}$, there is a single $J'\in \mathcal{J}_{\infty}$ such that $J\subseteq J'$. We define $\imath_{n,\infty}(J)=J'$. $\imath_{n,\infty}$ is injective. Trivially, for $n\leq n'$, $\imath_{n,\infty}=\imath_{n',\infty}\circ\imath_{n,n'}$.
We will sometimes write $\mathcal{Y}_{n}(x_{1},\dots,x_{n})$, $\mathcal{J}_{n}(x_{1},\dots,x_{n})$, $\mathcal{Y}_{\infty}((x_{n})_{n\geq 1})$ and $\mathcal{J}_{\infty}((x_{n})_{n\geq 1})$ in order to emphasize the dependence on the starting points $(x_{n})_{n\geq 1}$. In Sections \ref{Ch6Sec3} and \ref{Ch6Sec4} we will see that:
\begin{itemize}
\item The set $\mathcal{Y}_{\infty}$ is a.s. discrete.
\item A.s., for any intervals $J\in\mathcal{J}_{\infty}$, $J\setminus \partial I$ is open.
\item The subset $I\setminus \bigcup_{J\in \mathcal{J}_{\infty}}J$ is a.s. discrete.
\item The law of $(\mathcal{Y}_{\infty}, \mathcal{J}_{\infty})$ does not depend on the choice of starting points $(x_{n})_{n\geq 1}$.
\end{itemize}
We introduce $\mathcal{Z}_{\infty}:=I\setminus\Big(
\bigcup_{J\in\mathcal{J}_{\infty}}J\Big)$.
We will further see that $\mathcal{Y}_{\infty}$ and $\mathcal{Z}_{\infty}$ are determinantal point processes.

The couple $(\mathcal{Y}_{\infty}, \mathcal{J}_{\infty})$ may be interpreted as a spanning tree. Consider the following undirected "graph": Its set of "vertices" is $\bar{I}\cup\lbrace \dagger\rbrace$, where $\dagger$ is a cemetery point outside of $\bar{I}$. Ever point $x\in I$ is connected by an "edge" to its two infinitesimal neighbors $x-dx$ and $x+dx$. Every point in $\operatorname{Supp}(\kappa)$ is connected by an "edge" to $\dagger$. Finally, any point in $y\in\partial I$ such that $\mathbb{P}\big(X^{(x_{n})}_{\zeta_{n}^{-}}=y\big)>0$ is connected by an "edge" to $\dagger$. On this "graph" $(\mathcal{Y}_{\infty}, \mathcal{J}_{\infty})$ induces the following "spanning tree": 
Each point in $\bigcup_{J\in\mathcal{J}_{\infty}}J$ is connected to its infinitesimal neighbors in $I$ and $\mathcal{Z}_{\infty}$ represents "edges" on $I$ that are missing. Moreover, every point in $\mathcal{Y}_{\infty}$ is connected to $\dagger$.

There are two trivial cases in which $(\mathcal{Y}_{\infty}, \mathcal{J}_{\infty})$ is deterministic. In the first one $\kappa =0$ and $I$ has one single regular or exit boundary point $y$ characterized by $\mathbb{P}\big(X^{(x_{n})}_{\zeta_{n}^{-}}=y\big)>0$
(see \cite{Breiman1992Probability}, Chapter $16$, for the characterization of boundaries). Then, $\mathcal{Y}_{\infty}$ is made of this boundary point and $\mathcal{J}_{\infty}$ contains one single interval $I\cup \mathcal{Y}_{\infty}$. $\mathcal{Z}_{\infty}$ is empty. In the second case $I$ does not have regular or exit boundaries and $\kappa$ is proportional to a Dirac measure $c\delta_{y_{0}}$. Then, $\mathcal{Y}_{\infty}=\lbrace y_{0}\rbrace$ and $\mathcal{J}_{\infty}=\lbrace I\rbrace$. $\mathcal{Z}_{\infty}$ is again empty. In all other situation $\mathcal{Z}_{\infty}$ is non-empty and random. See Figure \ref{FigW1} for an illustration of $(\mathcal{Y}_{n}, \mathcal{J}_{n})$ for $1\leq n\leq 5$ and
Figure \ref{FigW2} for an illustration of $(\mathcal{Y}_{\infty}, \mathcal{Z}_{\infty})$.

\begin{figure}[H]
\centering{
\includegraphics[width=1\textwidth]{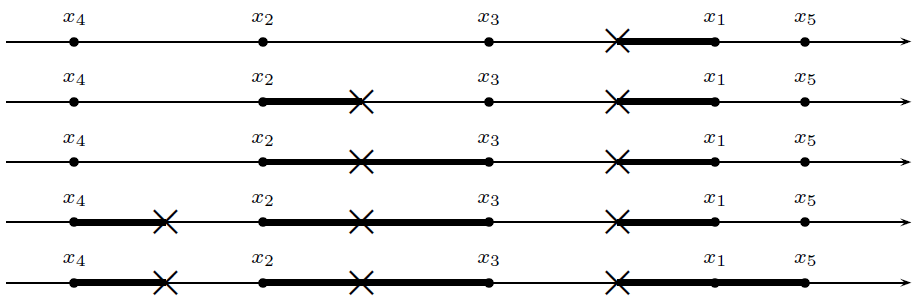}
}
\caption{Illustration of $((\mathcal{Y}_{n},\mathcal{J}_{n}))_{1\leq n\leq 5}$: x-dots represent the points of $\mathcal{Y}_{n}$
and thick lines the intervals in $\mathcal{J}_{n}$.}
\label{FigW1}
\end{figure}

\begin{figure}[H]
\centering{
\includegraphics[width=1\textwidth]{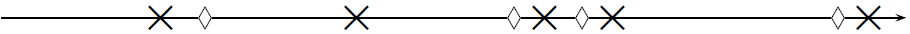}
}
\caption{Illustration of $(\mathcal{Y}_{\infty},\mathcal{J}_{\infty})$: x-dots represent the points of $\mathcal{Y}_{\infty}$ 
and diamonds the points of $\mathcal{Z}_{\infty}$.}
\label{FigW2}
\end{figure}

\section{The erased paths}
\label{Ch6Sec2}

During the execution of Wilson's algorithm we used the paths $\left(\big(X^{(x_{n})}_{t}\big)_{0\leq t<T_{n}}\right)_{n\geq 1}$. These paths can be further decomposed using the procedure described in Section \ref{Ch5Sec3}.

\begin{prop}
\label{Ch6Sec2: PropErasedLoops}
The family of unrooted loops
\begin{displaymath}
\bigcup_{n\geq 1} \mathscr{L}^{1}\left(\big(X^{(x_{n})}_{t}\big)_{0\leq t<T_{n}}\right)
\end{displaymath}
has the same law as the Poisson ensemble $\mathcal{L}_{1,L}$. Moreover it is independent from $(\mathcal{Y}_{\infty},\mathcal{J}_{\infty})$.
\end{prop}

\begin{proof}
Let $\mathcal{L}_{1,L}$ be a Poisson ensemble of loops independent from the family of paths $\left(\big(X^{(x_{n})}_{t}\big)_{0\leq t<\zeta_{n}}\right)_{n\geq 1}$. Using Proposition \ref{Ch5Sec3: PropTwoDecomp} and induction is it immediate to show that
the triple
\begin{displaymath}
\left(\mathcal{Y}_{n},\mathcal{J}_{n},\bigcup_{j=1}^{n} \mathscr{L}^{1}\left(\big(X^{(x_{j})}_{t}\big)_{0\leq t<T_{j}}\right)\right)
\end{displaymath}
has the same law as
\begin{displaymath}
\left(\mathcal{Y}_{n},\mathcal{J}_{n},
\Big\lbrace(\gamma(t))_{0\leq t\leq T(\gamma)}\in\mathcal{L}_{1,L}\vert
\gamma([0,T(\gamma)])\cap\bigcup_{J\in\mathcal{J}_{n}}J\neq\emptyset\Big\rbrace\right).
\end{displaymath}
Since $(\mathcal{Y}_{\infty},\mathcal{J}_{\infty})$ is by construction independent from
$\left(\big(X^{(x_{j})}_{t}\big)_{0\leq t<T_{j}}\right)_{1\leq j\leq n}$ conditional
on $(\mathcal{Y}_{n},\mathcal{J}_{n})$, we further get that the triple
\begin{displaymath}
\left(\mathcal{Y}_{\infty},\mathcal{J}_{\infty},\bigcup_{j=1}^{n} \mathscr{L}^{1}\left(\big(X^{(x_{j})}_{t}\big)_{0\leq t<T_{j}}\right)\right)
\end{displaymath}
has the same law as
\begin{displaymath}
\left(\mathcal{Y}_{\infty},\mathcal{J}_{\infty},
\Big\lbrace(\gamma(t))_{0\leq t\leq T(\gamma)}\in\mathcal{L}_{1,L}\vert
\gamma([0,T(\gamma)])\cap\bigcup_{J\in\mathcal{J}_{n}}J\neq\emptyset\Big\rbrace\right).
\end{displaymath}
Taking the limit of the third component as $n$ tends to infinity we get that
\begin{displaymath}
\left(\mathcal{Y}_{\infty},\mathcal{J}_{\infty},\bigcup_{j\geq 1} \mathscr{L}^{1}\left(\big(X^{(x_{j})}_{t}\big)_{0\leq t<T_{j}}\right)\right)
\end{displaymath}
has the same law as
\begin{displaymath}
\left(\mathcal{Y}_{\infty},\mathcal{J}_{\infty},
\Big\lbrace(\gamma(t))_{0\leq t\leq T(\gamma)}\in\mathcal{L}_{1,L}\vert
\gamma([0,T(\gamma)])\cap\bigcup_{J\in\mathcal{J}_{\infty}}J\neq\emptyset\Big\rbrace\right).
\end{displaymath}
To conclude we need only to show that almost surely
\begin{displaymath}
\Big\lbrace(\gamma(t))_{0\leq t\leq T(\gamma)}\in\mathcal{L}_{1,L}\vert
\gamma([0,T(\gamma)])\cap\bigcup_{J\in\mathcal{J}_{\infty}}J\neq\emptyset\Big\rbrace = 
\mathcal{L}_{1,L}.
\end{displaymath}
The latter is equivalent to $\bigcup_{J\in\mathcal{J}_{\infty}}J$ being dense in $I$, which will be proved in the next section.
\end{proof}

\section{Determinantal point processes $(\mathcal{Y}_{\infty},\mathcal{Z}_{\infty})$: Brownian case}
\label{Ch6Sec3}

In this section we will describe $(\mathcal{Y}_{\infty},\mathcal{J}_{\infty})$
in the Brownian case by giving the joint law of the point processes
$\mathcal{Y}_{\infty}$ and $\mathcal{Z}_{\infty}$. First, we will study the case of a Brownian motion on a bounded interval $(a,b)$, killed upon hitting $a$ or $b$, and without killing measure. Then, we will study the case of the Brownian motion on $\mathbb{R}$ with a non-zero Radon killing measure $\kappa$. We will write $\big(B^{(x_{n})}_{t}\big)_{0\leq t<\zeta_{n}}$ instead of $\big(X^{(x_{n})}_{t}\big)_{0\leq t<\zeta_{n}}$.

\begin{prop}
\label{Ch6Sec3: PropUnifSep}
In the case of a Brownian motion on a bounded interval $(a,b)$, killed upon hitting $a$ or $b$, and without killing measure, $\mathcal{Y}_{\infty}$ is deterministic and equals $\lbrace a,b\rbrace$ and $\mathcal{Z}_{\infty}$ is made of a single point distributed uniformly on $(a,b)$.
\end{prop}

\begin{proof}
For $n\geq 1$, we define $\tilde{x}_{n,0}<\tilde{x}_{n,1}<\dots <\tilde{x}_{n,n+1}$ as the family $x_{1},\dots,x_{n},a,b$ ordered increasingly. According to this definition, $\tilde{x}_{n,0}=a$ and $\tilde{x}_{n,n+1}=b$. As a convention we denote 
$\tilde{x}_{0,0}:=a$ and $\tilde{x}_{0,1}:=b$. For $n\geq 2$, one of the following situations may occur:
\begin{itemize}
\item $\mathcal{Y}_{n}=\lbrace b\rbrace$ and $\mathcal{J}_{n}=\lbrace[\tilde{x}_{n,1},b]\rbrace$,
\item $\mathcal{Y}_{n}=\lbrace a\rbrace$ and $\mathcal{J}_{n}=\lbrace[a,\tilde{x}_{n,n}]\rbrace$,
\item $\mathcal{Y}_{n}=\lbrace a,b\rbrace$ and for some $j\in \lbrace2,\dots,n\rbrace$,
$\mathcal{J}_{n}=\lbrace[a,\tilde{x}_{n,j-1}],[\tilde{x}_{n,j},b]\rbrace$.
\end{itemize}
In any case, $(a,b)\setminus\Big(\bigcup_{J\in\mathcal{J}_{n}}J\Big)$ is an interval of form $(\tilde{x}_{n,j-1},\tilde{x}_{n,j})$. 

We set $\mathcal\lbrace J\rbrace_{0}=\emptyset$. Let $n\geq 1$. There is a $j\in \lbrace 1,\dots, n\rbrace$ such that $x_{n}\in(\tilde{x}_{n-1,j-1},\tilde{x}_{n-1,j})$. Conditional on 
$(a,b)\setminus\Big(\bigcup_{J\in\mathcal{J}_{n-1}}J\Big)=(\tilde{x}_{n-1,j-1},\tilde{x}_{n-1,j})$, the point $B^{(x_{n})}_{T_{n}^{-}}$ equals $\tilde{x}_{n-1,j-1}$ with probability $\frac{\tilde{x}_{n-1,j}-x_{n}}{\tilde{x}_{n-1,j})-\tilde{x}_{n-1,j-1}}$
and $\tilde{x}_{n-1,j}$ with probability
$\frac{x_{n}-\tilde{x}_{n-1,j-1}}{\tilde{x}_{n-1,j})-\tilde{x}_{n-1,j-1}}$.
By induction we get that
\begin{displaymath}
\mathbb{P}\left((a,b)\setminus\Big(\bigcup_{J\in\mathcal{J}_{n}}J\Big)
=(\tilde{x}_{n,j-1},\tilde{x}_{n,j})\right)=
\dfrac{\tilde{x}_{n,j}-\tilde{x}_{n,j-1}}{b-a}.
\end{displaymath}
Hence,
\begin{displaymath}
\mathbb{P}(\mathcal{Y}_{\infty}=\lbrace a\rbrace)\leq \lim_{n\rightarrow +\infty}
\mathbb{P}\left((a,b)\setminus\Big(\bigcup_{J\in\mathcal{J}_{n}}J\Big)
=(\tilde{x}_{n,0},\tilde{x}_{n,1})\right)
=\lim_{n\rightarrow +\infty} \dfrac{\tilde{x}_{n,1}-\tilde{x}_{n,0}}{b-a} = 0,
\end{displaymath}
and similarly $\mathbb{P}(\mathcal{Y}_{\infty}=\lbrace b\rbrace)=0$. Thus,
$\mathcal{Y}_{\infty}=\lbrace a,b\rbrace$. Almost surely for $n$ large enough, 
$\mathcal{J}_{n}$ will be of form $\lbrace[a,\tilde{x}_{n,j-1}],[\tilde{x}_{n,j},b]\rbrace$ for a random $j\in \lbrace 2,\dots,n\rbrace$.  We denote by $p_{n,1}^{+}$, respectively $p_{n,2}^{-}$, the random values of $\tilde{x}_{n,j-1}$, respectively $\tilde{x}_{n,j}$. Almost surely, neither of the non-decreasing sequence
$(p_{n,1}^{+})_{n}$ or non-increasing sequence of $(p_{n,2}^{-})_{n}$ is stationary. This fact follows from the same argument according to which $\mathcal{Y}_{\infty}$ is not reduced to one point. Moreover, $p_{n,2}^{-}-p_{n,1}^{+}$, bounded by 
$\sup_{2\leq j\leq n}(\tilde{x}_{n,j}-\tilde{x}_{n,j-1})$, converges to $0$. It follows that a.s., $\mathcal{Z}_{\infty}$ is reduced to one point, the common limit of $p_{n,1}^{+}$ and $p_{n,2}^{-}$. Finally, if $\tilde{a}<\tilde{b}$ are two values taken by the sequence $(x_{n})_{n\geq 1}$, then
\begin{displaymath}
\mathbb{P}(\mathcal{Z}_{\infty}\subseteq (\tilde{a},\tilde{b}))
=\dfrac{\tilde{b}-\tilde{a}}{b-a}.
\end{displaymath} It follows that the unique point in $\mathcal{Z}_{\infty}$ is distributed uniformly on $(a,b)$.
\end{proof}

We consider now the case of the Brownian motion on $\mathbb{R}$ with a non-zero Radon killing measure $\kappa$. 
$G(x,y)=u_{\uparrow}(x\wedge y)u_{\downarrow}(x\vee y)$ will be the Green's function of 
$\frac{1}{2}\frac{d^{2}}{dx^{2}}-\kappa$. The law of $(\mathcal{Y}_{n},\mathcal{J}_{n})$ may be expressed explicitly. Let $Q_{n}$ be the cardinal of $\mathcal{Y}_{n}$. Let $Y_{n,1}, Y_{n,2},\dots$,$Y_{n,Q(n)}$ be the points in $\mathcal{Y}_{n}$ ordered in the increasing sense. Denote by $[p^{-}_{n,1},p^{+}_{n,1}], [p^{-}_{n,2},p^{+}_{n,2}],\dots, [p^{-}_{n,Q_{n}},p^{+}_{n,Q_{n}}]$ the intervals in $\mathcal{J}_{n}$ ordered in the increasing sense. For all $q\in \lbrace 1,\dots, Q_{n}\rbrace$, $Y_{n,q}\in [p^{-}_{n,q},p^{+}_{n,q}]$. It happens with positive probability that for some $q$, $p^{-}_{n,q}=p^{+}_{n,q}$, if one of the starting points $x_{1},\dots,x_{n}$ is an atom of $\kappa$. To compute recursively the joint law of above random variables we use the following facts: Given a killed Brownian path $\big(B^{(x)}_{t}\big)_{0\leq t<\zeta}$ starting from $x$, the distribution of $B^{(x)}_{\zeta^{-}}$ is $G(x,y)\kappa(dy)$ (see Section \ref{Ch2Sec2}).
Given $a< x$, let $T_{a}$ be the first time $B^{(x)}$ hits $a$. Then,
\begin{displaymath}
\mathbb{P}_{x}(T_{a}\leq\zeta)=\dfrac{u_{\downarrow}(x)}{u_{\downarrow}(a)}
=\dfrac{G(x,a)}{G(a,a)}.
\end{displaymath}
On the event $T_{a}>\zeta$, the distribution of $B^{(x)}_{\zeta^{-}}$ is
\begin{displaymath}
(G(x,y)-\mathbb{P}_{x}(T_{a}\leq\zeta)G(a,y))1_{y>a}\kappa(dy)=
\bigg(G(x,y)-\dfrac{G(x,a)G(a,y)}{G(a,a)}\bigg)1_{y>a}\kappa(dy).
\end{displaymath}
More generally, if $a<x<b$ and $\tilde{\zeta}$ is the first time $B^{(x)}$ gets either killed by the killing measure $\kappa$ or hits $a$ or $b$, then
\begin{itemize}
\item The probability that $B^{(x)}_{\tilde{\zeta}^{-}}=a$ is:
\begin{displaymath}
\dfrac{u_{\downarrow}(a)u_{\uparrow}(x)-u_{\downarrow}(x)u_{\uparrow}(a)}
{u_{\downarrow}(a)u_{\uparrow}(b)-u_{\downarrow}(b)u_{\uparrow}(a)}
=\dfrac{
\det
\left(
\begin{array}{cc}
G(x,b) & G(a,b) \\ 
G(a,x) & G(a,a)
\end{array} 
\right) 
}
{
\det
\left(
\begin{array}{cc}
G(b,b) & G(a,b) \\ 
G(a,b) & G(a,a)
\end{array} 
\right) 
}.
\end{displaymath}
\item The probability that $B^{(x)}_{\tilde{\zeta}^{-}}=b$ is:
\begin{displaymath}
\dfrac{u_{\downarrow}(x)u_{\uparrow}(b)-u_{\downarrow}(b)u_{\uparrow}(x)}
{u_{\downarrow}(a)u_{\uparrow}(b)-u_{\downarrow}(b)u_{\uparrow}(a)}
=\dfrac{
\det
\left(
\begin{array}{cc}
G(a,x) & G(a,b) \\ 
G(x,b) & G(b,b)
\end{array} 
\right) 
}
{
\det
\left(
\begin{array}{cc}
G(a,a) & G(a,b) \\ 
G(a,b) & G(b,b)
\end{array} 
\right) 
}.
\end{displaymath}
\item The distribution of $B^{(x)}_{\tilde{\zeta}^{-}}$ on $(a,b)$ is:
\begin{displaymath}
\dfrac{\det
\left( 
\begin{array}{ccc}
G(x,y) & G(a,y) & G(y,b) \\ 
G(a,x) & G(a,a) & G(a,b) \\ 
G(x,b) & G(a,b) & G(b,b)
\end{array} 
\right) 
}
{\det
\left(
\begin{array}{cc}
G(a,a) & G(a,b) \\ 
G(a,b) & G(b,b)
\end{array} 
\right) 
}
1_{a<y<b}\kappa(dy).
\end{displaymath}
\end{itemize}

Above expressions give the law of $(\mathcal{Y}_{1}, \mathcal{J}_{1})$ and the law of $(\mathcal{Y}_{n+1}, \mathcal{J}_{n+1})$ conditional on $(\mathcal{Y}_{n}, \mathcal{J}_{n})$. By induction, one can derive the law of $(\mathcal{Y}_{n}, \mathcal{J}_{n})$. We will express it using a single identity involving a determinant. However, this single identity may correspond to different configurations: We will divide the set of indices $\lbrace 1,\dots,Q_{n}\rbrace$ in three categories $E^{-}_{n}$, $E^{+}_{n}$ and $E^{-,+}_{n}$ where for $q\in E^{-}_{n}$, $Y_{n,q}=p^{-}_{n,q}$, for $q\in E^{+}_{n}$, $Y_{n,q}=p^{+}_{n,q}$ and for $q\in E^{-,+}_{n}$, $p^{-}_{n,q}<Y_{n,q}<p^{+}_{n,q}$. For instance, on Figure \ref{FigW1}, $Q_{5}=3$, $E^{-}_{5}=\lbrace 3\rbrace$, $E^{+}_{5}=\lbrace 1\rbrace$ and $E^{-,+}_{5}=\lbrace 2\rbrace$.

\begin{prop}
\label{Ch6Sec3: PropAfternSteps}
Let $q\in\lbrace 1,\dots,n\rbrace$. Let $(E_{n}^{-}, E_{n}^{+}, E_{n}^{-,+})$ be a partition of $\lbrace 1,\dots,q\rbrace$:
\begin{displaymath}
\lbrace 1,\dots,q\rbrace = E_{n}^{-}\amalg E_{n}^{+}\amalg E_{n}^{-,+}.
\end{displaymath}
Let $x^{-}$ be an increasing function from $E_{n}^{-}\amalg E_{n}^{-,+}$ to $\lbrace x_{1},\dots,x_{n}\rbrace$ and $x^{+}$ an increasing function from $E_{n}^{+}\amalg E_{n}^{-,+}$ to $\lbrace x_{1},\dots,x_{n}\rbrace$. We assume that the sets $x^{-}(E_{n}^{-}\amalg E_{n}^{-,+})$ and $x^{+}(E_{n}^{+}\amalg E_{n}^{-,+})$ are disjoint, that for every $i\in E_{n}^{-,+}$ $x^{-}(i)<x^{+}(i)$ and that for every $i\in E_{n}^{-}\amalg E_{n}^{-,+}$ and $j\in E_{q}^{+}\amalg E_{q}^{-,+}$ such that $i\neq j$, $(x^{+}(j)-x^{-}(i))$ has the same sign as $(j-i)$. Let $(\Delta_{i})_{1\leq i\leq n}$ be a family of disjoint bounded intervals each of which may be open, closed or semi-open such that for every $i<j$, $\max \Delta_{i} < \min \Delta_{j}$, that for every $i$, $\min \Delta_{i} \geq x^{-}(i)$ if $i\in E_{n}^{-}\amalg E_{n}^{-,+}$, $\max \Delta_{i} \leq x^{+}(i)$ if $i\in E_{n}^{+}\amalg E_{n}^{-,+}$, and that for all $i$,
\begin{displaymath}
x^{-}(i-1), x^{+}(i-1) <\min \Delta_{i},~\max \Delta_{i} < x^{-}(i+1), x^{+}(i+1),
\end{displaymath}
where in the previous inequalities one should only consider the terms that are defined. Let $p_{i}^{-}(y_{i})$ and $p_{i}^{+}(y_{i})$ be the functions defined by: $p_{i}^{-}(y_{i})=x^{-}(i)$ if $i\in E_{n}^{-}\amalg E_{n}^{-,+}$ and $y_{i}$ otherwise. $p_{i}^{+}(y_{i})=x^{+}(i)$ if $i\in E_{n}^{+}\amalg E_{n}^{-,+}$ and $y_{i}$ otherwise. Then,
\begin{multline}
\label{Ch6Sec3: EqDeterminantnSteps}
\mathbb{P}\big(Q_{n}=q,\forall i\in E_{n}^{-},p^{-}_{n,i}=x^{-}(i),p^{+}_{n,i}=Y_{n,i}
\forall i\in E_{n}^{+}, p^{+}_{n,i}=x^{+}(i), p^{-}_{n,i}=Y_{n,i},
\\\forall i\in E_{n}^{-,+}, p^{-}_{n,i}=x^{-}(i), p^{+}_{n,i}=x^{+}(i),
\forall r\in\lbrace 1,\dots,q\rbrace, Y_{n,r}\in\Delta_{r}\big) = 
\\\int_{y_{1}\in\Delta_{1}}\dots\int_{y_{q}\in\Delta_{q}}\det\left(G(p_{i}^{-}(y_{i}),p_{j}^{+}(y_{j}))\right)_{1\leq i,j\leq q}\prod_{1\leq r\leq q}\kappa(dy_{i}).
\end{multline}
$\det\left(G(p_{i}^{-}(y_{i}),p_{j}^{+}(y_{j}))\right)_{1\leq i,j\leq q}$ may be rewritten as a simpler product:
\begin{multline}
\label{Ch6Sec3: EqSimplerProduct1}
G(p_{1}^{-}(y_{1}),p_{1}^{+}(y_{1}))
\prod_{1\leq r\leq q-1}\bigg(G(p_{r+1}^{-}(y_{r+1}),p_{r+1}^{+}(y_{r+1}))
\\-\dfrac{G(p^{-}_{r}(y_{r}),p^{+}_{r+1}(y_{r+1}))G(p^{+}_{r}(y_{r}),p^{-}_{r+1}(y_{r+1}))}{G(p_{r}^{-}(y_{r}),p_{r}^{+}(y_{r}))}\bigg).
\end{multline}
If $\sigma$ is a permutation of $\lbrace 1,\dots,n\rbrace$, then $(\mathcal{Y}_{n}(x_{\sigma(1)},\dots,x_{\sigma(n)}),$ $\mathcal{J}_{n}(x_{\sigma(1)},\dots,x_{\sigma(n)}))$ has the same law as $(\mathcal{Y}_{n}(x_{1},\dots,x_{n}), \mathcal{J}_{n}(x_{1},\dots,x_{n}))$ Moreover, for any $n'>n$ and any permutation $\sigma$ of 
${\lbrace n+1,\dots,n'\rbrace}$, the law of 
$(\mathcal{Y}_{n'}(x_{1},\dots,x_{n},x_{\sigma(n+1)},\dots,x_{\sigma(n')}),$  $\mathcal{J}_{n'}(x_{1},\dots,x_{n},x_{\sigma(n+1)},\dots,x_{\sigma(n')}))$ conditional on $(\mathcal{Y}_{n}(x_{1},\dots,x_{n}), \mathcal{J}_{n}(x_{1},\dots,x_{n}))$ is the same as the law of $(\mathcal{Y}_{n'}(x_{1},\dots,x_{n},x_{n+1},\dots,x_{n'}),$  $\mathcal{J}_{n'}(x_{1},\dots,x_{n},x_{n+1},\dots,x_{n'}))$ conditional on $(\mathcal{Y}_{n}(x_{1},\dots,x_{n}), \mathcal{J}_{n}(x_{1},\dots,x_{n}))$.
\end{prop}

\begin{proof}
We will only give the sketch of a short proof. First, let us check that the determinant $\det\left(G(p_{i}^{-}(y_{i}),p_{j}^{+}(y_{j}))\right)_{1\leq i,j\leq q}$ may be indeed expressed as a product \eqref{Ch6Sec3: EqSimplerProduct1}. We use the fact that for any $a<b<\tilde{a}<\tilde{b}\in\mathbb{R}$,
\begin{displaymath}
G(a,\tilde{b})G(b,\tilde{a})=G(a,\tilde{a})G(b,\tilde{b})=
u_{\uparrow}(a)u_{\uparrow}(b)u_{\downarrow}(\tilde{a})u_{\downarrow}(\tilde{b}).
\end{displaymath}
By subtracting from the last line in the matrix $\left(G(p_{i}^{-}(y_{i}),p_{j}^{+}(y_{j}))\right)_{1\leq i,j\leq q}$, which is $(G(p_{q}^{-}(y_{q}),p_{j}^{+}(y_{j})))_{1\leq j\leq q}$, the second to last line 
$(G(p_{q-1}^{-}(y_{q-1}),p_{j}^{+}(y_{j})))_{1\leq j\leq q}$ multiplied by 
$\dfrac{G(p_{q-1}^{-}(y_{q-1}),p_{q}^{+}(y_{q}))}{G(p_{q-1}^{-}(y_{q-1}),
p_{q-1}^{+}(y_{q-1}))}$, we get zero for all coefficient on the last line, except the diagonal one. Thus, $\det\left(G(p_{i}^{-}(y_{i}),p_{j}^{+}(y_{j}))\right)_{1\leq i,j\leq q}$ equals
\begin{multline*}
\det\!\left(\!G(p_{i}^{-}(y_{i}),p_{j}^{+}(y_{j}))\right)_{1\leq i,j\leq q-1}\times
\\\bigg(G(p_{q}^{-}(y_{q}),p_{q}^{+}(y_{q}))-\dfrac{G(p^{-}_{q-1}(y_{q-1}),p^{+}_{q}(y_{q}))G(p^{+}_{q-1}(y_{q-1}),p^{-}_{q}(y_{q}))}{G(p_{q-1}^{-}(y_{q-1}),p_{q-1}^{+}(y_{q-1}))}\bigg).
\end{multline*}
By induction we get \eqref{Ch6Sec3: EqSimplerProduct1}.

Next step is to check that 
$(\mathcal{Y}_{n}(x_{1},\dots,x_{n-2},x_{n-1},x_{n})$, 
$\mathcal{J}_{n}(x_{1},\dots,x_{n-2},x_{n-1},x_{n}))$ and 
$(\mathcal{Y}_{n}(x_{1},\dots,x_{n-2},x_{n},x_{n-1}),
\mathcal{J}_{n}(x_{1},\dots,x_{n-2},x_{n},x_{n-1}))$ have the same law conditional on 
$(\mathcal{Y}_{n-2}(x_{1},\dots,x_{n-2}),\mathcal{J}_{n-2}(x_{1},\dots,x_{n-2}))$. This can be done using the explicit expressions for the conditional destitution of $B^{(x_{n-1})}_{T_{n-1}^{-}}$, $B^{(x_{n})}_{T_{n}^{-}}$, $B^{(x_{n})}_{T_{n-1}^{-}}$ and $B^{(x_{n-1})}_{T_{n}^{-}}$. This invariance by transposition of the two last starting points implies in turn all the invariances by permutation stated in the proposition.

From the invariance by permutation follows that one only needs to prove 
\eqref{Ch6Sec3: EqDeterminantnSteps} in case $x_{1}<x_{2}<\dots<x_{n}$. In this case one can prove \eqref{Ch6Sec3: EqDeterminantnSteps} by induction on $n$ using the expression \eqref{Ch6Sec3: EqSimplerProduct1} for $\det\left(G(p_{i}^{-}(y_{i}),p_{j}^{+}(y_{j}))\right)_{1\leq i,j\leq q}$.
\end{proof}

The fact that the law of the tree obtained after $n$ steps of Wilson's algorithm is invariant under permutations of the starting points $(x_{1},\dots,x_{n})$ is something that is also satisfied in case of random walks on a true finite graph. 
The product \eqref{Ch6Sec3: EqSimplerProduct1} can be further rewritten as
\begin{equation}
\label{Ch6Sec3: EqSimplerProduct2}
u_{\uparrow}(p_{1}^{-}(y_{1}))
u_{\downarrow}(p_{q}^{+}(y_{q}))
\prod_{1\leq r\leq q-1}(u_{\downarrow}(p_{r}^{+}(y_{r}))u_{\uparrow}(p_{r+1}^{-}(y_{r+1}))-u_{\uparrow}(p_{r}^{+}(y_{r}))u_{\downarrow}(p_{r+1}^{-}(y_{r+1}))).
\end{equation}

Next, we will show that $\mathcal{Y}_{\infty}$ and $\mathcal{Z}_{\infty}$ are a.s. discrete.

\begin{lemm}
\label{Ch6Sec3: LemHole}
For all $n\geq 2$ and $q\in\lbrace 2,\dots,n\rbrace$,
\begin{multline*}
\mathbb{P}\left(\mathcal{Y}_{\infty}\cap(p^{-}_{n,q-1},p^{+}_{n,q})=\emptyset
\vert p^{-}_{n,q-1}, p^{+}_{n,q}, Q_{n}\geq q\right)=
\\\dfrac{2(p^{+}_{n,q}-p^{-}_{n,q-1})}{u_{\downarrow}(p^{-}_{n,q-1})
u_{\uparrow}(p^{+}_{n,q})-u_{\uparrow}(p^{-}_{n,q-1})u_{\downarrow}(p^{+}_{n,q})}.
\end{multline*}
\end{lemm}

\begin{proof}
Let $n$ and $q$ be fixed. For $n'>n$, let
\begin{displaymath}
N(n'):=\#(\lbrace x_{n+1},\dots,x_{n'}\rbrace\cap(p^{-}_{n,q-1},p^{+}_{n,q}),
\end{displaymath}
and $\tilde{x}_{n',1}<\tilde{x}_{n',2}<\dots<\tilde{x}_{n',N(n')}$ the points of 
$\lbrace x_{n+1},\dots,x_{n'}\rbrace\cap(p^{-}_{n,q-1},p^{+}_{n,q})$ ordered increasingly. By convention, we define $\tilde{x}_{n',0}:=p^{-}_{n,q-1}$ and $\tilde{x}_{n',N(n')+1}:=p^{+}_{n,q}$. 
The condition $\mathcal{Y}_{n'}\cap(p^{-}_{n,q-1},p^{+}_{n,q})=\emptyset$ is satisfied if and only if for some $i\in \lbrace 1,2,\dots,N(n')+1\rbrace$, necessarily unique, the following holds:
\begin{displaymath}
[p^{-}_{n,q-1},\tilde{x}_{n',i-1}]\subseteq\bigcup_{J\in\mathcal{J}_{m}}J~~\text{and}~~
[\tilde{x}_{n',i},p^{+}_{n,q}]\subseteq\bigcup_{J\in\mathcal{J}_{n'}}J.
\end{displaymath}
Thus,
\begin{multline*}
\mathbb{P}\left(\mathcal{Y}_{n'}\cap(p^{-}_{n,q-1},p^{+}_{n,q})=\emptyset
\vert p^{-}_{n,q-1}, p^{+}_{n,q}, Q_{n}\geq q\right)=\\
\sum_{i=1}^{N(n')+1}
\mathbb{P}\Big([p^{-}_{n,q-1},\tilde{x}_{n',i-1}]\subseteq\bigcup_{J\in\mathcal{J}_{m}}J,
[\tilde{x}_{n',i},p^{+}_{n,q}]\subseteq\bigcup_{J\in\mathcal{J}_{n'}}J
\Big\vert p^{-}_{n,q-1}, p^{+}_{n,q}, Q_{n}\geq q\Big).
\end{multline*}
Let $T_{n',i}$ be the first time $B^{(\tilde{x}_{n',i})}$ hits either $p^{-}_{n,q-1}$ or $p^{+}_{n,q}$ or gets killed by the killing measure $\kappa$. For 
$i\in\lbrace 1,2,\dots,N(n')+1\rbrace$, let $T_{n',i,\tilde{x}_{n',i-1}}$ be the first time $B^{(\tilde{x}_{n',i})}$ hits $\tilde{x}_{n',i-1}$. Since the law of 
$(\mathcal{Y}_{n'},\mathcal{J}_{n'})$ conditional on $(\mathcal{Y}_{n},\mathcal{J}_{n})$ is invariant by permutation of points in $(x_{n+1},\dots,x_{n'})$, we get that
\begin{multline*}
\mathbb{P}\Big([p^{-}_{n,q-1},\tilde{x}_{n',i-1}]\subseteq\bigcup_{J\in\mathcal{J}_{m}}J,
[\tilde{x}_{n',i},p^{+}_{n,q}]\subseteq\bigcup_{J\in\mathcal{J}_{n'}}J
\Big\vert p^{-}_{n,q-1}, p^{+}_{n,q}, Q_{n}\geq q\Big)=
\\\mathbb{P}\left(B^{(\tilde{x}_{n',i-1})}_{T_{n',i-1}^{-}}=p^{-}_{n,q-1},
B^{(\tilde{x}_{n',i})}_{T_{n',i}^{-}}=p^{+}_{n,q},T_{n',i}<T_{n',i,\tilde{x}_{n',i-1}}
\Big\vert p^{-}_{n,q-1}, p^{+}_{n,q}, Q_{n}\geq q\right) =
\\\dfrac{u_{\downarrow}(\tilde{x}_{n',i-1})u_{\uparrow}(\tilde{x}_{n',i})-
u_{\uparrow}(\tilde{x}_{n',i-1})u_{\downarrow}(\tilde{x}_{n',i})}
{u_{\downarrow}(p^{-}_{n,q-1})u_{\uparrow}(p^{+}_{n,q})-
u_{\uparrow}(p^{-}_{n,q-1})u_{\downarrow}(p^{+}_{n,q})}.
\end{multline*}
It follows that
\begin{multline*}
\mathbb{P}\left(\mathcal{Y}_{n'}\cap(p^{-}_{n,q-1},p^{+}_{n,q})=\emptyset
\vert p^{-}_{n,q-1}, p^{+}_{n,q}, Q_{n}\geq q\right)=\\
\sum_{i=1}^{N(n')+1}
\dfrac{u_{\downarrow}(\tilde{x}_{n',i-1})u_{\uparrow}(\tilde{x}_{n',i})-
u_{\uparrow}(\tilde{x}_{n',i-1})u_{\downarrow}(\tilde{x}_{n',i})}
{u_{\downarrow}(p^{-}_{n,q-1})u_{\uparrow}(p^{+}_{n,q})-
u_{\uparrow}(p^{-}_{n,q-1})u_{\downarrow}(p^{+}_{n,q})}.
\end{multline*}
If $\tilde{x}_{n',i-1}$ is close to $\tilde{x}_{n',i}$, then
\begin{equation*}
\begin{split}
u_{\downarrow}(\tilde{x}_{n',i-1})u_{\uparrow}(\tilde{x}_{n',i})-
u_{\uparrow}(\tilde{x}_{n',i-1})&u_{\downarrow}(\tilde{x}_{n',i})
\\=&W(u_{\downarrow},u_{\uparrow})(\tilde{x}_{n',i-1})
(\tilde{x}_{n',i}-\tilde{x}_{n',i-1}) + o(\tilde{x}_{n',i}-\tilde{x}_{n',i-1})
\\=&2(\tilde{x}_{n',i}-\tilde{x}_{n',i-1}) + o(\tilde{x}_{n',i}-\tilde{x}_{n',i-1}).
\end{split}
\end{equation*}
The sequence $(x_{n'})_{n'\geq n+1}$ is dense in $(p^{-}_{n,q-1},p^{+}_{n,q})$. Thus, 
\begin{multline*}
\lim_{n'\rightarrow+\infty}
\mathbb{P}\left(\mathcal{Y}_{n'}\cap(p^{-}_{n,q-1},p^{+}_{n,q})=\emptyset
\vert p^{-}_{n,q-1}, p^{+}_{n,q}, Q_{n}\geq q\right)=
\\\dfrac{2(p^{+}_{n,q}-p^{-}_{n,q-1})}{u_{\downarrow}(p^{-}_{n,q-1})
u_{\uparrow}(p^{+}_{n,q})-u_{\uparrow}(p^{-}_{n,q-1})u_{\downarrow}(p^{+}_{n,q})}.
\qedhere
\end{multline*}
\end{proof}

\begin{prop}
\label{Ch6Sec3: PropFiniteExp}
Let $a<b\in\mathbb{R}$. Then, for all $n\geq 1$,
\begin{equation}
\label{Ch6Sec3: EqMajorEspNumPoints}
\mathbb{E}\left[\#(\mathcal{Y}_{n}\cap[a,b))\right] \leq \int_{[a,b)}G(x,x)\kappa(dx).
\end{equation}
It follows that a.s., for all $a<b\in\mathbb{R}$, $\mathcal{Y}_{\infty}\cap[a,b)$ is finite.
\end{prop}

\begin{proof}
Let $\tilde{a}<\tilde{b}\in[a,b]$, where $\tilde{a}$ is close to $\tilde{b}$. We will first show that for all $n\geq 1$,
\begin{equation}
\label{Ch6Sec3: EqMajorPresenceProb}
\mathbb{P}\left(\mathcal{Y}_{n}\cap[\tilde{a},\tilde{b})\neq \emptyset\right)\leq \int_{[\tilde{a},\tilde{b})}G(x,x)\kappa(dx) + o(\tilde{b}-\tilde{a}),
\end{equation}
where $o(\tilde{b}-\tilde{a})$ is uniform over $\tilde{a}$ and $\tilde{b}$ close to each other in $[a,b]$. Then, we will deduce \eqref{Ch6Sec3: EqMajorEspNumPoints} by partitioning the interval $[a,b)$ in small subintervals $[\tilde{a},\tilde{b})$ and approximating the expected number of points in $[\tilde{a},\tilde{b})$ by the probability of presence of one point. Let $n\geq 1$. Then,
\begin{displaymath}
\mathbb{P}\left(\mathcal{Y}_{n}(x_{1},\dots,x_{n})\cap[\tilde{a},\tilde{b})
\neq \emptyset\right)\leq
\mathbb{P}\left(\mathcal{Y}_{n+2}(x_{1},\dots,x_{n},\tilde{a},\tilde{b})
\cap[\tilde{a},\tilde{b})\neq \emptyset\right).
\end{displaymath}
Since the law of $\mathcal{Y}_{n+2}$ is invariant by permutation of the starting points,
\begin{displaymath}
\mathbb{P}\left(\mathcal{Y}_{n+2}(x_{1},\dots,x_{n},\tilde{a},\tilde{b})
\cap[\tilde{a},\tilde{b})\neq \emptyset\right)=
\mathbb{P}\left(\mathcal{Y}_{n+2}(\tilde{a},\tilde{b},x_{1},\dots,x_{n})
\cap[\tilde{a},\tilde{b})\neq \emptyset\right).
\end{displaymath}
But,
\begin{multline}
\label{Ch6Sec3: EqIneq1}
\mathbb{P}\left(\mathcal{Y}_{n+2}(\tilde{a},\tilde{b},x_{1},\dots,x_{n})
\cap[\tilde{a},\tilde{b})\neq \emptyset\right)
=\mathbb{P}\left(\mathcal{Y}_{2}(\tilde{a},\tilde{b})\cap[\tilde{a},\tilde{b})
\neq \emptyset\right)
\\+\mathbb{P}\left(\mathcal{Y}_{2}(\tilde{a},\tilde{b})\cap[\tilde{a},\tilde{b})
=\emptyset,\mathcal{Y}_{n+2}(\tilde{a},\tilde{b},x_{1},\dots,x_{n})
\cap[\tilde{a},\tilde{b})\neq \emptyset \right).
\end{multline}
We start Wilson's algorithm by launching first $B^{(\tilde{a})}$ starting from $\tilde{a}$ followed by $B^{(\tilde{b})}$ starting $\tilde{b}$. Then,
\begin{displaymath}
\mathbb{P}\left(\mathcal{Y}_{2}(\tilde{a},\tilde{b})\cap[\tilde{a},\tilde{b})
\neq \emptyset\right)=
\mathbb{P}\left( B^{(\tilde{a})}_{T_{1}^{-}}\in[\tilde{a},\tilde{b})\right) +
\mathbb{P}\left( B^{(\tilde{a})}_{T_{1}^{-}}\not
\in[\tilde{a},\tilde{b}),B^{(\tilde{a})}_{T_{1}^{-}}
\leq \tilde{a},B^{(\tilde{b})}_{T_{2}^{-}}\in[\tilde{a},\tilde{b})\right).
\end{displaymath}
Applying Proposition \ref{Ch6Sec3: PropAfternSteps}, we get that
\begin{multline*}
\mathbb{P}\left(\mathcal{Y}_{2}(\tilde{a},\tilde{b})
\cap[\tilde{a},\tilde{b})\neq \emptyset\right)=
\\\int_{x\in [\tilde{a},\tilde{b})}\left(G(\tilde{a},x)
+\int_{y\leq \tilde{a}}(G(y,\tilde{a})G(x,\tilde{b})
-G(y,\tilde{b})G(\tilde{a},x))\kappa(dy)\right)\kappa(dx).
\end{multline*}
For $x\in\mathbb{R}$, let $T_{1,x}$ be the first time $B^{(\tilde{a})}$ hits $x$. Then,
\begin{multline*}
G(\tilde{a},x)+\int_{y\leq\tilde{a}}(G(y,\tilde{a})G(x,\tilde{b})-
G(y,\tilde{b})G(\tilde{a},x))\kappa(dy)=
\\G(x,x)\left(\mathbb{P}(T_{1}\geq T_{1,x})+\dfrac{G(x,b)}{G(x,x)}
\mathbb{P}(T_{1}<T_{1,x},B^{(\tilde{a})}_{T_{1}^{-}}\leq \tilde{a})\right)\leq G(x,x).
\end{multline*}
Thus,
\begin{equation}
\label{Ch6Sec3: EqIneq2}
\mathbb{P}\left(\mathcal{Y}_{2}(\tilde{a},\tilde{b})\cap[\tilde{a},\tilde{b})\neq \emptyset\right)\leq\int_{[\tilde{a},\tilde{b})}G(x,x)\kappa(dx).
\end{equation}
Further,
\begin{multline*}
\mathbb{P}\left(\mathcal{Y}_{2}(\tilde{a},\tilde{b})\cap[\tilde{a},\tilde{b})=\emptyset,\mathcal{Y}_{n+2}(\tilde{a},\tilde{b},x_{1},\dots,x_{n})\cap[\tilde{a},\tilde{b})
\neq \emptyset \right)\leq
\\\mathbb{P}\left(\mathcal{Y}_{2}(\tilde{a},\tilde{b})\cap[\tilde{a},\tilde{b})=\emptyset,\mathcal{Y}_{\infty}(\tilde{a},\tilde{b},(x_{j})_{j\geq 1})\cap[\tilde{a},\tilde{b})
\neq \emptyset\right).
\end{multline*}
Applying Lemma \ref{Ch6Sec3: LemHole} and Proposition \ref{Ch6Sec3: PropAfternSteps},
we get that
\begin{equation*}
\begin{split}
\mathbb{P}\Big(\mathcal{Y}_{2}(&\tilde{a},\tilde{b})\cap[\tilde{a},\tilde{b})=\emptyset,\mathcal{Y}_{\infty}(\tilde{a},\tilde{b},(x_{j})_{j\geq 1})\cap[\tilde{a},\tilde{b})
\neq \emptyset \Big)=
\\\mathbb{P}\Big(\mathcal{Y}_{\infty}&(\tilde{a},\tilde{b},(x_{j})_{j\geq 1})
\cap[\tilde{a},\tilde{b})\neq \emptyset \vert\mathcal{Y}_{2}(\tilde{a},\tilde{b})\cap[\tilde{a},\tilde{b})=\emptyset\Big)
\times\mathbb{P}\left(\mathcal{Y}_{2}(\tilde{a},\tilde{b})
\cap[\tilde{a},\tilde{b})=\emptyset\right)
\\=\Bigg(1-&\dfrac{2(\tilde{b}-\tilde{a})}{u_{\downarrow}(\tilde{a})u_{\uparrow}(\tilde{b})-u_{\uparrow}(\tilde{a})u\downarrow(\tilde{b})}\Bigg)
\\&\times\int_{y\leq a}\int_{z\geq b}1_{y,z\not\in[\tilde{a},\tilde{b})}\det
\left( 
\begin{array}{cc}
G(y,a) & G(y,z) \\ 
G(a,b) & G(b,z)
\end{array} 
\right)\kappa(dy)\kappa(dz) 
\\\leq(u_{\downarrow}&(\tilde{a})u_{\uparrow}(\tilde{b})
-u_{\uparrow}(\tilde{a})u\downarrow(\tilde{b})-2(\tilde{b}-\tilde{a}))
\int_{y\leq b}u_{\uparrow}(y)\kappa(dy)\int_{z\geq a}u_{\downarrow}(z)\kappa(dz).
\end{split}
\end{equation*}
But,
\begin{displaymath}
u_{\downarrow}(\tilde{a})u_{\uparrow}(\tilde{b})
-u_{\uparrow}(\tilde{a})u\downarrow(\tilde{b})
-2(\tilde{b}-\tilde{a})= o(\tilde{b}-\tilde{a}).
\end{displaymath}
Thus,
\begin{equation}
\label{Ch6Sec3: EqIneq3}
\mathbb{P}\left(\mathcal{Y}_{2}(\tilde{a},\tilde{b})\cap[\tilde{a},\tilde{b})=\emptyset,\mathcal{Y}_{n+2}(\tilde{a},\tilde{b},x_{1},\dots,x_{n})\cap[\tilde{a},\tilde{b})
\neq \emptyset \right) = o(\tilde{b}-\tilde{a}).
\end{equation}
Combining \eqref{Ch6Sec3: EqIneq1}, \eqref{Ch6Sec3: EqIneq2} and \eqref{Ch6Sec3: EqIneq3} we get \eqref{Ch6Sec3: EqMajorPresenceProb}.

Now, for $j\in\mathbb{N}^{\ast}$ and $i\in\lbrace 1,\dots,2^{j}\rbrace$, consider the intervals $\Delta_{i,j}$ defined by 
\begin{displaymath}
\Delta_{i,j}=
\left\lbrace
\begin{array}{ll}
\left[ a+(i-1)2^{-j}(b-a), a + i 2^{-j}(b-a)\right)  & \text{if}~i\leq 2^{j}-1, \\ 
\left[a+(1-2^{-j})(b-a), b \right]  & \text{if}~i=2^{j}.
\end{array} 
 \right. 
\end{displaymath}
Then, $\mathbb{E}\left[\#(\mathcal{Y}_{n}\cap[a,b))\right]$ is the increasing limit of
$\sum_{i=1}^{2^{j}}\mathbb{P}\left(\mathcal{Y}_{n}\cap\Delta_{i,j}\neq \emptyset\right)$.
But,
\begin{displaymath}
\sum_{i=1}^{2^{j}}\mathbb{P}\left(\mathcal{Y}_{n}\cap\Delta_{i,j}\neq \emptyset\right)
\leq\sum_{i=1}^{2^{j}}\int_{\Delta_{i,j}}G(x,x)\kappa(dx)+ 2^{j}o(2^{-j}).
\end{displaymath}
\eqref{Ch6Sec3: EqMajorEspNumPoints} follows. Since \eqref{Ch6Sec3: EqMajorEspNumPoints} holds for all $n$, it also holds at the limit when $n$ tends to $+\infty$. This implies that $\mathcal{Y}_{\infty}\cap[a,b)$ is a.s. finite.
\end{proof}

\begin{prop}
\label{Ch6Sec3: PropOpenInt}
Almost surely, all the intervals in $\mathcal{J}_{\infty}$ are open.
\end{prop}

\begin{proof}
We need only to show that for any $n\geq 1$ and $q\in\lbrace 1,\dots,n\rbrace$,
\begin{equation}
\label{Ch6Sec3: EqMinNonStat}
\mathbb{P}\left(Q_{n}\geq q,\forall n'\geq n, \min(\imath_{n,n'}([p^{-}_{n,q},p^{+}_{n,q}]))=p^{-}_{n,q}\right) =0,
\end{equation}
and
\begin{displaymath}
\mathbb{P}\left(Q_{n}\geq q,\forall n'\geq n, \max(\imath_{n,n'}([p^{-}_{n,q},p^{+}_{n,q}]))=p^{+}_{q,n}\right) =0.
\end{displaymath}
Let $n$ and $q$ be fixed. We will show \eqref{Ch6Sec3: EqMinNonStat}. We will also assume that $q\geq 2$. The proof is similar if $q=1$. We need to show that a.s., the following conditional probability converges to $0$:
\begin{displaymath}
\lim_{n'\rightarrow +\infty}\mathbb{P}\left(\min(\imath_{n,n'}([p^{-}_{n,q},p^{+}_{n,q}]))=p^{-}_{n,q}\vert 
(\mathcal{Y}_{n},\mathcal{J}_{n}),Q_{n}\geq q\right) = 0.
\end{displaymath}
We recall that for $n''\geq n+1$, $B^{(x_{n''})}$ is a Brownian motion starting from $x_{n''}$ and it is independent from $(\mathcal{Y}_{n},\mathcal{J}_{n})$. Let $T_{n'',p^{-}_{q,n}}$ be the first time it hits $p^{-}_{q,n}$ and $\widetilde{T}_{n''}$ the first time it either hits $\bigcup_{J\in \mathcal{J}_{n}}J$ or gets killed by the killing measure $\kappa$. Since the law of $(\mathcal{Y}_{n'},\mathcal{J}_{n'})$ conditional on $(\mathcal{Y}_{n},\mathcal{J}_{n})$ is invariant by permutation of points in $(x_{n+1},\dots,x_{n'})$, we get that
\begin{multline*}
\mathbb{P}\left(\min(\imath_{n,n'}([p^{-}_{n,q},p^{+}_{n,q}]))=p^{-}_{n,q}\vert (\mathcal{Y}_{n},\mathcal{J}_{n}),Q_{n}\geq q\right)
\\\leq \inf_{n+1\leq n''\leq n'} 1-1_{p^{+}_{n,q-1}<x_{n''}<p^{-}_{n,q}}
\mathbb{P}\left(\widetilde{T}_{n''}=T_{n'',p^{-}_{q,n}}\vert p^{+}_{n,q-1},p^{-}_{n,q},Q_{n}\geq q\right).
\end{multline*}
But, $\mathbb{P}\left(\widetilde{T}_{n''}=T_{n'',p^{-}_{q,n}}\vert p^{+}_{n,q-1},p^{-}_{n,q}\right)$ is close to $1$ if $x_{n''}$ is close enough to $p^{-}_{n,q}$. There is always a subsequence of $(x_{n''})_{n''\geq n+1}$ made of points in $(p^{+}_{n,q-1},p^{-}_{n,q})$ which converges to $p^{-}_{n,q}$. It follows that
\begin{displaymath}
\inf_{n''\geq n+1}1-1_{p^{+}_{n,q-1}<x_{n''}<p^{-}_{n,q}}
\mathbb{P}\left(\widetilde{T}_{n''}=T_{n'',p^{-}_{q,n}}\vert p^{+}_{n,q-1},p^{-}_{n,q},Q_{n}\geq q\right)=0,
\end{displaymath}
which concludes the proof.
\end{proof}

From Proposition \ref{Ch6Sec3: PropOpenInt} follows that $\mathcal{Z}_{\infty}$ is closed. Moreover, it does not contain any of the points of the sequence $(x_{n})_{n\geq 1}$. Since the sequence $(x_{n})_{n\geq 1}$ is everywhere dense, the connected components of $\mathcal{Z}_{\infty}$ are single points. One can see that 
\begin{itemize}
\item If $y<\tilde{y}$ are two consecutive points in $\mathcal{Y}_{\infty}$, then ${\#(\mathcal{Z}_{\infty}\cap(y,\tilde{y}))=1}$.
\item If $\mathcal{Y}_{\infty}$ is bounded from below and $y=\min\mathcal{Y}_{\infty}$, then $\mathcal{Z}_{\infty}\cap(-\infty,y]=\emptyset$.
\item If $\mathcal{Y}_{\infty}$ is bounded from above and $y=\max\mathcal{Y}_{\infty}$, then $\mathcal{Z}_{\infty}\cap[y,+\infty)=\emptyset$.
\end{itemize}
See Figure \ref{FigW2}. The set $\mathcal{Z}_{\infty}$ may be empty, which for instance happens almost surely if $\kappa$ is a Dirac measure. 
For $n\geq 1$, we define
\begin{displaymath}
\mathcal{Z}_{n}:=\bigg\lbrace\dfrac{p^{-}_{n,q-1}+p^{+}_{n,q}}{2}\Big\vert 2\leq q\leq Q_{n}\bigg\rbrace.
\end{displaymath}
We will write $\mathcal{Z}_{n}(x_{1},\dots,x_{n})$ and $\mathcal{Z}_{\infty}((x_{n})_{n\geq 1})$ whenever we need to emphasize the dependence on the starting points.

\begin{prop}
\label{Ch6Sec3: PropInvPerm}
The law of $(\mathcal{Y}_{\infty},\mathcal{Z}_{\infty})$ does not depend on the starting points $(x_{n})_{n\geq 1}$.
\end{prop}

\begin{proof}
Let $(\tilde{x}_{n})_{n\geq 1}$ be another sequence of pairwise disjoint points in $\mathbb{R}$. We will show that the sequence
$(\mathcal{Y}_{2n}(x_{1},\dots,x_{n},\tilde{x}_{1},\dots,\tilde{x}_{n}),
\mathcal{Z}_{2n}(x_{1},\dots,x_{n},\tilde{x}_{1},\dots,\tilde{x}_{n}))$ 
converges in law to 
$(\mathcal{Y}_{\infty}((x_{n})_{n\geq 1}),\mathcal{Z}_{\infty}((x_{n})_{n\geq 1}))$
and that 
$(\mathcal{Y}_{2n}(\tilde{x}_{1},\dots,\tilde{x}_{n},x_{1},\dots,x_{n})$,
$\mathcal{Z}_{2n}(\tilde{x}_{1},\dots,\tilde{x}_{n},x_{1},\dots,x_{n}))$
converges to
$(\mathcal{Y}_{\infty}((\tilde{x}_{n})_{n\geq 1}),\mathcal{Z}_{\infty}((\tilde{x}_{n})_{n\geq 1}))$.
Since the two couples of point processes
$(\mathcal{Y}_{2n}(x_{1},\dots,x_{n},\tilde{x}_{1},\dots,\tilde{x}_{n}),
\mathcal{Z}_{2n}(x_{1},\dots,x_{n},\tilde{x}_{1},\dots,\tilde{x}_{n}))$
and 
$(\mathcal{Y}_{2n}(\tilde{x}_{1},\dots,\tilde{x}_{n},x_{1},\dots,x_{n}),
\mathcal{Z}_{2n}(\tilde{x}_{1},\dots,\tilde{x}_{n},x_{1},\dots,x_{n}))$
have the same law, this will finish the proof.

For the convergence in law we will use the topology of uniform convergence on compact sets of collections of points in $\mathbb{R}$. It can be defined  using the following metric: Let $d_{H}$ be the Hausdorff metric on compact subsets of $\mathbb{R}$. One may use the metric $d_{PP}$ on point processes:
\begin{displaymath}
d_{PP}(\mathcal{X},\widetilde{\mathcal{X}}):=
d_{H}(\tan^{-1}(\mathcal{X})\cup\lbrace -1,1\rbrace,
\tan^{-1}(\widetilde{\mathcal{X}})\cup\lbrace -1,1\rbrace).
\end{displaymath}

In order to simplify the notations we will write
\begin{displaymath}
(\mathcal{Y}_{n},\mathcal{Z}_{n}):=(\mathcal{Y}_{n}(x_{1},\dots,x_{n}),
\mathcal{Z}_{n}(x_{1},\dots,x_{n})),
\end{displaymath}
\begin{displaymath}
(\mathcal{Y}_{\infty},\mathcal{Z}_{\infty}):=
(\mathcal{Y}_{\infty}((x_{n})_{n\geq 1}),\mathcal{Z}_{\infty}((x_{n})_{n\geq 1})),
\end{displaymath}
\begin{displaymath}
(\widetilde{\mathcal{Y}}_{2n},\widetilde{\mathcal{Z}}_{2n}):=
(\mathcal{Y}_{2n}(x_{1},\dots,x_{n},\tilde{x}_{1},\dots,\tilde{x}_{n}),
\mathcal{Z}_{2n}(x_{1},\dots,x_{n},\tilde{x}_{1},\dots,\tilde{x}_{n})).
\end{displaymath}
We can construct $((\mathcal{Y}_{n},\mathcal{Z}_{n}))_{n\geq 1}$, $(\mathcal{Y}_{\infty},\mathcal{Z}_{\infty})$ and 
$((\widetilde{\mathcal{Y}}_{2n},\widetilde{\mathcal{Z}}_{2n}))_{n\geq 1}$
on the same probability space using independent Brownian motions starting from the points in $(x_{n})_{n\geq 1}$ and $(\tilde{x}_{n})_{n\geq 1}$ and killed by the measure $\kappa$. We construct the sequence $((\mathcal{Y}_{n},\mathcal{Z}_{n}))_{n\geq 1}$ using the Wilson's algorithm described in the introduction. This way, $\mathcal{Y}_{n}\subseteq\mathcal{Y}_{n+1}$ and 
$\mathcal{Y}_{\infty}=\bigcup_{n\geq 1}\mathcal{Y}_{n}$.
In order to construct $\widetilde{\mathcal{Y}}_{2n}$, we first construct $\mathcal{Y}_{n}$ and then continue the Wilson's algorithm using the Brownian motions starting from $\tilde{x}_{1},\dots,\tilde{x}_{n}$. This way, $\mathcal{Y}_{n}\subseteq\widetilde{\mathcal{Y}}_{2n}$ but not necessarily $\widetilde{\mathcal{Y}}_{2n}\subseteq\widetilde{\mathcal{Y}}_{2(n+1)}$.

Let $C>0$ and $\varepsilon\in(0,\frac{C}{2})$. Let $\delta\in(0,1)$, $\delta$ small. There is $N\in\mathbb{N}^{\ast}$ such that
\begin{displaymath}
\mathbb{P}\left(\mathcal{Y}_{N}\cap[-C,C]=
\mathcal{Y}_{\infty}\cap[-C,C]\right)\geq 1-\delta.
\end{displaymath}
There is $\varepsilon'\in(0,\varepsilon)$ such that for all $a<b\in[-C,C]$ satisfying 
$b-a\leq \varepsilon'$, the following holds:
\begin{displaymath}
1-\dfrac{2(b-a)}{u_{\downarrow}(a)u_{\uparrow}(b)-
u_{\uparrow}(a)u_{\downarrow}(b)}\leq \dfrac{\delta}{N}.
\end{displaymath}
There is $N'\geq N$ such that with probability $1-2\delta$ the following two conditions hold:
\begin{equation}
\label{Ch6Sec3: EqInvCond1}
\mathcal{Y}_{N}\cap[-C,C]=\mathcal{Y}_{\infty}\cap[-C,C],
\end{equation}
\begin{equation}
\label{Ch6Sec3: EqInvCond2}
\operatorname{Leb}\big([-C,C]\setminus \bigcup_{J\in\mathcal{J}_{N'}}J\big)\leq \varepsilon'.
\end{equation}
We define the following two random variables:
\begin{displaymath}
K^{-}:=\min_{J\in\mathcal{J}_{N'}, J\subseteq[-C,C]}(\min J),\qquad
K^{+}:=\max_{J\in\mathcal{J}_{N'}, J\subseteq[-C,C]}(\max J).
\end{displaymath}
If \eqref{Ch6Sec3: EqInvCond2} holds, then 
$[-\frac{C}{2},\frac{C}{2}]\subseteq[K^{-},K^{+}]$.
If \eqref{Ch6Sec3: EqInvCond1} and \eqref{Ch6Sec3: EqInvCond2} hold than for $n\geq N'$, ${[K^{-},K^{+}]\setminus \bigcup_{J\in\mathcal{J}_{n}}J}$ is made of at most $N$ intervals, each of length at most $\varepsilon'$. Consider the following condition on $\widetilde{\mathcal{Y}}_{2n}$:
\begin{equation}
\label{Ch6Sec3: EqInvCond3}
\widetilde{\mathcal{Y}}_{2n}\cap[K^{-},K^{+}]=\mathcal{Y}_{n}\cap[K^{-},K^{+}].
\end{equation}
Applying Lemma \ref{Ch6Sec3: LemHole}, we get that, for all $n\geq N'$,
\begin{displaymath}
\mathbb{P}\left(\widetilde{\mathcal{Y}}_{2n}~\text{satisfies}~\eqref{Ch6Sec3: EqInvCond3}~\vert ~ \eqref{Ch6Sec3: EqInvCond1}~\text{and}~\eqref{Ch6Sec3: EqInvCond2}~\text{hold}\right)\geq 1-\delta.
\end{displaymath}
This implies that, for all $n\geq N'$,
\begin{displaymath}
\mathbb{P}\left(\widetilde{\mathcal{Y}}_{2n}~\text{satisfies}~\eqref{Ch6Sec3: EqInvCond3},~\text{and}~ \eqref{Ch6Sec3: EqInvCond1}~\text{and}~\eqref{Ch6Sec3: EqInvCond2}~\text{hold}.\right)\geq 1-3\delta.
\end{displaymath}
Let $n\geq N'$. On the event when \eqref{Ch6Sec3: EqInvCond1} and 
\eqref{Ch6Sec3: EqInvCond2} hold and $\widetilde{\mathcal{Y}}_{2n}$ satisfies \eqref{Ch6Sec3: EqInvCond3}, which happens with probability at least $1-3\delta$, the following is true:
\begin{itemize}
\item $\widetilde{\mathcal{Y}}_{2n}\cap[K^{-},K^{+}]=\mathcal{Y}_{\infty}\cap[K^{-},K^{+}]$,
\item $d_{H}(\widetilde{\mathcal{Z}}_{2n}\cap[K^{-},K^{+}],\mathcal{Z}_{\infty}\cap[K^{-},K^{+}])\leq\varepsilon$.
\end{itemize}
In particular, with probability at least $1-3\delta$,
\begin{itemize}
\item $d_{PP}(\widetilde{\mathcal{Y}}_{2n},\mathcal{Y}_{\infty})\leq 
1-\tan^{-1}(\frac{C}{2})$,
\item $d_{H}(\widetilde{\mathcal{Z}}_{2n},\mathcal{Z}_{\infty})\leq \varepsilon
+(1-\tan^{-1}(\frac{C}{2}))$.
\end{itemize}
Since $C$ is arbitrary large and $\varepsilon$ and $\delta$ are arbitrary small, this implies that $(\widetilde{\mathcal{Y}}_{2n},\widetilde{\mathcal{Z}}_{2n})$ converges in law as $n\rightarrow +\infty$ to $(\mathcal{Y}_{\infty},\mathcal{Z}_{\infty})$.
\end{proof}

Next we identify the law of $\mathcal{Y}_{\infty}$ as a determinantal fermionic point process. For generalities on this processes, see \cite{HoughKrishPeresVirag2009GAFDet}, Chapter $4$, and \cite{Soshnikov2000Determinantal}.

\begin{prop}
\label{Ch3Sec3: PropDeterminantalY}
Let $n\geq 1$ and $a_{1}<b_{1}<a_{2}<b_{2}<\dots<a_{n}<b_{n}\in\mathbb{R}$. Then,
\begin{equation}
\label{Ch3Sec3: EqDeterminantalIdentity}
\mathbb{E}\left[\prod_{r=1}^{n}\#(\mathcal{Y}_{\infty}\cap[a_{r},b_{r}))\right] =
\int_{[a_{1},b_{1})}\dots\int_{[a_{n},b_{n})}
\det\left(G(y_{i},y_{j})\right)_{1\leq i,j\leq n}
\prod_{r=1}^{n}\kappa(dy_{r}).
\end{equation}
In other words $\mathcal{Y}_{\infty}$ is a determinantal point process on $\mathbb{R}$ with reference measure $\kappa$ and determinantal kernel $G$.
\end{prop}

\begin{proof}
Consider points $\tilde{a}_{r}<\tilde{b}_{r}\in[a_{r},b_{r}]$ for 
$r\in\lbrace 1,\dots,n\rbrace$. We will show that
\begin{multline}
\label{Ch3Sec3: EqDeterm1}
\mathbb{P}\left(\forall r\in\lbrace 1,\dots,n\rbrace,\mathcal{Y}_{\infty}\cap[\tilde{a}_{r},\tilde{b}_{r})
\neq\emptyset \right) =
\\\int_{[\tilde{a}_{1},\tilde{b}_{1})}\dots
\int_{[\tilde{a}_{n},\tilde{b}_{n})}
\det\left(G(y_{i},y_{j})\right)_{1\leq i,j\leq n}\prod_{r=1}^{n}\kappa(dy_{r})
\\+\Big(\sum_{r=1}^{n}O(\tilde{b}_{r}-\tilde{a}_{r})\Big)
\times \prod_{r=1}^{n}\kappa([\tilde{a}_{r},\tilde{b}_{r}))+
\sum_{
\substack{
E\subseteq\lbrace 1,\dots,n\rbrace \\ 
E\neq\emptyset
}}
\prod_{r\in E}o(\tilde{b}_{r}-\tilde{a}_{r})\prod_{r\not\in E}
\kappa([\tilde{a}_{r},\tilde{b}_{r})),
\end{multline}
where the quantities $O(\tilde{b}_{r}-\tilde{a}_{r})$ and $o(\tilde{b}_{r}-\tilde{a}_{r})$ are uniform over $\tilde{a}_{r}<\tilde{b}_{r}\in[a_{r},b_{r}]$, $\tilde{a}_{r}$ close to $\tilde{b}_{r}$. From \eqref{Ch3Sec3: EqDeterm1} one deduces 
\eqref{Ch3Sec3: EqDeterminantalIdentity} by splitting the intervals $[a_{r},b_{r}]$ in small subintervals and approximating the number of points in $\mathcal{Y}_{\infty}\cap[a_{r},b_{r})$ by the number of subintervals of $[a_{r},b_{r})$ that contain a point in $\mathcal{Y}_{\infty}$.

Since the law of $\mathcal{Y}_{\infty}$ does not depend on the choice of everywhere dense sequence of starting points, we will assume that the first $2n$ starting points in Wilson's algorithm are in order $\tilde{a}_{1},\tilde{b}_{1},\dots,\tilde{a}_{n},\tilde{b}_{n}$. We will show that for all non-empty subsets $E$ of $\lbrace 1,\dots,n\rbrace$,
\begin{multline}
\label{Ch3Sec3: EqDeterm2}
\mathbb{P}\left(\forall r\in E,\mathcal{Y}_{2n}\cap[\tilde{a}_{r},\tilde{b}_{r})
=\emptyset,\mathcal{Y}_{\infty}\cap[\tilde{a}_{r},\tilde{b}_{r})\neq\emptyset,
\forall r\not\in E,\mathcal{Y}_{2n}\cap[\tilde{a}_{r},\tilde{b}_{r})\neq\emptyset\right)
\\=\prod_{r\in E}o(\tilde{b}_{r}-\tilde{a}_{r})\prod_{r\not\in E}
\kappa([\tilde{a}_{r},\tilde{b}_{r})).
\end{multline}
Further, we will show that for any $r_{0}\in \lbrace 1,\dots,n\rbrace$,
\begin{equation}
\label{Ch3Sec3: EqDeterm3}
\mathbb{P}\big(\forall r\in\lbrace 1,\dots,n\rbrace,
\mathcal{Y}_{2n}\cap[\tilde{a}_{r},\tilde{b}_{r})\neq\emptyset,
[\tilde{a}_{r_{0}},\tilde{b}_{r_{0}}]\nsubseteq\bigcup_{J\in\mathcal{J}_{2n}}J\big)=
O(\tilde{b}_{r_{0}}-\tilde{a}_{r_{0}})\prod_{r=1}^{n}
\kappa([\tilde{a}_{r},\tilde{b}_{r})).
\end{equation}
If for all $r\in\lbrace 1,\dots,n\rbrace$, 
$\mathcal{Y}_{2n}\cap[\tilde{a}_{r},\tilde{b}_{r})\neq\emptyset$ and 
$[\tilde{a}_{r},\tilde{b}_{r}]\subseteq\bigcup_{J\in\mathcal{J}_{2n}}J$,
 then, necessarily, $Q_{2n}=n$ and 
$\mathcal{J}_{2n}=\lbrace[\tilde{a}_{r},\tilde{b}_{r}]\vert 1\leq r\leq n\rbrace$.
We will use the fact that according to \eqref{Ch6Sec3: EqDeterminantnSteps},
\begin{multline}
\label{Ch3Sec3: EqDeterm4}
\mathbb{P}\left(Q_{2n}=n,\mathcal{J}_{2n}=\lbrace[\tilde{a}_{r},\tilde{b}_{r}]
\vert 1\leq r\leq n\rbrace\right)
\\=\int_{[\tilde{a}_{1},\tilde{b}_{1})}\dots\int_{[\tilde{a}_{n},\tilde{b}_{n})}
\det\left(G^{\tilde{a}_{i},\tilde{b}_{j}}\right)_{1\leq i,j\leq n}
\prod_{r=1}^{n}\kappa(dy_{r})
\\=\int\limits_{[\tilde{a}_{1},\tilde{b}_{1})}\dots\int\limits_{[\tilde{a}_{n},\tilde{b}_{n})}
\det\left(G(y_{i},y_{j})\right)_{1\leq i,j\leq n}
\prod_{r=1}^{n}\kappa(dy_{r})+
\Big(\sum_{r=1}^{n}O(\tilde{b}_{r}-\tilde{a}_{r})\Big) 
\times\prod_{r=1}^{n}\kappa([\tilde{a}_{r},\tilde{b}_{r})).
\end{multline}

Let us show \eqref{Ch3Sec3: EqDeterm2}. A closed expression of the probability in \eqref{Ch3Sec3: EqDeterm2} can be computed using \eqref{Ch6Sec3: EqDeterminantnSteps} and Lemma \ref{Ch6Sec3: LemHole}. Since many different configurations (different values of $Q_{2n}$ and configurations of $J_{2n}$) contribute to the probability in \eqref{Ch3Sec3: EqDeterm2}, we will not give the closed expression and only give the estimates. Let $E$ be a non-empty subset of $\lbrace 1,\dots,n\rbrace$. If $r\not\in E$, then the condition $\mathcal{Y}_{2n}\cap[\tilde{a}_{r},\tilde{b}{r})\neq\emptyset$ contributes by a factor $O(\kappa([\tilde{a}_{r},\tilde{b}_{r})))$ to the probability in 
\eqref{Ch3Sec3: EqDeterm2}. If $r\in E$, then the two conditions $\mathcal{Y}_{2n}\cap[\tilde{a}_{r},\tilde{b}{r})=\emptyset$ and $\mathcal{Y}_{\infty}\cap[\tilde{a}_{r},\tilde{b}{r})\neq\emptyset$ imply that $(\tilde{a}_{r},\tilde{b}{r})\cap\bigcup_{J\in\mathcal{J}_{2n}}J = \emptyset$. According to the identity \eqref{Ch6Sec3: EqSimplerProduct2}, the condition 
$(\tilde{a}_{r},\tilde{b}{r})\cap\bigcup_{J\in\mathcal{J}_{2n}}J = \emptyset$ contributes to the probability in \eqref{Ch3Sec3: EqDeterm2} by a factor
\begin{displaymath}
O(u_{\downarrow}(\tilde{a}_{r})u_{\uparrow}(\tilde{b}_{r})-
u_{\uparrow}(\tilde{a}_{r})u_{\downarrow}(\tilde{b}_{r}))=O(\tilde{b}_{r}-\tilde{a}_{r}).
\end{displaymath}
According to Lemma \ref{Ch6Sec3: LemHole}, the additional condition $\mathcal{Y}_{\infty}\cap[\tilde{a}_{r},\tilde{b}{r})\neq\emptyset$ contributes to the probability in \eqref{Ch3Sec3: EqDeterm2} by a factor
\begin{displaymath}
1-\dfrac{2(\tilde{b}_{r}-\tilde{a}_{r})}{u_{\downarrow}(\tilde{a}_{r})
u_{\uparrow}(\tilde{b}_{r})-u_{\uparrow}(\tilde{a}_{r})u_{\downarrow}(\tilde{b}_{r})}=o(1).
\end{displaymath}
\eqref{Ch3Sec3: EqDeterm2} follows.

We deal now with \eqref{Ch3Sec3: EqDeterm3}. As in the previous case, the condition that for all $r\in\lbrace 1,\dots,n\rbrace$, 
$\mathcal{Y}_{2n}\cap[\tilde{a}_{r},\tilde{b}_{r})\neq\emptyset$, contributes by a factor $O\left(\prod_{r=1}^{n}\kappa([\tilde{a}_{r},\tilde{b}{r}))\right) $ to the probability in \eqref{Ch3Sec3: EqDeterm3}. The condition 
$[\tilde{a}_{r_{0}},\tilde{b}_{r_{0}}]\nsubseteq\bigcup_{J\in\mathcal{J}_{2n}}J$ implies that there is $i\in\lbrace 2,\dots,Q_{2n}\rbrace$ such that ${\tilde{a}_{r_{0}}<p^{+}_{2n,i-1}<p^{-}_{2n,i}<\tilde{b}_{r_{0}}}$. As previously, this contributes by a factor $O(\tilde{b}_{r_{0}}-\tilde{a}_{r_{0}})$ to the probability. Combining \eqref{Ch3Sec3: EqDeterm2}, \eqref{Ch3Sec3: EqDeterm3} and \eqref{Ch3Sec3: EqDeterm4} yields \eqref{Ch3Sec3: EqDeterm1}.
\end{proof}

Let $\mathfrak{G}_{\kappa}$ be the following operator defined for functions in $\mathbb{L}^{2}(dk)$ with compact support:
\begin{displaymath}
(\mathfrak{G}_{\kappa}f)(x):=\int_{\mathbb{R}}G(x,y)f(y)\kappa(dy).
\end{displaymath}
A standard condition for a determinantal point process with kernel $G$ relative to the measure $\kappa$ to be well defined is $\mathfrak{G}_{\kappa}$ to be positive semi-definite, contracting and locally trace class. We explain why this is true. Let $f$ be a compactly supported $\mathbb{L}^{2}(d\kappa)$ function. Then, the weak second derivative of $\mathfrak{G}_{\kappa}f$ is
\begin{displaymath}
d\Big(\dfrac{d(\mathfrak{G}_{\kappa}f)}{dx}\Big) = 2(\mathfrak{G}_{\kappa}f-f)d\kappa.
\end{displaymath}
$\mathfrak{G}_{\kappa}f$ and $\dfrac{d(\mathfrak{G}_{\kappa}f)}{dx}$
 are square-integrable and
\begin{equation}
\begin{split}
\label{Ch3Sec3: EqIPPG}
\int_{\mathbb{R}}(\mathfrak{G}_{\kappa}f)fd\kappa=&
\int_{\mathbb{R}}(\mathfrak{G}_{\kappa}f)^{2}d\kappa+
\dfrac{1}{2}\int_{\mathbb{R}}(\mathfrak{G}_{\kappa}f)
d\Big(\dfrac{d(\mathfrak{G}_{\kappa}f)}{dx}\Big)
\\=&\int_{\mathbb{R}}(\mathfrak{G}_{\kappa}f)^{2}d\kappa+
\dfrac{1}{2}\int_{\mathbb{R}}\Big(\dfrac{d(\mathfrak{G}_{\kappa}f)}{dx}\Big)^{2}dx.
\end{split}
\end{equation}
Identity \eqref{Ch3Sec3: EqIPPG} shows that $\mathfrak{G}_{\kappa}$ is positive semi-definite. It also shows that $\int_{\mathbb{R}}(\mathfrak{G}_{\kappa}f)^{2}d\kappa\leq
\int_{\mathbb{R}}(\mathfrak{G}_{\kappa}f)fd\kappa$, which implies that 
$\mathfrak{G}_{\kappa}$ is contracting and hence can be continuously extended to a contraction of the whole space $\mathbb{L}^{2}(d\kappa)$. $\mathfrak{G}_{\kappa}$ is locally trace class because it is positive semi-definite and its functional kernel is continuous (see Theorem $2.12$ in \cite{Simon2005TraceIdeals}, Chapter $2$). 

Next we give a criterion for $\mathcal{Y}_{\infty}$ to be finite or just to be finite in the neighborhood of either  $+\infty$ or $-\infty$.

\begin{prop}
\label{Ch6Sec3: PropFiniteness}
If $\int_{(0,+\infty)}x\kappa(dx) <+\infty$, then, almost surely, 
$\#(\mathcal{Y}_{\infty}\cap(0,+\infty))$ is finite. Moreover,
\begin{equation}
\label{Ch6Sec3: EqFiniteTrace}
\mathbb{E}\left[\#(\mathcal{Y}_{\infty}\cap(0,+\infty))\right]=
\int_{(0,+\infty)}G(x,x)\kappa(dx) <+\infty.
\end{equation}
If $\int_{(0,+\infty)}x\kappa(dx) =+\infty$, then, almost surely, $\#(\mathcal{Y}_{\infty}\cap(0,+\infty))=+\infty$. In general, for all $a\in\mathbb{R}$,
\begin{equation}
\label{Ch6Sec3: EqSemiFinProb}
\mathbb{P}(\mathcal{Y}_{\infty}\cap(a,+\infty)=\emptyset)=
u_{\downarrow}(+\infty)\int_{(-\infty,a]}u_{\uparrow}(x)\kappa(dx).
\end{equation}
Similarly, if $\int_{\mathbb{R}}\vert x\vert\kappa(dx)<+\infty$, then, a.s., 
$\# \mathcal{Y}_{\infty}$ is finite and
\begin{displaymath}
\mathbb{E}\left[\#\mathcal{Y}_{\infty}\right]=
\int_{\mathbb{R}}G(x,x)\kappa(dx)<+\infty.
\end{displaymath}
If $\int_{\mathbb{R}}\vert x\vert\kappa(dx)=+\infty$, then, a.s., 
$\# \mathcal{Y}_{\infty}=+\infty$.
\end{prop}

\begin{proof}
We need only to deal with the finiteness of $\#(\mathcal{Y}_{\infty}\cap(0,+\infty))$. If $\int_{(0,+\infty)}x\,k(dx) <+\infty$, then \eqref{Ch6Sec3: EqFiniteTrace} holds according to \ref{Ch2Sec1: PropAsymptotics}, and hence, $\#(\mathcal{Y}_{\infty}\cap(0,+\infty))$ is finite is finite a.s.

We will prove \eqref{Ch6Sec3: EqSemiFinProb}. If $\int_{(0,+\infty)}x\kappa(dx)=+\infty$, then, according to \ref{Ch2Sec1: PropAsymptotics}, $u_{\downarrow}(+\infty)>0$, and thus, $\#(\mathcal{Y}_{\infty}\cap(0,+\infty))=+\infty$ a.s. Let $a<b\in\mathbb{R}$. We assume that the two first starting points in Wilson's algorithm are $a$ and $b$. Then,
\begin{equation}
\label{Ch6Sec3: EqNoPointProb}
\begin{split}
\mathbb{P}(\mathcal{Y}_{\infty}\cap(a,b]=\emptyset)=&
\mathbb{P}\big(B^{(a)}_{T_{1}^{-}}>b\big)+
\mathbb{P}\big(B^{(a)}_{T_{1}^{-}}\leq a, B^{(b)}_{T_{2}^{-}}=a\big)
\\=&\mathbb{P}\big(B^{(a)}_{\zeta_{1}^{-}}>b\big)+
\mathbb{P}\big(B^{(a)}_{\zeta_{1}^{-}}\leq a\big)\times\mathbb{P}\big(B^{(b)}~\text{hits a
before time}~\zeta_{2}\big)
\\=&\int_{(b,+\infty)}G(a,x)\kappa(dx) + \Big(\int_{(-\infty,a]}G(a,x)\kappa(dx)\Big)
\times \dfrac{u_{\downarrow}(b)}{u_{\downarrow}(a)}
\\=&\int_{(b,+\infty)}G(a,x)\kappa(dx) + u_{\downarrow}(b)
\int_{(-\infty,a]}u_{\uparrow}(x)\kappa(dx).
\end{split}
\end{equation}
Letting $b$ go to $+\infty$ in \eqref{Ch6Sec3: EqNoPointProb} gives 
\eqref{Ch6Sec3: EqSemiFinProb}.
\end{proof}

Next we will show that $\mathcal{Z}_{\infty}$ is a determinantal point process with kernel $\mathcal{K}$ relatively to the Lebesgue measure where
\begin{equation*}
\begin{split}
\mathcal{K}(y,z):=&-\dfrac{1}{2}\dfrac{du_{\uparrow}}{dx}((y\wedge z)^{+})
\dfrac{du_{\downarrow}}{dx}((y\vee z)^{-})
\\=&2\int_{(-\infty, y\wedge z]}u_{\uparrow}(x)\kappa(dx) \times
\int_{[y\vee z,+\infty)}u_{\downarrow}(x)\kappa(dx).
\end{split}
\end{equation*}

\begin{prop}
\label{Sec3PropDiamondsDet}
Let $n\geq 1$ and $a_{1}<b_{1}<a_{2}<b_{2}<\dots<a_{n}<b_{n}\in\mathbb{R}$. Then,
\begin{equation}
\label{Ch6Sec3: EqDiamondsDet}
\mathbb{E}\left[\prod_{r=1}^{n}\#
(\mathcal{Z}_{\infty}\cap(a_{r},b_{r}))\right] =
\int_{(a_{1},b_{1})}\dots\int_{(a_{n},b_{n})}\det(\mathcal{K}(z_{i},z_{j}))
_{1\leq i,j\leq n}\prod_{r=1}^{n}dz_{r}.
\end{equation}
If for $r\in\lbrace 1,2,\dots,n\rbrace$, 
$\kappa(\lbrace a_{r}\rbrace)=\kappa(\lbrace b_{r}\rbrace)=0$, then
\begin{equation}
\label{Ch6Sec3: EqExclDiamonds}
\mathbb{P}\left(\forall r\in\lbrace 1,2,\dots,n\rbrace, 
\#(\mathcal{Z}_{\infty}\cap(a_{r},b_{r}))=1 \right)=
\det(\mathcal{K}(a_{i},b_{j}))_{1\leq i,j\leq n}\times \prod_{r=1}^{n}(b_{r}-a_{r}).
\end{equation}
\end{prop}

\begin{proof}
We will only prove \eqref{Ch6Sec3: EqExclDiamonds}. \eqref{Ch6Sec3: EqDiamondsDet} can be deduced from \eqref{Ch6Sec3: EqExclDiamonds} by diving the intervals $(a_{r},b_{r})$ in small subintervals and approximating the expected number of points in these subintervals by the probability to have one single point per subinterval. Observe that, if the measure 
$\kappa$ has atoms, then $\mathcal{K}$ is not continuous. Yet, 
$z\mapsto \frac{du_{\uparrow}}{dx}(z^{+})$ is right-continuous and 
$z\mapsto \frac{du_{\downarrow}}{dx}(z^{-})$ is left-continuous.
 So the approximation can still be done.

Consider the Wilson's algorithm where the $2n$ first starting points are in order 
$a_{1},b_{1},a_{2},b_{2},\dots,a_{n},b_{n}$. Then,
\begin{multline}
\label{Ch6Sec3: EqMultiHoles}
\mathbb{P}\left(\forall r\in\lbrace 1,2,\dots,n\rbrace,
\#(\mathcal{Z}_{\infty}\cap(a_{r},b_{r}))=1 \right) =
\\\mathbb{P}\bigg(\forall r\in\lbrace 1,2,\dots,n\rbrace, 
(a_{r},b_{r})\subseteq\mathbb{R}\setminus \bigcup_{J\in \mathcal{J}_{2n}}J,
(a_{r},b_{r})\cap\mathcal{Y}_{\infty}=\emptyset \bigg).
\end{multline}
Applying Lemma \ref{Ch6Sec3: LemHole}, we get that \eqref{Ch6Sec3: EqMultiHoles} equals
\begin{equation}
\label{Ch6Sec3: EqProd1}
\mathbb{P}\bigg(\forall r\in\lbrace 1,2,\dots,n\rbrace, 
(a_{r},b_{r})\subseteq\mathbb{R}\setminus \bigcup_{J\in \mathcal{J}_{2n}}J\bigg)
\times 
\prod_{r=1}^{n}\dfrac{2(b_{r}-a_{r})}
{u_{\downarrow}(a_{r})u_{\uparrow}(b_{r})-
u_{\uparrow}(a_{r})u_{\downarrow}(b_{r})}.
\end{equation}
Further,
\begin{multline}
\label{Ch6Sec3: EqDiamondsKillPts}
\mathbb{P}\bigg(\forall r\in\lbrace 1,2,\dots,n\rbrace, 
(a_{r},b_{r})\subseteq\mathbb{R}\setminus \bigcup_{J\in \mathcal{J}_{2n}}J\bigg)=
\\\mathbb{P}\left(B^{(a_{1})}_{T_{1}^{-}}\leq a_{1},
B^{(b_{n})}_{T_{2n}^{-}}\geq b_{n},
\forall r\in\lbrace1,\dots,n-1\rbrace,
b_{r}\leq B^{(b_{r})}_{T_{2r}^{-}}\leq B^{(a_{r+1})}_{T_{2r+1}^{-}}
\leq a_{r+1}\right).
\end{multline}
Applying \eqref{Ch6Sec3: EqDeterminantnSteps} and \eqref{Ch6Sec3: EqSimplerProduct2}, we get that \eqref{Ch6Sec3: EqDiamondsKillPts} equals
\begin{multline}
\label{Ch6Sec3: EqDiamondsKillPts2}
\prod_{r=1}^{n}(u_{\downarrow}(a_{r})u_{\uparrow}(b_{r})-
u_{\uparrow}(a_{r})u_{\downarrow}(b_{r}))\times
\int_{(-\infty, a_{1}]}u_{\uparrow}(y_{1})\kappa(dy_{1})\times
\int_{[b_{n},+\infty)}u_{\downarrow}(z_{n})\kappa(dy_{n})
\\\times \prod_{r=1}^{n-1}\bigg(\kappa([b_{r},a_{r+1}])+ 
\int_{b_{r}\leq y_{r}<\tilde{y}_{r}\leq a_{r+1}}
(u_{\downarrow}(y_{r})u_{\uparrow}(\tilde{y}_{r})-
u_{\uparrow}(y_{r})u_{\downarrow}(\tilde{y}_{r}))
\kappa(dy_{r})\kappa(d\tilde{y}_{r})\bigg).
\end{multline}
But,
\begin{multline}
\label{Ch6Sec3: EqIPP1}
\int_{b_{r}\leq y_{r}<\tilde{y}_{r}\leq a_{r+1}}
u_{\downarrow}(y_{r})u_{\uparrow}(\tilde{y}_{r})\kappa(dy_{r})\kappa(d\tilde{y}_{r})
\\=\dfrac{1}{2}\int_{b_{r}\leq y_{r}\leq a_{r+1}}u_{\downarrow}(y_{r})
\Big(\dfrac{du_{\uparrow}}{dx}(a_{r+1})
-\dfrac{du_{\uparrow}}{dx}(y_{r}^{+})\Big)\kappa(dy_{r})
\\=\dfrac{1}{4}\Big(\dfrac{du_{\downarrow}}{dx}(a_{r+1})
-\dfrac{du_{\downarrow}}{dx}(b_{r})\Big)
\dfrac{du_{\uparrow}}{dx}(a_{r+1})-\dfrac{1}{2}\int_{b_{r}\leq y_{r}\leq a_{r+1}} 
u_{\downarrow}(y_{r})\dfrac{du_{\uparrow}}{dx}(y_{r}^{+})\kappa(dy_{r}),
\end{multline}
and
\begin{multline}
\label{Ch6Sec3: EqIPP2}
-\int_{b_{r}\leq y_{r}<\tilde{y}_{r}\leq a_{r+1}}
u_{\uparrow}(y_{r})u_{\downarrow}(\tilde{y}_{r})\kappa(dy_{r})\kappa(d\tilde{y}_{r})
\\=-\dfrac{1}{2}\int_{b_{r}\leq y_{r}\leq a_{r+1}}u_{\uparrow}(y_{r})
\Big(\dfrac{du_{\downarrow}}{dx}(a_{r+1})-
\dfrac{du_{\downarrow}}{dx}(y_{r}^{+})\Big)\kappa(dy_{r})
\\=-\dfrac{1}{4}\Big(\dfrac{du_{\uparrow}}{dx}(a_{r+1})-
\dfrac{du_{\uparrow}}{dx}(b_{r})\Big)\dfrac{du_{\downarrow}}{dx}(a_{r+1})+
\dfrac{1}{2}\int_{b_{r}\leq y_{r}\leq a_{r+1}} 
u_{\uparrow}(y_{r})\dfrac{du_{\downarrow}}{dx}(y_{r})\kappa(dy_{r}).
\end{multline}
Combining \eqref{Ch6Sec3: EqIPP1} and \eqref{Ch6Sec3: EqIPP2} we get that
\begin{equation*}
\begin{split}
\int_{b_{r}\leq y_{r}<\tilde{y}_{r}\leq a_{r+1}}
(u_{\downarrow}(y_{r})u_{\uparrow}(\tilde{y}_{r})-&
u_{\uparrow}(y_{r})u_{\downarrow}(\tilde{y}_{r}))
\kappa(dy_{r})\kappa(d\tilde{y}_{r})
\\=\dfrac{1}{4}\Big(\dfrac{du_{\uparrow}}{dx}(b_{r})&\dfrac{du_{\downarrow}}{dx}(a_{r+1})-
\dfrac{du_{\downarrow}}{dx}(b_{r})\dfrac{du_{\uparrow}}{dx}(a_{r+1})\Big)
\\-&\dfrac{1}{2}\int_{b_{r}\leq y_{r}\leq a_{r+1}}\Big(u_{\downarrow}(y_{r})
\dfrac{du_{\uparrow}}{dx}(y_{r}^{+})-u_{\uparrow}(y_{r})
\dfrac{du_{\downarrow}}{dx}(y_{r}^{+})\Big)\kappa(dy_{r})
\\=\dfrac{1}{4}\Big(\dfrac{du_{\uparrow}}{dx}(b_{r})&\dfrac{du_{\downarrow}}{dx}(a_{r+1})-
\dfrac{du_{\downarrow}}{dx}(b_{r})\dfrac{du_{\uparrow}}{dx}(a_{r+1})\Big)
-\kappa([b_{r},a_{r+1}]).
\end{split}
\end{equation*}
It follows that \eqref{Ch6Sec3: EqDiamondsKillPts2} equals
\begin{multline}
\label{Ch6Sec3: EqDiamondsKillPts3}
\prod_{r=1}^{n}(u_{\downarrow}(a_{r})u_{\uparrow}(b_{r})-
u_{\uparrow}(a_{r})u_{\downarrow}(b_{r}))\times\Big(-
\dfrac{1}{4}\dfrac{du_{\uparrow}}{dx}(a_{1})\dfrac{du_{\downarrow}}{dx}(b_{n})\Big)
\\\times\prod_{r=1}^{n-1}\Big(\dfrac{1}{4}\Big(\dfrac{du_{\uparrow}}{dx}(b_{r})
\dfrac{du_{\downarrow}}{dx}(a_{r+1})-\dfrac{du_{\downarrow}}{dx}(b_{r})
\dfrac{du_{\uparrow}}{dx}(a_{r+1})\Big)\Big)
\\=\dfrac{1}{2^{n}}\prod_{r=1}^{n}(u_{\downarrow}(a_{r})u_{\uparrow}(b_{r})-
u_{\uparrow}(a_{r})u_{\downarrow}(b_{r}))\times
\det(\mathcal{K}(a_{i},b_{j}))_{1\leq i,j\leq n}.
\end{multline}
\eqref{Ch6Sec3: EqProd1} together with \eqref{Ch6Sec3: EqDiamondsKillPts3} gives \eqref{Ch6Sec3: EqExclDiamonds}.
\end{proof}

To see that the operator induced by the kernel $\mathcal{K}$ on $\mathbb{L}^{2}(\operatorname{Leb})$ is positive semi-definite, one can check that for any $\mathbb{L}^{2}$ function $f$ with compact support,
\begin{displaymath}
\int_{\mathbb{R}^{2}}f(y)\mathcal{K}(y,z)f(z) dy dz = 
\int_{\mathbb{R}^{2}} G(\tilde{y},\tilde{z}) 
\left(\int_{\tilde{y}}^{\tilde{z}} f(x)dx\right)^{2}\kappa(d\tilde{y})\kappa(d\tilde{z}).
\end{displaymath}
Too see that $\mathcal{K}$ induces a contraction, one can check that, for any $\mathcal{C}^{1}$ function $f$ with compact support,
\begin{displaymath}
\int_{\mathbb{R}^{2}}f(y)\mathcal{K}(y,z)f(z) dy dz = 
\int_{\mathbb{R}}f(x)^{2} dx-
\dfrac{1}{2}\int_{\mathbb{R}^{2}}\dfrac{df}{dx}(\tilde{y})G(\tilde{y}, \tilde{z})\dfrac{df}{dx}(\tilde{z})d\tilde{y} d\tilde{z},
\end{displaymath}
and that $\int_{\mathbb{R}^{2}}\frac{df}{dx}(\tilde{y})G(\tilde{y}, \tilde{z})\frac{df}{dx}(\tilde{z})d\tilde{y} d\tilde{z}\geq 0$.

The determinantal kernels $G$ and $\mathcal{K}$ both satisfy the following relation: for any $x\leq y\leq z\in\mathbb{R}$,
\begin{equation}
\label{Ch6Sec3: EqCocyc}
G(x,y)G(y,z)=G(x,z)G(y,y) \qquad 
\mathcal{K}(x,y)\mathcal{K}(y,z)=\mathcal{K}(x,z)\mathcal{K}(y,y).
\end{equation}
For $x\in\mathbb{R}$ and $y,z>x$, we define
\begin{equation}
\label{Ch6Sec3: EqGtimesKtriang}
G^{(x\times)}(y,z):=G(y,z)-\dfrac{G(x,y)G(x,z)}{G(x,x)},\qquad
\mathcal{K}^{(x\triangleright)}(y,z):=\mathcal{K}(y,z)
-\dfrac{\mathcal{K}(x,y)\mathcal{K}(x,z)}{\mathcal{K}(x,x)}.
\end{equation}
Relation \eqref{Ch6Sec3: EqCocyc} ensures that 
$\det(G(y_{i},y_{j}))_{1\leq i,j\leq n}$ and 
$\det(\mathcal{K}(z_{i},z_{j}))_{1\leq i,j\leq n}$ can be factorized as follows:
If $y_{1}<y_{2}<\dots<y_{n}$, then
\begin{equation}
\label{Ch6Sec3: EqFacG}
\det(G(y_{i},y_{j}))_{1\leq i,j\leq n}=
G(y_{1},y_{1})\prod_{r=2}^{n}G^{(y_{r-1}\times)}(y_{r},y_{r}).
\end{equation}
If $z_{1}<z_{2}<\dots<z_{n}$, then
\begin{equation}
\label{Ch6Sec3: EqFacK}
\det(\mathcal{K}(z_{i},z_{j}))_{1\leq i,j\leq n}=
\mathcal{K}(z_{1},z_{1})\prod_{r=2}^{n}\mathcal{K}^{(z_{r-1}\triangleright)}(z_{r},z_{r}).
\end{equation}
The relations \eqref{Ch6Sec3: EqCocyc}, or equivalently the factorizations 
\eqref{Ch6Sec3: EqFacG} and \eqref{Ch6Sec3: EqFacK}, imply that the spacings between consecutive points of $\mathcal{Y}_{\infty}$, respectively $\mathcal{Z}_{\infty}$,
are independent, that is to say conditional on $\mathcal{Y}_{\infty}$ having a point at $y_{0}$, the position of the next higher point $y$ is independent on 
${\mathcal{Y}_{\infty}\cap(-\infty,y_{0})}$, and similarly for $\mathcal{Z}_{\infty}$ (\cite{Soshnikov2000Determinantal}, Section $2.4$). 
Conditional on $y_{0}\in \mathcal{Y}_{\infty}$, the distribution of its higher neighbor in $\mathcal{Y}_{\infty}$ is of the form $f_{G}(y_{0},y)\kappa(dy)$. Similarly, denote
$f_{\mathcal{K}}(z_{0},z)dz$ the distribution between two consecutive points in $\mathcal{Z}_{\infty}$ conditional on $z_{0}$ be the lowest one. Following relations relate $G^{(y_{0}\times)}(y,y)$, respectively 
$\mathcal{K}^{(z_{0}\triangleright)}(z,z)$, to $f_{G}$, respectively $f_{\mathcal{K}}$:
\begin{multline*}
G^{(y_{0}\times)}(y,y) = f_{G}(y_{0},y)
\\+\sum_{j\geq 2}\int_{y_{0}<y_{1}<\dots<y_{j-1}<y}
f_{G}(y_{0},y_{1})f_{G}(y_{1},y_{2})\dots f_{G}(y_{j-1},y)\kappa(dy_{1})\dots \kappa(dy_{j-1}),
\end{multline*}
\begin{multline*}
\mathcal{K}^{(z_{0}\triangleright)}(z,z) = f_{\mathcal{K}}(z_{0},z)
\\+\sum_{j\geq 2}\int_{z_{0}<z_{1}<\dots<z_{j-1}<z}
f_{\mathcal{K}}(z_{0},z_{1})f_{\mathcal{K}}(z_{1},z_{2})\dots f_{\mathcal{K}}(z_{j-1},z)dz_{1}\dots dz_{j-1}.
\end{multline*}
If $\int_{(0,+\infty)}x\,k(dx)<+\infty$, i.e. $\mathcal{Y}_{\infty}\cap(0,+\infty)$ a.s. finite, then $\int_{(y_{0},+\infty)}f_{G}(y_{0},y)\kappa(dy)<1$ and
$\int_{z_{0}}^{+\infty}f_{\mathcal{K}}(z_{0},z) dz<1$.

Given a couple of interwoven point processes $(\mathcal{Y},\mathcal{Z})$ on $\mathbb{R}$, such that between any two consecutive point in $\mathcal{Y}$ lies one single point of
$\mathcal{Z}$, and such that for any $J$ bounded subinterval of $\mathbb{R}$, $\mathcal{Y}$ satisfies the constraint
\begin{displaymath}
\mathbb{E}\big[\#(\mathcal{Y}\cap J)\big]<+\infty,
\end{displaymath}
the joint distribution of $(\mathcal{Y},\mathcal{Z})$ can be fully described using the family of measures $(M_{n}(\mathcal{Y},\mathcal{Z}))_{n\geq 0}$ defined by
\begin{displaymath}
\int_{\mathbb{R}}f(y_{0})M_{0}(\mathcal{Y},\mathcal{Z})(dy_{0})=
\mathbb{E}\Big[\sum_{y_{0}\in \mathcal{Y}}f(y_{0})\Big],
\end{displaymath}
\begin{multline*}
\int_{y_{0}<z_{1}<y_{1}<\dots z_{n}<y_{n}}
f(y_{0},z_{1},y_{1},\dots z_{n},y_{n})
M_{n}(\mathcal{Y},\mathcal{Z})(dy_{0},dz_{1},dy_{1},\dots dz_{n},dy_{n})\\=
\mathbb{E}\Bigg[
\sum_{
\substack{
y_{0},\dots,y_{n}\\ 
n+1~\text{consecutive points in}~\mathcal{Y}\\
z_{1},\dots, z_{n}\in \mathcal{Z}\\
y_{0}<z_{1}<y_{1}<\dots z_{n}<y_{n}
 }
}
f(y_{0},z_{1},y_{1},\dots z_{n},y_{n})\Bigg].
\end{multline*}
$M_{n}(\mathcal{Y},\mathcal{Z})(dy_{0},dz_{1},dy_{1},\dots dz_{n},dy_{n})$ is the infinitesimal probability for $y_{0},y_{1},\dots y_{n}$ being $n+1$ consecutive points in $\mathcal{Y}$, and $z_{1},\dots z_{n}$ being the $n$ points in $\mathcal{Z}$ separating them. In case of $(\mathcal{Y}_{\infty},\mathcal{Z}_{\infty})$,
$M_{0}(\mathcal{Y}_{\infty},\mathcal{Z}_{\infty})(dy_{0})=G(y_{0},y_{0})\kappa(dy_{0})$.

\begin{prop}
\label{Ch6Sec3: PropJointLaw}
For $n\geq 1$,
\begin{equation}
\label{Ch6Sec3: EqMn}
\begin{split}
M_{n}(\mathcal{Y}_{\infty},\mathcal{Z}_{\infty})
(dy_{0},dz_{1},\dots dz_{n},dy_{n})=&2^{n}u_{\uparrow}(y_{0})u_{\downarrow}(y_{n})
\kappa(dy_{0})dz_{1}\dots dz_{n}\kappa(dy_{n})\\=&
2^{n}G(y_{0},y_{n})\kappa(dy_{0})dz_{1}\dots dz_{n}\kappa(dy_{n}).
\end{split}
\end{equation}
Moreover,
\begin{displaymath}
f_{G}(y_{0},y)=2(y-y_{0})\dfrac{u_{\downarrow}(y)}{u_{\downarrow}(y_{0})}
\qquad\kappa(dy)-\text{almost everywhere},
\end{displaymath}
\begin{displaymath}
f_{\mathcal{K}}(z_{0},z)=2\kappa((z_{0},z))
\Big(\dfrac{du_{\downarrow}}{dx}(z_{0})\Big)^{-1}
\dfrac{du_{\downarrow}}{dx}(z)\qquad dz-\text{almost everywhere}.
\end{displaymath}
The distribution on $\mathcal{Z}_{\infty}$ conditional on $\mathcal{Y}_{\infty}$ is the following: given two consecutive points $y_{1}<y_{2}$ in $\mathcal{Y}_{\infty}$, then the point of $\mathcal{Z}_{\infty}$ lying between them is distributed uniformly on $(y_{1},y_{2})$ and independently on the behavior of $\mathcal{Z}_{\infty}$ on $(-\infty,y_{1})\cup(y_{2},+\infty)$.
The distribution on $\mathcal{Y}_{\infty}$ conditional on $\mathcal{Z}_{\infty}$ is the following: given two consecutive points $z_{1}<z_{2}$ in $\mathcal{Z}_{\infty}$, then the point of $\mathcal{Y}_{\infty}$ lying between them is distributed on $(z_{1},z_{2})$ according the measure 
$1_{z_{1}<y<z_{2}}\frac{\kappa(dy)}{\kappa((z_{1},z_{2}))}$
and independently on the behavior of $\mathcal{Y}_{\infty}$ on 
$(-\infty,z_{1})\cup(z_{2},+\infty)$. If $\int_{(-\infty,0)}\vert x\vert\kappa(dx)<+\infty$, then $\min\mathcal{Y}_{\infty}$ is distributed conditional on $\mathcal{Z}_{\infty}$ according to the measure
$1_{y<\min\mathcal{Z}_{\infty}}\frac{\kappa(dy)}{k((-\infty, \min\mathcal{Z}_{\infty}))}$ and it is independent on the behavior of $\mathcal{Y}_{\infty}$ on 
$(-\infty,\min\mathcal{Z}_{\infty})$. Similarly, for the distribution of $\max\mathcal{Y}_{\infty}$ conditional on $\max\mathcal{Z}_{\infty}$ if 
$\int_{(0,+\infty)}x\kappa(dx)<+\infty$.
\end{prop}

\begin{proof}
Let $a_{0}<b_{0}<\tilde{a}_{1}<\tilde{b}_{1}<a_{1}<b_{1}<\dots
<\tilde{a}_{n}<\tilde{b}_{n}<a_{n}<b_{n}\in\mathbb{R}$. Let be the event
$\mathscr{C}_{n}(a_{0},b_{0},\tilde{a}_{1},\tilde{b}_{1},a_{1},b_{1},\dots
,\tilde{a}_{n},\tilde{b}_{n},a_{n},b_{n})$, corresponding to the following conditions:
\begin{itemize}
\item $\mathcal{Y}_{\infty}\cap [a_{0},b_{0}]\neq\emptyset$, 
$\mathcal{Y}_{\infty}\cap [a_{n},b_{n}]\neq\emptyset$,
\item $\forall r\in\lbrace 1,\dots,n\rbrace, 
\#(\mathcal{Y}_{\infty}\cap [a_{r},b_{r}])=1$,
\item  $\forall r\in\lbrace 1,\dots,n\rbrace, 
\#(\mathcal{Z}_{\infty}\cap (\tilde{a}_{r},\tilde{b}_{r}))=1$,
\item $\forall r\in\lbrace 0,\dots,n-1\rbrace, (\mathcal{Y}_{\infty}\cup\mathcal{Z}_{\infty})\cap(b_{r},\tilde{a}_{r}]=
\emptyset,(\mathcal{Y}_{\infty}\cup\mathcal{Z}_{\infty})\cap[\tilde{b}_{r},a_{r+1})=
\emptyset$.
\end{itemize}
We will compute the probability of 
$\mathscr{C}_{n}(a_{0},b_{0},\tilde{a}_{1},\tilde{b}_{1},a_{1},b_{1},\dots
,\tilde{a}_{n},\tilde{b}_{n},a_{n},b_{n})$. Consider that we execute the Wilson's algorithm where the $2n$ first starting points are $\tilde{a}_{1}, \tilde{b}_{1},\dots,\tilde{a}_{n}, \tilde{b}_{n}$. The only configurations that contribute to the studied event are those where $B^{(\tilde{a}_{1})}_{T_{1}^{-}}\in[a_{0},b_{0}]$, $B^{(\tilde{b}_{n})}_{T_{2n}^{-}}\in[a_{n},b_{n}]$ and for 
$r\in\lbrace 1,\dots,n-1\rbrace$, $B^{(\tilde{b}_{r})}_{T_{2r}^{-}}=
B^{(\tilde{a}_{r+1})}_{T_{2r+1}^{-}}\in[a_{r+1},b_{r+1}]$. We further need that for $r\in\lbrace 1,\dots,n\rbrace$, $\mathcal{Y}_{\infty}\cap(\tilde{a}_{r},\tilde{b}_{r})=\emptyset$. Thus, applying \eqref{Ch6Sec3: EqDeterminantnSteps}, 
\eqref{Ch6Sec3: EqSimplerProduct2} and lemma \ref{Ch6Sec3: LemHole} we get that
the probability of the event $\mathscr{C}_{n}(a_{0},b_{0},\tilde{a}_{1},\tilde{b}_{1},a_{1},b_{1},\dots,\tilde{a}_{n},\tilde{b}_{n},a_{n},b_{n})$ equals
\begin{displaymath}
\int_{[a_{0},b_{0}]}u_{\uparrow}(y_{0})\kappa(dy_{0})\times
\int_{[a_{n},b_{n}]}u_{\downarrow}(y_{n})\kappa(dy_{n})\times
\prod_{r=1}^{n-1}\kappa([a_{r},b_{r}])\times
\prod_{r=1}^{n}2(\tilde{b}_{r}-\tilde{a}_{r}).
\end{displaymath}
The above probability also equals 
$M_{n}(\mathcal{Y}_{\infty},\mathcal{Z}_{\infty})([a_{0},b_{0}]\times
[\tilde{a}_{1},\tilde{b}_{1}]\times[a_{1},b_{1}]\times\dots\times
[\tilde{a}_{n},\tilde{b}_{n}]\times[a_{n},b_{n}])$, and gives the expression of 
\eqref{Ch6Sec3: EqMn}. To get the expressions of $f_{G}$ and $f_{\mathcal{K}}$, just observe that
\begin{displaymath}
G(y_{0},y_{0})f_{G}(y_{0},y)\kappa(dy_{0})\kappa(dy)=
M_{1}([y_{0},y_{0}+dy_{0}]\times(y_{0},y)\times[y,y+dy]),
\end{displaymath}
\begin{displaymath}
\mathcal{K}(z_{0},z_{0})f_{\mathcal{K}}(z_{0},z) dz_{0} dz=
M_{3}((-\infty,z_{0})\times[z_{0},z_{0}+dz_{0}]\times(z_{0},z)\times[z,z+dz]
\times (z,+\infty)).
\end{displaymath}

Expression \eqref{Ch6Sec3: EqMn} gives also the law of $\mathcal{Z}_{\infty}$ conditional on $\mathcal{Y}_{\infty}$, and the law of $\mathcal{Y}_{\infty}$ conditional on $\mathcal{Z}_{\infty}$, except for the possible extremal points of $\mathcal{Y}_{\infty}$. Let's deal with the distribution of $\max \mathcal{Y}_{\infty}$ conditional on $\max \mathcal{Z}_{\infty}$ in case 
$\int_{(0,+\infty)}x\kappa(dx)<+\infty$. Again, according to \eqref{Ch6Sec3: EqMn}, conditional on $z_{0}\in\mathcal{Z}_{\infty}$, the distribution of 
$\min \mathcal{Y}_{\infty}\cap(z_{0},+\infty)$ is proportional to $1_{y>z_{0}}u_{\downarrow}(y)\kappa(dy)$.
To obtain the distribution of $\max \mathcal{Y}_{\infty}$ conditional on 
$\max \mathcal{Z}_{\infty}$, one must weight $u_{\downarrow}(y)$ by 
$1-\int_{\tilde{y}>y}f_{G}(y,\tilde{y})\kappa(d\tilde{y})$, i.e. the probability of not having any point in $\mathcal{Y}_{\infty}$ consecutive to $y$. But,
\begin{equation*}
\begin{split}
\int_{\tilde{y}>y}f_{G}(y,\tilde{y})\kappa(d\tilde{y})=&2\int_{\tilde{y}>y}(\tilde{y}-y)
\dfrac{u_{\downarrow}(\tilde{y})}{u_{\downarrow}(y)}\kappa(d\tilde{y})\\=&
\lim_{\tilde{y}\rightarrow +\infty}\dfrac{\tilde{y}-y}{u_{\downarrow}(y)}
\dfrac{du_{\downarrow}}{dx}(\tilde{y}^{+})-\dfrac{1}{u_{\downarrow}(y)}\int_{\tilde{y}>y}
\dfrac{du_{\downarrow}}{dx}(\tilde{y}^{+}) d\tilde{y}.
\end{split}
\end{equation*}
Further,
\begin{displaymath}
(\tilde{y}-y)\dfrac{du_{\downarrow}}{dx}(\tilde{y}^{+})=
(\tilde{y}-y)\int_{(\tilde{y},+\infty)}2u_{\downarrow}(x)\kappa(dx)
\leq 2\int_{(\tilde{y},+\infty)}(x-y)u_{\downarrow}(x)\kappa(dx)\rightarrow 0.
\end{displaymath}
It follows that
\begin{displaymath}
\int_{\tilde{y}>y}f_{G}(y,\tilde{y})\kappa(d\tilde{y})=-\dfrac{1}{u_{\downarrow}(y)}\int_{\tilde{y}>y}\dfrac{du_{\downarrow}}{dx}(\tilde{y}^{+}) d\tilde{y}=
1-\dfrac{u_{\downarrow}(+\infty)}{u_{\downarrow}(y)}.
\end{displaymath}
Thus, $1_{y>z_{0}}u_{\downarrow}(y)\big(1-\int_{\tilde{y}>y}
f_{G}(y,\tilde{y})\kappa(d\tilde{y})\big)\kappa(dy)$ is simply proportional to $1_{y>z_{0}}\kappa(dy)$.
\end{proof}

\begin{prop}
\label{Ch6Sec3: PropUniqPoint}
In case $\int_{\mathbb{R}}\vert x\vert\kappa(dx)<+\infty$,
\begin{displaymath}
\mathbb{P}(\# \mathcal{Y}_{\infty}=1)=u_{\uparrow}(-\infty)u_{\downarrow}(+\infty).
\kappa(\mathbb{R})
\end{displaymath}
Conditional on $\# \mathcal{Y}_{\infty}=1$, the unique point in $\mathcal{Y}_{\infty}$ is distributed according $\frac{\kappa(dy)}{\kappa(\mathbb{R})}$.
\end{prop}

\begin{proof}
The distribution of the unique point $y_{0}$ of $\mathcal{Y}_{\infty}$ on the event 
$\# \mathcal{Y}_{\infty}=1$ is given by the following sieve identity:
\begin{equation*}
\begin{split}
\bigg(G(y_{0},y_{0})-&\int_{y_{-1}<y_{0}}G(y_{-1},y_{-1})f_{G}(y_{-1},y_{0})\kappa(dy_{-1})
\\-&\int_{y_{1}>y_{0}}G(y_{0},y_{0})f_{G}(y_{0},y_{1})\kappa(dy_{1})\\
+&\int_{y_{-1}<y_{0}}\int_{y_{1}>y_{0}}G(y_{-1},y_{-1})f_{G}(y_{-1},y_{0})f_{G}(y_{0},y_{1})
\kappa(dy_{-1})\kappa(dy_{1})\bigg)\kappa(dy_{0}).
\end{split}
\end{equation*}
It is the infinitesimal probability of $\mathcal{Y}_{\infty}$ having a point at $y_{0}$, minus the probability of having a point at $y_{0}$ and an other lower,
minus the probability of having a point at $y_{0}$ and an other higher, plus the 
probability of having a point at $y_{0}$ surrounded by two neighbors on both sides. The identity can be further factorized as
\begin{multline*}
\bigg(u_{\uparrow}(y_{0})-2\int_{y_{-1}<y_{0}}(y_{0}-y_{-1})u_{\uparrow}(y_{-1})\kappa(dy_{-1})\bigg)
\\\times\bigg(u_{\downarrow}(y_{0})-2\int_{y_{1}>y_{0}}(y_{1}-y_{0})u_{\downarrow}(y_{1})\kappa(dy_{1})\bigg)\times\kappa(dy_{0}). 
\end{multline*}
According to the calculation done in the proof of Proposition \ref{Ch6Sec3: PropJointLaw}, the above equals $u_{\uparrow}(-\infty)u_{\downarrow}(+\infty) \kappa(dy_{0})$.
\end{proof}

Now, let us describe $(\mathcal{Y}_{\infty},\mathcal{Z}_{\infty})$ in two particular cases. If the killing rate is uniform, that is to say $\kappa(dy)=c dy$, where $c$ is constant, then
\begin{displaymath}
c f_{G}(x_{0},x)=f_{\mathcal{K}}(x_{0},x)=2c(x-x_{0})e^{-\sqrt{2c}(x-x_{0})}.
\end{displaymath}
Both the spacings of $\mathcal{Y}_{\infty}$ and $\mathcal{Z}_{\infty}$ are
i.i.d. gamma-$2$ variables with mean $\sqrt{\frac{2}{c}}$. Actually, the union 
$\mathcal{Y}_{\infty}\cup\mathcal{Z}_{\infty}$ is a Poisson point process with intensity $\sqrt{2c} dx$. If the killing measure is of form $\kappa=c\sum_{j\in\mathbb{Z}}\delta_{j}$,
where $c$ is constant, then again the spacings between consecutive points in $\mathcal{Y}_{\infty}$ are i.i.d random variables, this time integer valued. Let
$N_{2}$ be a random variable with same distribution as this spacings. For any $j\in\mathbb{N}$
\begin{displaymath}
\mathbb{P}(N_{2}=j)=2c j (1+\sqrt{2c})^{-j}.
\end{displaymath}
$N_{2}$ can be written as $N_{2}=N_{1}+\widetilde{N}_{1}-1$, where $N_{1}$ and $\widetilde{N}_{1}$ are two independent geometric variables of parameter $(1+\sqrt{2c})^{-1}$. Actually, if $y_{0}<y$ are two consecutive points in $\mathcal{Y}_{\infty}$ and $z$ the point of $\mathcal{Z}_{\infty}$ lying between them, then, conditional on $y_{0}$, $(\lfloor z\rfloor-y_{0}, y-\lfloor z\rfloor)$ has the same law as $(N_{1}-1,\widetilde{N})$. Moreover, 
$\lbrace\lfloor z\rfloor\vert z\in\mathcal{Z}_{\infty}\rbrace$ has the same law as
$\mathcal{Y}_{\infty}$.

\section{Determinantal point processes $(\mathcal{Y}_{\infty},\mathcal{Z}_{\infty})$: general case}
\label{Ch6Sec4}

Let $I$ be an open subinterval of $\mathbb{R}$ and $L$ be the generator of a transient diffusion on $I$ of form
$L=\frac{1}{m(x)}\frac{d}{dx}\left(\frac{1}{w(x)}\frac{d}{dx}\right)-\kappa$, with zero Dirichlet boundary conditions on 
$\partial I$ with sample path denoted $(X_{t})_{0\leq t<\zeta}$.
We will describe, without proof, the law of
$(\mathcal{Y}_{\infty},\mathcal{Z}_{\infty})$ in this generic case. It can be derived in the same way as it was done in the previous section. Let $G$ be the Green's function of $L$ relatively to the measure 
$m(y) dy$, factorizable as 
$G(x,y)=u_{\uparrow}(x\wedge y)u_{\downarrow}(x\vee y)$.

\begin{prop}
\label{Ch6Sec4: PropJointLawGeneral}
$\mathcal{Y}_{\infty}$ and $\mathcal{Z}_{\infty}$ are a.s. discrete point processes. Let $\partial I$ be the boundary of $I$ in 
$\mathbb{R}\cup\lbrace -\infty,+\infty\rbrace$. Almost surely,
\begin{displaymath}
\mathcal{Y}_{\infty}\cap\partial I=
\left\lbrace y\in\partial I\vert
\mathbb{P}(X_{\zeta^{-}}=y)>0\right\rbrace.
\end{displaymath}
If $\kappa\neq 0$, the points in $\mathcal{Y}_{\infty}\cap I$ are a determinantal point process with determinantal kernel $G(x,y)$ relatively the reference measure $m(y)\kappa(dy)$. $\mathcal{Z}_{\infty}$ is a determinantal point process on $I$ with determinantal kernel
\begin{displaymath}
\dfrac{du_{\uparrow}}{dx}((y\wedge z)^{+})
\dfrac{du_{\downarrow}}{dx}((y\vee z)^{-}),
\end{displaymath}
relatively to the reference measure $\frac{dz}{w(z)}$.
Given two consecutive points $y_{1}<y_{2}$ in $\mathcal{Y}_{\infty}$, then the point of $\mathcal{Z}_{\infty}$ lying between them is distributed according to the measure 
$1_{y_{1}<z<y_{2}}\frac{w(z) dz}{\int_{(y_{1},y_{2})}w(a) da}$,
and independently on the behavior of $\mathcal{Z}_{\infty}$ on $(-\infty,y_{1})\cup(y_{2},+\infty)$.
Given two consecutive points $z_{1}<z_{2}$ in $\mathcal{Z}_{\infty}$, then the point of $\mathcal{Y}_{\infty}$ lying between them is distributed on $(z_{1},z_{2})$ according the measure 
$1_{z_{1}<y<z_{2}}\frac{m(y)\kappa(dy)}{\int_{(z_{1},z_{2})}m(q)\kappa(dq)}$,
and independently on the behavior of $\mathcal{Y}_{\infty}$ on 
$(-\infty,z_{1})\cup(z_{2},+\infty)$.
\end{prop}

\chapter{Monotone couplings for the point processes $(\mathcal{Y}_{\infty},\mathcal{Z}_{\infty})$}
\label{Ch7}

\section{Conditioning}
\label{Ch7Sec1}

In this chapter we will deal with monotone coupling for the determinantal point processes $\mathcal{Y}_{\infty}$ and $\mathcal{Z}_{\infty}$ intruded in Chapter \ref{Ch6}. We will restrict to the Brownian case. Consider two different killing measures $\kappa$ and $\tilde{\kappa}$ on $\mathbb{R}$, with $\kappa\leq \tilde{\kappa}$, and the couples of determinantal point processes $(\mathcal{Y}_{\infty}, \mathcal{Z}_{\infty})$, respectively $(\widetilde{\mathcal{Y}}_{\infty}, \widetilde{\mathcal{Z}}_{\infty})$ corresponding to the Brownian motion on $\mathbb{R}$ with killing measure $\kappa$, respectively $\tilde{\kappa}$. We will show that one can couple $(\mathcal{Y}_{\infty}, \mathcal{Z}_{\infty})$ and $(\widetilde{\mathcal{Y}}_{\infty}, \widetilde{\mathcal{Z}}_{\infty})$ on the same probability space such that $\mathcal{Z}_{\infty}\subseteq\widetilde{\mathcal{Z}}_{\infty}$ and
${\widetilde{\mathcal{Y}}_{\infty}\subseteq\mathcal{Y}_{\infty}\cup \operatorname{Supp}(\tilde{\kappa}-\kappa)}$. Moreover, if $\kappa$ and $\tilde{\kappa}$ are proportional, we may also have $\mathcal{Y}_{\infty}\subseteq\widetilde{\mathcal{Y}}_{\infty}$. We will provide an explicit construction for the this couplings in Section \ref{Ch7Sec2}. 

In Section \ref{Ch7Sec1}, we will prove conditioning results for $(\mathcal{Y}_{\infty},\mathcal{Z}_{\infty})$: what is obtained if $\mathcal{Y}_{\infty}$ or $\mathcal{Z}_{\infty}$ is conditioned by either containing a point at a given location or not containing any points in a given interval. These results will be used in the next section.
The conditional law we will obtain are analogous to those of the uniform spanning tree on a finite undirected connected graph: Let $\mathbb{G}$ be such a graph, $E$ the set of its edges, $C$ a weight function on $E$ and $\Upsilon$ the corresponding uniform spanning tree on $\mathbb{G}$.
Let $E_{1}$ and $E_{2}$ be two disjoint subsets of $E$ such that $E_{1}$ contains no cycles and such that erasing the edges in $E_{2}$ does not disconnect $\mathbb{G}$. The law of $\Upsilon$ conditioned by $E_{1}\subseteq \Upsilon$ and $\Upsilon\cap E_{2} = \emptyset$ can be described as follows: Let $\mathbb{G}'$ be the graph obtained from $\mathbb{G}$ trough erasing the edges in $E_{2}$ and contracting (i.e. identifying the two end vertices) the edges in $E_{1}$. The edges of $\mathbb{G}'$ are in one to one correspondence with $E\setminus E_{2}$. If we keep the same weight function $C$ on these edges and take 
$\Upsilon'$ an uniform spanning tree on $\mathbb{G}'$, then $\Upsilon'\cup E_{1}$ has the same law as $\Upsilon$ conditional on $E_{1}\subseteq \Upsilon$ and $\Upsilon\cap E_{2} = \emptyset$ (see Proposition $4.2$ in \cite{BenjaminiLyonsPeresSchramm2001UnifSpanFor}).

Let $\kappa$ be a Radon measure on $\mathbb{R}$ and 
$G(x,y)=u_{\uparrow}(x\wedge y)u_{\downarrow}(x\vee y)$ the Green's function of $\frac{1}{2}\frac{d^{2}}{dx^{2}}-\kappa$.
First, we will restrict the Brownian motion with killing measure $\kappa$ to a half-line by adding either a killing or a reflecting boundary point and describe what is obtained if we apply the Wilson's algorithm to it.
This is related to some of the conditional laws we are interested in.
Diffusions with reflection were not discussed so far.

For $x_{0}<y$, let
\begin{displaymath}
u^{(x_{0}\times)}_{\uparrow}(y):=u_{\uparrow}(y)
-\dfrac{u_{\uparrow}(x_{0})}{u_{\downarrow}(x_{0})}u_{\uparrow}(y),
\end{displaymath}
and for $x_{0}<y,z$, let
\begin{displaymath}
G^{(x_{0}\times)}(y,z):=u^{(x_{0}\times)}_{\uparrow}(y\wedge z)
u_{\downarrow}(y\vee z),
\end{displaymath}
\begin{displaymath}
\mathcal{K}^{(x_{0}\times)}(y,z):=-\dfrac{1}{2}
\dfrac{du^{(x_{0}\times)}_{\uparrow}}{dx}((y\wedge z)^{+})
\dfrac{du_{\downarrow}}{dx}((y\vee z)^{-}).
\end{displaymath}
$G^{(x_{0}\times)}$ was already introduced in 
\eqref{Ch6Sec3: EqGtimesKtriang}.
For $y<x_{0}$, let
\begin{displaymath}
u^{(\times x_{0})}_{\downarrow}(y):=u_{\downarrow}(y)
-\dfrac{u_{\downarrow}(x_{0})}{u_{\uparrow}(x_{0})}u_{\uparrow}(y),
\end{displaymath}
and for $y,z<x_{0}$, let
\begin{displaymath}
G^{(\times x_{0})}(y,z):=u_{\uparrow}(y\wedge z)
u^{(\times x_{0})}_{\downarrow}(y\vee z),
\end{displaymath}
\begin{displaymath}
\mathcal{K}^{(\times x_{0})}(y,z):=-\dfrac{1}{2}
\dfrac{du_{\uparrow}}{dx}((y\wedge z)^{+})
\dfrac{du^{(\times x_{0})}_{\downarrow}}{dx}((y\vee z)^{-}).
\end{displaymath}
$G^{(x_{0}\times)}$, respectively $G^{(\times x_{0})}$, is the Green's function of $\frac{1}{2}\frac{d^{2}}{dx^{2}}-\kappa$ restricted to the interval $(x_{0}, +\infty)$, respectively $(-\infty,x_{0})$, with zero Dirichlet boundary condition at $x_{0}$.

Let $x_{0}\in\mathbb{R}$ such that $\kappa(\lbrace x_{0}\rbrace)=0$.
For $x_{0}<y$, let
\begin{displaymath}
u^{(x_{0}\triangleright)}_{\uparrow}(y):=u_{\uparrow}(y)+
\Big(\dfrac{du_{\downarrow}}{dx}(x_{0})\Big)^{-1}
\dfrac{du_{\uparrow}}{dx}(x_{0})u_{\downarrow}(y),
\end{displaymath}
and for $y,z<x_{0}$, let
\begin{displaymath}
G^{(x_{0}\triangleright)}(y,z):=u^{(x_{0}\triangleright)}_{\uparrow}(y\wedge z)u_{\downarrow}(y\vee z),
\end{displaymath}
\begin{displaymath}
\mathcal{K}^{(x_{0}\triangleright)}(y,z):=-\dfrac{1}{2}
\dfrac{du^{(x_{0}\triangleright)}_{\uparrow}}{dx}((y\wedge z)^{+})
\dfrac{du_{\downarrow}}{dx}((y\vee z)^{-}).
\end{displaymath}
$\mathcal{K}^{(x_{0}\triangleright)}$ was already introduced in 
\eqref{Ch6Sec3: EqGtimesKtriang}.
For $y<x_{0}$, let
\begin{displaymath}
u^{(\triangleleft x_{0})}_{\downarrow}(y):=u_{\downarrow}(y)+
\Big(\dfrac{du_{\uparrow}}{dx}(x_{0})\Big)^{-1}
\dfrac{du_{\downarrow}}{dx}(x_{0})u_{\uparrow}(y),
\end{displaymath}
and for $y,z<x_{0}$, let
\begin{displaymath}
G^{(\triangleleft x_{0})}(y,z):=u_{\uparrow}(y\wedge z)
u^{(\triangleleft x_{0})}_{\downarrow}(y\vee z),
\end{displaymath}
\begin{displaymath}
\mathcal{K}^{(\triangleleft x_{0})}(y,z):=-\dfrac{1}{2}
\dfrac{du_{\uparrow}}{dx}((y\wedge z)^{+})
\dfrac{du^{(\triangleleft x_{0})}_{\downarrow}}{dx}((y\vee z)^{-}).
\end{displaymath}
$G^{(x_{0}\triangleright)}$, respectively $G^{(\triangleleft x_{0})}$, is the Green's function of $\frac{1}{2}\frac{d^{2}}{dx^{2}}-\kappa$ restricted to the interval $[x_{0}, +\infty)$, respectively $(-\infty,x_{0}]$, with zero Neumann boundary condition at $x_{0}$. Equivalently, $G^{(x_{0}\triangleright)}$, respectively 
$G^{(\triangleleft x_{0})}$, is the restriction to $[x_{0}, +\infty)$, respectively $(-\infty,x_{0}]$, of the Green's function on $\mathbb{R}$ of 
$\frac{1}{2}\frac{d^{2}}{dx^{2}}-1_{[x_{0},+\infty)}\kappa$, respectively
$\frac{1}{2}\frac{d^{2}}{dx^{2}}-1_{(-\infty,x_{0}]}\kappa$.

Consider now $x_{0}\in\mathbb{R}$ and $(x_{n})_{n\geq 1}$ a dense sequence of pairwise disjoint points in $(x_{0},+\infty)$. We consider the Wilson's algorithm applied to the Brownian motion on $(x_{0},+\infty)$, with killing measure $\kappa$ and killing boundary $x_{0}$, where $(x_{n})_{n\geq 0}$ is the sequence of starting points. 
Let $\mathcal{Y}_{\infty}^{(x_{0}\times)}$ and $\mathcal{Z}_{\infty}^{(x_{0}\times)}$  be the interwoven point processes in $[x_{0},+\infty)$ obtained as result. See Figure \ref{FigC1} for an illustration of the first four steps of Wilson's algorithm and of $(\mathcal{Y}_{\infty}^{(x_{0}\times)},\mathcal{Z}_{\infty}^{(x_{0}\times)})$. According to Proposition \ref{Ch6Sec4: PropJointLawGeneral}, $x_{0}\in\mathcal{Y}_{\infty}^{(x_{0}\times)}$ a.s.,
$\mathcal{Y}_{\infty}^{(x_{0}\times)}\cap(x_{0},+\infty)$  is a determinantal point process with determinantal kernel $G^{(x_{0}\times)}$ relatively to the measure $1_{(x_{0},+\infty)}\kappa$ and $\mathcal{Z}_{\infty}^{(x_{0}\times)}$ is a determinantal point process with kernel $\mathcal{K}^{(x_{0}\times)}$ relatively to the measure
$1_{z>x_{0}} dz$. The distribution of the $2n$ closest to $x_{0}$ points in $(\mathcal{Y}_{\infty}^{(x_{0}\times)}\cap(x_{0},+\infty))\cup\mathcal{Z}_{\infty}^{(x_{0}\times)}$, the odd-numbered belonging to $\mathcal{Y}_{\infty}^{(x_{0}\times)}\cap(x_{0},+\infty)$ and the even-numbered to $\mathcal{Z}_{\infty}^{(x_{0}\times)}$, is given by the measure
\begin{displaymath}
M_{n}^{(x_{0}\times)}(\mathcal{Y}_{\infty}^{(x_{0}\times)},\mathcal{Z}_{\infty}^{(x_{0}\times)})
(dz_{1},dy_{1},\dots ,dz_{n},dy_{n}):=
2^{n}\dfrac{u_{\downarrow}(y_{n})}{u_{\downarrow}(x_{0})}
dz_{1}\kappa(dy_{1})\dots dz_{n}\kappa(dy_{n}).
\end{displaymath}
Its total mass equals 
$\mathbb{P}(\# \mathcal{Y}_{\infty}^{(x_{0}\times)}\geq n+1)$. If the Wilson's algorithm is applied to the Brownian motion on $(-\infty,x_{0})$, killed at $x_{0}$ and with killing measure $\kappa$, and $(\mathcal{Y}_{\infty}^{(\times x_{0})},
\mathcal{Z}_{\infty}^{(\times x_{0})})$ are  the point processes returned by the algorithm, then the distribution of the $2n$ closest to $x_{0}$ points in $(\mathcal{Y}_{\infty}^{(\times x_{0})}\cap
(-\infty,x_{0} ))\cup\mathcal{Z}_{\infty}^{(\times x_{0})}$ is given by the measure
\begin{multline*}
M_{n}^{(\times x_{0})}(\mathcal{Y}_{\infty}^{(\times x_{0})},\mathcal{Z}_{\infty}^{(\times x_{0})})
(dz_{-1},dy_{-1},\dots ,dz_{-n},dy_{-n}):=\\
2^{n}\dfrac{u_{\uparrow}(y_{-n})}{u_{\uparrow}(x_{0})}
dz_{-1}\kappa(dy_{-1})\dots dz_{-n}\kappa(dy_{-n}).
\end{multline*}

Let now $x_{0}\in\mathbb{R}$ such that $\kappa(\lbrace x_{0}\rbrace)=0$. If we replace the Brownian motion on $(x_{0},+\infty)$ killed in $x_{0}$ by a Brownian motion on $[x_{0},+\infty)$ reflected in $x_{0}$, and keep the killing measure $\kappa$, we get another pair 
$(\mathcal{Y}_{\infty}^{(x_{0}\triangleright)},\mathcal{Z}_{\infty}^{(x_{0}\triangleright)})$ of interwoven point processes on $[x_{0},+\infty)$. The pair $(\mathcal{Y}_{\infty}^{(x_{0}\triangleright)},\mathcal{Z}_{\infty}^{(x_{0}\triangleright)})$ can be also obtained by applying Wilson's algorithm to a Brownian motion on $\mathbb{R}$ with the killing measure $1_{(x_{0},+\infty)}\kappa$. See Figure \ref{FigC2} for an illustration of $(\mathcal{Y}_{\infty}^{(x_{0}\triangleright)},\mathcal{Z}_{\infty}^{(x_{0}\triangleright)})$. Observe the difference with Figure \ref{FigC1} at the third step of Wilson's algorithm. $\mathcal{Y}_{\infty}^{(x_{0}\triangleright)}$ is a determinantal point process with determinantal kernel $G^{(x_{0}\triangleright)}$ relatively to the measure $1_{(x_{0},+\infty)}\kappa$. $\mathcal{Z}_{\infty}^{(x_{0}\triangleright)}$ is a determinantal point process with kernel $\mathcal{K}^{(x_{0}\triangleright)}$ relatively to the measure $1_{z>x_{0}} dz$. The distribution of the $2n-1$ closest to $x_{0}$ points in $\mathcal{Y}_{\infty}^{(x_{0}\triangleright)}\cup
\mathcal{Z}_{\infty}^{(x_{0}\triangleright)}$, the odd-numbered belonging to $\mathcal{Z}_{\infty}^{(x_{0}\triangleright)}$ and the even-numbered to $\mathcal{Y}_{\infty}^{(x_{0}\triangleright)}$, is given by the measure
\begin{multline*}
M_{n}^{(x_{0}\triangleright)}(\mathcal{Y}_{\infty}^{(x_{0}\triangleright)},\mathcal{Z}_{\infty}^{(x_{0}\triangleright)})
(dy_{1},dz_{1},\dots dz_{n-1},dy_{n}):=\\
-2^{n}\Big(\dfrac{du_{\downarrow}}{dx}(x_{0})\Big)^{-1}
u_{\downarrow}(y_{n})\kappa(dy_{1})dz_{1}\dots dz_{n-1}\kappa(dy_{n}).
\end{multline*}
If the Wilson's algorithm is applied to the Brownian motion on $(-\infty,x_{0}]$, reflected at $x_{0}$ and with killing measure $\kappa$, and $(\mathcal{Y}_{\infty}^{(\triangleleft x_{0})},
\mathcal{Z}_{\infty}^{(\triangleleft x_{0})})$ are  the point processes returned by the algorithm, then the distribution of the $2n-1$ closest to $x_{0}$ points in $\mathcal{Y}_{\infty}^{(\triangleleft x_{0})}
\cup\mathcal{Z}_{\infty}^{(\triangleleft x_{0})}$ is given by the measure
\begin{multline*}
M_{n}^{(\triangleleft x_{0})}
(\mathcal{Y}_{\infty}^{(\triangleleft x_{0})},\mathcal{Z}_{\infty}^{(\triangleleft x_{0})})
(dy_{-1},dz_{-1},\dots dz_{-n+1},dy_{-n}):=\\
2^{n}\Big(\dfrac{du_{\uparrow}}{dx}(x_{0})\Big)^{-1}
u_{\uparrow}(y_{-n})\kappa(dy_{-1})dz_{-1}\dots dz_{-n+1}\kappa(dy_{-n}).
\end{multline*}

\begin{figure}[H]
\centering{
\includegraphics[width=0.95\textwidth]{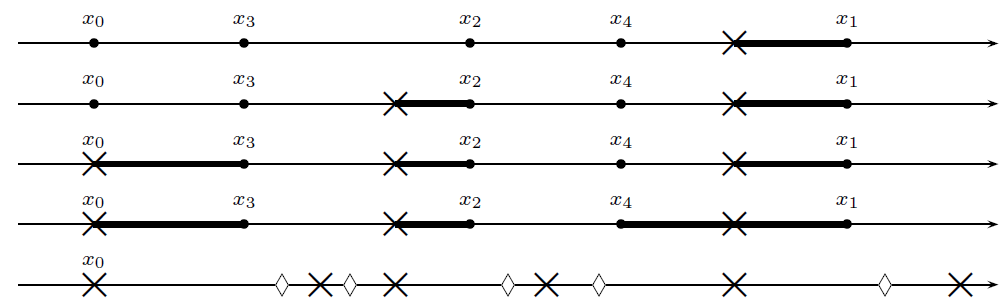}
}
\caption{Illustration of the first four steps of Wilson's algorithm in case of killing at $x_{0}$ and of $(\mathcal{Y}_{\infty}^{(x_{0}\times)},\mathcal{Z}_{\infty}^{(x_{0}\times)})$:
x-dots represent the points of $\mathcal{Y}_{n}^{(x_{0}\times)}$,
diamonds the points of $\mathcal{Z}_{n}^{(x_{0}\times)}$,
and thick lines the intervals in $\mathcal{J}_{n}^{(x_{0}\times)}$.
}
\label{FigC1}
\end{figure} 

\begin{figure}[H]
\centering{\includegraphics[width=0.95\textwidth]{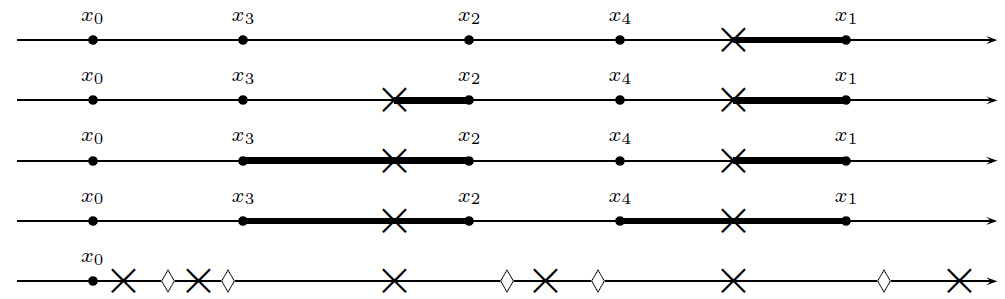}
}
\caption{Illustration of the first four steps of Wilson's algorithm in case of reflection at $x_{0}$
and of $(\mathcal{Y}_{\infty}^{(x_{0}\triangleright)},
\mathcal{Z}_{\infty}^{(x_{0}\triangleright)})$: 
x-dots represent the points of $\mathcal{Y}_{n}^{(x_{0}\triangleright)}$,
diamonds the points of $\mathcal{Z}_{n}^{(x_{0}\triangleright)}$,
and thick lines the intervals in $\mathcal{J}_{n}^{(x_{0}\triangleright)}$.
}
\label{FigC2}
\end{figure}

Let $\mathcal{Y}_{\infty}$ and $\mathcal{Z}_{\infty}$ be the determinantal point processes associated to the Brownian motion on $\mathbb{R}$ with killing measure $\kappa$. Let $n,n'\in\mathbb{N}^{\ast}$. The following two factorizations hold:
\begin{multline*}
M_{n+n'}(\mathcal{Y}_{\infty},\mathcal{Z}_{\infty})
(dy_{-n'},dz_{-n'},\dots dy_{-1},dz_{-1},dy_{0},dz_{1},dy_{1},\dots,
dz_{n},dy_{n})=\\
M_{n'}^{(\times y_{0})}(\mathcal{Y}_{\infty}^{(\times y_{0})},\mathcal{Z}_{\infty}^{(\times y_{0})})
(dz_{-1},dy_{-1},\dots ,dz_{-n'},dy_{-n'})\times
G(y_{0},y_{0})\kappa(dy_{0})\\\times
M_{n}^{(y_{0}\times)}(\mathcal{Y}_{\infty}^{(y_{0}\times)},\mathcal{Z}_{\infty}^{(y_{0}\times)})
(dz_{1},dy_{1},\dots ,dz_{n},dy_{n}),
\end{multline*}
\begin{multline*}
M_{n+n'-1}(\mathcal{Y}_{\infty},\mathcal{Z}_{\infty})
(dy_{-n'},dz_{-n'+1},\dots dz_{-1},dy_{-1},dz_{0},dy_{1},dz_{1},\dots,
dz_{n-1},dy_{n})=\\
M_{n'}^{(\triangleleft z_{0})}
(\mathcal{Y}_{\infty}^{(\triangleleft z_{0})},\mathcal{Z}_{\infty}^{(\triangleleft z_{0})})
(dy_{-1},dz_{-1},\dots dz_{-n'+1},dy_{-n'})\times
\mathcal{K}(z_{0},z_{0}) dz_{0}\\\times
M_{n}^{(z_{0}\triangleright)}(\mathcal{Y}_{\infty}^{(z_{0}\triangleright)},\mathcal{Z}_{\infty}^{(z_{0}\triangleright)})
(dy_{1},dz_{1},\dots dz_{n-1},dy_{n}).
\end{multline*} 
The above factorizations imply the following:

\begin{property}
\label{Ch7Sec1: PropertyCondPres}
Let $\varepsilon>0$ and let $F_{1}$ and $F_{2}$ be two measurable non-negative functionals on couples of point processes on $\mathbb{R}$ and $f$ a measurable non-negative function on $\mathbb{R}$. Then,
\begin{multline*}
\mathbb{E}\bigg[
\sum_{y_{0}\in\mathcal{Y}_{\infty}}
f(y_{0})F_{1}(\mathcal{Y}_{\infty}\cap(-\infty, y_{0}],
\mathcal{Z}_{\infty}\cap(-\infty, y_{0}])
F_{2}(\mathcal{Y}_{\infty}\cap[y_{0},+\infty),
\mathcal{Z}_{\infty}\cap[y_{0},+\infty))\bigg]
\\=\int_{\mathbb{R}}f(y_{0})G(y_{0},y_{0})
\mathbb{E}[F_{1}(\mathcal{Y}_{\infty}^{(\times y_{0})},
\mathcal{Z}_{\infty}^{(\times y_{0})})]
\mathbb{E}[F_{2}(\mathcal{Y}_{\infty}^{(y_{0}\times)},
\mathcal{Z}_{\infty}^{(y_{0}\times)})]\kappa(dy_{0}),
\end{multline*}
and
\begin{multline*}
\mathbb{E}\bigg[
\sum_{z_{0}\in\mathcal{Z}_{\infty}}f(z_{0})
F_{1}(\mathcal{Y}_{\infty}\cap(-\infty, z_{0}],
\mathcal{Z}_{\infty}\cap(-\infty, z_{0}])
F_{2}(\mathcal{Y}_{\infty}\cap[z_{0},+\infty),
\mathcal{Z}_{\infty}\cap[z_{0},+\infty))\bigg]
\\=\int_{\mathbb{R}}f(z_{0})\mathcal{K}(z_{0},z_{0})
\mathbb{E}[F_{1}(\mathcal{Y}_{\infty}^{(\triangleleft z_{0})},
\mathcal{Z}_{\infty}^{(\triangleleft z_{0})})]
\mathbb{E}[F_{2}(\mathcal{Y}_{\infty}^{(z_{0}\triangleright)},
\mathcal{Z}_{\infty}^{(z_{0}\triangleright)})] dz_{0}.
\end{multline*}

If $y_{0}\in \operatorname{Supp}(\kappa)$, then, conditional on $y_{0}\in\mathcal{Y}_{\infty}$, 
$(\mathcal{Y}_{\infty}\cap(-\infty,y_{0}],\mathcal{Z}_{\infty}\cap(-\infty,y_{0}])$
and ${(\mathcal{Y}_{\infty}\cap[y_{0},+\infty),\mathcal{Z}_{\infty}
\cap[y_{0},+\infty))}$ are independent, $(\mathcal{Y}_{\infty}\cap(-\infty,y_{0}],\mathcal{Z}_{\infty}\cap(-\infty,y_{0}])$ has the same law as 
$(\mathcal{Y}_{\infty}^{(\times y_{0})},
\mathcal{Z}_{\infty}^{(\times y_{0})})$
and ${(\mathcal{Y}_{\infty}\cap[y_{0},+\infty),
\mathcal{Z}_{\infty}\cap[y_{0},+\infty))}$ has the same law as $(\mathcal{Y}_{\infty}^{(y_{0}\times)},\mathcal{Z}_{\infty}^{(y_{0}\times)})$.

If $\kappa((-\infty,z_{0}))>0$, $\kappa((z_{0},+\infty))>0$ and
$\kappa(\lbrace z_{0}\rbrace)=0$, then, conditional on $z_{0}\in\mathcal{Z}_{\infty}$, 
$(\mathcal{Y}_{\infty}\cap(-\infty,z_{0}],\mathcal{Z}_{\infty}\cap(-\infty,z_{0}])$
and $(\mathcal{Y}_{\infty}\cap[z_{0},+\infty),
\mathcal{Z}_{\infty}\cap[z_{0},+\infty))$ are independent, $(\mathcal{Y}_{\infty}\cap(-\infty,z_{0}],\mathcal{Z}_{\infty}\cap(-\infty,z_{0}])$ has the same law as 
$(\mathcal{Y}_{\infty}^{(\triangleleft z_{0})},
\mathcal{Z}_{\infty}^{(\triangleleft z_{0})})$
and $(\mathcal{Y}_{\infty}\cap[z_{0},+\infty) 
\mathcal{Z}_{\infty}\cap[z_{0},+\infty))$ has the same law as $(\mathcal{Y}_{\infty}^{(z_{0}\triangleright)},
\mathcal{Z}_{\infty}^{(z_{0}\triangleright)})$.
\end{property}

Let $y_{0}\in\mathbb{R}$ and $c>0$. We will denote by $(\mathcal{Y}^{(y_{0})}_{\infty},\mathcal{Z}^{(y_{0})}_{\infty})$ the pair of interwoven determinantal point processes corresponding to the killing measure $\kappa+c\delta_{y_{0}}$, conditioned on $\mathcal{Y}^{(y_{0})}_{\infty}$ containing $y_{0}$. The law of $(\mathcal{Y}^{(y_{0})}_{\infty},\mathcal{Z}^{(y_{0})}_{\infty})$ does not depend on the value of $c$ according to 
Property \ref{Ch7Sec1: PropertyCondPres}. $(\mathcal{Y}^{(y_{0})}_{\infty}\cap (y_{0},+\infty),\mathcal{Z}^{(y_{0})}_{\infty}\cap (y_{0},+\infty))$ and $(\mathcal{Y}^{(y_{0})}_{\infty}\cap (-\infty,y_{0}),\mathcal{Z}^{(y_{0})}_{\infty}\cap(-\infty,y_{0}))$ are independent. The
distribution of the $2n$ closest to $y_{0}$ points in $(\mathcal{Y}^{(y_{0})}_{\infty}\cup\mathcal{Z}^{(y_{0})}_{\infty})\cap(y_{0},+\infty)$, on the event 
$\#(\mathcal{Y}^{(y_{0})}_{\infty}\cap(y_{0},+\infty))\geq n$, is
\begin{equation}
\label{Ch7Sec1: EqCondPresy0}
1_{y_{0}<z_{1}<y_{1}<\dots<z_{n}<y_{n}}2^{n}\dfrac{u_{\downarrow}(y_{n})}{u_{\downarrow}(y_{0})} dz_{1}\kappa(dy_{1})\dots dz_{n}\kappa(dy_{n}).
\end{equation}
The distribution of the $2n$ closest to $y_{0}$ points in $(\mathcal{Y}^{(y_{0})}_{\infty}\cup\mathcal{Z}^{(y_{0})}_{\infty})
\cap(-\infty,y_{0})$ is
\begin{equation}
\label{Ch7Sec1: EqCondPresy0bis}
1_{y_{0}>z_{-1}>y_{-1}>\dots>z_{-n}>y_{-n}}2^{n}\dfrac{u_{\uparrow}(y_{-n})}{u_{\uparrow}(y_{0})} dz_{-1}\kappa(dy_{-1})
\dots dz_{-n}\kappa(dy_{-n}).
\end{equation}

Let $a<b\in\mathbb{R}$. Next we will describe what happens if we condition by $\mathcal{Z}_{\infty}\cap[a,b]=\emptyset$. This condition implies in particular that
$\#(\mathcal{Y}_{\infty}\cap[a,b])\leq 1$. Let $\widehat{\mathbb{R}}$ be the quotient space, where in $\mathbb{R}$ we identify to one point all the points lying in $[a,b]$. $\widehat{\mathbb{R}}$ is homeomorphic to $\mathbb{R}$. Let $\hat{\pi}$ be the projection from $\mathbb{R}$ to $\widehat{\mathbb{R}}$. Let $\theta$ be the class of $[a,b]$ in $\widehat{\mathbb{R}}$. We define on $\widehat{\mathbb{R}}$ the metric $d_{\widehat{\mathbb{R}}}$:
\begin{itemize}
\item If $x<y<a$ or $b<x<y$, then $d_{\widehat{\mathbb{R}}}
(\hat{\pi}(x),\hat{\pi}(y))=y-x$.
\item If $x<a$ and $y>b$, then 
$d_{\widehat{\mathbb{R}}}(\hat{\pi}(x),\hat{\pi}(y))=(y-x)-(b-a)$.
\item If $x<a$, then $d_{\widehat{\mathbb{R}}}(\hat{\pi}(x),\theta)=a-x$.
\item If $x>b$, then $d_{\widehat{\mathbb{R}}}(\hat{\pi}(x),\theta)=x-b$.
\end{itemize}
$\widehat{\mathbb{R}}$ endowed with $d_{\widehat{\mathbb{R}}}$ is isometric to $\mathbb{R}$. So, we can define a standard Brownian motion on $\widehat{\mathbb{R}}$. Let $\hat{\kappa}$ be the measure $\kappa$ pushed forward by $\hat{\pi}$ on $\widehat{\mathbb{R}}$. In particular 
$\hat{\kappa}(\lbrace \theta\rbrace)=\kappa([a,b])$. Let $(\widehat{\mathcal{Y}}_{\infty},\widehat{\mathcal{Z}}_{\infty})$ be the pair of interwoven determinantal point processes on $\widehat{\mathbb{R}}$ obtained by applying the Wilson's algorithm to the Brownian motion on $\widehat{\mathbb{R}}$ with killing measure $\hat{\kappa}$.

\begin{prop}
\label{Ch7Sec1: PropEdgeContr}
Conditional on $\mathcal{Z}_{\infty}\cap[a,b]=\emptyset$, 
$(\hat{\pi}(\mathcal{Y}_{\infty}),\hat{\pi}(\mathcal{Z}_{\infty}))$ has the same distribution as
$(\widehat{\mathcal{Y}}_{\infty},\widehat{\mathcal{Z}}_{\infty})$. Moreover, on the event $\mathcal{Y}_{\infty}\cap[a,b]\neq\emptyset$, the unique point in $\mathcal{Y}_{\infty}\cap[a,b]$ is distributed according the probability measure 
$\frac{1_{a\leq y\leq b}\kappa(dy)}{\kappa([a,b])}$.
\end{prop}

\begin{proof}
First, we compute $\mathbb{P}(\mathcal{Z}_{\infty}\cap[a,b]=\emptyset)$. We consider that $a$ and $b$ are the first two starting points in the Wilson's algorithm. Then,
\begin{equation*}
\begin{split}
\mathbb{P}(\mathcal{Z}_{\infty}\cap[a,b]=\emptyset)=&
\mathbb{P}\Big(B^{(a)}_{T_{1}^{-}}>b\Big)+
\mathbb{P}\Big(B^{(a)}_{T_{1}^{-}}<a, B^{(b)}_{T_{2}^{-}}=a\Big)+
\mathbb{P}\Big(B^{(a)}_{T_{1}^{-}}=B^{(b)}_{T_{2}^{-}}\in[a,b]\Big)
\\=&\dfrac{1}{2}\dfrac{du_{\uparrow}}{dx}(a^{-})u_{\downarrow}(b)
-\dfrac{1}{2}u_{\uparrow}(a)\dfrac{du_{\downarrow}}{dx}(b^{+})+
u_{\uparrow}(a)u_{\downarrow}(b)\kappa([a,b]).
\end{split}
\end{equation*}

Next, we determine the Green's function $\widehat{G}$ of 
$\frac{1}{2}\frac{d^{2}}{d\tilde{x}^{2}}-\hat{\kappa}$ on $\widehat{\mathbb{R}}$. Let $\hat{u}_{\uparrow}$ and $\hat{u}_{\downarrow}$ be two solutions on $\widehat{\mathbb{R}}$ to
\begin{displaymath}
\dfrac{1}{2}\dfrac{d\hat{u}}{dx}
-\hat{u}\hat{\kappa}=0,
\end{displaymath}
with the initial conditions $\hat{u}_{\uparrow}(\theta)=u_{\uparrow}(a)$, 
$\frac{d\hat{u}_{\uparrow}}{dx}(\theta^{-})
=\frac{du_{\uparrow}}{dx}(a^{-})$, 
$\hat{u}_{\downarrow}(\theta)=u_{\downarrow}(b)$ and 
$\frac{d\hat{u}_{\downarrow}}{dx}(\theta^{+})
=\frac{du_{\downarrow}}{dx}(b^{+})$. Then, for $x\leq a$, 
$\hat{u}_{\uparrow}(\hat{\pi}(x))=u_{\uparrow}(x)$ and for $x\geq b$, 
$\hat{u}_{\downarrow}(\hat{\pi}(x))=u_{\downarrow}(x)$. $\hat{u}_{\uparrow}$ and $\hat{u}_{\downarrow}$ are positive, $\hat{u}_{\uparrow}$ is non-decreasing and 
$\hat{u}_{\downarrow}$ non-increasing. Moreover,
\begin{displaymath}
\dfrac{d\hat{u}_{\uparrow}}{dx}(\theta^{+})=
\dfrac{d\hat{u}_{\uparrow}}{dx}(\theta^{-})+
2\hat{u}_{\uparrow}(\theta)\hat{\kappa}(\lbrace\theta\rbrace)=
\dfrac{du_{\uparrow}}{dx}(a^{-})+2u_{\uparrow}(a)\kappa([a,b]).
\end{displaymath}
The Wronskian of $\hat{u}_{\downarrow}$ and $\hat{u}_{\uparrow}$ equals
\begin{equation*}
\begin{split}
W(\hat{u}_{\downarrow},\hat{u}_{\uparrow})=&
\hat{u}_{\downarrow}(\theta)\dfrac{d\hat{u}_{\uparrow}}{dx}(\theta^{+})-
\hat{u}_{\uparrow}(\theta)
\dfrac{d\hat{u}_{\downarrow}}{dx}(\theta^{+})\\=&
\dfrac{du_{\uparrow}}{dx}(a^{-})u_{\downarrow}(b)
-u_{\uparrow}(a)\dfrac{du_{\downarrow}}{dx}(b^{+})+
2u_{\uparrow}(a)u_{\downarrow}(b)\kappa([a,b])
\\=&2\mathbb{P}(\mathcal{Z}_{\infty}\cap[a,b]=\emptyset).
\end{split}
\end{equation*}
Thus, $\widehat{G}$ equals
\begin{displaymath}
\widehat{G}(\tilde{x},\tilde{y})=
\dfrac{\hat{u}_{\uparrow}(\tilde{x}\wedge\tilde{y})
\hat{u}_{\downarrow}(\tilde{x}\vee\tilde{y})}
{\mathbb{P}(\mathcal{Z}_{\infty}\cap[a,b]=\emptyset)}.
\end{displaymath}
In particular, if $x\leq a$ and $y\geq b$, then
\begin{equation}
\label{Ch7Sec1: EqGreenFuncProb2}
\widehat{G}(\hat{\pi}(x),\hat{\pi}(y))
=\dfrac{u_{\uparrow}(x)u_{\downarrow}(y)}
{\mathbb{P}(\mathcal{Z}_{\infty}\cap[a,b]=\emptyset)}=
\dfrac{G(x,y)}{\mathbb{P}(\mathcal{Z}_{\infty}\cap[a,b]=\emptyset)}.
\end{equation}

To prove the equality in law, we need to consider the probabilities of the events
$\mathscr{C}_{n}(a_{0},b_{0},\tilde{a}_{1},\tilde{b}_{1},a_{1},b_{1},\dots,\tilde{a}_{n},\tilde{b}_{n},a_{n},b_{n})$, 
where $n\geq 1$ and $a_{0}<b_{0}<\tilde{a}_{1}<\tilde{b}_{1}<a_{1}<b_{1}<\dots
<\tilde{a}_{n}<\tilde{b}_{n}<a_{n}<b_{n}\in\mathbb{R}$,
corresponding to following conditions:
\begin{itemize}
\item $\mathcal{Y}_{\infty}\cap [a_{0},b_{0}]\neq\emptyset$, 
$\mathcal{Y}_{\infty}\cap [a_{n},b_{n}]\neq\emptyset$,
\item $\forall r\in\lbrace 1,\dots,n\rbrace, 
\#(\mathcal{Y}_{\infty}\cap [a_{r},b_{r}])=1$,
\item  $\forall r\in\lbrace 1,\dots,n\rbrace, 
\#(\mathcal{Z}_{\infty}\cap (\tilde{a}_{r},\tilde{b}_{r}))=1$,
\item $\forall r\in\lbrace 0,\dots,n-1\rbrace, (\mathcal{Y}_{\infty}\cup\mathcal{Z}_{\infty})\cap(b_{r},\tilde{a}_{r}]=
\emptyset,(\mathcal{Y}_{\infty}\cup\mathcal{Z}_{\infty})\cap[\tilde{b}_{r},a_{r+1})=
\emptyset$.
\end{itemize}
We will also assume that either all of the $[a_{r},b_{r}]$ do not intersect $[a,b]$, or one of the $[a_{r},b_{r}]$ is contained in 
$[a,b]$ and the other do not intersect $[a,b]$.
The probabilities of such events determine the joint law of 
$(\mathcal{Y}_{\infty},\mathcal{Z}_{\infty})$ on the event
$\# \mathcal{Y}_{\infty}\geq 2, \mathcal{Z}_{\infty}\cap[a,b]=\emptyset$. We will denote $\widehat{\mathscr{C}}_{n}(\cdot)$ the analogously defined events, where we replace 
$(\mathcal{Y}_{\infty},\mathcal{Z}_{\infty})$ by $(\widehat{\mathcal{Y}}_{\infty},\widehat{\mathcal{Z}}_{\infty})$. We do not need to deal with the event $\# \mathcal{Y}_{\infty}=1$, because then $\mathcal{Z}_{\infty}=\emptyset$.

We first consider the case of $[a,b]\cap\big(\bigcup_{r=0}^{n}[a_{r},b_{r}]\big)=\emptyset$. If there is $r_{0}\in \lbrace 0, n-1\rbrace$ such that
$b_{r_{0}}<a$ and $b<a_{r_{0}+1}$, then
\begin{multline*}
\mathbb{P}\left(\mathscr{C}_{n}(a_{0},b_{0},\tilde{a}_{1},\tilde{b}_{1},a_{1},b_{1},\dots,\tilde{a}_{n},\tilde{b}_{n},a_{n},b_{n}),
\mathcal{Z}_{\infty}\cap[a,b]=\emptyset\right)
\\=\int_{[a_{0},b_{0}]}u_{\uparrow}(y_{0})\kappa(dy_{0})\times
\int_{[a_{n},b_{n}]}u_{\downarrow}(y_{n})\kappa(dy_{n})\times
\prod_{r=1}^{n-1}\kappa([a_{r},b_{r}])\\\times
\prod_{r\neq r_{0}}2(\tilde{b}_{r}-\tilde{a}_{r})\times
2\operatorname{Leb}([\tilde{a}_{r_{0}},\tilde{b}_{r_{0}}]\setminus[a,b]).
\end{multline*}
Using \eqref{Ch7Sec1: EqGreenFuncProb2}, we get that the above equals
\begin{displaymath}
\mathbb{P}(\mathcal{Z}_{\infty}\cap[a,b]=\emptyset)\times
\mathbb{P}\left(\widehat{\mathscr{C}}_{n}(\hat{\pi}(a_{0}),\hat{\pi}(b_{0}),\hat{\pi}(\tilde{a}_{1}),\hat{\pi}(\tilde{b}_{1}),\dots
,\hat{\pi}(a_{n}),\hat{\pi}(b_{n}))\right).
\end{displaymath}
If $b<a_{0}$, then we consider a Wilson's algorithm, where the $2(n+1)$ first starting points are $\tilde{a}_{1},\tilde{b}_{1},\dots,\tilde{a}_{n},\tilde{b}_{n},a,b$.
The conditions $\mathscr{C}_{n}(a_{0},b_{0},\tilde{a}_{1},\tilde{b}_{1},a_{1},b_{1},\dots,\tilde{a}_{n},\tilde{b}_{n},a_{n},b_{n})$ and ${\mathcal{Z}_{\infty}\cap[a,b]=\emptyset}$ are satisfied if and only if the following is true:
\begin{itemize}
\item $B^{(\tilde{a}_{1})}_{T_{1}^{-}}\in[a_{0},b_{0}]$,
$B^{(\tilde{b}_{n})}_{T_{2n}^{-}}\in[a_{n},b_{n}]$, for all 
$r\in\lbrace 1,\dots,n-1\rbrace$, $B^{(\tilde{b}_{r})}_{T_{2r}^{-}}=B^{(\tilde{a}_{r+1})}_{T_{2r+1}^{-}}
\in[a_{r},b_{r}]$ and for all $r\in\lbrace 1,\dots,n\rbrace$, 
$\mathcal{Y}_{\infty}\cap(\tilde{a}_{r},\tilde{b}_{r})=\emptyset$.
\item Either $B^{(a)}_{T_{2n+1}^{-}}\in (b,B^{(\tilde{a}_{1})}_{T_{1}^{-}}]$ or
$B^{(b)}_{T_{2n+2}^{-}}<a$ or $B^{(a)}_{T_{2n+1}^{-}}=B^{(b)}_{T_{2n+2}^{-}}\in [a,b]$.
\end{itemize}
Then,
\begin{equation*}
\begin{split}
\mathbb{P}&\Big(\mathscr{C}_{n}(a_{0},b_{0},\tilde{a}_{1},\tilde{b}_{1},a_{1},b_{1},\dots,\tilde{a}_{n},\tilde{b}_{n},a_{n},b_{n}),
\mathcal{Z}_{\infty}\cap[a,b]=\emptyset\Big)
\\=&\int_{[a_{n},b_{n}]}u_{\downarrow}(y_{n})\kappa(dy_{n})\times
\prod_{r=1}^{n-1}\kappa([a_{r},b_{r}])\times
\prod_{r=1}^{n}2(\tilde{b}_{r}-\tilde{a}_{r})
\\&\times \bigg(u_{\uparrow}(a)\int\limits_{b<y<y_{0},y_{0}\in[a_{0},b_{0}]}(u_{\downarrow}(y)u_{\uparrow}(y_{0})-u_{\uparrow}(y)u_{\downarrow}(y_{0}))\kappa(dy)\kappa(dy_{0})+
u_{\uparrow}(a)\kappa([a_{0},b_{0}])
\\&+\Big(\int_{y_{-1}<a}u_{\uparrow}(y_{-1})\kappa(dy_{-1})
+u_{\uparrow}(a)\kappa([a_{0},b_{0}])\Big)
\\&\times\int_{[a_{0},b_{0}]}(u_{\downarrow}(b)u_{\uparrow}(y_{0})
-u_{\uparrow}(b)u_{\downarrow}(y_{0}))\kappa(dy_{0})\bigg)
\\=&\int_{[a_{n},b_{n}]}u_{\downarrow}(y_{n})\kappa(dy_{n})\times
\prod_{r=1}^{n-1}\kappa([a_{r},b_{r}])\times
\prod_{r=1}^{n}2(\tilde{b}_{r}-\tilde{a}_{r})
\\&\times\Big(\dfrac{1}{2}u_{\uparrow}(a)\int_{[a_{0},b_{0}]}
\Big(\dfrac{du_{\downarrow}}{dx}(b^{+})u_{\uparrow}(y_{0})-
\dfrac{du_{\uparrow}}{dx}(b^{+})u_{\downarrow}(y_{0})\Big)\kappa(dy_{0})
\\&+\Big(\dfrac{1}{2}\dfrac{du_{\uparrow}}{dx}(a^{-})+u_{\uparrow}(a)\kappa([a_{0},b_{01}])\Big)
\int_{[a_{0},b_{0}]}(u_{\downarrow}(b)u_{\uparrow}(y_{0})-u_{\uparrow}(b)u_{\downarrow}(y_{0}))\kappa(dy_{0})\Big).
\end{split}
\end{equation*}
But, for $y_{0}\geq b$,
\begin{equation*}
\begin{split}
\hat{u}_{\uparrow}(\hat{\pi}(y_{0}))=&\dfrac{1}{2}u_{\uparrow}(a)\Big(\dfrac{du_{\downarrow}}{dx}(b^{+})u_{\uparrow}(y_{0})-
\dfrac{du_{\uparrow}}{dx}(b^{+})u_{\downarrow}(y_{0})\Big)
\\+&\Big(\dfrac{1}{2}\dfrac{du_{\uparrow}}{dx}(a^{-})
+u_{\uparrow}(a)\kappa([a_{0},b_{0}])\Big)(u_{\downarrow}(b)
u_{\uparrow}(y_{0})-u_{\uparrow}(b)u_{\downarrow}(y_{0})).
\end{split}
\end{equation*}
Indeed, one can check the initial conditions 
$\hat{u}_{\uparrow}(\hat{\pi}(b))=u_{\uparrow}(a)$ and $\frac{d\hat{u}_{\uparrow}}{dx}(\hat{\pi}(b)^{+})=
\frac{du_{\uparrow}}{dx}(a^{-})+2u_{\uparrow}(a)\kappa([a_{0},b_{0}])$. It follows that
\begin{equation*}
\begin{split}
\mathbb{P}&\Big(\mathscr{C}_{n}(a_{0},b_{0},\tilde{a}_{1},\tilde{b}_{1},\dots,a_{n},b_{n}),\mathcal{Z}_{\infty}\cap[a,b]=\emptyset\Big)
\\=&\int\limits_{[\hat{\pi}(a_{0}),\hat{\pi}(b_{0})]}\hat{u}_{\uparrow}(\tilde{y}_{0})\kappa(d\tilde{y}_{0})\times
\int\limits_{[\hat{\pi}(a_{n}),\hat{\pi}(b_{n})]}\hat{u}_{\downarrow}(\tilde{y}_{n})\kappa(d\tilde{y}_{n})\times
\prod_{r=1}^{n-1}\kappa([a_{r},b_{r}])\times
\prod_{r=1}^{n}2(\tilde{b}_{r}-\tilde{a}_{r})
\\=&\mathbb{P}(\mathcal{Z}_{\infty}\cap[a,b]=\emptyset)\times
\mathbb{P}\left(\widehat{\mathscr{C}}_{n}(\hat{\pi}(a_{0}),\hat{\pi}(b_{0}),\hat{\pi}(\tilde{a}_{1}),\hat{\pi}(\tilde{b}_{1}),\dots
,\hat{\pi}(a_{n}),\hat{\pi}(b_{n}))\right).
\end{split}
\end{equation*}
Similar holds if $b_{n}<a$.

Now, we consider the case when there is $r_{0}\in\lbrace 0,\dots,n\rbrace$ such that
$[a_{r_{0}},b_{r_{0}}]\subseteq[a,b]$ and 
$[a,b]\cap\big(\bigcup_{r\neq r_{0}}[a_{r},b_{r}]\big)=\emptyset$. If 
$1\leq r_{0}\leq n-1$, then
\begin{equation*}
\begin{split}
\mathbb{P}&\left(\mathscr{C}_{n}(a_{0},b_{0},\tilde{a}_{1},\tilde{b}_{1},\dots,a_{n},b_{n}),
\mathcal{Z}_{\infty}\cap[a,b]=\emptyset\right)
\\=&\int_{[a_{0},b_{0}]}u_{\uparrow}(y_{0})\kappa(dy_{0})\times
\int_{[a_{n},b_{n}]}u_{\downarrow}(y_{n})\kappa(dy_{n})\times
\prod_{r=1}^{n-1}\kappa([a_{r},b_{r}])
\\&\times\prod_{r\neq r_{0},r_{0}+1}2(\tilde{b}_{r}-\tilde{a}_{r})
\times 2 \operatorname{Leb}([\tilde{a}_{r_{0}},\tilde{b}_{r_{0}}]\setminus [a,b])
\times 2 \operatorname{Leb}([\tilde{a}_{r_{0}+1},\tilde{b}_{r_{0}+1}]\setminus [a,b])
\\=&\dfrac{\kappa([a_{r_{0}},b_{r_{0}}])}{\kappa([a,b])}\times
\mathbb{P}(\mathcal{Z}_{\infty}\cap[a,b]=\emptyset)
\\&\times
\mathbb{P}\left(\widehat{\mathscr{C}}_{n}(\hat{\pi}(a_{0}),\hat{\pi}(b_{0}),\hat{\pi}(\tilde{a}_{1}),\hat{\pi}(\tilde{b}_{1}),\dots
,\hat{\pi}(a_{n}),\hat{\pi}(b_{n}))\right).
\end{split}
\end{equation*}
Moreover, $\hat{\pi}(a_{r_{0}})=\hat{\pi}(b_{r_{0}})=\theta$.
If $r_{0}=0$, then
\begin{equation*}
\begin{split}
\mathbb{P}\Big(\mathscr{C}_{n}(a_{0},b_{0},\tilde{a}_{1},\tilde{b}_{1},\dots&,a_{n},b_{n}),\mathcal{Z}_{\infty}\cap[a,b]=\emptyset\Big)
\\=&u_{\uparrow}(a)\kappa([a_{0},b_{0}])\times
\int_{[a_{n},b_{n}]}u_{\downarrow}(y_{n})\kappa(dy_{n})\\&\times
\prod_{r=1}^{n-1}\kappa([a_{r},b_{r}])\times
\prod_{r=2}^{n}2(\tilde{b}_{r}-\tilde{a}_{r})\times
2 \operatorname{Leb}([\tilde{a}_{1},\tilde{b}_{1}]\setminus [a,b])
\\=&\dfrac{\kappa([a_{0},b_{0}])}{\kappa([a,b])}\times
\mathbb{P}(\mathcal{Z}_{\infty}\cap[a,b]=\emptyset)
\\&\times
\mathbb{P}\left(\widehat{\mathscr{C}}_{n}(\hat{\pi}(a_{0}),\hat{\pi}(b_{0}),\hat{\pi}(\tilde{a}_{1}),\hat{\pi}(\tilde{b}_{1}),\dots
,\hat{\pi}(a_{n}),\hat{\pi}(b_{n}))\right),
\end{split}
\end{equation*}
and $\hat{\pi}(a_{0})=\hat{\pi}(b_{0})=\theta$. We have a similar expression if $r_{0}=n$.
\end{proof}

Next, we deal with the condition of the determinantal point process 
$\mathcal{Y}_{\infty}$ not charging a given subinterval of $\mathbb{R}$. We will consider the following more general situation:
Let $\kappa$ and $\tilde{\kappa}$ be two different killing measures on $\mathbb{R}$, with $\kappa\leq \tilde{\kappa}$, and the couples of determinantal point processes $(\mathcal{Y}_{\infty}, \mathcal{Z}_{\infty})$ respectively $(\widetilde{\mathcal{Y}}_{\infty}, \widetilde{\mathcal{Z}}_{\infty})$ corresponding to the Brownian motion on $\mathbb{R}$ with killing measure $\kappa$ respectively $\tilde{\kappa}$. Let $\widetilde{G}$ be the Green's function of 
$\frac{1}{2}\frac{d^{2}}{dx^{2}}-\tilde{\kappa}$, factorized as 
\begin{displaymath}
\widetilde{G}(x,y)=\tilde{u}_{\uparrow}(x\wedge y)
\tilde{u}_{\downarrow}(x\vee y).
\end{displaymath}
Let
\begin{displaymath}
\widetilde{\mathcal{K}}(y,z):=-\dfrac{1}{2}
\dfrac{d\tilde{u}_{\uparrow}}{dx}((y\wedge z)^{+})
\dfrac{d\tilde{u}_{\downarrow}}{dx}((y\vee z)^{-}).
\end{displaymath}
We will assume that $\tilde{\kappa}-\kappa$ has a first moment, that is to say
\begin{displaymath}
\int_{\mathbb{R}}\vert x\vert (\tilde{\kappa}(dx)-\kappa(dx))<+\infty.
\end{displaymath}
Let $\chi$ be the Radon-Nikodym derivative
\begin{displaymath}
\chi:=\dfrac{d\kappa}{d\tilde{\kappa}}.
\end{displaymath}
By definition, $0\leq \chi\leq 1$. Let $\Delta\widetilde{\mathcal{Y}}$ be the point process obtained from $\widetilde{\mathcal{Y}}_{\infty}$ as follows: Given a point $y$ in $\widetilde{\mathcal{Y}}_{\infty}$, we chose to erase it with probability $\chi(y)$ and keep it with probability $1-\chi(y)$, each choice being independent from the other choices and the position of other points. It is immediate to check that
$\Delta\widetilde{\mathcal{Y}}$ is a determinantal point process with determinantal kernel 
$(\widetilde{G}(x,y))_{x,y\in\mathbb{R}}$ relatively to the measure $(1-\chi)\tilde{\kappa}$, that is to say the measure $\tilde{\kappa}-\kappa$. We will show that conditional on 
$\Delta\widetilde{\mathcal{Y}}=\emptyset$, 
$(\widetilde{\mathcal{Y}}_{\infty}, \widetilde{\mathcal{Z}}_{\infty})$,
has the same law as $(\mathcal{Y}_{\infty}, \mathcal{Z}_{\infty})$.
In case $1-\chi$ being the indicator function of a bounded subinterval of $\mathbb{R}$, this gives the law of $(\widetilde{\mathcal{Y}}_{\infty}, \widetilde{\mathcal{Z}}_{\infty})$ conditioned on $\widetilde{\mathcal{Y}}_{\infty}$ not charging this subinterval.

\begin{lemm}
\label{Ch7Sec1: LemExtrDet}
$\Delta\widetilde{\mathcal{Y}}$ is a.s. finite. Let
\begin{multline*}
v_{\kappa,\tilde{\kappa}}(y):=\bigg(\tilde{u}_{\uparrow}(y)
-\int_{y_{-1}<y}\tilde{u}_{\uparrow}(y_{-1})(u_{\downarrow}(y_{-1})
u_{\uparrow}(y)-u_{\uparrow}(y_{-1})
u_{\downarrow}(y))(\tilde{\kappa}-\kappa)(dy_{-1})\bigg)\\
\times\bigg(\tilde{u}_{\downarrow}(y)
-\int_{y_{1}>y}\tilde{u}_{\downarrow}(y_{1})(u_{\uparrow}(y_{1})
u_{\downarrow}(y)-u_{\downarrow}(y_{1})
u_{\uparrow}(y))(\tilde{\kappa}-\kappa)(dy_{1})\bigg).
\end{multline*}
Then,
\begin{displaymath}
\mathbb{P}\big(\# \Delta\widetilde{\mathcal{Y}}=1\big)= \int_{\mathbb{R}}v_{\kappa,\tilde{\kappa}}(y)
(\tilde{\kappa}-\kappa)(dy).
\end{displaymath}
The distribution of the unique point in $\Delta\widetilde{\mathcal{Y}}$ conditional on $\# \Delta\widetilde{\mathcal{Y}}=1$ is
\begin{displaymath}
\dfrac{v_{\kappa,\tilde{\kappa}}(y)(\tilde{\kappa}-\kappa)(dy)}
{\mathbb{P}\big(\# \Delta\widetilde{\mathcal{Y}}=1\big)}.
\end{displaymath}
Furthermore,
\begin{displaymath}
\mathbb{P}\big(\# \Delta\widetilde{\mathcal{Y}}\geq 2\big)\leq
\dfrac{1}{2}\Big(\int_{\mathbb{R}}\widetilde{G}(y,y)
(\tilde{\kappa}(dy)-\kappa(dy))\Big)^{2},
\end{displaymath}
and $\mathbb{P}(\Delta\widetilde{\mathcal{Y}}=\emptyset)>0$.
\end{lemm}

\begin{proof}
First, let us check that 
$\int_{\mathbb{R}}\widetilde{G}(y,y)(\tilde{\kappa}(dy)-\kappa(dy))<+\infty$. 
Since $\tilde{\kappa}-\kappa$ has a first moment, we need only to show that $\widetilde{G}(y,y)$ grows sub-linearly in the neighborhood of $-\infty$ and $+\infty$. 
Let $a<b\in\mathbb{R}$ such that $\tilde{\kappa}((a,b))>0$. 
Let $\widetilde{G}_{a,b}$ be the Green's function of $\frac{1}{2}\frac{d^{2}}{dx^{2}}-1_{(a,b)}\tilde{\kappa}$. 
Then $\widetilde{G}_{a,b}(y,y)$ is affine on $(-\infty,a)$ and on $(b,+\infty)$. 
Moreover $\widetilde{G}(y,y)\leq\widetilde{G}_{a,b}(y,y)$. 
Thus, we get
\begin{displaymath}
\mathbb{E}\big[\# \Delta\widetilde{\mathcal{Y}}\big]=
\int_{\mathbb{R}}\widetilde{G}(y,y)
(\tilde{\kappa}(dy)-\kappa(dy))<+\infty.
\end{displaymath}
In particular $\Delta\widetilde{\mathcal{Y}}$ is a.s. finite.

To bound $\mathbb{P}\big(\# \Delta\widetilde{\mathcal{Y}}\geq 2\big)$
we use the following:
\begin{equation*}
\begin{split}
\mathbb{P}\big(\# \Delta\widetilde{\mathcal{Y}}\geq 2\big)&
\leq\dfrac{1}{2}\mathbb{E}\big[\#\Delta\widetilde{\mathcal{Y}}(\#\Delta\widetilde{\mathcal{Y}}-1)\big]
\\&=\dfrac{1}{2}\int_{\mathbb{R}^{2}}(\widetilde{G}(x,x)
\widetilde{G}(y,y)-\widetilde{G}(x,y)^{2})
(\tilde{\kappa}(dx)-\kappa(dx))(\tilde{\kappa}(dy)-\kappa(dy))
\\&\leq \dfrac{1}{2}\Big(\int_{\mathbb{R}}\widetilde{G}(y,y)
(\tilde{\kappa}(dy)-\kappa(dy))\Big)^{2}.
\end{split}
\end{equation*}
The expression of $\mathbb{E}\big[\#\Delta\widetilde{\mathcal{Y}}(\#\Delta\widetilde{\mathcal{Y}}-1)\big]$ that we used is general for determinantal point processes.

Let us prove now that 
$\mathbb{P}(\Delta\widetilde{\mathcal{Y}}=\emptyset)>0$. $\Delta\widetilde{\mathcal{Y}}$ is determinantal point process associated to a trace-class self-adjoint positive semi-definite contraction operator on $\mathbb{L}^{2}(d\tilde{\kappa}-d\kappa)$. 
$\mathbb{P}(\Delta\widetilde{\mathcal{Y}}=\emptyset)>0$ if and only if all the eigenvalues of the operator are strictly less then $1$ (see Theorem $4.5.3$ in \cite{HoughKrishPeresVirag2009GAFDet}). Let $f\in\mathbb{L}^{2}(\tilde{\kappa}-\kappa)$. Let
\begin{displaymath}
F(x):=\int_{\mathbb{R}}\widetilde{G}(x,y)f(y)
(\tilde{\kappa}(dy)-\kappa(dy)).
\end{displaymath}
$F$ is continuous, dominated by
\begin{displaymath}
\widetilde{G}(x,x)^{\frac{1}{2}}
\Big(\int_{\mathbb{R}}\widetilde{G}(y,y)
(\tilde{\kappa}(dy)-\kappa(dy))\Big)^{\frac{1}{2}}
\Big(\int_{\mathbb{R}}f(y)^{2}
(\tilde{\kappa}(dy)-\kappa(dy))\Big)^{\frac{1}{2}},
\end{displaymath}
and has left-side and right-side derivatives at every point. $F$ satisfies the equation
\begin{displaymath}
-\dfrac{1}{2}\dfrac{d^{2}F}{dx^{2}}
+F\tilde{\kappa}=f(\tilde{\kappa}-\kappa).
\end{displaymath}
Assume by absurd that $f=F$ $(\tilde{\kappa}-\kappa)$-almost everywhere. Then,
\begin{equation*}
\begin{split}
\int_{\mathbb{R}}F(x)^{2}(\tilde{\kappa}(dx)-\kappa(dx))&=
\int_{\mathbb{R}}f(x)F(x)(\tilde{\kappa}(dx)-\kappa(dx))\\&=
\int_{\mathbb{R}}F(x)^{2}\tilde{\kappa}(dx)
+\dfrac{1}{2}\int_{\mathbb{R}}\dfrac{dF}{dx}(x)^{2}dx.
\end{split}
\end{equation*}
Thus, $F$ is necessarily constant. But then, this means that 
$(\tilde{\kappa}-\kappa)(\mathbb{R})=\tilde{\kappa}(\mathbb{R})$, which is impossible because $\kappa$ is non zero. Thus, $1$ is not an eigenvalue of the operator defining the determinantal process $\Delta\widetilde{\mathcal{Y}}$ and thus, 
$\mathbb{P}(\Delta\widetilde{\mathcal{Y}}=\emptyset)>0$.

As for $\widetilde{\mathcal{Y}}_{\infty}$, the spacing between consecutive points of $\Delta\widetilde{\mathcal{Y}}$ are independent. By construction, $\Delta\widetilde{\mathcal{Y}}\subseteq \operatorname{Supp}(\tilde{\kappa}-\kappa)$. Given $y_{0}\in \operatorname{Supp}(\tilde{\kappa}-\kappa)$, let
\begin{displaymath}
1_{y>y_{0}}f_{\Delta\widetilde{\mathcal{Y}}}(y_{0},y)
(\tilde{\kappa}(dy)-\kappa(dy))
\end{displaymath} 
be the distribution of the lowest point in $\Delta\widetilde{\mathcal{Y}}\cap(y_{0},+\infty)$ conditional on $y_{0}\in \Delta\widetilde{\mathcal{Y}}$. Since $y_{0}$ may be the maximum of $\Delta\widetilde{\mathcal{Y}}$, $f_{\Delta\widetilde{\mathcal{Y}}}(y_{0},y)
(\tilde{\kappa}(dy)-\kappa(dy))<1$. For $y$ to be $\min\Delta\widetilde{\mathcal{Y}}\cap(y_{0},+\infty)$, $y$ must belong to $\widetilde{\mathcal{Y}}_{\infty}$, all points in $y'\in\widetilde{\mathcal{Y}}_{\infty}\cap(y_{0},y)$ must be erased (probability $\chi(y')$ for each), and $y$ must be kept (probability $1-\chi(y)$). For $y'>y_{0}$, let $f_{\widetilde{G}}(y_{0},y')$ be
\begin{displaymath}
f_{\widetilde{G}}(y_{0},y')=2(y'-y_{0})\dfrac{\tilde{u}_{\downarrow}(y')}
{\tilde{u}_{\downarrow}(y_{0})}.
\end{displaymath}
$1_{y'>y_{0}}f_{\widetilde{G}}(y_{0},y')\tilde{\kappa}(dy')$ is the distribution of 
$\min\widetilde{\mathcal{Y}}_{\infty}\cap(y_{0},+\infty)$ conditional on $y_{0}\in\widetilde{\mathcal{Y}}_{\infty}$ 
(Proposition \ref{Ch6Sec3: PropJointLaw}).
$f_{\Delta\widetilde{\mathcal{Y}}}$ and $f_{\widetilde{G}}$ are related as follows:
\begin{equation*}
\begin{split}
f_{\Delta\widetilde{\mathcal{Y}}}&(y_{0},y)\\&=f_{\widetilde{G}}(y_{0},y)
+\sum_{j\geq 2}\int_{y_{0}<\dots<y_{j-1}<y}
f_{\widetilde{G}}(y_{0},y_{1})\dots 
f_{\widetilde{G}}(y_{j-1},y)\prod_{i=1}^{j-1}\chi(y_{i})\kappa(dy_{i})\\
\\&=\dfrac{\tilde{u}_{\downarrow}(y)}{\tilde{u}_{\downarrow}(y_{0})}
\bigg(2(y-y_{0})+\sum_{j\geq 2}2^{j}\int_{y_{0}<\dots<y_{j-1}<y}
(y_{1}-y_{0})\dots(y-y_{j-1})\prod_{i=1}^{j-1}\kappa(dy_{i})\bigg).
\end{split}
\end{equation*}
But,
\begin{equation*}
\begin{split}
2(y-&y_{0})+\sum_{j\geq 2}2^{j}
\int_{y_{0}<\dots<y_{j-1}<y}(y_{1}-y_{0})\dots(y-y_{j-1})
\prod_{i=1}^{j-1}\kappa(dy_{i})
\\&=\dfrac{u_{\downarrow}(y_{0})}{u_{\downarrow}(y)}
\bigg(f_{G}(y_{0},y)+\sum_{j\geq 2}\int_{y_{0}<\dots<y_{j-1}<y}
f_{G}(y_{0},y_{1})\dots f_{G}(y_{j-1},y)\prod_{i=1}^{j-1}
\kappa(dy_{i})\bigg)
\\&=\dfrac{u_{\downarrow}(y_{0})}{u_{\downarrow}(y)}\Big(G(y,y)-\dfrac{G(y_{0},y)^{2}}{G(y_{0},y_{0})}\Big)=
u_{\downarrow}(y_{0})u_{\uparrow}(y)-u_{\uparrow}(y_{0})u_{\downarrow}(y).
\end{split}
\end{equation*}
(see Section \ref{Ch6Sec3}). It follows that
\begin{displaymath}
f_{\Delta\widetilde{\mathcal{Y}}}(y_{0},y)
=\dfrac{\tilde{u}_{\downarrow}(y)}{\tilde{u}_{\downarrow}(y_{0})}
(u_{\downarrow}(y_{0})u_{\uparrow}(y)
-u_{\uparrow}(y_{0})u_{\downarrow}(y)).
\end{displaymath}
In particular, if $y_{0}<y_{1}<\dots<y_{n}\in\mathbb{R}$, the infinitesimal probability that $\Delta\widetilde{\mathcal{Y}}$ has a point at each of the locations $y_{i}$ and no points in-between is
\begin{multline*}
\widetilde{G}(y_{0},y_{0})f_{\Delta\widetilde{\mathcal{Y}}}(y_{0},y_{1})\dots f_{\Delta\widetilde{\mathcal{Y}}}(y_{n-1},y_{n})
\prod_{i=0}^{n}(\tilde{\kappa}(dy_{i})-\kappa(dy_{i}))\\=
\tilde{u}_{\uparrow}(y_{0})\tilde{u}_{\downarrow}(y_{n})
\prod_{i=1}^{n}(u_{\downarrow}(y_{i})u_{\uparrow}(y_{i-1})-
u_{\uparrow}(y_{i})u_{\downarrow}(y_{i-1}))
\prod_{i=0}^{n}(\tilde{k}(dy_{i})-k(dy_{i})).
\end{multline*}
Thus, the expression of $v_{\kappa,\tilde{\kappa}}(y)$ is an inclusion-exclusion identity obtained as follows: 
$v_{\kappa,\tilde{\kappa}}(y)(\tilde{\kappa}-\kappa)(dy)$ is the infinitesimal probability that  $\Delta\widetilde{\mathcal{Y}}$ contains a point at $y$, from which we subtract the infinitesimal probabilities to have a point at $y$ at another below respectively above, and to which we add the infinitesimal probability to have a point at $y$ and points both below and above $y$.
\end{proof}

Next we deal with the law of 
$(\widetilde{\mathcal{Y}}_{\infty},\widetilde{\mathcal{Z}}_{\infty})$ conditional on $\Delta\widetilde{\mathcal{Y}}=0$.
Let $y_{0}\in \operatorname{Supp}(\tilde{\kappa}-\kappa)$. First, we will compute the probability that $\Delta\widetilde{\mathcal{Y}}\cap(y_{0},+\infty)\neq\emptyset$ conditional on $y_{0}\in\widetilde{\mathcal{Y}}_{\infty}$.

\begin{lemm}
\label{Ch7Sec1: LemProportionality}
There are positive constants $c_{1}$ and $c_{2}$ such that for all $x\in\mathbb{R}$,
\begin{equation}
\label{Ch7Sec1: EqLemProp1}
\int_{y<x}(u_{\downarrow}(y)u_{\uparrow}(x)
-u_{\uparrow}(y)u_{\downarrow}(x))
\tilde{u}_{\uparrow}(y)(\tilde{\kappa}(dy)-\kappa(dy))
=\tilde{u}_{\uparrow}(x)-c_{1}u_{\uparrow}(x),
\end{equation}
\begin{equation}
\label{Ch7Sec1: EqLemProp2}
\int_{y>x}(u_{\uparrow}(y)u_{\downarrow}(x)
-u_{\downarrow}(y)u_{\uparrow}(x))
\tilde{u}_{\downarrow}(y)(\tilde{\kappa}(dy)-\kappa(dy))
=\tilde{u}_{\downarrow}(x)-c_{2}u_{\downarrow}(x).
\end{equation}
In particular,
\begin{displaymath}
v_{\kappa,\tilde{\kappa}}(y)=c_{1}c_{2}u_{\uparrow}(y)u_{\downarrow}(y).
\end{displaymath}
\end{lemm}

\begin{proof}
We will prove \eqref{Ch7Sec1: EqLemProp2}. 
The proof of \eqref{Ch7Sec1: EqLemProp1} is similar.
Let $f$ be the function
\begin{displaymath}
f(x):=\tilde{u}_{\downarrow}(x)-\int_{y>x}
(u_{\uparrow}(y)u_{\downarrow}(x)
-u_{\downarrow}(y)u_{\uparrow}(x))
\tilde{u}_{\downarrow}(y)(\tilde{\kappa}(dy)-\kappa(dy)).
\end{displaymath}
The derivative of $f$, defined everywhere except at most countably many points, is
\begin{displaymath}
\dfrac{df}{dx}(x)=\dfrac{d\tilde{u}_{\downarrow}}{dx}(x)
-\int_{y>x}\Big(u_{\uparrow}(y)\dfrac{du_{\downarrow}}{dx}(x)
-u_{\downarrow}(y)\dfrac{du_{\uparrow}}{dx}(x)\Big)
\tilde{u}_{\downarrow}(y)(\tilde{\kappa}(dy)-\kappa(dy)).
\end{displaymath}
The weak second derivative of $f$ is
\begin{equation*}
\begin{split}
\dfrac{d^{2}f}{dx^{2}}(x)=&\dfrac{d^{2}\tilde{u}_{\downarrow}}{dx^{2}}(x)
-\int_{y>x}\Big(u_{\uparrow}(y)\dfrac{d^{2}u_{\downarrow}}{dx^{2}}(x)
-u_{\downarrow}(y)\dfrac{d^{2}u_{\uparrow}}{dx^{2}}(x)\Big)
\tilde{u}_{\downarrow}(y)(\tilde{\kappa}(dy)-\kappa(dy))\\
&+\Big(u_{\uparrow}(x)\dfrac{du_{\downarrow}}{dx}(x)
-u_{\downarrow}(x)\dfrac{du_{\uparrow}}{dx}(x)\Big)
\tilde{u}_{\downarrow}(x)(\tilde{\kappa}(dx)-\kappa(dx))\\
=&2\tilde{u}_{\downarrow}(x)\tilde{\kappa}(dx)
\\&-\int_{y>x}(u_{\uparrow}(y)u_{\downarrow}(x)
-u_{\downarrow}(y)u_{\uparrow}(x))\tilde{u}_{\downarrow}(y)
(\tilde{\kappa}(dy)-\kappa(dy))\times \kappa(dx)\\
&+2\tilde{u}_{\downarrow}(x)(\tilde{k}(dx)-k(dx))\\
=&2\tilde{u}_{\downarrow}(x)\kappa(dx)
\\&-\int_{y>x}(u_{\uparrow}(y)u_{\downarrow}(x)
-u_{\downarrow}(y)u_{\uparrow}(x))\tilde{u}_{\downarrow}(y)
(\tilde{\kappa}(dy)-\kappa(dy))\times \kappa(dx)\\
=&2f(x)\kappa(dx).
\end{split}
\end{equation*}
Thus, $f$ satisfies the same differential equation as $u_{\downarrow}$. Moreover, $\vert f\vert$ is dominated by
\begin{displaymath}
\tilde{u}_{\downarrow}(x)+u_{\downarrow}(x)\int_{y>x}
G(y,y)(\tilde{\kappa}(dy)-\kappa(dy)).
\end{displaymath}
Thus, $f$ is bounded on the intervals of the type $(a,+\infty)$. It follows that there is a constant $c_{2}\in\mathbb{R}$ such that 
$f\equiv c_{2}u_{\downarrow}$. Thus, we get the identity 
\eqref{Ch7Sec1: EqLemProp2}. Let us show that $c_{2}>0$. Let 
$x\in \operatorname{Supp}(\tilde{\kappa})$. Then,
\begin{multline*}
1-\dfrac{1}{\tilde{u}_{\downarrow}(x)}\int_{y>x}(u_{\uparrow}(y)u_{\downarrow}(x)-u_{\downarrow}(y)u_{\uparrow}(x))\tilde{u}_{\downarrow}(y)(\tilde{\kappa}(dy)-\kappa(dy))\\
=1-\int_{y>x}f_{\Delta\widetilde{\mathcal{Y}}}(x,y)
(\tilde{\kappa}(dy)-\kappa(dy))=
\mathbb{P}\big(\Delta\widetilde{Y}\cap(x,+\infty)=\emptyset
\vert x\in\widetilde{Y}_{\infty}\big).
\end{multline*}
The above conditional probability is positive because according to Lemma \ref{Ch7Sec1: LemExtrDet}, 
$\mathbb{P}\big(\Delta\widetilde{Y}=\emptyset\big)>0$. Thus, $f$ is positive and $c_{2}>0$.
\end{proof}

\begin{lemm}
\label{Ch7Sec1: LemCondEmpt}
Conditional on the event $\Delta\widetilde{\mathcal{Y}}=\emptyset$, $(\widetilde{\mathcal{Y}}_{\infty},\widetilde{\mathcal{Z}}_{\infty})$ has the same law as $(\mathcal{Y}_{\infty},\mathcal{Z}_{\infty})$.
\end{lemm}

\begin{proof}
It is enough to show that conditional on $\Delta\widetilde{\mathcal{Y}}=\emptyset$, $\widetilde{\mathcal{Y}}_{\infty}$ has the same law as $\mathcal{Y}_{\infty}$. Indeed, in both cases the points of $\widetilde{\mathcal{Z}}_{\infty}$, respectively $\mathcal{Z}_{\infty}$, are distributed independently and uniformly between any two consecutive points of $\widetilde{\mathcal{Y}}_{\infty}$, respectively $\mathcal{Y}_{\infty}$. 
For $n\geq 1$ and $y_{1}<\dots<y_{n}$, 
let $\varrho_{n}(dy_{1},\dots dy_{n})$ be the infinitesimal probability
for $\widetilde{\mathcal{Y}}_{\infty}$ having a point at each of the locations $y_{i}$ and none in-between, conditional on $\Delta\widetilde{\mathcal{Y}}=\emptyset$. We need only to show that
\begin{equation}
\label{Ch7Sec1: EqIdentDistribY}
\varrho_{n}(dy_{1},\dots dy_{n})
=2^{n-1}u_{\uparrow}(y_{1})u_{\downarrow}(y_{n})
\prod_{i=2}^{n}(y_{i}-y_{i-1})\prod_{i=1}^{n}\kappa(dy_{i}).
\end{equation}
For $y_{1}<\dots<y_{n}$ to be $n$ consecutive points in $\widetilde{\mathcal{Y}}_{\infty}$ and for $\Delta\widetilde{\mathcal{Y}}=\emptyset$, we need $y_{1}<\dots<y_{n}$ to be $n$ consecutive points in $\widetilde{\mathcal{Y}}_{\infty}$, to choose not to erase any of $y_{i}$ (probability $\chi(y_{i})$) and finally we need that $\Delta\widetilde{\mathcal{Y}}\cap(-\infty,y_{1})=\emptyset$ and  
$\Delta\widetilde{\mathcal{Y}}\cap(y_{n},+\infty)=\emptyset$. Thus,
\begin{equation*}
\begin{split}
\varrho_{n}(dy_{1},\dots &,dy_{n})=\dfrac{1}{\mathbb{P}(\Delta\widetilde{\mathcal{Y}}=\emptyset)}2^{n-1}\tilde{u}_{\uparrow}(y_{1})
\tilde{u}_{\downarrow}(y_{n})\\
&\times\Big(1-\dfrac{1}{\tilde{u}_{\uparrow}(y_{1})}\int_{y<y_{1}}(u_{\downarrow}(y)u_{\uparrow}(y_{1})-u_{\uparrow}(y)u_{\downarrow}(y_{1}))\tilde{u}_{\uparrow}(y)(\tilde{\kappa}(dy)-\kappa(dy))\Big)
\\&\times\Big(1-\dfrac{1}{\tilde{u}_{\downarrow}(y_{n})}\int_{y>y_{n}}(u_{\uparrow}(y)u_{\downarrow}(y_{n})-u_{\downarrow}(y)u_{\uparrow}(y_{n}))\tilde{u}_{\downarrow}(y)(\tilde{\kappa}(dy)-\kappa(dy))\Big)\\
&\times\prod_{i=2}^{n}(y_{i}-y_{i-1})\prod_{i=1}^{n}\chi(y_{i})
\tilde{\kappa}(dy_{i}).
\end{split}
\end{equation*}
Applying Lemma \ref{Ch7Sec1: LemProportionality}, we get that
\begin{displaymath}
\varrho_{n}(dy_{1},\dots,dy_{n})
=\dfrac{c_{1}c_{2}}{\mathbb{P}(\Delta\widetilde{\mathcal{Y}}=\emptyset)} 2^{n-1}u_{\uparrow}(y_{1})u_{\downarrow}(y_{n})
\prod_{i=2}^{n}(y_{i}-y_{i-1})\prod_{i=1}^{n}\kappa(dy_{i}).
\end{displaymath}
Since the constant $\frac{c_{1}c_{2}}{\mathbb{P}(\Delta\widetilde{\mathcal{Y}}=\emptyset)}$ does not depend on $n$, the previous equations implies that
\begin{displaymath}
\mathbb{P}(\widetilde{\mathcal{Y}}_{\infty}\neq\emptyset\vert \Delta\widetilde{\mathcal{Y}}=\emptyset)=
\dfrac{c_{1}c_{2}}{\mathbb{P}(\Delta\widetilde{\mathcal{Y}}=\emptyset)}
\mathbb{P}(\mathcal{Y}_{\infty}\neq\emptyset).
\end{displaymath}
But $\mathbb{P}(\widetilde{\mathcal{Y}}_{\infty}\neq\emptyset\vert \Delta\widetilde{\mathcal{Y}}=\emptyset)=
\mathbb{P}(\mathcal{Y}_{\infty}\neq\emptyset)=1$. Thus, 
\begin{displaymath}
\dfrac{c_{1}c_{2}}{\mathbb{P}(\Delta\widetilde{\mathcal{Y}}=\emptyset)}=1,
\end{displaymath}
and \ref{Ch7Sec1: EqIdentDistribY} holds.
\end{proof}

\begin{coro}
\label{Ch7Sec1: CorAbsenceRoot}
Let $a<b\in\mathbb{R}$ such that 
$\tilde{\kappa}(\mathbb{R}\setminus [a,b])>0$. Conditional on $\widetilde{\mathcal{Y}}_{\infty}\cap [a,b]=\emptyset$,
$(\widetilde{\mathcal{Y}}_{\infty},\widetilde{\mathcal{Z}}_{\infty})$
has the same law as the pair of interwoven determinantal point processes
obtained from the Wilson's algorithm applied to the Brownian motion with
killing measure $1_{\mathbb{R}\setminus [a,b]}\kappa$.
\end{coro}

\begin{lemm}
\label{Ch7Sec1: LemCondOne}
Conditional on 
$\# \Delta\widetilde{\mathcal{Y}}=1$ and on the position of the unique point $Y$ in $\Delta\widetilde{\mathcal{Y}}$, $(\widetilde{\mathcal{Y}}_{\infty},\widetilde{\mathcal{Z}}_{\infty})$ has the same law as $(\mathcal{Y}_{\infty}^{(Y)},\mathcal{Z}_{\infty}^{(Y)})$.
\end{lemm}

\begin{proof}
It is enough to show that conditional on 
$\# \Delta\widetilde{\mathcal{Y}}=1$ and on the position of the unique point $Y$ in $\Delta\widetilde{\mathcal{Y}}$, $\widetilde{\mathcal{Y}}_{\infty}$ has the same law as $\mathcal{Y}_{\infty}^{(Y)}$. Indeed, the points of $\widetilde{\mathcal{Z}}_{\infty}$, respectively $\mathcal{Z}_{\infty}^{(Y)}$, are independently and uniformly distributed between any two consecutive points in $\widetilde{\mathcal{Y}}_{\infty}$, respectively $\mathcal{Z}_{\infty}^{(Y)}$.

Let $n\geq 1$ and $i_{0}\in \lbrace 1,\dots,n\rbrace$. Let $y_{1}<\dots<y_{n}\in\mathbb{R}$. The infinitesimal probability for $y_{1},\dots,y_{n}$ being $n$ consecutive points in $\widetilde{\mathcal{Y}}_{\infty}$ and $\Delta\widetilde{\mathcal{Y}}=\lbrace y_{i_{0}}\rbrace$ is
\begin{equation*}
\begin{split}
2^{n-1}\tilde{u}_{\uparrow}(y_{1})&\tilde{u}_{\downarrow}(y_{n})
\\&\times\Big(1-\dfrac{1}{\tilde{u}_{\uparrow}(y_{1})}\int_{y<y_{1}}(u_{\downarrow}(y)u_{\uparrow}(y_{1})
-u_{\uparrow}(y)u_{\downarrow}(y_{1}))\tilde{u}_{\uparrow}(y)
(\tilde{\kappa}(dy)-\kappa(dy))\Big)\\&\times
\Big(1-\dfrac{1}{\tilde{u}_{\downarrow}(y_{n})}\int_{y>y_{n}}(u_{\uparrow}(y)u_{\downarrow}(y_{n})-u_{\downarrow}(y)u_{\uparrow}(y_{n}))\tilde{u}_{\downarrow}(y)(\tilde{\kappa}(dy)-\kappa(dy))\Big)\\
&\times\prod_{i=2}^{n}(y_{i}-y_{i-1})\prod_{i\neq i_{0}}
\kappa(dy_{i})\times(\tilde{\kappa}-\kappa)(dy_{i_{0}})\\
=c_{1}c_{2}&2^{n-1}u_{\uparrow}(y_{1})
u_{\downarrow}(y_{n})\prod_{i=2}^{n}(y_{i}-y_{i-1})
\prod_{i\neq i_{0}}\kappa(dy_{i})
\times(\tilde{\kappa}-\kappa)(dy_{i_{0}})
\end{split}
\end{equation*}
\begin{equation}
\label{Ch7Sec1: EqProd3}
\begin{split}
=v_{\kappa,\tilde{\kappa}}(y_{i_{0}})(\tilde{\kappa}-\kappa)(dy_{i_{0}})
&\times 2^{i_{0}-1}\dfrac{u_{\uparrow}(y_{1})}{u_{\uparrow}(y_{i_{0}})}
\prod_{i=1}^{i_{0}-1}(y_{i+1}-y_{i})\kappa(dy_{i})\\&\times
2^{n-i_{0}}\dfrac{u_{\downarrow}(y_{n})}{u_{\downarrow}(y_{i_{0}})}\prod_{i=i_{0}+1}^{n}(y_{i}-y_{i-1})\kappa(dy_{i}).
\end{split}
\end{equation}
In \ref{Ch7Sec1: EqProd3} appears the infinitesimal probability for $\Delta\widetilde{\mathcal{Y}}=\lbrace y_{i_{0}}\rbrace$ times the infinitesimal probability for $y_{1},\dots,y_{n}$ being $n$ consecutive points in $\mathcal{Y}_{\infty}^{(y_{0})}$ (compare with expressions \ref{Ch7Sec1: EqCondPresy0} and \ref{Ch7Sec1: EqCondPresy0bis}).
\end{proof}

\section{Couplings}
\label{Ch7Sec2}

In this section we will prove the monotone coupling results for
$(\mathcal{Y}_{\infty},\mathcal{Z}_{\infty})$ stated at the begining of Section \ref{Ch7Sec1}. The construction of the coupling will be explicit. However, it will not appeal to Wilson's algorithm used to define
$(\mathcal{Y}_{\infty},\mathcal{Z}_{\infty})$. First, we will describe analogous monotone coupling results for uniform spanning trees on finite graphs. In this case no explicit construction is known in general and the proof relies on Strassen's theorem and the conditions for stochastic domination between determinantal processes shown in \cite{Lyons2003DetProbMes}.

\begin{prop}
\label{Ch7Sec2: PropCouplDiscr}
Let $\mathbb{G}$ be a finite connected undirected graph with $E$ its set of edges, and $(C(e))_{e\in E}$ a positive weight function on $E$. Let $F$ be a subset of $E$. Let 
$(\widetilde{C}(e))_{e\in E}$ be an other weight function such $\widetilde{C}\geq C$ and 
$\widetilde{C}=C$ on $E\setminus F$. Let $\Upsilon$ be the uniform spanning tree of $\mathbb{G}$ corresponding to the weights $C$ and $\widetilde{\Upsilon}$ the uniform spanning tree of $\mathbb{G}$ corresponding to the weights $\widetilde{C}$. There is a coupling of $\Upsilon$ and $\widetilde{\Upsilon}$ such that
\begin{equation}
\label{Ch7Sec2: EqDiscrCoup1}
\widetilde{\Upsilon}\cap(E\setminus F)\subseteq\Upsilon\cap(E\setminus F). 
\end{equation}
In case $F$ is made of all edges adjacent to a particular vertex $x_{0}$, and $\widetilde{C}$ is proportional to $C$ on $F$, then there is a coupling satisfying the additional condition
\begin{equation}
\label{Ch7Sec2: EqDiscrCoup2}
\Upsilon\cap F\subseteq \widetilde{\Upsilon}\cap F.
\end{equation}
\end{prop}

\begin{proof}
It is enough to prove the first coupling (\eqref{Ch7Sec2: EqDiscrCoup1}) in case $F$ is a single edge ($F=\lbrace e\rbrace$). Then, by induction on $\# F$ the general result will follow. From definition of uniform spanning trees is clear that 
$\mathbb{P}(e\in\Upsilon)\leq\mathbb{P}(e\in\widetilde{\Upsilon})$. Moreover, $\Upsilon$ conditional on $e\in\Upsilon$, respectively $e\not\in\Upsilon$, has the same law as $\widetilde{\Upsilon}$ conditional on $e\in\widetilde{\Upsilon}$, respectively $e\not\in\widetilde{\Upsilon}$. A possible coupling is the following: first we couple $1_{e\in\Upsilon}$ with $1_{e\in\widetilde{\Upsilon}}$ in a way such that $1_{e\in\Upsilon}\leq 1_{e\in\widetilde{\Upsilon}}$. In case 
$1_{e\in\Upsilon}= 1_{e\in\widetilde{\Upsilon}}=0$, respectively 
$1_{e\in\Upsilon}= 1_{e\in\widetilde{\Upsilon}}=1$, we sample for both $\Upsilon$ and
$\widetilde{\Upsilon}$ the same tree having the law of $\Upsilon$ conditioned by 
$e\not\in\Upsilon$, respectively $e\in\Upsilon$. In case $1_{e\in\Upsilon}=0$ and
$1_{e\in\widetilde{\Upsilon}}=1$, we use the fact that on the edges in $E\setminus \lbrace e\rbrace$, the law of $\Upsilon$ conditioned on $e\in\Upsilon$ is stochastically dominated by the law of $\Upsilon$ conditioned on $e\not\in\Upsilon$, which implies the existence of a monotone coupling by Strassen's theorem. See Theorems $5.2$, $5.3$ and $5.5$ in \cite{Lyons2003DetProbMes}. 

Now we consider the case of $F$ made of all edges adjacent to a particular vertex $x_{0}$, and $\widetilde{C}$ is proportional to $C$ on $F$. Let $(\Upsilon,\widetilde{\Upsilon})$ be a coupling satisfying \eqref{Ch7Sec2: EqDiscrCoup1}. In general it does not satisfy \eqref{Ch7Sec2: EqDiscrCoup2}. To deal with this issue we will re-sample the edges of $\Upsilon$ and $\widetilde{\Upsilon}$ contained in $F$, that is to say sample $\Upsilon'$ having the same law as $\Upsilon$, $\widetilde{\Upsilon}'$ having the same law as $\widetilde{\Upsilon}$, such that 
$\Upsilon'\cap(E\setminus F)=\Upsilon\cap(E\setminus F)$, 
$\widetilde{\Upsilon}'\cap(E\setminus F)=\widetilde{\Upsilon}\cap(E\setminus F)$ and such that $\Upsilon'\cap F\subseteq\widetilde{\Upsilon}'\cap F$. Let $\mathcal{T}_{1},\dots,\mathcal{T}_{N}$ be the connected components of $\Upsilon\cap(E\setminus F)$. \eqref{Ch7Sec2: EqDiscrCoup1} ensures that each connected component of $\widetilde{\Upsilon}'\cap(E\setminus F)$ is contained in one of the $\mathcal{T}_{i}$. Let 
$\mathcal{T}_{1,1},\dots,\mathcal{T}_{1,q_{1}},\dots,\mathcal{T}_{N,1},\dots,\mathcal{T}_{N,q_{N}}$ be the connected components of 
$\widetilde{\Upsilon}'\cap(E\setminus F)$, where $\mathcal{T}_{i,j}\subseteq \mathcal{T}_{i}$. Conditional on $\mathcal{T}_{1} ,\dots,\mathcal{T}_{N}$, $\Upsilon\cap F$ has the following law: for each $\mathcal{T}_{i}$ one chooses an edge connecting $x_{0}$ to $\mathcal{T}_{i}$ with probability proportional to $C$, and independently from the edges of $\Upsilon$ that will connect $x_{0}$ to other $(\mathcal{T}_{i'})_{i'\neq i}$. Similarly, for the law of $\widetilde{\Upsilon}$ conditional on $\mathcal{T}_{1,1},\dots,\mathcal{T}_{1,q_{1}},\dots,\mathcal{T}_{N,1},\dots,\mathcal{T}_{N,q_{N}}$. To construct $\Upsilon'$ and
$\widetilde{\Upsilon}'$ we use the fact that $\widetilde{C}$ is proportional to $C$ on $F$:
\begin{itemize}
\item We start with $\Upsilon$ and $\widetilde{\Upsilon}$ satisfying \eqref{Ch7Sec2: EqDiscrCoup1}.
\item Then we remove from $\Upsilon$ and $\widetilde{\Upsilon}$ the edges contained in $F$.
\item For each $\mathcal{T}_{i,j}$, we add to $\widetilde{\Upsilon}'$ an edge connecting $x_{0}$ to $\mathcal{T}_{i,j}$, chosen proportionally to its weight under $C$, each choice being independent from the others.
\item For each $i\in \lbrace 1,\dots,N\rbrace$, there are $q_{i}$ edges in $\widetilde{\Upsilon}'$ connecting $x_{0}$ to $\mathcal{T}_{i}$, one for each $(\mathcal{T}_{i,j})_{1\leq j\leq q_{i}}$. In order to construct $\Upsilon'$, we need to chose one out of $q_{i}$ to keep and remove the others. We chose to keep the edge corresponding to $\mathcal{T}_{i,j}$ with probability proportional to
\begin{displaymath}
\sum_{
\substack{
e~\text{connecting} \\ 
x_{0}~\text{to}~\mathcal{T}_{i,j}
}}C(e).
\end{displaymath}
The choice is done independently for each 
$i\in \lbrace 1,\dots,N\rbrace$.
\end{itemize}
By construction, $\Upsilon'\cap F\subseteq\widetilde{\Upsilon}'\cap F$.
\end{proof}

Consider now two different killing measures $\kappa$ and $\tilde{\kappa}$ on $\mathbb{R}$, with $\kappa\leq \tilde{\kappa}$, and the couples of determinantal point processes $(\mathcal{Y}_{\infty}, \mathcal{Z}_{\infty})$, respectively $(\widetilde{\mathcal{Y}}_{\infty}, \widetilde{\mathcal{Z}}_{\infty})$, corresponding to the Brownian motion on $\mathbb{R}$ with killing measure $\kappa$, respectively $\tilde{\kappa}$. We want to show that one can couple $(\mathcal{Y}_{\infty}, \mathcal{Z}_{\infty})$ and $(\widetilde{\mathcal{Y}}_{\infty}, \widetilde{\mathcal{Z}}_{\infty})$ on the same probability space such that $\mathcal{Z}_{\infty}\subseteq\widetilde{\mathcal{Z}}_{\infty}$ and
${\widetilde{\mathcal{Y}}_{\infty}\subseteq\mathcal{Y}_{\infty}\cup \operatorname{Supp}(\tilde{\kappa}-\kappa)}$, and if $\kappa$ and $\tilde{\kappa}$ are proportional, also have $\mathcal{Y}_{\infty}\subseteq\widetilde{\mathcal{Y}}_{\infty}$.
The condition $\mathcal{Z}_{\infty}\subseteq\widetilde{\mathcal{Z}}_{\infty}$ and
${\widetilde{\mathcal{Y}}_{\infty}\subseteq\mathcal{Y}_{\infty}\cup \operatorname{Supp}(\tilde{\kappa}-\kappa)}$ is analogous to 
\eqref{Ch7Sec2: EqDiscrCoup1}. The condition $\mathcal{Y}_{\infty}\subseteq\widetilde{\mathcal{Y}}_{\infty}$ is analogous to \eqref{Ch7Sec2: EqDiscrCoup2}, where the cemetery $\dagger$ plays the role of the distinguished vertex $x_{0}$. We used the stochastic domination principle (\cite{Lyons2003DetProbMes}) for determinantal point process with determinantal kernel a projection operator. It ensures the existence of a monotone coupling but does not give one explicitly (see open questions \cite{Lyons2003DetProbMes}). However, for 
$(\mathcal{Y}_{\infty}, \mathcal{Z}_{\infty})$ and $(\widetilde{\mathcal{Y}}_{\infty}, \widetilde{\mathcal{Z}}_{\infty})$ we will construct a whole family of rather explicit monotone couplings.

Let $\widetilde{G}$ be the Green's function of 
$\frac{1}{2}\frac{d^{2}}{dx^{2}}-\tilde{\kappa}$, factorized as 
\begin{displaymath}
\widetilde{G}(x,y)=\tilde{u}_{\uparrow}(x\wedge y)
\tilde{u}_{\downarrow}(x\vee y).
\end{displaymath}
Let
\begin{displaymath}
\widetilde{\mathcal{K}}(y,z):=-\dfrac{1}{2}
\dfrac{d\tilde{u}_{\uparrow}}{dx}((y\wedge z)^{+})
\dfrac{d\tilde{u}_{\downarrow}}{dx}((y\vee z)^{-}).
\end{displaymath}
Let $\mathfrak{G}_{\tilde{\kappa}}$ be the operator on $\mathbb{L}^{2}(d\tilde{\kappa})$ defined on functions with compact support as follows:
\begin{displaymath}
(\mathfrak{G}_{\tilde{\kappa}}f)(x):=\int_{\mathbb{R}}\widetilde{G}(x,y)f(y)\tilde{\kappa}(dy).
\end{displaymath}
In case $\tilde{\kappa}=c\kappa$, where $c$ is a constant, $c>1$, we have the following resolvent identity, which follows from Lemma 
\ref{Ch2Sec3: LemResolvent}:
\begin{equation}
\label{Ch7Sec2: EqResolv}
\dfrac{1}{c}\mathfrak{G}_{ck}\mathfrak{G}_{\kappa}=
\dfrac{1}{c}\mathfrak{G}_{\kappa}\mathfrak{G}_{c\kappa}=
\dfrac{1}{c-1}\Big(\mathfrak{G}_{\kappa}
-\frac{1}{c}\mathfrak{G}_{c\kappa}\Big).
\end{equation}

Next we prove that a simple necessary but not sufficient condition for monotone couplings to exist is satisfied. It won't be used in the sequel but we prefer to give a direct proof for it.

\begin{prop}
\label{Ch7Sec2: PropStochDom}
For any $z_{1},\dots,z_{n}\in \mathbb{R}$ such that 
$\tilde{\kappa}(\lbrace z_{i}\rbrace)=0$,
\begin{equation}
\label{Ch7Sec2: EqStochDomK}
\det(\widetilde{\mathcal{K}}(z_{i},z_{j}))_{1\leq i,j\leq n}\geq
\det(\mathcal{K}(z_{i},z_{j}))_{1\leq i,j\leq n}.
\end{equation}
If $\tilde{\kappa}=c\kappa$, $c>1$, then for any 
$y_{1},\dots,y_{n}\in \operatorname{Supp}(\kappa)$,
\begin{equation}
\label{Ch7Sec2: EqStochDomG}
c^{n}\det(\widetilde{G}(y_{i},y_{j}))_{1\leq i,j\leq n}\geq
\det(G(y_{i},y_{j}))_{1\leq i,j\leq n}.
\end{equation}
\end{prop}

\begin{proof}
We will first show \eqref{Ch7Sec2: EqStochDomK}. To begin with we will show that for any $z_{1}\in\mathbb{R}$, $\widetilde{\mathcal{K}}(z_{1},z_{1})\geq \mathcal{K}(z_{1},z_{1})$. The Wronskian
\begin{displaymath}
W(u_{\uparrow},\tilde{u}_{\uparrow})(z):=u_{\uparrow}(z)
\dfrac{d\tilde{u}_{\uparrow}}{dx}(z^{+})-
\tilde{u}_{\uparrow}(z)\dfrac{du_{\uparrow}}{dx}(z^{+})
\end{displaymath}
is non-negative. Indeed, $W(u_{\uparrow},\tilde{u}_{\uparrow})(-\infty)=0$ and
\begin{displaymath}
dW(u_{\uparrow},\tilde{u}_{\uparrow})=2u_{\uparrow}
\tilde{u}_{\uparrow}(d\tilde{\kappa}-d\kappa)\geq 0.
\end{displaymath}
Similarly, the Wronskian
\begin{displaymath}
W(u_{\downarrow},\tilde{u}_{\downarrow})(z):=u_{\downarrow}(z)\dfrac{d\tilde{u}_{\downarrow}}{dx}(z^{+})
-\tilde{u}_{\downarrow}(z)\dfrac{du_{\downarrow}}{dx}(z^{+})
\end{displaymath}
is non-positive. Using the fact that
\begin{displaymath}
W(u_{\downarrow},u_{\uparrow})=W(\tilde{u}_{\downarrow},\tilde{u}_{\uparrow})\equiv 2,
\end{displaymath}
we get
\begin{equation*}
\begin{split}
\widetilde{\mathcal{K}}(z_{1},z_{1})-\mathcal{K}(z_{1},z_{1})
=&\dfrac{1}{2}\Big(\dfrac{du_{\uparrow}}{dx}(z_{1}^{+})
\dfrac{du_{\downarrow}}{dx}(z_{1}^{+})-
\dfrac{d\tilde{u}_{\uparrow}}{dx}(z_{1}^{+})
\dfrac{d\tilde{u}_{\downarrow}}{dx}(z_{1}^{+})\Big)\\
=&\dfrac{1}{4}\Big(\dfrac{du_{\uparrow}}{dx}(z_{1}^{+})
\dfrac{du_{\downarrow}}{dx}(z_{1}^{+})W(\tilde{u}_{\downarrow},\tilde{u}_{\uparrow})-\dfrac{d\tilde{u}_{\uparrow}}{dx}(z_{1}^{+})
\dfrac{d\tilde{u}_{\downarrow}}{dx}(z_{1}^{+})
W(u_{\downarrow},u_{\uparrow})\Big)\\=&
\dfrac{1}{4}\Big(\dfrac{du_{\downarrow}}{dx}(z_{1}^{+})\dfrac{d\tilde{u}_{\downarrow}}{dx}(z_{1}^{+})
W(u_{\uparrow},\tilde{u}_{\uparrow})(z_{1})\\&-
\dfrac{du_{\uparrow}}{dx}(z_{1}^{+})
\dfrac{d\tilde{u}_{\uparrow}}{dx}(z_{1}^{+})
W(u_{\downarrow},\tilde{u}_{\downarrow})(z_{1})\Big)\geq 0.
\end{split}
\end{equation*}
To prove \eqref{Ch7Sec2: EqStochDomK} in general, we will use the factorization \eqref{Ch6Sec3: EqFacK}. For $x_{0}<z$, let
\begin{displaymath}
\tilde{u}^{(x_{0}\triangleright)}_{\uparrow}(z):=\tilde{u}_{\uparrow}(z)
+\Big(\dfrac{d\tilde{u}_{\downarrow}}{dx}(x_{0}^{-})\Big)^{-1}
\dfrac{d\tilde{u}_{\uparrow}}{dx}(x_{0}^{-})\tilde{u}_{\downarrow}(z).
\end{displaymath}
Factorization \eqref{Ch6Sec3: EqFacK} ensures that we only need to prove that, for $x_{0}<z$, with $\kappa(\lbrace x_{0}\rbrace)=0$,
\begin{displaymath}
-\dfrac{d\tilde{u}^{(x_{0}\triangleright)}_{\uparrow}}{dx}(z^{+})
\dfrac{d\tilde{u}_{\downarrow}}{dx}(z)\geq
-\dfrac{du^{(x_{0}\triangleright)}_{\uparrow}}{dx}(z^{+})
\dfrac{du_{\downarrow}}{dx}(z).
\end{displaymath}
First, observe that the Wronskian
\begin{displaymath}
W(u^{(x_{0}\triangleright)}_{\uparrow},\tilde{u}^{(x_{0}\triangleright)}_{\uparrow})(z)
:=u^{(x_{0}\triangleright)}_{\uparrow}(z)
\dfrac{d\tilde{u}^{(x_{0}\triangleright)}_{\uparrow}}{dx}(z^{+})-
\tilde{u}^{(x_{0}\triangleright)}_{\uparrow}(z)
\dfrac{du^{(x_{0}\triangleright)}_{\uparrow}}{dx}(z^{+})
\end{displaymath}
is non-negative on $[x_{0},+\infty)$. Indeed, $W(u^{(x_{0}\triangleright)}_{\uparrow},\tilde{u}^{(x_{0}\triangleright)}_{\uparrow})(x)=0$, and
\begin{displaymath}
dW(u^{(x_{0}\triangleright)}_{\uparrow},\tilde{u}^{(x_{0}\triangleright)}_{\uparrow})=
2u^{(x_{0}\triangleright)}_{\uparrow}(z)
\tilde{u}^{(x_{0}\triangleright)}_{\uparrow}(z)
(d\tilde{\kappa}-d\kappa)\geq 0.
\end{displaymath}
The sequel of the proof works as in the previous case.

Let's prove now \eqref{Ch7Sec2: EqStochDomG}. First, we consider the case $n=1$. From the resolvent identity \eqref{Ch7Sec2: EqResolv} follows that
\begin{displaymath}
\mathfrak{G}_{c\kappa}-\mathfrak{G}_{\kappa}=(c-1)(\mathfrak{G}_{\kappa}-\mathfrak{G}_{c\kappa}\mathfrak{G}_{\kappa}).
\end{displaymath}
Since $\mathfrak{G}_{c\kappa}$ is contracting, this implies that $\mathfrak{G}_{\kappa}\leq \mathfrak{G}_{c\kappa}$, where the inequality stands for positive semi-definite operators on $\mathbb{L}^{2}(d\kappa)$. Let $y_{1}\in \operatorname{Supp}(\kappa)$. Then, for any $\varepsilon>0$,
\begin{equation}
\label{Ch7Sec2: EqDomG1Int}
c\int_{(y_{1}-\varepsilon, y_{1}+\varepsilon)^{2}}
\widetilde{G}(x,y)\kappa(dx)\kappa(dy)\geq
\int_{(y_{1}-\varepsilon, y_{1}+\varepsilon)^{2}}G(x,y)
\kappa(dx)\kappa(dy).
\end{equation}
Since $y_{1}\in \operatorname{Supp}(\kappa)$, both sides of \eqref{Ch7Sec2: EqDomG1Int} are positive. The continuity of $G$ and $\widetilde{G}$ ensures that 
$c\widetilde{G}(y_{1},y_{1})\geq G(y_{1},y_{1})$. In case of general $n$, we use the factorization \eqref{Ch6Sec3: EqFacG}. It is enough to prove that for any $x_{0}<y$, $y\in \operatorname{Supp}(\kappa)$,
\begin{equation}
\label{Ch7Sec2: EqStochDomGx0}
c\widetilde{G}^{(x_{0}\times)}(y,y)\geq G^{(x_{0}\times)}(y,y),
\end{equation}
where
\begin{displaymath}
\widetilde{G}^{(x_{0}\times)}(y,y):=\widetilde{G}(y,y)
-\dfrac{\widetilde{G}(x_{0},y)^{2}}{\widetilde{G}(x_{0},x_{0})}.
\end{displaymath}
$\widetilde{G}$ is the restriction to $(x_{0},+\infty)^{2}$ of the Green's function of 
$\frac{1}{2}\frac{d^{2}}{dx^{2}}-1_{(x_{0},+\infty)}\tilde{\kappa}$. 
Let $\mathfrak{G}_{\kappa}^{(x_{0}\times)}$ and $\mathfrak{G}_{c\kappa}^{(x_{0}\times)}$ be the operators on $\mathbb{L}^{2}(1_{(x_{0},+\infty)}d\kappa)$ defined for functions $f$ with compact support as
\begin{displaymath}
(\mathfrak{G}_{\kappa}^{(x_{0}\times)}f)(x):=
\int_{(x_{0},+\infty)}G^{(x_{0}\times)}(x,y)f(y)\kappa(dy),
\end{displaymath}
\begin{displaymath}
(\mathfrak{G}_{c\kappa}^{(x_{0}\times)}f)(x):=
c\int_{(x_{0},+\infty)}\widetilde{G}^{(x_{0}\times)}(x,y)f(y)\kappa(dy).
\end{displaymath}
$\mathfrak{G}_{\kappa}^{(x_{0}\times)}$ and $\mathfrak{G}_{c\kappa}^{(x_{0}\times)}$ are contractions and satisfy a resolvent identity similar to \eqref{Ch7Sec2: EqResolv}, which similarly implies \eqref{Ch7Sec2: EqStochDomGx0}.
\end{proof}

The resolvent identity \eqref{Ch7Sec2: EqResolv} implies that $\mathfrak{G}_{\kappa}$ and $\mathfrak{G}_{c\kappa}$ commute and that $\mathfrak{G}_{\kappa}\leq \mathfrak{G}_{c\kappa}$. It was shown in case of determianatal point processes on discrete space that this a sufficient condition for a monotone coupling to exist. See Theorem $7.1$ in \cite{Lyons2003DetProbMes}.

To construct the couplings we will give several procedures that take deterministic arguments, among which pairs of interwoven sets of points, and return pairs of interwoven random point processes. The first procedure we describe will be used as sub-procedure in subsequent procedures.

\begin{proc}
\label{Ch7Sec2: ProcSubProc}
Arguments:
\begin{itemize}
\item a pair $(\mathcal{Y},\mathcal{Z})$ of disjoint discrete sets of points in $\mathbb{R}$, such that between any two points in $\mathcal{Y}$ lies a single point in $\mathcal{Z}$ and vice-versa, and such that $\inf \mathcal{Y}\cup\mathcal{Z} \in \mathcal{Y}\cup\lbrace -\infty\rbrace$, $\sup \mathcal{Y}\cup\mathcal{Z} \in \mathcal{Y}\cup\lbrace +\infty\rbrace$,
\item a positive Radon measure $\kappa$,
\item a point $y_{0}\in\mathbb{R}$, such that $y_{0}\not\in \mathcal{Z}$.
\end{itemize}
Procedure:
\begin{itemize}
\item[(i)] If $y_{0}\not\in \mathcal{Y}$, we define a random variable $Z$ distributed as follows:
\begin{itemize}
\item[(i a)] If there are $y'\in\mathcal{Y}$, $z'\in\mathcal{Z}\cup\lbrace +\infty\rbrace$, such that $y'<z'$, $y_{0}\in (y',z')$, and 
$\mathcal{Y}\cap(y',z')=\mathcal{Z}_{\infty}\cap(y',z')=\emptyset$, then $Z$ is distributed according to 
\begin{displaymath}
\dfrac{1_{z\in(y',y_{0})}}{u_{\uparrow}(y_{0})
-u_{\uparrow}(y')}\dfrac{du_{\uparrow}}{dx}(z)dz.
\end{displaymath}
\item[(i b)] If there are $y'\in\mathcal{Y}$, 
$z'\in\mathcal{Z}\cup\lbrace -\infty\rbrace$,
such that $z'<y'$, $y_{0}\in (z',y')$ and 
$\mathcal{Y}\cap(z',y')=\mathcal{Z}\cap(z',y')=\emptyset$, then $Z$ is distributed according to 
\begin{displaymath}
\dfrac{-1_{z\in(y_{0},y')}}{u_{\downarrow}(y')
-u_{\downarrow}(y_{0})}\dfrac{du_{\downarrow}}{dx}(z)dz.
\end{displaymath}
\end{itemize}
\item[(ii)] If there are $y'\in\mathcal{Y}$, $z'\in\mathcal{Z}\cup\lbrace +\infty\rbrace$, such that $y'<z'$, $y_{0}\in (y',z')$, and 
$\mathcal{Y}\cap(y',z')=\mathcal{Z}\cap(y',z')=\emptyset$, then
\begin{itemize}
\item[(ii a)] with probability $\dfrac{u_{\uparrow}(y')}{u_{\uparrow}(y_{0})}$ we set
\begin{displaymath}
(\widetilde{\mathcal{Y}},\widetilde{\mathcal{Z}})=(\mathcal{Y}\cup\lbrace y_{0}\rbrace\setminus\lbrace y'\rbrace,\mathcal{Z}),
\end{displaymath}
\item[(ii b)] and with probability $1-\dfrac{u_{\uparrow}(y')}{u_{\uparrow}(y_{0})}$,
we set
\begin{displaymath}
(\widetilde{\mathcal{Y}},\widetilde{\mathcal{Z}})=(\mathcal{Y}\cup\lbrace y_{0}\rbrace,\mathcal{Z}\cup\lbrace Z\rbrace).
\end{displaymath}
\end{itemize}
\item[(iii)] If there are $y'\in\mathcal{Y}$, $z'\in\mathcal{Z}\cup\lbrace -\infty\rbrace$, such that $z'<y'$, $y_{0}\in (z',y')$, and 
$\mathcal{Y}\cap(z',y')=\mathcal{Z}\cap(z',y')=\emptyset$, then
\begin{itemize}
\item[(iii a)] with probability $\dfrac{u_{\downarrow}(y')}{u_{\downarrow}(y_{0})}$ 
we set
\begin{displaymath}
(\widetilde{\mathcal{Y}},\widetilde{\mathcal{Z}})=(\mathcal{Y}\cup\lbrace y_{0}\rbrace\setminus\lbrace y'\rbrace,\mathcal{Z}),
\end{displaymath}
\item[(iii b)] and with probability $1-\dfrac{u_{\downarrow}(y')}{u_{\downarrow}(y_{0})}$ we set
\begin{displaymath}
(\widetilde{\mathcal{Y}},\widetilde{\mathcal{Z}})=(\mathcal{Y}\cup\lbrace y_{0}\rbrace,\mathcal{Z}\cup\lbrace Z\rbrace).
\end{displaymath}
\end{itemize}
\item[(iv)] If $y_{0}\in\mathcal{Y}$, we set $(\widetilde{\mathcal{Y}},\widetilde{\mathcal{Z}})=(\mathcal{Y},\mathcal{Z})$.
\end{itemize}
Return: $(\widetilde{\mathcal{Y}},\widetilde{\mathcal{Z}})$.
\end{proc}

\begin{lemm}
\label{Ch7Sec2: LemAfterSubProc}
If Procedure \ref{Ch7Sec2: ProcSubProc} is applied to the pair of interwoven determinantal point processes $(\mathcal{Y}_{\infty},\mathcal{Z}_{\infty})$ corresponding to the killing measure $\kappa$, then its result $(\widetilde{\mathcal{Y}},\widetilde{\mathcal{Z}})$ has the same law as $(\mathcal{Y}^{(y_{0})}_{\infty},\mathcal{Z}^{(y_{0})}_{\infty})$.
\end{lemm}

\begin{proof}
By construction, $y_{0}\in\widetilde{\mathcal{Y}}$. 
Let $\widetilde{Z}_{1}<\widetilde{Y}_{1}<\dots<\widetilde{Z}_{n}
<\widetilde{Y}_{n}$ be the $2n$ closest points to $y_{0}$ in $(\widetilde{\mathcal{Y}}\cup\widetilde{\mathcal{Z}})\cap(y_{0},+\infty)$. On the event $\min(\mathcal{Y}_{\infty}\cup\mathcal{Z}_{\infty})
\cap (y_{0},+\infty)\in \mathcal{Z}_{\infty}$ (point (ii) in Procedure \ref{Ch7Sec2: ProcSubProc}), their distribution is given by
\begin{equation}
\label{Ch7Sec2: EqE1}
1_{y_{0}<z_{1}<y_{1}<\dots<z_{n}<y_{n}}2^{n}
\Big(\int_{(-\infty,y_{0})}u_{\uparrow}(y')\kappa(dy')\Big)u_{\downarrow}(y_{n}) dz_{1}\kappa(dy_{1})\dots dz_{n}\kappa(dy_{n}).
\end{equation}
On the event $\min (\mathcal{Y}_{\infty}\cup\mathcal{Z}_{\infty})\cap (y_{0},+\infty)\in \mathcal{Y}_{\infty}$ (point (iii) 
in Procedure \ref{Ch7Sec2: ProcSubProc}), the distribution of 
$\min (\mathcal{Y}_{\infty}\cup\mathcal{Z}_{\infty})\cap (y_{0},+\infty)$ is (see Proposition \ref{Ch6Sec3: PropJointLaw})
\begin{equation*}
\begin{split}
1_{y'>y_{0}}2\Big(\int_{(-\infty,y_{0})}&u_{\uparrow}(y_{-1})
(y_{0}-y_{-1})\kappa(dy_{-1})\Big)u_{\downarrow}(y')\kappa(dy')\\
&+1_{y'>y_{0}}\dfrac{u_{\uparrow}(+\infty)}{u_{\uparrow}(y')}G(y_{0},y_{0})\kappa(dy')\\
=&1_{y'>y_{0}}(u_{\uparrow}(y_{0})
-u_{\uparrow}(+\infty))u_{\downarrow}(y')\kappa(dy')+
1_{y'>y_{0}}u_{\uparrow}(+\infty)u_{\downarrow}(y')\kappa(dy')\\
=&1_{y'>y_{0}}u_{\uparrow}(y_{0})u_{\downarrow}(y')\kappa(dy').
\end{split}
\end{equation*}
Thus, on the event $\min(\mathcal{Y}_{\infty}\cup\mathcal{Z}_{\infty})
\cap (y_{0},+\infty)\in \mathcal{Y}_{\infty}$ (point (iii) in Procedure \ref{Ch7Sec2: ProcSubProc}), the distribution of 
$(Z_{1},Y_{1},\dots,Z_{n},Y_{n})$ is
\begin{equation}
\label{Ch7Sec2: EqE2}
1_{y_{0}<z_{1}<\dots<y_{n}}\Big(\int_{y_{0}<y'<z_{1}}
\dfrac{u_{\downarrow}(y')}{u_{\downarrow}(y_{0})}u_{\uparrow}(y_{0})
u_{\downarrow}(y')2^{n}\dfrac{u_{\downarrow}(y_{n})}
{u_{\downarrow}(y')}\kappa(dy')\Big)
dz_{1}\kappa(dy_{1})\dots dz_{n}\kappa(dy_{n}),
\end{equation}
\begin{equation}
\label{Ch7Sec2: EqE2bis}
+1_{y_{0}<z_{1}<\dots<y_{n}}\dfrac{-1}{u_{\downarrow}(y_{0})}
\dfrac{du_{\downarrow}}{dx}(z_{1})
u_{\uparrow}(y_{0})u_{\downarrow}(y_{1})2^{n-1}
\dfrac{u_{\downarrow}(y_{n})}{u_{\downarrow}(y_{1})}
dz_{1}\kappa(dy_{1})\dots dz_{n}\kappa(dy_{n}).
\end{equation}
The term \eqref{Ch7Sec2: EqE2} corresponds to the case when a point is removed from $\mathcal{Y}_{\infty}$ (case (iii a) in Procedure \ref{Ch7Sec2: ProcSubProc}) and \eqref{Ch7Sec2: EqE2bis} to the case when $Z$ is added to $\mathcal{Z}_{\infty}$ (case (iii b) in Procedure \ref{Ch7Sec2: ProcSubProc}). The sum of the densities that appear in \eqref{Ch7Sec2: EqE1}, \eqref{Ch7Sec2: EqE2}
and \eqref{Ch7Sec2: EqE2bis} is
\begin{equation*}
\begin{split}
2^{n}\Big(\int_{(-\infty,y_{0})}u_{\uparrow}(y')\kappa&(dy')\Big)
u_{\downarrow}(y_{n})+\Big(\int_{y_{0}<y'<z_{1}}
\dfrac{u_{\downarrow}(y')}{u_{\downarrow}(y_{0})}u_{\uparrow}(y_{0})
u_{\downarrow}(y')2^{n}
\dfrac{u_{\downarrow}(y_{n})}{u_{\downarrow}(y')}\kappa(dy')\Big)
\\&+\dfrac{-1}{u_{\downarrow}(y_{0})}
\dfrac{du_{\downarrow}}{dx}(z_{1})
u_{\uparrow}(y_{0})u_{\downarrow}(y_{1})2^{n-1}
\dfrac{u_{\downarrow}(y_{n})}{u_{\downarrow}(y_{1})}
\\=&2^{n-1}\dfrac{du_{\uparrow}}{dx}(y_{0})u_{\downarrow}(y_{n})
+\dfrac{2^{n-1}}{u_{\downarrow}(y_{0})}
\Big(\dfrac{du_{\downarrow}}{dx}(z_{1})-
\dfrac{du_{\downarrow}}{dx}(y_{0}^{+})\Big)
u_{\uparrow}(y_{0})u_{\downarrow}(y_{n})
\\&+\dfrac{-2^{n-1}}{u_{\downarrow}(y_{0})}
\dfrac{du_{\downarrow}}{dx}(z_{1})
u_{\uparrow}(y_{0})u_{\downarrow}(y_{n})
\\=&2^{n-1}u_{\downarrow}(y_{n})
\Big(\dfrac{du_{\uparrow}}{dx}(y_{0})+\dfrac{-1}{u_{\downarrow}(y_{0})}
\dfrac{du_{\downarrow}}{dx}(y_{0}^{+})u_{\uparrow}(y_{0})\Big)=
2^{n}\dfrac{u_{\downarrow}(y_{n})}{u_{\downarrow}(y_{0})}.
\end{split}
\end{equation*}
So, we obtain the density which appears in \eqref{Ch7Sec1: EqCondPresy0}.

It remains to prove that $(\widetilde{\mathcal{Y}}\cap(y_{0},+\infty),\widetilde{\mathcal{Z}}\cap(y_{0},+\infty))$ and $(\widetilde{\mathcal{Y}}\cap(-\infty,y_{0}),\widetilde{\mathcal{Z}}\cap(-\infty,y_{0}))$ are independent. Let $Z_{-1}>Y_{-1}>\dots>Z_{-n'}>Y_{-n'}$
be the $n'$ closest points to $y_{0}$ in $(\widetilde{\mathcal{Y}}\cup\widetilde{\mathcal{Z}})\cap(-\infty,y_{0})$. The distribution of the family of points $(Z_{-1},Y_{-1}\dots,Z_{-n'},Y_{-n'},Z_{1},Y_{1},\dots,Z_{n},Y_{n})$ on the event $\#(\widetilde{\mathcal{Y}}\cap(-\infty,y_{0}))\geq n, \#(\widetilde{\mathcal{Y}}\cap(y_{0},+\infty))\geq n'$ is
\begin{equation}
\label{Ch7Sec2: EqEE1}
\bigg(\int_{y_{0}<y'<z_{1}}2^{n+n'}u_{\uparrow}(y_{-n'})u_{\downarrow}(y_{n})\dfrac{u_{\downarrow}(y')}{u_{\downarrow}(y_{0})}\kappa(dy')
-\dfrac{2^{n+n'-1}u_{\uparrow}(y_{-n'})u_{\downarrow}(y_{n})}{u_{\downarrow}(y_{0})}\dfrac{du_{\downarrow}}{dx}(z_{1})
\end{equation}
\begin{equation}
\label{Ch7Sec2: EqEE2}
+\int_{z_{-1}<y'<y_{0}}2^{n+n'}u_{\uparrow}(y_{-n'})u_{\downarrow}(y_{n})\dfrac{u_{\uparrow}(y')}{u_{\uparrow}(y_{0})}\kappa(dy')
+\dfrac{2^{n+n'-1}u_{\uparrow}(y_{-n'})u_{\downarrow}(y_{n})}{u_{\uparrow}(y_{0})}\dfrac{du_{\uparrow}}{dx}(z_{-1})\bigg)
\end{equation}
\begin{displaymath}
\times 1_{y_{-n'}<z_{-n'}<\dots<z_{-1}<y_{0}<z_{1}<
\dots<z_{n}<y_{n}}\kappa(dy_{-n'})dz_{-n'}\dots
dz_{-1}dz_{1}\dots dz_{n}\kappa(dy_{n}).
\end{displaymath}
The term \eqref{Ch7Sec2: EqEE1} corresponds to point (iii) in Procedure \ref{Ch7Sec2: ProcSubProc} and \eqref{Ch7Sec2: EqEE2} to point (ii) in Procedure \ref{Ch7Sec2: ProcSubProc}. One can check that the sum of the densities equals
\begin{displaymath}
2^{n+n'}\dfrac{u_{\uparrow}(y_{-n'})}{u_{\uparrow}(y_{0})}
\dfrac{u_{\downarrow}(y_{n})}{u_{\downarrow}(y_{0})}.
\end{displaymath}
Thus, $(\widetilde{\mathcal{Y}}\cap(y_{0},+\infty),\widetilde{\mathcal{Z}}\cap(y_{0},+\infty))$ and ${(\widetilde{\mathcal{Y}}\cap(-\infty,y_{0}),\widetilde{\mathcal{Z}}\cap(-\infty,y_{0}))}$ are independent.
\end{proof}

\begin{lemm}
\label{Ch7Sec2: LemContinuitySubproc}
We consider the subspace of triples $((\mathcal{Y},\mathcal{Z}),\kappa,y_{0})$ consisting of a pair of discrete sets of points $(\mathcal{Y},\mathcal{Z})$, a Radon measure $\kappa$ and a point $y_{0}\in\mathbb{R}$, and which satisfies the restrictions on the arguments of Procedure \ref{Ch7Sec2: ProcSubProc}. We assume this subspace endowed with the product topology obtained from the topology of uniform convergence on compact subsets for the pairs $(\mathcal{Y},\mathcal{Z})$, the vague topology for the measures $\kappa$ and standard order topology on $\mathbb{R}$. If $(\widetilde{\mathcal{Y}},\widetilde{\mathcal{Z}})$ is the pair of point processes obtained by applying Procedure \ref{Ch7Sec2: ProcSubProc} to the arguments $((\mathcal{Y},\mathcal{Z}),\kappa,y_{0})$, then its law depends continuously on $((\mathcal{Y},\mathcal{Z}),\kappa,y_{0})$.
\end{lemm}

\begin{proof}
From Lemma \ref{Ch2Sec1: LemContKappa}, it follows that the cumulative distribution function of $Z$ (point (i) in Procedure 
\ref{Ch7Sec2: ProcSubProc}) depends uniformly continuously on 
$((\mathcal{Y},\mathcal{Z}),\kappa,y_{0})$ in the neighborhood of triples where $y_{0}\not\in\mathcal{Y}$. Moreover, the probabilities to make either the choice (ii a) or the choice (ii b), as well as to make either the choice (iii a) or the choice (iii b), depend continuously on $((\mathcal{Y},\mathcal{Z}),\kappa,y_{0})$. Thus, the law of $(\widetilde{\mathcal{Y}},\widetilde{\mathcal{Z}})$ depends continuously on $((\mathcal{Y},\mathcal{Z}),\kappa,y_{0})$ in the neighborhood of triples where $y_{0}\not\in\mathcal{Y}$. Moreover, in the neighborhood of triples where $y_{0}\in\mathcal{Y}$, with high probability, converging to $1$, $(\widetilde{\mathcal{Y}},\widetilde{\mathcal{Z}})=(\mathcal{Y},\mathcal{Z})$. Thus, the law of 
$(\widetilde{\mathcal{Y}},\widetilde{\mathcal{Z}})$ is continuous also at these triples.
\end{proof}

First, we will describe a coupling in case when $\tilde{\kappa}$ and $\kappa$ differ by an atom: $\tilde{\kappa}=\kappa+c\delta_{y_{0}}$. We construct the coupling as follows: 

\begin{proc}
\label{Ch7Sec2: ProcCouplDirac}
Arguments:
\begin{itemize}
\item a pair $(\mathcal{Y},\mathcal{Z})$ of disjoint discrete sets of points in $\mathbb{R}$ such that between any two points in $\mathcal{Y}$ lies a single point in $\mathcal{Z}$ and vice-versa, and such that $\inf \mathcal{Y}\cup\mathcal{Z} \in \mathcal{Y}\cup\lbrace -\infty\rbrace$, $\sup \mathcal{Y}\cup\mathcal{Z} \in \mathcal{Y}\cup\lbrace +\infty\rbrace$,
\item two positive Radon measures $\kappa$ and $\tilde{\kappa}$ where $\tilde{\kappa}$ is of form $\tilde{\kappa}=\kappa+c\delta_{y_{0}}$ and $y_{0}\not\in\mathcal{Z}$.
\end{itemize}
Procedure:
\begin{itemize}
\item[(i)] Let $\beta$ be a Bernoulli r.v. of parameter $c\widetilde{G}(y_{0},y_{0})$.
\item[(ii)] If $\beta=0$, we set $(\widetilde{\mathcal{Y}},\widetilde{\mathcal{Z}})=(\mathcal{Y},\mathcal{Z})$.
\item[(iii)] If $\beta=1$, we apply Procedure 
\ref{Ch7Sec2: ProcSubProc} to the arguments 
$(\mathcal{Y},\mathcal{Z})$, $\kappa$ and $y_{0}$ and set $(\widetilde{\mathcal{Y}},\widetilde{\mathcal{Z}})$ to be its result.
\end{itemize}
Return: $(\widetilde{\mathcal{Y}},\widetilde{\mathcal{Z}})$.
\end{proc}

$(\widetilde{\mathcal{Y}},\widetilde{\mathcal{Z}})$ constructed this way satisfies the following: between any two consecutive points in $\widetilde{\mathcal{Y}}$ lies a single point in $\widetilde{\mathcal{Z}}$, and between any two consecutive points in $\widetilde{\mathcal{Z}}$ lies a point in $\widetilde{\mathcal{Y}}$. By construction, $\mathcal{Z}\subseteq\widetilde{\mathcal{Z}}$ and $\widetilde{\mathcal{Y}}\subseteq \mathcal{Y}\cup\lbrace y_{0}\rbrace$.

\begin{prop}
\label{Ch7Sec2: PropCoupleDirac}
If Procedure \ref{Ch7Sec2: ProcCouplDirac} is applied to to the pair of interwoven determinantal point processes 
$(\mathcal{Y}_{\infty},\mathcal{Z}_{\infty}),$ corresponding to the measure $\kappa$, then the returned pair of point processes $(\widetilde{\mathcal{Y}},\widetilde{\mathcal{Z}})$ has the law of the interwoven determinantal point processes $(\widetilde{\mathcal{Y}}_{\infty},\widetilde{\mathcal{Z}}_{\infty})$, corresponding to $\tilde{\kappa}=\kappa+c\delta_{y_{0}}$.
\end{prop}

\begin{proof}
Observe that a.s., $y_{0}\not\in \mathcal{Z}_{\infty}$.
First, we deal with the case $\kappa(\lbrace y_{0}\rbrace)=0$. Then, almost surely, $y_{0}\not\in\mathcal{Y}_{\infty}$ and 
$y_{0}\in \widetilde{\mathcal{Y}}$ if and only if $\beta=1$. But,
\begin{displaymath}
\mathbb{P}(\beta=1)=\mathbb{P}(y_{0}\in \widetilde{\mathcal{Y}})=c\widetilde{G}(y_{0},y_{0}).
\end{displaymath}
According to Corollary \ref{Ch7Sec1: CorAbsenceRoot}, conditional on $y_{0}\not\in\widetilde{\mathcal{Y}}$,  $(\widetilde{\mathcal{Y}}_{\infty},\widetilde{\mathcal{Z}}_{\infty})$ has the same law as $(\mathcal{Y}_{\infty},\mathcal{Z}_{\infty})$, that is to say the same law as $(\widetilde{\mathcal{Y}},\widetilde{\mathcal{Z}})$ conditional on $\beta=0$. According to Lemma 
\ref{Ch7Sec2: LemAfterSubProc}, conditional on $\beta=1$, $(\widetilde{\mathcal{Y}},\widetilde{\mathcal{Z}})$ follows the same law as ${(\widetilde{\mathcal{Y}}\cap(-\infty,y_{0}),\widetilde{\mathcal{Z}}\cap(-\infty,y_{0}))}$, which is also the law of $(\widetilde{\mathcal{Y}}_{\infty},\widetilde{\mathcal{Z}}_{\infty})$ conditioned on $y_{0}\in\widetilde{\mathcal{Y}}_{\infty}$.

We deal now with the case $\kappa(\lbrace y_{0}\rbrace)>0$.
\begin{displaymath}
\mathbb{P}(y_{0}\in\widetilde{\mathcal{Y}}_{\infty})
=\tilde{\kappa}(\lbrace y_{0}\rbrace)\widetilde{G}(y_{0},y_{0}),
\end{displaymath}
\begin{equation*}
\begin{split}
\mathbb{P}(y_{0}\in\widetilde{\mathcal{Y}})&=\mathbb{P}(\beta =1)+
\mathbb{P}(\beta =0,y_{0}\in\mathcal{Y}_{\infty})\\&=
c\widetilde{G}(y_{0},y_{0})+(1-c\widetilde{G}(y_{0},y_{0}))
\kappa(\lbrace y_{0}\rbrace)G(y_{0},y_{0}).
\end{split}
\end{equation*}
But $G$ and $\widetilde{G}$ satisfy the resolvent identity (see Lemma
\ref{Ch2Sec3: LemResolvent}):
\begin{displaymath}
\widetilde{G}(y_{0},y_{0})\kappa(\lbrace y_{0}\rbrace)G(y_{0},y_{0})=
\dfrac{\kappa(\lbrace y_{0}\rbrace)}{\tilde{\kappa}(\lbrace y_{0}\rbrace)
-\kappa(\lbrace y_{0}\rbrace)}(G(y_{0},y_{0})-\widetilde{G}(y_{0},y_{0})).
\end{displaymath}
It follows that $\mathbb{P}(y_{0}\in\widetilde{\mathcal{Y}})=\mathbb{P}(y_{0}\in\widetilde{\mathcal{Y}}_{\infty})$. Let 
$\check{\kappa}:=\kappa-\kappa(\lbrace y_{0}\rbrace)\delta_{y_{0}}$ and $(\widecheck{\mathcal{Y}}_{\infty},\widecheck{\mathcal{Z}}_{\infty})$ be the interwoven determinantal point processes corresponding to $\check{\kappa}$.et $\tilde{\kappa}':=\tilde{\kappa}
-\kappa(\lbrace y_{0}\rbrace)\delta_{y_{0}}$ and 
$(\widetilde{\mathcal{Y}}'_{\infty},\widetilde{\mathcal{Z}}'_{\infty})$
be the interwoven determinantal processes corresponding to $\tilde{\kappa}'$. According to Corollary 
\ref{Ch7Sec1: CorAbsenceRoot}, $(\widetilde{\mathcal{Y}},\widetilde{\mathcal{Z}})$ conditioned by $y_{0}\not\in\widetilde{\mathcal{Y}}$ has the same law as $(\mathcal{Y}_{\infty},\mathcal{Z}_{\infty})$ conditioned on $y_{0}\not\in \mathcal{Y}_{\infty}$, which is the same law as $(\widetilde{\mathcal{Y}}_{\infty},\widetilde{\mathcal{Z}}_{\infty})$ conditioned on $y_{0}\not\in\widetilde{\mathcal{Y}}_{\infty}$, and it is the law of 
$(\widecheck{\mathcal{Y}}_{\infty},\widecheck{\mathcal{Z}}_{\infty})$. For $y_{0}\in\widetilde{\mathcal{Y}}$, there are two possibilities: either $y_{0}\in \mathcal{Y}_{\infty}$ or $y_{0}\not\in\mathcal{Y}_{\infty}$ and $\beta=1$. In the first case, it follows from Proposition 
\ref{Ch7Sec1: PropertyCondPres} that 
$(\mathcal{Y}_{\infty},\mathcal{Z}_{\infty})$ conditioned on 
$y_{0}\in \mathcal{Y}_{\infty}$ has the same law as $(\widetilde{\mathcal{Y}}_{\infty},\widetilde{\mathcal{Z}}_{\infty})$ conditioned on $y_{0}\in\widetilde{\mathcal{Y}}_{\infty}$. In the second case, $(\mathcal{Y}_{\infty},\mathcal{Z}_{\infty})$ conditioned on $y_{0}\not\in\mathcal{Y}_{\infty}$, has the same law as $(\widecheck{\mathcal{Y}}_{\infty},\widecheck{\mathcal{Z}}_{\infty})$. This brings us back to the situation $\kappa(\lbrace y_{0}\rbrace)=0$. According to what was proved earlier, conditional on $y_{0}\not\in\mathcal{Y}_{\infty}$ and $\beta=1$,  $(\widetilde{\mathcal{Y}},\widetilde{\mathcal{Z}})$ has the same law as 
$(\widetilde{\mathcal{Y}}'_{\infty},\widetilde{\mathcal{Z}}'_{\infty})$ conditioned on $y_{0}\in\widetilde{\mathcal{Y}}'_{\infty}$. But this is the same law as for 
$(\widetilde{\mathcal{Y}}_{\infty},\widetilde{\mathcal{Z}}_{\infty})$ conditioned on $y_{0}\in\widetilde{\mathcal{Y}}_{\infty}$. So again, $(\widetilde{\mathcal{Y}},\widetilde{\mathcal{Z}})$ has the same law as $(\widetilde{\mathcal{Y}}_{\infty},\widetilde{\mathcal{Z}}_{\infty})$.
\end{proof}

Next we consider the more general case where the measure 
$\tilde{\kappa}-\kappa$ has a first moment:
\begin{displaymath}
\int_{\mathbb{R}}\vert x\vert (\tilde{\kappa}(dx)-\kappa(dx))<+\infty.
\end{displaymath}
First, we describe a procedure that does not give a coupling between $(\mathcal{Y}_{\infty},\mathcal{Z}_{\infty})$ and $(\widetilde{\mathcal{Y}}_{\infty},\widetilde{\mathcal{Z}}_{\infty})$, but allows to approach it.

\begin{proc}
\label{Ch7Sec2: ProcAlmost}
Arguments:
\begin{itemize}
\item a pair $(\mathcal{Y},\mathcal{Z})$ of disjoint discrete sets of points in $\mathbb{R}$ such that between any two points in $\mathcal{Y}$ lies a single point in $\mathcal{Z}$ and vice-versa, and such that $\inf \mathcal{Y}\cup\mathcal{Z} \in \mathcal{Y}\cup\lbrace -\infty\rbrace$, $\sup \mathcal{Y}\cup\mathcal{Z} \in \mathcal{Y}\cup\lbrace +\infty\rbrace$
\item two positive Radon measures $\kappa$, $\tilde{\kappa}$ such that $\kappa\leq \tilde{\kappa}$ and $\int_{\mathbb{R}}\vert x\vert\,(\tilde{\kappa}(dx)-\kappa(dx))<+\infty$ and 
$(\tilde{\kappa}-\kappa)(\mathcal{Z})=0$.
\end{itemize}
Procedure:
\begin{itemize}
\item[(i)] Let $\beta$ be a Bernoulli r.v. of parameter
\begin{displaymath}
\int_{\mathbb{R}}v_{\kappa,\tilde{\kappa}}(y)(\tilde{\kappa}-\kappa)(dy)
\end{displaymath}
(see notations of Proposition \ref{Ch7Sec1: LemExtrDet}).
\item[(ii)] Let $Y$ be a real r.v. independent from $\beta$, distributed according to
\begin{displaymath}
\dfrac{v_{\kappa,\tilde{\kappa}}(y)(\tilde{\kappa}-\kappa)(dy)}{\mathbb{P}(\beta =1)}.
\end{displaymath} 
\item[(iii)] If $\beta=0$, we set $(\widetilde{\mathcal{Y}},\widetilde{\mathcal{Z}})=(\mathcal{Y},\mathcal{Z})$.
\item[(iv)] If $\beta=1$, we apply Procedure 
\ref{Ch7Sec2: ProcSubProc} to the arguments $(\mathcal{Y},\mathcal{Z})$, $\kappa$ and $Y$, and set $(\widetilde{\mathcal{Y}},\widetilde{\mathcal{Z}})$ to be its result.
\end{itemize}
Return: $(\widetilde{\mathcal{Y}},\widetilde{\mathcal{Z}})$.
\end{proc}

Observe that in case $\tilde{\kappa}$ and $\kappa$ differ only by an atom, Procedure \ref{Ch7Sec2: ProcAlmost} is the same as Procedure \ref{Ch7Sec2: ProcCouplDirac}.

\begin{lemm}
\label{Ch7Sec2: LemAlmost}
Let $(\mathcal{Y}_{\infty},\mathcal{Z}_{\infty})$, respectively $(\widetilde{\mathcal{Y}}_{\infty},\widetilde{\mathcal{Z}}_{\infty})$, be the pair of interwoven determinantal point processes corresponding to the killing measure $\kappa$, respectively $\tilde{\kappa}$. We assume that Procedure \ref{Ch7Sec2: ProcAlmost} is applied to $(\mathcal{Y}_{\infty},\mathcal{Z}_{\infty})$ and that $(\widetilde{\mathcal{Y}},\widetilde{\mathcal{Z}})$ is the returned pair of point processes. Then, the total variation distance between the law of $(\widetilde{\mathcal{Y}},\widetilde{\mathcal{Z}})$ and the law of $(\widetilde{\mathcal{Y}}_{\infty},\widetilde{\mathcal{Z}}_{\infty})$ is less or equal to $\Big(\int_{\mathbb{R}}\widetilde{G}(y,y)(\tilde{\kappa}(dy)-\kappa(dy))\Big)^{2}$.
\end{lemm}

\begin{proof}
Let $\Delta\widetilde{\mathcal{Y}}$ be the determinantal point process
defined in Section \ref{Ch7Sec1} (see Lemma \ref{Ch7Sec1: LemExtrDet}).
According to Lemma \ref{Ch7Sec1: LemCondEmpt}, the law of $(\widetilde{\mathcal{Y}},\widetilde{\mathcal{Z}})$, conditional on $\beta=0$, is the same as the law of $(\widetilde{\mathcal{Y}}_{\infty},\widetilde{\mathcal{Z}}_{\infty})$, conditional on $\Delta\widetilde{\mathcal{Y}}=\emptyset$. From Lemmas 
\ref{Ch7Sec2: LemAfterSubProc} and \ref{Ch7Sec1: LemCondOne} follows that the law of $(\widetilde{\mathcal{Y}},\widetilde{\mathcal{Z}})$, conditional on $\beta=1$, is the same as the law of $(\widetilde{\mathcal{Y}}_{\infty},\widetilde{\mathcal{Z}}_{\infty})$, conditional on $\#\Delta\widetilde{\mathcal{Y}}=1$. Moreover, $\mathbb{P}(\beta=1)=\mathbb{P}(\#\Delta\widetilde{\mathcal{Y}}=1)$. However,
\begin{displaymath}
\mathbb{P}(\beta=0)=\mathbb{P}(\Delta\widetilde{\mathcal{Y}}=\emptyset)+\mathbb{P}(\#\Delta\widetilde{\mathcal{Y}}\geq 2)\geq 
\mathbb{P}(\Delta\widetilde{\mathcal{Y}}=\emptyset).
\end{displaymath}
It follows that the total variation distance between the law of $(\widetilde{\mathcal{Y}},\widetilde{\mathcal{Z}})$ and the law of $(\widetilde{\mathcal{Y}}_{\infty},\widetilde{\mathcal{Z}}_{\infty})$ is less or equal to $2\mathbb{P}(\#\Delta\widetilde{\mathcal{Y}}\geq 2)$, which, according to Lemma \ref{Ch7Sec1: LemExtrDet}, is less or equal to $\Big(\int_{\mathbb{R}}\widetilde{G}(y,y)
(\tilde{\kappa}(dy)-\kappa(dy))\Big)^{2}$.
\end{proof}

\begin{coro}
\label{Ch7Sec2: CorApproxTV}
Let $\kappa_{0}\leq \kappa_{1}\leq\dots\leq \kappa_{j}$ be positive Radon measures such that 
$\int_{\mathbb{R}}\vert x\vert(\kappa_{j}(dx)-\kappa_{0}(dx))<+\infty$. Let $G_{i}$ be the Green's function of 
$\frac{1}{2}\frac{d^{2}}{dx^{2}}-\kappa_{i}$ and $(\mathcal{Y}^{(i)}_{\infty},\mathcal{Z}^{(i)}_{\infty})$ the pair of interwoven determinantal point processes corresponding to $\kappa_{i}$. Let $((\mathcal{Y}^{(i)},\mathcal{Z}^{(i)}))_{0\leq i\leq j}$ be the sequence of pairs of interwoven point processes defined as follows: $(\mathcal{Y}^{(0)},\mathcal{Z}^{(0)}):=
(\mathcal{Y}^{(0)}_{\infty},\mathcal{Z}^{(0)}_{\infty})$; given $(\mathcal{Y}^{(i-1)},\mathcal{Z}^{(i-1)})$, $(\mathcal{Y}^{(i)},\mathcal{Z}^{(i)})$ is obtained by applying Procedure 
\ref{Ch7Sec2: ProcAlmost} to the arguments $(\mathcal{Y}^{(i-1)},\mathcal{Z}^{(i-1)})$, $\kappa_{i-1}$ and $\kappa_{i}$. Then, the total variation distance between the law of $(\mathcal{Y}^{(j)},\mathcal{Z}^{(j)})$ and the law of $(\mathcal{Y}^{(j)}_{\infty},\mathcal{Z}^{(j)}_{\infty})$ is less or equal to
\begin{displaymath}
\sum_{i=1}^{j}\Big(\int_{\mathbb{R}}G_{i-1}(y,y)
(\kappa_{i}(dy)-\kappa_{i-1}(dy))\Big)^{2}.
\end{displaymath}
\end{coro}

\begin{proof}
Let $(\mathcal{Y}'^{(i)},\mathcal{Z}'^{(i)})$ be the pair of point processes obtained by applying Procedure \ref{Ch7Sec2: ProcAlmost} to the arguments 
$(\mathcal{Y}^{(i-1)}_{\infty},\mathcal{Z}^{(i-1)}_{\infty})$, $\kappa_{i-1}$ and $\kappa_{i}$.
According to Lemma \ref{Ch7Sec2: LemAlmost}, the total variation distance between the law of $(\mathcal{Y}'^{(i)},\mathcal{Z}'^{(i)})$ and the law of $(\mathcal{Y}^{(i)}_{\infty},\mathcal{Z}^{(i)}_{\infty})$ is less or equal to 
$\Big(\int_{\mathbb{R}}G_{i-1}(y,y)(\kappa_{i}(dy)-\kappa_{i-1}(dy))\Big)^{2}$. We denote by $d_{q}$ the total variation distance between the law of $(\mathcal{Y}^{(q)},\mathcal{Z}^{(q)})$ and the law of $(\mathcal{Y}^{(q)}_{\infty},\mathcal{Z}^{(q)}_{\infty})$. The total variation distance between the law of $(\mathcal{Y}'^{(i)},\mathcal{Z}'^{(i)})$ and the law of $(\mathcal{Y}^{(i)},\mathcal{Z}^{(i)})$ is less or equal to $d_{i-1}$. It follows that
\begin{displaymath}
d_{i}\leq d_{i-1}+\Big(\int_{\mathbb{R}}G_{i-1}(y,y)
(\kappa_{i}(dy)-\kappa_{i-1}(dy))\Big)^{2},
\end{displaymath}
and thus,
\begin{displaymath}
d_{j}\leq \sum_{i=1}^{j}\Big(\int_{\mathbb{R}}G_{i-1}(y,y)
(\kappa_{i}(dy)-\kappa_{i-1}(dy))\Big)^{2}.
\qedhere
\end{displaymath}
\end{proof}

Next, we give a true monotone coupling between 
$(\mathcal{Y}_{\infty},\mathcal{Z}_{\infty})$ and $(\widetilde{\mathcal{Y}}_{\infty},\widetilde{\mathcal{Z}}_{\infty})$. We still consider that $\kappa\leq \tilde{\kappa}$ and that $\int_{\mathbb{R}}\vert x\vert(\tilde{\kappa}(dx)-\kappa(dx))<+\infty$. To construct the coupling we will use a continuous monotonic increasing path in the space of measures, $(\kappa_{q})_{0\leq q\leq 1}$, joining $\kappa$ to $\tilde{\kappa}$ ($\kappa_{0}=\kappa$, $\kappa_{1}=\tilde{\kappa}$). Such a path is defined as follows: Let $\Lambda$ be a positive Radon measure on $\mathbb{R}\times [0,1]$ satisfying the following constraints:
\begin{itemize}
\item For any $q\in [0,1]$, 
$\Lambda(\mathbb{R}\times \lbrace q\rbrace)=0$.
\item For any $A$ Borel subset of $\mathbb{R}, \Lambda(A\times[0,1])=\tilde{\kappa}(A)$.
\end{itemize}
For $q\in[0,1]$, we define $\kappa_{q}$ as the measure on $\mathbb{R}$ satisfying, for any $A$ Borel subset of $\mathbb{R}$,
\begin{displaymath}
\kappa_{q}(A)=\kappa_{0}(A)+\Lambda(A\times [0,q]).
\end{displaymath}
For any $q\leq q'\in[0,1]$, $\kappa_{q}\leq \kappa_{q'}$. Moreover, the map $q\mapsto \kappa_{q}$ is continuous for the vague topology. In the sequel we will denote $G_{q}$ the Green's function of 
$\frac{1}{2}\frac{d^{2}}{dx^{2}}-\kappa_{q}$ (for $x\leq y$, $G_{q}(x,y)=u_{q,\uparrow}(x)u_{q,\downarrow}(y)$) and use the measure 
$G_{q}(y,y)\Lambda(dy,dq)$, which is finite.

\begin{proc}
\label{Ch7Sec2: ProcCoupling}
Arguments:
\begin{itemize}
\item a pair $(\mathcal{Y},\mathcal{Z})$ of disjoint discrete sets of points in $\mathbb{R}$ such that between any two points in $\mathcal{Y}$ lies a single point in $\mathcal{Z}$ and vice-versa, and such that $\inf \mathcal{Y}\cup\mathcal{Z} \in \mathcal{Y}\cup\lbrace -\infty\rbrace$, $\sup \mathcal{Y}\cup\mathcal{Z} \in \mathcal{Y}\cup\lbrace +\infty\rbrace$,
\item two positive Radon measures $\kappa$, $\tilde{\kappa}$, such that $\kappa\leq \tilde{\kappa}$ and 
$\int_{\mathbb{R}}\vert x\vert(\tilde{\kappa}(dx)-\kappa(dx))<+\infty$ and $(\tilde{\kappa}-\kappa)(\mathcal{Z})=0$,
\item a continuous monotonic increasing path in the space of measures, $(\kappa_{q})_{0\leq q\leq 1}$, joining $\kappa$ to $\tilde{\kappa}$, obtained by integrating the Radon measure $\Lambda$ on $\mathbb{R}\times[0,1]$.
\end{itemize}
Procedure:
\begin{itemize}
\item[(i)] First, sample a Poisson point process of intensity 
$G_{q}(y,y)\Lambda(dy,dq)$ on $\mathbb{R}\times[0,1]$: $((Y_{j},q_{j}))_{1\leq j\leq N}$, the points being ordered in the increasing sense of $q_{j}$.
\item[(ii)] Then, construct recursively the sequence $((\mathcal{Y}^{(j)},\mathcal{Z}^{(j)}))_{0\leq j\leq N}$ of pairs of interwoven point processes as follows: $(\mathcal{Y}^{(0)},\mathcal{Z}^{(0)})$ is set to be $(\mathcal{Y},\mathcal{Z})$. $(\mathcal{Y}^{(j)},\mathcal{Z}^{(j)})$ is obtained by applying Procedure \ref{Ch7Sec2: ProcSubProc} to the arguments $(\mathcal{Y}^{(j-1)},\mathcal{Z}^{(j-1)})$,
$\kappa_{q_{j}}$ and $Y_{j}$.
\item[(iii)] $(\widetilde{\mathcal{Y}},\widetilde{\mathcal{Z}})$ is set to be $(\mathcal{Y}^{(N)},\mathcal{Z}^{(N)})$.
\end{itemize}
Return: $(\widetilde{\mathcal{Y}},\widetilde{\mathcal{Z}})$.
\end{proc}

The condition $(\tilde{\kappa}-\kappa)(\mathcal{Z})=0$ ensures that a.s., none of $\mathcal{Y}^{(j)}$ lies in $\mathcal{Z}$. By construction, $\mathcal{Z}\subseteq\widetilde{\mathcal{Z}}$ and $\widetilde{\mathcal{Y}}\subseteq\mathcal{Y}\cup \operatorname{Supp}(\tilde{\kappa}-\kappa)$. $(\widetilde{\mathcal{Y}},\widetilde{\mathcal{Z}})$ differs from $(\mathcal{Y},\mathcal{Z})$ only by a finite number of points. The law of $(\widetilde{\mathcal{Y}},\widetilde{\mathcal{Z}})$ depends only on the "geometrical path" $(\kappa_{q})_{0\leq q\leq 1}$ and not on its parametrization: if $\vartheta$ is an increasing homeomorphism from $[0,1]$ to itself, then Procedure \ref{Ch7Sec2: ProcCoupling} applied the path 
$(\kappa_{\vartheta(q)})_{0\leq q\leq 1}$ returns the same result (in law). Below, an illustration of Procedure \ref{Ch7Sec2: ProcCoupling}:

\begin{figure}[H]
\centering{ 
\includegraphics[width=1\textwidth]{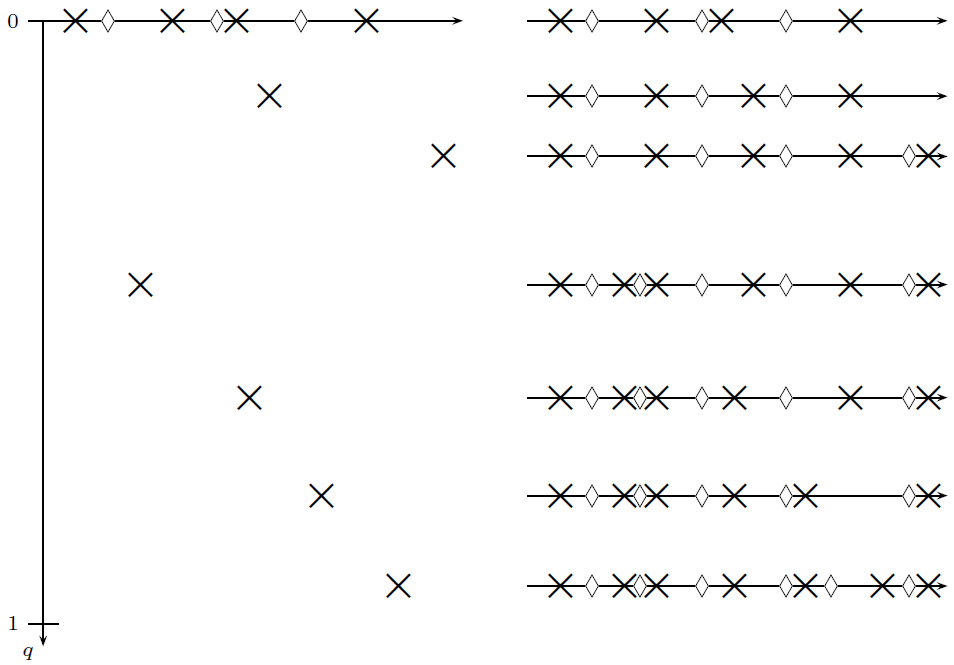}
}
\caption{Illustration of Procedure \ref{Ch7Sec2: ProcCoupling}: On the left are represented $(\mathcal{Y},\mathcal{Z})$ and
the Poisson process $((Y_{j},q_{j}))_{1\leq j\leq N}$. On the right are represented the successive
$((\mathcal{Y}^{(j)},\mathcal{Z}^{(j)}))_{0\leq j\leq N}$. x-dots represent the points of $\mathcal{Y}^{(j)}$ and diamonds the points of $\mathcal{Z}^{(j)}$.
}
\end{figure}

\begin{prop}
\label{Ch7Sec2: PropCouplingPoisson}
Let $(\mathcal{Y}_{\infty},\mathcal{Z}_{\infty})$, respectively $(\widetilde{\mathcal{Y}}_{\infty},\widetilde{\mathcal{Z}}_{\infty})$, be the couple of interwoven determinantal point processes corresponding to the killing measure $\kappa$, respectively $\tilde{\kappa}$. We assume that Procedure \ref{Ch7Sec2: ProcCoupling} is applied to $(\mathcal{Y}_{\infty},\mathcal{Z}_{\infty})$ and that $(\widetilde{\mathcal{Y}},\widetilde{\mathcal{Z}})$ is the returned couple of point processes. Then, $(\widetilde{\mathcal{Y}},\widetilde{\mathcal{Z}})$ has the same law as $(\widetilde{\mathcal{Y}}_{\infty},\widetilde{\mathcal{Z}}_{\infty})$.
\end{prop}

\begin{proof}
Observe that a.s., $(\tilde{\kappa}-\kappa)(\mathcal{Z}_{\infty})=0$.
Let $n\in\mathbb{N}^{\ast}$. We define the family $((\mathcal{Y}^{(j,n)},\mathcal{Z}^{(j,n)}))_{0\leq j\leq n}$ of interwoven point processes as follows: 
$(\mathcal{Y}^{(0,n)},\mathcal{Z}^{(0,n)})$ equals $(\mathcal{Y}_{\infty},\mathcal{Z}_{\infty})$. Given $(\mathcal{Y}^{(j-1,n)},\mathcal{Z}^{(j-1,n)})$, $(\mathcal{Y}^{(j,n)},\mathcal{Z}^{(j,n)})$ is obtained by applying Procedure 
\ref{Ch7Sec2: ProcAlmost} to the arguments $(\mathcal{Y}^{(j-1,n)},\mathcal{Z}^{(j-1,n)})$, $\kappa_{\frac{j-1}{n}}$ and 
$\kappa_{\frac{j}{n}}$. We will show that, as $n$ tends to infinity,
the law of $(\mathcal{Y}^{(n,n)},\mathcal{Z}^{(n,n)})$ converges in total variation to the law of $(\widetilde{\mathcal{Y}}_{\infty},\widetilde{\mathcal{Z}}_{\infty})$ and converges weakly to the law of $(\widetilde{\mathcal{Y}},\widetilde{\mathcal{Z}})$, which will imply that $(\widetilde{\mathcal{Y}},\widetilde{\mathcal{Z}})$ and $(\widetilde{\mathcal{Y}}_{\infty},\widetilde{\mathcal{Z}}_{\infty})$ have the same law.

Applying Corollary \ref{Ch7Sec2: CorApproxTV}, we get that the total variation distance between the law of 
$(\mathcal{Y}^{(n,n)},\mathcal{Z}^{(n,n)})$ and the law of $(\widetilde{\mathcal{Y}}_{\infty},\widetilde{\mathcal{Z}}_{\infty})$ is bounded by
\begin{equation*}
\begin{split}
\sum_{j=1}^{n}\Big(\int_{\mathbb{R}}G_{\frac{j-1}{n}}(y,y)
(\kappa_{\frac{j}{n}}&(dy)-\kappa_{\frac{j-1}{n}}(dy))\Big)^{2}\\\leq &
\sup_{x\in\mathbb{R}}\left(\dfrac{G_{0}(x,x)}{1+\vert x\vert}\right)^{2}
\sum_{j=1}^{n}\Big(\int_{\mathbb{R}}(1+\vert y\vert)
(\kappa_{\frac{j}{n}}(dy)-\kappa_{\frac{j-1}{n}}(dy))\Big)^{2}\\\leq &
\sup_{x\in\mathbb{R}}\left(\dfrac{G_{0}(x,x)}{1+\vert x\vert}\right)^{2}
\int_{\mathbb{R}}(1+\vert y\vert)(\tilde{\kappa}(dy)-\kappa(dy))
\\&\times\sup_{1\leq j\leq n}\int_{\mathbb{R}}(1+\vert y\vert)
(\kappa_{\frac{j}{n}}(dy)-\kappa_{\frac{j-1}{n}}(dy)).
\end{split}
\end{equation*}
The continuity of the path $(\kappa_{q})_{0\leq q\leq 1}$ ensures that
\begin{displaymath}
\lim_{n\rightarrow +\infty}\sup_{1\leq j\leq n}\int_{\mathbb{R}}
(1+\vert y\vert)(\kappa_{\frac{j}{n}}(dy)-\kappa_{\frac{j-1}{n}}(dy))=0,
\end{displaymath}
and hence, the total variation distance between the law of $(\mathcal{Y}^{(n,n)},\mathcal{Z}^{(n,n)})$ and the law of $(\widetilde{\mathcal{Y}}_{\infty},\widetilde{\mathcal{Z}}_{\infty})$ converges to $0$ as $n$ tends to infinity.

We define a random finite set $E_{n}$ of points in $\mathbb{R}\times\big\lbrace \frac{1}{n},\frac{2}{n},\dots,\frac{n}{n}\big\rbrace$ as follows: Let $(\beta_{1,n},\beta_{2,n},\dots,\beta_{n,n})$ be a family of independent Bernoulli variables, $\beta_{i,n}$ being of parameter
\begin{displaymath}
\int_{\mathbb{R}}v_{\kappa_{\frac{i-1}{n}},\kappa_{\frac{i}{n}}}(y)(\kappa_{\frac{i}{n}}-\kappa_{\frac{i-1}{n}})(dy).
\end{displaymath}
Whenever $\beta_{i,n}=1$, we add to $E_{n}$ a point 
$(Y_{i,n},\frac{i-1}{n})$ to $E_{n}$, where $Y_{i,n}$ is a r.v. distributed according the measure
\begin{displaymath}
\dfrac{1}{\mathbb{P}(\beta_{i,n}=1)} 
v_{\kappa_{\frac{i-1}{n}},\kappa_{\frac{i}{n}}}(y)
(\kappa_{\frac{i}{n}}-\kappa_{\frac{i-1}{n}})(dy).
\end{displaymath} 
The $(Y_{i,n},\frac{i-1}{n})$ are assumed to be independent and independent from the family
$(\beta_{1,n},\beta_{2,n},\dots,\beta_{n,n})$. 
The pair $(\mathcal{Y}^{(n,n)},\mathcal{Z}^{(n,n)})$ is sampled as follows: starting from $(\mathcal{Y}_{\infty},\mathcal{Z}_{\infty})$, independent from $E_{n}$, we apply successively, for $i$ ranging from $1$ to $n$, Procedure \ref{Ch7Sec2: ProcSubProc} with the arguments $\kappa_{\frac{i-1}{n}}$, and $Y_{i,n}$ whenever $\beta_{i,n}=1$. At the end, we get $(\mathcal{Y}^{(n,n)},\mathcal{Z}^{(n,n)})$. According to Lemma \ref{Ch2Sec1: LemContKappa}, the law of the pair of point processes returned by Procedure \ref{Ch7Sec2: ProcSubProc} depends continuously on the arguments. So, to prove that 
$(\mathcal{Y}^{(n,n)},\mathcal{Z}^{(n,n)})$ converges in law to $(\widetilde{\mathcal{Y}},\widetilde{\mathcal{Z}})$, we only need to show that the random set of point $E_{n}$ converges in law to the Poisson point process $((Y_{j},q_{j}))_{1\leq j\leq N}$ used in Procedure \ref{Ch7Sec2: ProcCoupling}. All of the functions 
$v_{\kappa_{\frac{i-1}{n}},\kappa_{\frac{i}{n}}}(y)$ are dominated by $G_{0}(y,y)$. Moreover,
\begin{equation*}
\begin{split}
\big\vert v_{\kappa_{\frac{i-1}{n}},\kappa_{\frac{i}{n}}}(y)
&-G_{\frac{i}{n}}(y,y)\big\vert\\
\leq&u_{\frac{i}{n},\downarrow}(y)\int_{y_{-1}<y}
u_{\frac{i}{n},\uparrow}(y_{-1})(u_{\frac{i-1}{n},\downarrow}(y_{-1})
u_{\frac{i-1}{n},\uparrow}(y)-u_{\frac{i-1}{n},\uparrow}(y_{-1})
u_{\frac{i-1}{n},\downarrow}(y))\\
&\times(\kappa_{\frac{i}{n}}-\kappa_{\frac{i-1}{n}})(dy_{-1})\\
+&u_{\frac{i}{n},\uparrow}(y)\int_{y_{1}>y}u_{\frac{i}{n},\downarrow}(y_{1})(u_{\frac{i-1}{n},\uparrow}(y_{1})u_{\frac{i-1}{n},\downarrow}(y)-u_{\frac{i-1}{n},\downarrow}(y_{1})u_{\frac{i-1}{n},\uparrow}(y))
\\&\times(\kappa_{\frac{i}{n}}-\kappa_{\frac{i-1}{n}})(dy_{1})
\\+&\int_{y_{-1}<y}u_{\frac{i}{n},\uparrow}(y_{-1})
(u_{\frac{i-1}{n},\downarrow}(y_{-1})u_{\frac{i-1}{n},\uparrow}(y)-u_{\frac{i-1}{n},\uparrow}(y_{-1})u_{\frac{i-1}{n},\downarrow}(y))
\\&\times(\kappa_{\frac{i}{n}}-\kappa_{\frac{i-1}{n}})(dy_{-1})
\\\times&\int_{y_{1}>y}u_{\frac{i}{n},\downarrow}(y_{1})
(u_{\frac{i-1}{n},\uparrow}(y_{1})u_{\frac{i-1}{n},\downarrow}(y)-u_{\frac{i-1}{n},\downarrow}(y_{1})u_{\frac{i-1}{n},\uparrow}(y))
\\&\times(\kappa_{\frac{i}{n}}-\kappa_{\frac{i-1}{n}})(dy_{1})
\\\leq & G_{0}(y,y)\int_{y_{-1}<y}G_{0}(y_{-1},y_{-1})
(\kappa_{\frac{i}{n}}-\kappa_{\frac{i-1}{n}})(dy_{-1})
\\+&G_{0}(y,y)\int_{y_{1}>y}G_{0}(y_{1},y_{1})
(\kappa_{\frac{i}{n}}-\kappa_{\frac{i-1}{n}})(dy_{1})\\+&
G_{0}(y,y)\int_{y_{-1}<y}G_{0}(y_{-1},y_{-1})
(\kappa_{\frac{i}{n}}-\kappa_{\frac{i-1}{n}})(dy_{-1})
\\&\times\int_{y_{1}>y}G_{0}(y_{1},y_{1})
(\kappa_{\frac{i}{n}}-\kappa_{\frac{i-1}{n}})(dy_{1}).
\end{split}
\end{equation*}
Thus, given any bounded interval $J$,
\begin{displaymath}
\lim_{n\rightarrow +\infty}\sup_{1\leq i\leq n}\sup_{y\in J}\big\vert v_{\kappa_{\frac{i-1}{n}},\kappa_{\frac{i}{n}}}(y)
-G_{\frac{i}{n}}(y,y)\big\vert =0.
\end{displaymath}
It follows that
\begin{displaymath}
\lim_{n\rightarrow +\infty}\sup_{1\leq i\leq n}
\mathbb{P}(\beta_{i,n}=1)=0,
\end{displaymath}
and the measure
\begin{displaymath}
\sum_{i=1}^{n}v_{\kappa_{\frac{i-1}{n}},\kappa_{\frac{i}{n}}}(y)(\kappa_{\frac{i}{n}}-\kappa_{\frac{i-1}{n}})(dy)\otimes
\delta_{\frac{i}{n}}(dq)
\end{displaymath}
converges weekly to $G_{q}(y,y)\Lambda(dy,dq)$, which is the intensity  of the Poisson point process $((Y_{j},q_{j}))_{1\leq j\leq N}$. Thus, the random sets $E_{n}$ are compound Bernoulli approximations of the Poisson point process $((Y_{j},q_{j}))_{1\leq j\leq N}$ and converge in law to the latter.
\end{proof}

Given a continuous monotonic increasing path 
$(\kappa_{q})_{0\leq q\leq 1}$ in the space of Radon measures, and a pair of interwoven determinantal point processes $(\mathcal{Y}_{\infty},\mathcal{Z}_{\infty})$ corresponding to $\kappa_{0}$, used as argument, Procedure \ref{Ch7Sec2: ProcCoupling} yields non-homogeneous Markov $q$-parametrized process in the space of interwoven pairs of discrete sets of points, whose one-dimensional marginal at any value $q_{0}$ of the parameter is the pair of interwoven determinantal point processes corresponding to the killing measure $\kappa_{q_{0}}$. This corresponds to sampling only the partial Poisson point process of intensity 
$1_{0\leq q\leq q_{0}}G_{q}(y,y)\Lambda(dy,dq)$ and successively applying Procedure \ref{Ch7Sec2: ProcSubProc} for each of its points. In general, multidimensional marginals corresponding to $q_{1}<\dots<q_{n}$ depend not only on $\kappa_{q_{1}},\dots,\kappa_{q_{n}}$, but on the whole path $(\kappa_{q})_{q_{1}\leq q\leq q_{n}}$. For instance, consider two different paths $(\kappa_{q})_{0\leq q\leq 1}$ and 
$(\hat{\kappa}_{q})_{0\leq q\leq 1}$, where
\begin{itemize}
\item $\kappa_{0}=\hat{\kappa}_{0}=\delta_{-\frac{1}{2}}
+\delta_{\frac{1}{2}}$,
\item $\kappa_{1}=\hat{\kappa}_{1}=\delta_{-\frac{3}{2}}
+\delta_{-\frac{1}{2}}+\delta_{\frac{1}{2}}+\delta_{\frac{3}{2}}$,
\item $\kappa_{q}=2q\delta_{-\frac{1}{2}}+\delta_{-\frac{1}{2}}
+\delta_{\frac{1}{2}}$ for $q\in\big[0,\frac{1}{2}\big],$ and 
$\kappa_{q}=\delta_{-\frac{1}{2}}+\delta_{-\frac{1}{2}}+
\delta_{\frac{1}{2}}+(2q-1)\delta_{\frac{3}{2}}$ for 
$q\in\big[\frac{1}{2},1\big]$,
\item $\hat{\kappa}_{q}=\delta_{-\frac{1}{2}}+\delta_{\frac{1}{2}}
+2q\delta_{\frac{1}{2}}$ for $q\in\big[0,\frac{1}{2}\big]$, and 
$\hat{\kappa}_{q}=(2q-1)\delta_{-\frac{1}{2}}+\delta_{-\frac{1}{2}}
+\delta_{\frac{1}{2}}+\delta_{\frac{3}{2}}$ for 
$q\in\big[\frac{1}{2},1\big]$.
\end{itemize}
Let $G_{q}(x,y)=u_{q,\uparrow}(x\wedge y)u_{q,\downarrow}(x\vee y)$ be the Green's function of $\frac{1}{2}\frac{d^{2}}{dx^{2}}-\kappa_{q}$, and 
$\widehat{G}_{q}(x,y)={\hat{u}_{q,\uparrow}(x\wedge y)}\hat{u}_{q,\downarrow}(x\vee y)$ the Green's function of $\frac{1}{2}\frac{d^{2}}{dx^{2}}-\hat{\kappa}_{q}$. Let 
$((\mathcal{Y}_{\infty},\mathcal{Z}_{\infty})$, 
$(\widetilde{\mathcal{Y}}_{\infty},\widetilde{\mathcal{Z}}_{\infty}))$ be the coupling between the point process corresponding to $\kappa_{0}$, respectively $\kappa_{1}$,
induced by the path $(\kappa_{q})_{0\leq q\leq 1}$, and 
$((\mathcal{Y}_{\infty},\mathcal{Z}_{\infty}),(
\widehat{\mathcal{Y}}_{\infty},\widehat{\mathcal{Z}}_{\infty}))$ the coupling induced by the path $(\hat{\kappa}_{q})_{0\leq q\leq 1}$. Then,
\begin{multline*}
\mathbb{P}\Big(\mathcal{Y}_{\infty}=\Big\lbrace -\frac{1}{2},\frac{1}{2}\Big\rbrace,
\widetilde{\mathcal{Y}}_{\infty}=\Big\lbrace-\frac{3}{2},\frac{1}{2},\frac{3}{2}\Big\rbrace\Big)\\=
\mathbb{P}\Big(\mathcal{Y}_{\infty}=\Big\lbrace -\frac{1}{2},\frac{1}{2}\Big\rbrace\Big)
\times G_{\frac{1}{2}}\Big(-\frac{3}{2}\Big)\dfrac{u_{\frac{1}{2},\downarrow}\big(-\frac{1}{2}\big)}{u_{\frac{1}{2},\downarrow}\big(-\frac{3}{2}\big)}G_{1}\Big(\frac{3}{2}\Big)
\bigg(1-\dfrac{u_{1,\uparrow}\big(\frac{1}{2}\big)}{u_{1,\uparrow}\big(\frac{3}{2}\big)}\bigg),
\end{multline*}
\begin{multline*}
\mathbb{P}\Big(\mathcal{Y}_{\infty}=\Big\lbrace -
\frac{1}{2},\frac{1}{2}\Big\rbrace,
\widehat{\mathcal{Y}}_{\infty}=\Big\lbrace-
\frac{3}{2},\frac{1}{2},\frac{3}{2}\Big\rbrace\Big)\\=
\mathbb{P}\Big(\mathcal{Y}_{\infty}=\Big\lbrace 
-\frac{1}{2},\frac{1}{2}\Big\rbrace\Big)
\times \widehat{G}_{\frac{1}{2}}\Big(\frac{3}{2}\Big)\,\bigg(1-\dfrac{\hat{u}_{\frac{1}{2},\uparrow}\big(\frac{1}{2}\big)}{\hat{u}_{\frac{1}{2},\uparrow}\big(\frac{3}{2}\big)}\bigg)\,\widehat{G}_{1}\Big(-\frac{3}{2}\Big)
\dfrac{\hat{u}_{1,\downarrow}\big(-\frac{1}{2}\big)}
{\hat{u}_{1,\downarrow}\big(-\frac{3}{2}\big)}.
\end{multline*}
But,
\begin{displaymath}
\widehat{G}_{\frac{1}{2}}\big(\frac{3}{2}\big)=
G_{\frac{1}{2}}\big(-\frac{3}{2}\big)
\qquad\widehat{G}_{1}\big(-\frac{3}{2}\big)=G_{1}\big(\frac{3}{2}\big),
\end{displaymath}
and
\begin{displaymath}
\dfrac{\hat{u}_{\frac{1}{2},\uparrow}\big(\frac{1}{2}\big)}{\hat{u}_{\frac{1}{2},\uparrow}\big(\frac{3}{2}\big)}=\dfrac{u_{\frac{1}{2},\downarrow}\big(-\frac{1}{2}\big)}{u_{\frac{1}{2},\downarrow}\big(-\frac{3}{2}\big)}\qquad
\dfrac{\hat{u}_{1,\downarrow}\big(-\frac{1}{2}\big)}
{\hat{u}_{1,\downarrow}\big(-\frac{3}{2}\big)}=
\dfrac{u_{1,\uparrow}\big(\frac{1}{2}\big)}{u_{1,\uparrow}\big(\frac{3}{2}\big)}.
\end{displaymath}
Thus,
\begin{displaymath}
\dfrac{\mathbb{P}\Big(\mathcal{Y}_{\infty}=\Big\lbrace -\frac{1}{2},\frac{1}{2}\Big\rbrace,\widetilde{\mathcal{Y}}_{\infty}=\Big\lbrace-\frac{3}{2},\frac{1}{2},\frac{3}{2}\Big\rbrace\Big)}
{\mathbb{P}\Big(\mathcal{Y}_{\infty}=\Big\lbrace -\frac{1}{2},\frac{1}{2}\Big\rbrace,
\widehat{\mathcal{Y}}_{\infty}=\Big\lbrace-\frac{3}{2},\frac{1}{2},\frac{3}{2}\Big\rbrace\Big)}=\dfrac{\dfrac{u_{\frac{1}{2},\downarrow}\big(-\frac{1}{2}\big)}{u_{\frac{1}{2},\downarrow}\big(-\frac{3}{2}\big)}}{1-\dfrac{u_{\frac{1}{2},\downarrow}\big(-\frac{1}{2}\big)}{u_{\frac{1}{2},\downarrow}\big(-\frac{3}{2}\big)}}\times
\dfrac{1-\dfrac{u_{1,\uparrow}\big(\frac{1}{2}\big)}{u_{1,\uparrow}\big(\frac{3}{2}\big)}}{\dfrac{u_{1,\uparrow}\big(\frac{1}{2}\big)}{u_{1,\uparrow}\big(\frac{3}{2}\big)}}.
\end{displaymath}
But,
\begin{displaymath}
\dfrac{u_{\frac{1}{2},\downarrow}\big(-\frac{1}{2}\big)}{u_{\frac{1}{2},\downarrow}\big(-\frac{3}{2}\big)}=\dfrac{3}{11}~~~~~~
\dfrac{u_{1,\uparrow}\big(\frac{1}{2}\big)}{u_{1,\uparrow}\big(\frac{3}{2}\big)}=\dfrac{11}{41}.
\end{displaymath}
Thus,
\begin{displaymath}
\dfrac{\mathbb{P}\Big(\mathcal{Y}_{\infty}=\Big\lbrace -\frac{1}{2},\frac{1}{2}\Big\rbrace,\widetilde{\mathcal{Y}}_{\infty}=\Big\lbrace-\frac{3}{2},\frac{1}{2},\frac{3}{2}\Big\rbrace\Big)}
{\mathbb{P}\Big(\mathcal{Y}_{\infty}=\Big\lbrace -\frac{1}{2},\frac{1}{2}\Big\rbrace,
\widehat{\mathcal{Y}}_{\infty}=\Big\lbrace-\frac{3}{2},\frac{1}{2},\frac{3}{2}\Big\rbrace\Big)}=\dfrac{45}{44}\neq 1.
\end{displaymath}
The two couplings are different.

If $\tilde{\kappa}-\kappa$ does not have a first moment,
we can still construct a coupling between $(\mathcal{Y}_{\infty},\mathcal{Z}_{\infty})$ and $(\widetilde{\mathcal{Y}}_{\infty},\widetilde{\mathcal{Z}}_{\infty})$ as follows: Consider a continuous monotonic increasing path $(\kappa_{q})_{0\leq q\leq 1}$ joining $\kappa$ to $\tilde{\kappa}$, satisfying the constraint
\begin{displaymath}
\forall q\in[0,1), \int_{\mathbb{R}}\vert x\vert
(\kappa_{q}(dx)-\kappa_{0}(dx))<+\infty.
\end{displaymath}
Given $q_{0}\in (0,1)$, one can apply Procedure \ref{Ch7Sec2: ProcCoupling} to the arguments $(\mathcal{Y}_{\infty},\mathcal{Z}_{\infty})$,$\kappa$, $\kappa_{q_{0}}$, and the partial path $(\kappa_{q})_{0\leq q\leq q_{0}}$. As result, we get a two interwoven determinantal point processes corresponding to the killing measure $\kappa_{q_{0}}$. At the limit, as $q_{0}$ tends to $1$, we get something that has the same law as 
$(\widetilde{\mathcal{Y}}_{\infty},\widetilde{\mathcal{Z}}_{\infty})$.

Next, we prove the existence of stronger couplings in case $\tilde{\kappa}=c\kappa$, where $c>1$ is a constant.

\begin{prop}
\label{Ch7Sec2: PropStrongCoup}
If $\tilde{\kappa}=c\kappa$, with $c>1$, then there is a coupling between $(\mathcal{Y}_{\infty},\mathcal{Z}_{\infty})$ and $(\widetilde{\mathcal{Y}}_{\infty},\widetilde{\mathcal{Z}}_{\infty})$, such that $\mathcal{Z}_{\infty}\subseteq \widetilde{\mathcal{Z}}_{\infty}$, and $\mathcal{Y}_{\infty}\subseteq \widetilde{\mathcal{Y}}_{\infty}$.
\end{prop}

\begin{proof}
Consider a coupling between $(\mathcal{Y}_{\infty},\mathcal{Z}_{\infty})$ and $(\widetilde{\mathcal{Y}}_{\infty},\widetilde{\mathcal{Z}}_{\infty})$ given by Procedure \ref{Ch7Sec2: ProcCoupling}, possibly extended to the case where $\kappa$ does not have a first moment. Then, $\mathcal{Z}_{\infty}\subseteq \widetilde{\mathcal{Z}}_{\infty}$, but in general, 
$\mathcal{Y}_{\infty}\not\subseteq\widetilde{\mathcal{Y}}_{\infty}$. So, we will sample other point processes $\mathcal{Y}'_{\infty}$ and $\widetilde{\mathcal{Y}}'_{\infty}$ that conditional on $\mathcal{Z}_{\infty}$, respectively $\widetilde{\mathcal{Z}}_{\infty}$, have the same law as $\mathcal{Y}_{\infty}$ respectively $\widetilde{\mathcal{Y}}_{\infty}$, and such that $\mathcal{Y}'_{\infty}\subseteq\widetilde{\mathcal{Y}}'_{\infty}$. 
For each connected component $\widetilde{J}$ of $\mathbb{R}\setminus\widetilde{\mathcal{Z}}_{\infty}$, we sample a point $\widetilde{Y}_{\widetilde{J}}$ according the measure $\frac{1_{y\in\widetilde{J}}\tilde{\kappa}(dy)}
{\tilde{\kappa}(\widetilde{J})}$. We assume that, conditional on $\widetilde{\mathcal{Z}}_{\infty}$, all the $\widetilde{Y}_{\widetilde{J}}$ are independent from $\mathcal{Z}_{\infty}$ and independent one from another. We set
\begin{displaymath}
\widetilde{\mathcal{Y}}'_{\infty}:=\lbrace \widetilde{Y}_{\widetilde{J}}\vert \widetilde{J}~\text{connected component of}~\mathbb{R}\setminus
\widetilde{\mathcal{Z}}_{\infty}\rbrace.
\end{displaymath}
Then, $(\widetilde{\mathcal{Y}}'_{\infty},\widetilde{\mathcal{Z}}_{\infty})$ has the same law as $(\widetilde{\mathcal{Y}}_{\infty},\widetilde{\mathcal{Z}}_{\infty})$.
Let be $J$ a connected component of $\mathbb{R}\setminus\mathcal{Z}_{\infty}$, and $\widetilde{J}_{1},\dots,\widetilde{J}_{N_{J}}$ the connected components of $J\setminus\widetilde{\mathcal{Z}}_{\infty}$. On $J$, we define the r.v. $Y_{J}$ as follows: $Y_{J}$ is equal to one of the $\widetilde{Y}_{\widetilde{J}_{n}}$-s, and 
\begin{displaymath}
\mathbb{P}\big(Y_{J}=\widetilde{Y}_{\widetilde{J}_{n}}\vert J,\widetilde{J}_{1},\dots,\widetilde{J}_{N_{J}}\big)=\dfrac{\kappa(\widetilde{J}_{n})}{\kappa(J)}.
\end{displaymath}
We set
\begin{displaymath}
\mathcal{Y}'_{\infty}:=\lbrace Y_{J}\vert J~\text{connected component of}~\mathbb{R}\setminus
\mathcal{Z}_{\infty}\rbrace.
\end{displaymath}
By construction, $\mathcal{Y}'_{\infty}\subseteq\widetilde{\mathcal{Y}}'_{\infty}$. Moreover, the proportionality of $\kappa$ and $\tilde{\kappa}$ ensures that $(\mathcal{Y}'_{\infty},\mathcal{Z}_{\infty})$ has the same law as $(\mathcal{Y}_{\infty},\mathcal{Z}_{\infty})$.
\end{proof}

\backmatter

\bibliographystyle{smfplain}
\bibliography{tituslupu}

\end{document}